\documentclass[reqno,12pt]{gen-j-l}
\usepackage{amssymb,amsmath,amsbsy,amsthm,esint,setspace,color}
\usepackage{hyperref}
\usepackage{amsrefs}
\usepackage{enumitem}
\usepackage{graphicx}
\usepackage{caption}
\usepackage[normalem]{ulem}

\usepackage{xcolor}

\usepackage[utf8]{inputenc}

\usepackage{mathtools}

\graphicspath{ {images/} }

\newtheorem{theorem}{Theorem}[section]
\newtheorem*{theorem*}{Theorem}
\newtheorem{lemma}[theorem]{Lemma}
\newtheorem{proposition}[theorem]{Proposition}
\newtheorem{corollary}[theorem]{Corollary}
\newtheorem{definition}{Definition}[section]

\newtheorem{remark}{Remark}[section]

\numberwithin{equation}{section}

\let\oldtocsection=\tocsection
 
\let\oldtocsubsection=\tocsubsection
 
\let\oldtocsubsubsection=\tocsubsubsection
 
\renewcommand{\tocsection}[2]{\hspace{0em}\oldtocsection{#1}{#2}}
\renewcommand{\tocsubsection}[2]{\hspace{1em}\oldtocsubsection{#1}{#2}}
\renewcommand{\tocsubsubsection}[2]{\hspace{2em}\oldtocsubsubsection{#1}{#2}}

\begin{document}

\title[On an $L^2$ critical Boltzmann equation]{On an $L^2$ critical Boltzmann equation}
\author{Thomas Chen}
\address{T. Chen,  
Department of Mathematics, University of Texas at Austin.}
\email{tc@math.utexas.edu}
\author{Ryan Denlinger}
\address{R. Denlinger,
School of Biomedical Informatics,
University of Texas Health Science Center at Houston.}
\email{Ryan.A.Denlinger@uth.tmc.edu}
\author{Nata\v{s}a Pavlovi\'c}
\address{N. Pavlovi\'c,  
Department of Mathematics, University of Texas at Austin.}
\email{natasa@math.utexas.edu}

\begin{abstract}
We prove the existence of a class of large global scattering solutions of Boltzmann's equation with constant collision
kernel in two dimensions. These solutions are found for $L^2$ perturbations of an underlying
 initial data which is Gaussian jointly in
space and velocity. Additionally, the perturbation is required to satisfy natural physical
constraints for the total mass and second moments, corresponding to conserved or controlled quantities.
The space $L^2$ is a scaling critical space for the equation under consideration. If the initial
data is Schwartz then the solution is unique and again Schwartz on any bounded time interval.
\end{abstract}

\maketitle

\tableofcontents

\section{Introduction}

We consider the Boltzmann equation posed for a non-negative function $f \left( t,x,v \right)$, $t \in \mathbb{R}$, $x,v \in \mathbb{R}^2$, so that
\[
f : \left[0,T\right) \times \mathbb{R}^2_x \times \mathbb{R}^2_v \rightarrow
\mathbb{R}
\]
the collision kernel being constant.
Thus
\begin{equation}
\label{eq:BE}
\left( \partial_t + v \cdot \nabla_x \right) f = Q^+ \left(f,f\right) - Q^- \left(f,f\right)
\end{equation}
where we have the \emph{gain term}
\[
Q^+ \left(g,h\right) = \frac{1}{2\pi}
\int_{\mathbb{R}^2_v \times \mathbb{S}^1}
g^\prime h_*^\prime dv_* d\sigma
\]
with $f_* = f\left(v_*\right)$, $f^\prime = f\left(v^\prime\right)$,
$f_*^\prime = f \left(v_*^\prime\right)$ and
\[
\begin{aligned}
v^\prime & = \frac{v+v_*}{2} + \frac{\left|v-v_*\right|}{2}\sigma \\
v_*^\prime & = \frac{v+v_*}{2} - \frac{\left|v-v_*\right|}{2}\sigma
\end{aligned}
\]
the collisional change of variables defined for unit vectors $\sigma \in \mathbb{S}^1 \subset \mathbb{R}^2$.
The \emph{loss term} is written
\[
Q^- \left(g,h\right) = g \rho_h
\]
where 
\[
\rho_f \left(t,x\right) = \int_{\mathbb{R}^2_v} f\left(t,x,v\right) dv
\]
is the \emph{spatial density}, a quantity of direct interest in the study of
hydrodynamic limits of (\ref{eq:BE}).
More generally, the operators $Q^{\pm}$ may be replaced by $Q_b^{\pm}$ where
$b = b \left( u, \sigma \right)$ is the \emph{collision kernel}
\[
b : \mathbb{R}^2 \times \mathbb{S}^1 \rightarrow \mathbb{R}
\]
everywhere non-negative, and \emph{locally integrable} (referred to as the \emph{Grad cutoff}),
in particular being integrable in $\sigma$ for almost every $v$,
and
\begin{equation}
\label{eq:defQbPlus}
Q^+_b \left( g,h \right) =
\int_{\mathbb{R}^2_v \times \mathbb{S}^1_\sigma}
b g^\prime h_*^\prime dv_* d\sigma
\end{equation}
\begin{equation}
\label{eq:defQbMinus}
Q^-_b \left( g,h \right) =
\int_{\mathbb{R}^2_v \times \mathbb{S}^1_\sigma} b
g h_* dv_* d\sigma
\end{equation}
where the notation $b$ in $Q^{\pm}_b$ implicitly denotes
\[
b \equiv \tilde{b} \left( \left| v-v_* \right| \; , \; \sigma \cdot \frac{ v-v_*}
{\left| v-v_* \right|}\right)
\]
the dependence on the first argument being only of a radial nature. Clearly
the equation of interest (\ref{eq:BE}) in this paper corresponds to the
choice $b = \left( 2\pi \right)^{-1}$. The choice $b = \left| v-v_*\right|$ is known
as \emph{hard spheres}, and arises physically from a Newtonian (deterministic) ``gas'' of
hard sphere billiards via the
so-called \emph{Boltzmann-Grad limit}, first established rigorously by Lanford. \cite{L1975}

Although Boltzmann's equation is typically viewed as a dissipative
equation, following Arsenio \cite{Ar2011} we choose to view it as a dispersive
equation instead. Homogeneous Strichartz estimates  for kinetic equations
have been known since Castella and Perthame \cite{CaP1996}; in the same reference, some
 inhomogeneous Strichartz estimates were also proven. The complete set
of inhomogeneous kinetic Strichartz estimates was obtained by Ovcharov,
\cite{ovcharov2011strichartz}. The failure of endpoint homogeneous kinetic Strichartz estimates
was established by Bennett et al., \cite{Bennettetal2014}.

The novelty of Arsenio's contribution was to demonstrate, for the first time, the possibility of
applying the standard techniques of inhomogeneous Strichartz estimates, well-known from dispersive
theory, directly to Boltzmann's equation, under some highly restrictive assumptions for the collision kernel.
An alternative approach, avoiding the inhomogeneous Strichartz estimates
entirely, has been developed by the present authors, \cite{CDP2017,CDP2018,CDP2019}.
The approach, originating in works by Klainerman and Machedon, e.g. \cite{KM1993,KM2008},
and later extended by Pavlovi{\' c} and Chen, e.g. \cite{PC2010},
is based on a method of multilinear Strichartz estimates, 
and has seen substantial developments in various directions in recent years. 
Much of the more recent work motivated by the results of Klainerman and Machedon has been towards
alternative methods for rigorously deriving nonlinear Schr{\" o}dinger equations from quantum mechanical models of
many particle systems (e.g. Bose-Einstein condensation); work in this direction was pioneered by 
Erd{\" o}s, Schlein and Yau by other techniques, \cite{ESY2006,ESY2007,EY2001}.

In fact, following in the direction set forth by Klainerman and Machedon,
 a scaling-critical bilinear Strichartz estimate for Boltzmann's gain operator $Q^+$ has been proven in
 \cite{CDP2019}
using an endpoint homogeneous Strichartz estimate of Keel and Tao \cite{KT1998}; note that this bilinear Strichartz estimate
 was not subject to the negative results of Bennett et al.
\cite{Bennettetal2014} because its proof actually relied upon an endpoint homogeneous Strichartz estimate for a
\emph{hyperbolic Schr{\" o}dinger equation in dimension four}. Any
kinetic equation in dimension two is formally equivalent to a hyperbolic Schr{\" o}dinger equation in dimension four by
the Wigner transform; on the other hand, endpoint homogeneous Strichartz estimates
are true for the hyperbolic Schr{\"o}dinger equation in  dimension four.
Combining this dispersive estimate with a convolutive bound for $Q^+$ \emph{on the Fourier side}, and ultimately moving back to
the kinetic domain, it was possible to prove the bilinear $Q^+$ estimate, a quite unexpected outcome.

\subsection{A new notion of solution.}

The main new technical tool (and a main novelty) of the present article is the introduction of a new
class of solutions to (\ref{eq:BE}), which we refer to as ($*$)-solutions. This is a class of
global renormalized solutions (in the sense of DiPerna and Lions, \cite{DPL1989,DPL1991})
 which satisfy better bounds on some initial interval
$\left[ 0, T^* \left( f \right) \right)$. In \cite{L1994II}, Lions established a
 \emph{weak-strong uniqueness theorem}
 in a class of \emph{dissipative solutions} which allowed (for e.g. Schwartz initial data)
  the construction  of
global renormalized solutions to (\ref{eq:BE}) which are classical on some initial interval (and unique
\emph{on the initial interval}, in the class of all dissipative solutions). The notion of
($*$)-solutions is in no way (of which we are aware) related to the dissipative solutions of Lions.
However, the \emph{idea} of a critical time, past which the strength of the solution is diminished,
is quite similar. 

The main difference, with the present work, is that contrary to dissipative solutions,
 the notion of ($*$)-solutions is finely tuned to mirror the \emph{dispersive} properties
 of (\ref{eq:BE}) by way of Strichartz estimates. In order to fully employ Strichartz in
 scaling-critical spaces, it is necessary to use the convolutive properties known to hold
 for $Q^+$ (cf. \cite{Ar2011, AC2010, ACG2010}). However, the \emph{stability} of solutions against
 perturbations of the data
 as provided by Strichartz is not well-understood (even on small time intervals),
 due to that fact that $Q^-$ does not satisfy
the full range of convolutive estimates known for $Q^+$. This inflexibility due to $Q^-$
 has, so far, been the limiting
factor in further development of the well-posedness theory for Boltzmann's equation in scaling-critical
functional spaces. However, all ($*$)-solutions are renormalized solutions by definition, so
we can hope to employ the known (weak and strong) compactness properties of 
renormalized solutions of (\ref{eq:BE}) from \cite{DPL1989,L1994I,L1994II}. In fact, by playing the
renormalized theory and entropy dissipation against the dispersive properties of free transport and
the convolutive properties of $Q^+$, and by a careful \emph{choice of limiting process} (which is itself
new), we prove that the class of ($*$)-solutions is closed under certain types of limits.
Moreover, we are able to transfer certain information about the \emph{limit} back to the underlying
sequence. 

\begin{remark}
We do not address uniqueness in the ($*$)-solution class (our methods are non-constructive and neither require
nor imply uniqueness). 
However, even if ($*$)-solutions are unique in general (which implies a sense of continuity for the solution map
by \cite{L1994I,L1994II} and the methods of this article),
 we \textbf{\emph{do not}} expect the solution map to be (locally)
\textbf{\emph{uniformly}} continuous on $L^2 \left( \mathbb{R}^2_x \times \mathbb{R}^2_v \right)$,
 due to a recent announcement by Xuwen Chen and
Justin Holmer demonstrating that (at least for a constant collision kernel in $d=3$) the bifurcation for
(Hadamard) well-posedness falls \emph{far above} the scaling-critical threshold (for the problem considered therein).
\cite{2206.11931}
\end{remark}

Informally, our main theorem provides for the existence of a class of \emph{large global 
distributional solutions}
to (\ref{eq:BE}). These are not obtained for general initial data; instead, they are derived by
considering perturbations of known solutions, specifically the moving Maxwellians taking the form
\[
a \exp\left( - b \left| v \right|^2 - c \left| x - v t \right|^2 \right)
\]
Crucially, the numbers $a,b,c > 0$ are \emph{arbitrary} (although the allowable size of perturbation
depends on $a,b,c$ in a manner we cannot quantify). Small global perturbations of arbitrarily large
moving Maxwellians were obtained decades ago by Toscani in \cite{Ti1988} using the Kaniel-Shinbrot
iteration with a very clever choice of beginning condition. (Also see
\cite{Alonso2009,AlonsoGamba2011} and references therein for refined results along the same lines.) 
However, Toscani was only able to handle
(weighted) $L^\infty$ perturbations. The present article (which does not use Kaniel-Shinbrot) allows
for perturbations at \emph{scaling-critical regularity}; this improvement appears to be completely new.
Moreover, the proof of the main theorem brings to bear the full force of both the dispersive theory
and the theory of renormalized solutions.

\subsection{Scale invariance.}

Let us define for parameters $\lambda,\mu > 0$
\[
f^{\left(\lambda,\mu\right)} \left( t,x,v \right) =
\frac{1}{\lambda \mu} f \left( \frac{\mu}{\lambda} t, \frac{x}{\lambda}, \frac{v}{\mu} \right)
\]
and
\[
f^{\left(\lambda,\mu\right)}_0 \left( x,v \right) =
\frac{1}{\lambda \mu} f_0 \left( \frac{x}{\lambda}, \frac{v}{\mu} \right)
\]
Then there holds
\[
\left\Vert f^{\left(\lambda,\mu\right)}_0 \right\Vert_{L^2} = \left\Vert f_0 \right\Vert_{L^2}
\]
and
\[
\left\Vert f^{\left(\lambda,\mu\right)} \right\Vert_{L^\infty \left( 
I_{\lambda,\mu} , L^2 \right)} =
\left\Vert f \right\Vert_{L^\infty \left( I, L^2 \right)}
\]
where $I = \left[ 0,T \right) \subset \mathbb{R}$ for some $0 < T \leq \infty$,
 $I_{\lambda,\mu} = \left[ 0, \lambda T / \mu \right)$, and $L^2$ is the space of
 square-integrable functions on $\mathbb{R}^2_x \times \mathbb{R}^2_v$.
Moreover, if $f$ is a Schwartz solution of (\ref{eq:BE}) on $I$ with initial
data $f_0$, then $f^{\left( \lambda, \mu \right)}$ is a solution of (\ref{eq:BE}) on
$I_{\lambda,\mu}$ with initial data $f_0^{\left( \lambda,\mu \right)}$.

\subsection{Summary of results.}

The overall objective of this article is to detail a thorough treatment of
the Boltzmann equation (\ref{eq:BE}) (henceforth ``Boltzmann's equation'' unless otherwise indicated).
While we will rely upon key results from
 prior works in this series \cite{CDP2017,CDP2018,CDP2019}, it should be possible to understand
this article with minimal reference to the prior works: indeed,
the  aim of the present work is to provide a coherent picture
of the local and even, to some extent, the global behavior of Boltzmann's equation. This extends
our previous scaling-critical article
\cite{CDP2019}, in which only initial data with small
$L^2$ norm was considered.\footnote{Add that technical regularity and decay
 conditions, albeit non-quantitative, were imposed on the initial
data as well, whereas (in the $L^2$ setting) 
such regularity conditions have been disposed of entirely in the present work, at the cost of possible loss of uniqueness.}
That result relied upon balancing the dispersive properties of free transport against the convolutive properties
of $Q^+$, similar to the work of Arsenio in \cite{Ar2011}.

The bulk of the present work 
aims to lift the small data limitation in \cite{CDP2019}, at the cost of limiting the time of existence,
and thereby provide a \emph{general} scaling-critical local theory
for Boltzmann's equation. Only local existence, without uniqueness, will be proven in the scaling-critical
space $L^2$, although a rather general weak-strong uniqueness theorem will be supplied.
We also aim to prove (in a specific sense) the \emph{stability} of solutions under
perturbations: this tendency towards stability is naturally limited due to the possible lack of uniqueness. 
In fact, the stability properties will be proven with respect to the class of ($*$)-solutions; ($*$)-solutions are 
formally introduced in Section \ref{sec:starSolutions}. It will turn out that ($*$)-solutions are distributional
solutions on some \emph{initial interval}, which we call $I^*$ in the definition of ($*$)-solutions; thus, this formalism
provides a framework for proving existence results for distributional solutions. The most striking
application of this stability result will be to show, in the class of ($*$)-solutions,
 that $L^2$ perturbations of a \emph{global scattering solution}
(satisfying a technical criterion which is proven to hold for, e.g., moving Maxwellians), are again
global and scattering, subject only to natural constraints on physically conserved or controlled quantities,
namely the total ($L^1$) mass and ($L^1$) second moments in space and velocity. Along the way, sharp
scaling-critical critera
for both scattering and finite-time breakdown of continuity will be proven, which hold even far from vacuum or
Maxwellians.

Starting from Section \ref{sec:higherRegularity}, we aim
 to establish \emph{propagation of regularity}, in the class of ($*$)-solutions, to arbitrarily high smoothness and decay
thresholds, including the Schwartz class,
up to the full interval of existence (in $L^2$, namely $I^*$) for ($*$)-solutions.
Indeed, it will be proven that \emph{any} ($*$)-solution, corresponding to initial data with sufficient
regularity and decay, is again regular and decaying (for $t \in I^*$); in particular, the solution is unique
(on $I^*$).  The sharp scaling-critical
criteria\footnote{for either scattering or finite time breakdown of continuity}
 mentioned above, therefore, apply again in the 
  setting of \emph{classical} solutions, even (as before) far from vacuum or Maxwellians.
  This allows us to identify \emph{breakdown of continuity} with
\emph{breakdown of regularity}.

\subsection{Comparison of models.}

We will next elucidate the context in which (\ref{eq:BE}) fits with similar models analyzed
in the literature. 
Let us remark from the outset that, given the 
 dimension $d \geq 2$ (the equation (\ref{eq:BE}) addresses the case where $d=2$),
 any Boltzmann equation (with or without the Grad cutoff) with a collision kernel
which is homogeneous with respect to scaling
in the relative velocity possesses a full set of scaling symmetries, respecting separately
and simultaneously the
 spatial and velocity variables. (One never considers homogeneity of the collision kernel $b$
 with respect to the
 angular variable $\sigma \in \mathbb{S}^{d-1}$, for obvious reasons.)
Now \emph{in the special case} that the collision kernel is homogeneous of degree
$2-d$, the functional space $L^d \left( \mathbb{R}^d_x \times \mathbb{R}^d_v \right)$ for the initial data is preserved
by the full set of scaling symmetries. This seems to be essentially a technical convenience,
relating to the fact that the space $L^d$ is  preserved under the free transport group
$e^{-t v \cdot \nabla_x}$.
 In the present article, we will be exclusively concerned with
(\ref{eq:BE}), which satisfies the Grad cutoff condition, and for which $L^2$ constitutes a
scaling-critical space, being above all a Hilbert space: the best of all possible worlds.
We note that  the constant collision kernel appearing in
 (\ref{eq:BE}) is a member of the family of \emph{Maxwell molecule collision kernels}.

There \emph{is} a physically
meaningful analogue of (\ref{eq:BE}) known as \emph{true Maxwell molecules} (\emph{tMm}), but while the
\emph{tMm} collision kernel is homogeneous of degree zero (in any dimension $d$)
 and expresses the same scaling properties as the case of a constant collision
kernel (in the same dimension $d$), \emph{tMm} does not satisfy the
Grad cutoff condition (due to the non-integrable angular dependence),
and none of the analysis of this article applies to \emph{tMm} even in $d=2$.
The hard sphere model, mentioned above, \emph{does} satisfy the Grad cutoff and is homogeneous of degree one
(in any dimension), and is again
physically meaningful as is \emph{tMm}; unfortunately, just as with \emph{tMm}, the hard sphere model seems completely
out of reach by the present methods. There are a few hints about how to approach hard spheres dispersively, at least in
certain functional spaces far from the scaling critical threshold \cite{CDP2017}, but the dispersive treatment
of hard spheres at scaling critical regularities remains a subject of ongoing investigation.

Collision kernels homogeneous of degree $-1$ respect scaling in the space
$L^3$ in $d=3$; this case (assuming Grad cutoff) seems to be the only Boltzmann
equation other than (\ref{eq:BE})
that is remotely tractable (at the scaling-critical level)
 using current dispersive technology. Unfortunately, even in the $L^3$ setting, a more technical analysis is
 required, due to the role played in this work by 
 the special properties of $L^2$: in particular, we use Plancherel in the proof of the
key bilinear gain operator $Q^+$ estimate
\begin{equation}
\label{eq:crucialEstimateSummary}
L^2 \times L^2 \rightarrow L^1 \left( \mathbb{R}, L^2 \right)
\end{equation}
to be discussed later, the corresponding bilinear estimate
\[
L^3 \times L^3 \rightarrow L^1 \left( \mathbb{R}, L^3 \right)
\]
expected to be \emph{false}
 for any collision kernel homogeneous of degree $-1$ in $d=3$.
Substitutes for this estimate are available in the literature \cite{ACG2010} even for the $L^3$-critical case,
 but these estimates require far more effort to
apply correctly \cite{Ar2011} to Boltzmann's equation (for starters, one must employ inhomogeneous Strichartz estimates).

\begin{remark}
It is somewhat reasonable to view (\ref{eq:crucialEstimateSummary}) as a substitute for a scaling-critical estimate
in Bourgain spaces $X^{s,b}$, formally taking $\left( s, b\right) =\left( 0, - \frac{1}{2}\right)$ (note that the
 endpoint case $b = -\frac{1}{2}$ is not admissible in the 
classical theory of Bourgain spaces: the standard range for the \emph{nonlinearity}
 is $b \in \left( -\frac{1}{2} ,  \frac{1}{2}\right)$).
Indeed, observe that for any separable Hilbert space $\mathcal{H}$, the space
$L^1 \left( \mathbb{R}, \mathcal{H} \right)$ formally scales like $\dot{H}^{-\frac{1}{2}} \left( \mathbb{R}, \mathcal{H} \right)$.
This suggests that one may be able to salvage parts of the scaling-critical
theory expounded in this article for other collision kernels
through the use of one  (or some) of the multitude of techniques in dispersive PDE theory which have been inspired by Bourgain spaces. 
A crucial difficulty would be to understand \emph{non-negativity}, which plays a central role in this work, in these types of
functional spaces.
\end{remark}

\section{Organization of this paper}

The main results are stated in Section \ref{sec:mainResult}, using the (somewhat extensive and occasionally 
subtle) notation from
Section \ref{sec:secNotation}. Fundamental abstract results and dispersive estimates are recalled and/or established in
Sections \ref{sec:squareIntegrability}, \ref{sec:abstractTheorem}, \ref{sec:dispersiveEstimates},
and \ref{sec:nonNegativeEstimates}; these will primarily
(but by no means exclusively)
be used in developing solutions to the gain-only Boltzmann equation (i.e. the equation obtained by discarding the loss term
$Q^-$ from Boltzmann's equation), as well as many basic properties of such solutions. 
Sections \ref{sec:QplusSection}, \ref{sec:comparison}, and \ref{sec:ptwiseConvFund} will develop deeper results concerning
the gain-only Boltzmann equation along with a comparison principle which will ultimately be used to transfer certain
knowledge about
the gain-only Boltzmann equation to the full Boltzmann equation. Up to this point, there is no mention of renormalized solutions
or entropy.

In Sections \ref{sec:entropyDissipation}, \ref{sec:starSolutions}, and \ref{sec:breakdownCriterion}, we introduce a new notion
of solutions to (\ref{eq:BE}), which we call ($*$)-solutions. 
The definition of ($*$)-solutions uses the notion of renormalized solutions as well as the
entropy dissipation.
 Sections \ref{sec:ExistenceProof} and
\ref{sec:starLimitsProof} establish the existence of ($*$)-solutions, as well as the closure of the class of ($*$)-solutions 
under a suitable limiting process.

Results concerning scattering solutions of (\ref{eq:BE}) are proven in Sections \ref{sec:scatteringBasic} and
\ref{sec:scatteringAdvanced}. These are combined with an important weak-strong uniqueness result from Section
\ref{sec:WkStrong} to establish Part One of the main theorem in Section \ref{sec:mainTheoremOne}. 
Higher regularity results are proven in Section \ref{sec:higherRegularity}, leading finally to the proof of Part Two
of the main theorem in Section \ref{sec:mainTheoremTwo}.

\section{Notation}
\label{sec:secNotation}

For any $p \in \left[ 1, \infty \right]$ we denote by $p^\prime \in \left[ 1, \infty \right]$ the unique
extended real number satisfying
\[
\frac{1}{p} + \frac{1}{p^\prime} = 1
\]
We will require the norms defined for measurable functions $h \left(x,v\right)$
\[
\left\Vert h \right\Vert_{L^p}^p = \int_{\mathbb{R}^2 \times \mathbb{R}^2} 
\left| h \left( x,v \right) \right|^p dx dv
\]
for $1 \leq p < \infty$, and
\[
\left\Vert h \right\Vert_{L^\infty} = \textnormal{ess. sup.}_{\left( x,v \right) \in
\mathbb{R}^2 \times \mathbb{R}^2 } \left| h \left( x,v \right) \right|
\]
We will also require mixed Lebesgue norms $L^p_x L^q_v$ or even with time
$L^r_t L^p_x L^q_v$ as in \cite{Ar2011}; in such cases, subscripts $t,x,v$ will be provided, along with
precise domains of integration.
Other permutations such as $L^r_t L^q_v L^p_x$ or
$L^p_x L^q_v L^r_t$ may also arise, but unusual orderings such as these
will only be introduced if absolutely required to carry out an argument.

For any separable Banach space $\mathfrak{G}$ and any interval $I$, the notation $L^p \left(
I, \mathfrak{G} \right)$ with $p \in \left[ 1, \infty \right]$ denotes the usual Bochner space, considering
$t \in I$ to be a time variable; it may be that a function is only Bochner $p$-integrable when
\emph{restricted} to $I$, in which case we would still write $f \in L^p \left( I, \mathfrak{G} \right)$.
The independent variable
corresponding to the interval $I$ is \emph{always} denoted by the symbol $t$: thus
if $A : I \times I \rightarrow \mathfrak{G}$ then
$\left\Vert A \left( s, t \right) \right\Vert_{L^1 \left( I, \mathfrak{G} \right)}$ is equal to
$\int_I \left\Vert A \left( s,t \right) \right\Vert_{\mathfrak{G}} dt$.

\begin{remark}
Thus without further annotation (an annotation being a subscript \emph{or} an explicit domain of integration
\emph{or} both), the reader may safely assume that norms denoted by the symbol $L^p$ are taken with
respect to $\left( x,v \right) \in \mathbb{R}^2 \times \mathbb{R}^2$; on the other hand, norms denoted by the symbol
$L^q \left( I, L^p\right)$ for an interval $I$  refer to $t \in I$ with $q$ power
and $\left( x,v \right) \in \mathbb{R}^2 \times \mathbb{R}^2$ with $p$ power.
By contrast, we may write an expression such as $\rho_f \in L^2 \left( I, L^4_x \left( \mathbb{R}^2 \right) \right)$,
which means that the spatial density $\rho_f \left( t,x \right)$ is square-integrable in time $t \in I$ into
the separable Banach space $L^4_x \left( \mathbb{R}^2 \right)$: this statement could be equivalently
written $\rho_f \in L^2_t \left( I, L^4_x \left( \mathbb{R}^2 \right) \right)$ or
$\rho_f \in L^2_t L^4_x \left( I \times \mathbb{R}^2 \right)$, but it could \emph{not} be written
$\rho_f \in L^2_t \left( I, L^4 \right)$ (this last version, in our notation, implies
 that a constant function of $v \in \mathbb{R}^2$ is fourth-power-integrable
over $\mathbb{R}^2$, which is plainly false).
\end{remark}

We will also rely upon the norm
\begin{equation}
\label{eq:defSubscriptedNorm}
\left\Vert h \right\Vert_{L^1_{2,t}} =
\int_{\mathbb{R}^2 \times \mathbb{R}^2}
\left( 1 + \left| x-vt \right|^2 + \left| v \right|^2 \right)
\left| h \left( x,v \right) \right| dx dv
\end{equation}
 where $t$ is a subscripted parameter \emph{on the left side}; note that the ambiguity of $t$ in $L^1_{2,t}$ is only of importance when $t$ is very large, since
for fixed $t$ the $L^1_{2,t}$ norm is equivalent to the $L^1_{2,0}$ norm, the constant diverging as $\mathcal{O} \left( t^2 \right)$
as $\left| t \right| \rightarrow \infty$.
 We will denote
\[
L^1_2 \coloneqq L^1_{2,0}
\]
for convenience.
The space $L^2 \bigcap L^1_{2,t}$ is normed by
\begin{equation}
\label{eq:baseSpaceDef}
\left\Vert h_0 \right\Vert_{L^2 \bigcap L^1_{2,t}} = \left\Vert h_0 \right\Vert_{L^2}
+ \left\Vert h_0 \right\Vert_{L^1_{2,t}}
\end{equation}
for a measurable function $h_0 \left( x,v \right)$.

\begin{remark}
Given a sufficiently regular and decaying 
solution $f$ of (\ref{eq:BE}), the \emph{time-dependent} quantity
\[
\left\Vert f\left(t\right) \right\Vert_{L^1_{2,t}}
\]
 is equal to
\[
\left\Vert f_0 \right\Vert_{L^1_{2,0}}
\]
 for all $t \geq 0$, although this may be
only an upper bound at low regularity.
\end{remark}

It will be useful to introduce the unusual $X$-norm defined on $L^2 \bigcap L^1_2$,
\begin{equation}
\label{eq:defXnorm}
\left\Vert h \right\Vert_X \coloneqq
\left\Vert h \right\Vert_{L^2} + \sum_{\varphi} 
\left| \int_{\mathbb{R}^2 \times \mathbb{R}^2} \varphi \left( x,v \right)
 h \left( x,v \right) dx dv \right|
\end{equation}
where the sum ranges over
\[
\varphi \in
\left\{ 1, \; v_1, \; v_2, \; \left| v \right|^2, \; \left| x \right|^2,  \; x \cdot v \right\}
\]
where $v = \left( v_1, v_2 \right) \in \mathbb{R}^2$.
Note carefully that 
the absolute value bars on the second term of (\ref{eq:defXnorm}) 
have been deliberately placed on the \emph{outside} of the integral (note that
$h$ need not be non-negative in (\ref{eq:defXnorm})). Additionally, observe that the sum is over
six test functions $\varphi$ (three of which are everywhere non-negative, five of which are
\emph{unbounded}),
 each of which is integrated over the whole phase-space.
We define
\begin{equation}
\label{eq:defXspace}
X \coloneqq \left( L^{2,+} \bigcap L^1_2 \; , \; d_X \right)
\end{equation}
where $L^{2,+}$ is the set of non-negative functions in $L^2$, and
\begin{equation}
\label{eq:defXmetric}
d_X \left( h, \tilde{h} \right) \coloneqq \left\Vert h-\tilde{h} \right\Vert_{X}
\end{equation}
In particular, $X$ is an \emph{incomplete} metric space.

We denote the free transport group
\[
\mathcal{T} (t) = e^{-t v \cdot \nabla_x}
\]
which is related to the free transport equation in that for any initial data $f_0 \in L^2$
there holds
\begin{equation}
\label{eq:freeTransport}
\left( \partial_t + v \cdot \nabla_x \right) \left( \mathcal{T} (t) f_0 \right) = 0
\end{equation}
in the sense of distributions. For any $t \in \mathbb{R}$, $\mathcal{T} \left( t \right)$ 
preserves all $L^p$ norms on $\mathbb{R}^2 \times \mathbb{R}^2$; also, it can be written via
an explicit formula
\[
\left[ \mathcal{T} \left( t \right) f_0 \right] \left( x,v \right) =
f_0 \left( x - v t , v \right)
\]
The function
\[
t \mapsto \mathcal{T} \left( t \right) f_0
\]
may be referred to by the shorthand
\[
\mathcal{T} f_0
\]
Additionally, following DiPerna and Lions \cite{DPL1989}, for any measurable
function $h\left( t,x,v \right)$ we will
use the \emph{pointwise} shorthand
\begin{equation}
\label{eq:defHashNotation}
h^\# \left( t,x,v \right) = h \left( t, x+ v t, v \right)
\end{equation}
defined at almost all $\left(t,x,v\right)$; this is more closely related to
the \emph{inverse} free transport operator $\mathcal{T} \left( -t \right)$.

\begin{remark}
Using the identities, for $a \in \mathbb{R}$,
\[
\left| x  + a v \right|^2 = \left| x \right|^2 + 2 a x \cdot v + a^2 \left| v \right|^2
\]
and
\[
\left(x + a v \right) \cdot v = x \cdot v + a \left| v \right|^2
\]
it is possible to show that
\[
\left\Vert \mathcal{T} \left( a \right) h_0 \right\Vert_X \leq
\left( 1 + 3 \left| a \right| + a^2 \right) \left\Vert h_0 \right\Vert_{X}
\]
for all $h_0 \in L^2 \bigcap L^1_2$. Note that $h_0$ does \emph{not} need to be non-negative;
in particular, we find from this that, for any non-negative functions
 $f_0, g_0 \in L^{2,+} \bigcap L^1_2$, letting $h_0 = f_0 - g_0$,
\[
d_X \left( \mathcal{T} \left( a \right) f_0, \mathcal{T} \left( a \right) g_0
\right) \leq \left( 1 + 3 \left| a \right| + a^2 \right)
d_X \left( f_0, g_0 \right)
\]
so the $X$-norm is, in this sense, compatible with free transport.
\end{remark}

We will use the well-known notation $\left< v \right>^2 = 1 + \left| v \right|^2$;
moreover, in discussing propagation of regularity in Section \ref{sec:higherRegularity}, we shall require the Sobolev norms
indexed by non-negative numbers $\alpha,\beta$,
\begin{equation}
\label{eq:SobolevNorms}
\left\Vert h \right\Vert_{H^{\alpha,\beta}} =
\left\Vert \left< v \right>^\beta \left< \nabla_x \right>^\alpha h \right\Vert_{L^2}
\end{equation}
defined for a measurable and locally integrable function $h\left( x,v \right)$, where
$\left< \nabla_x \right> = \left( 1 - \Delta_x \right)^{\frac{1}{2}}$ and $\Delta_x$ is
the usual Laplacian operator, extended by duality from the Schwartz class to the space of tempered
distributions,
but acting in the $x$ variable only.

For any interval $I$ (possibly open, closed, or half-open, and possibly unbounded), and any
 topological space $\mathfrak{G}$, the symbol $C \left( I, \mathfrak{G} \right)$ denotes the \emph{set} of continuous functions from 
$I$ into $\mathfrak{G}$. If $\mathfrak{G}$ is, additionally, a (convex subset of a) topological vector space, then 
$C \left( I, \mathfrak{G} \right)$ is a (convex subset of a) \emph{vector space}, but does not inherit any topological structure unless
otherwise noted. For example, even if $\mathfrak{G}$ is a Banach space, elements of 
$C \left( \left[ 0, \infty \right) , \mathfrak{G} \right)$ are \emph{not} required
to be in $L^\infty \left( \left[ 0, \infty \right), \mathfrak{G} \right)$, since we do not view $C \left( \left[0,\infty \right), \mathfrak{G} \right)$
as a \emph{normed} vector space
(although it is clearly a vector space in view of the vector space structure of $\mathfrak{G}$).
Note carefully that, under the canonical identification,
\[
C \left( \left[ 0,1 \right), \mathbb{R} \right) \neq C \left( \left[ 0,1 \right], \mathbb{R} \right)
\]
For example, the former contains each of
\[
q\mapsto \frac{1}{1-q} \qquad \textnormal{ and } \qquad q\mapsto\sin \left(\frac{1}{1-q}\right)
\]
whereas the latter does not contain either of these (regardless of any finite candidate value chosen at $q=1$).

For any measurable subset of a Euclidean space, say $E \subset \mathbb{R}^k$ for some $k \in \mathbb{N}$, taking care
\emph{not} to identify sets which differ by a set of measure zero, we define $L^1_{\textnormal{loc}} \left(E \right)$ to be
the set of measurable functions on $E$ which are in $L^1 \left( K \right)$ for each compact $K \subset E$. 
Thus, even though there is a canonical isomorphism
\[
L^1 \left( \left[ 0,1 \right], \mathbb{R} \right) \simeq
L^1 \left( \left[ 0,1 \right), \mathbb{R} \right)
\]
there is no canonical isomorphism between
\[
L^1_{\textnormal{loc}} \left( \left[ 0,1 \right), \mathbb{R} \right) \qquad \textnormal{ and }\qquad
L^1_{\textnormal{loc}} \left( \left[ 0,1 \right], \mathbb{R} \right)
\]
For example, the former contains
\[
q \mapsto \frac{1}{1-q}
\]
whereas the latter does not.

We denote the Schwartz class
\[
\mathcal{S} \coloneqq \mathcal{S} \left( \mathbb{R}^2_x \times \mathbb{R}^2_v \right) =
\mathcal{S} \left( \mathbb{R}^4 \right)
\]
Given a (possibly unbounded) interval $I$, and a measurable and locally
integrable function $f \left( t,x,v \right)$ on $I \times \mathbb{R}^2 \times \mathbb{R}^2$, we shall write
\[
f \in C^1 \left( I, \mathcal{S} \right)
\]
precisely if
\[
f \in C \left( I, \mathcal{S} \right) \quad \textnormal{and} \quad
\frac{\partial f}{\partial t} \in C \left( I, \mathcal{S} \right)
\]
Note that if $f \in C^1 \left( I, \mathcal{S} \right)$ then it automatically holds
\[
Q^{\pm}\left( f,f \right) ,\; \; v \cdot \nabla_x f \; \; \in C^1 \left( I, \mathcal{S} \right)
\]
Thus $f \in C^1 \left( I, \mathcal{S} \right)$ supplies a simple sufficient criterion for indentifying
``classical solutions'' of (\ref{eq:BE}).

Constants indicated by the symbol $C$ (or e.g. $C_{z_1,\dots,z_k}$, depending on 
free real parameters $z_1,\dots,z_k > 0$) are allowed
to vary from one line to the next, but are always supposed to be finite and non-zero. If it is desired
to track constants precisely, then $\mathbb{Z}$-indexed subscripted notation $C_0, C_1, C_2, \dots$ will be
used instead of $C$.

\section{Main result}
\label{sec:mainResult}

\subsection{Preliminary remarks.}

The difficulty in solving Boltzmann's equation in the presence of scaling-criticality 
(by which we mean that the collision kernel is homogeneous in velocity \emph{and} that the functional
space of interest is critical 
\emph{jointly} with respect to scalings in space and velocity)
stems from the fact that
the loss term $Q^-$, despite having the same scaling behavior as the gain term $Q^+$, does
\emph{not} satisfy the same estimates. Indeed, simply by examining (\ref{eq:BE}), we see that
while the (unsymmetrized) gain operator $Q^+ \left(g,h\right) $ treats its two arguments on similar footing
in many respects, the loss operator $Q^- \left( g,h \right) = g\rho_h$ is highly asymmetric between 
its two arguments.
 Unfortunately, unlike $L^2$, there is no dispersion in $L^1$; indeed, the only
hint of dispersion at the $L^1$ level occurs via velocity averaging \emph{in the presence of uniform integrability},
which plays an essential role in
 the theory of renormalized solutions. \cite{DPL1989} So we see that there is little hope of applying dispersive
principles to the full Boltzmann equation (\ref{eq:BE}) directly, without some deeper insights.

The key realization is that the gain term $Q^+$ expresses certain convolutive and compactifying properties,
 well-known to kinetic theorists,
 which do not hold for the loss term. (e.g. see \cite{Ar2011,L1994I} and references therein)
 So it is very natural to simply \emph{discard} the loss term altogether, in
the hopes of constructing an \emph{upper envelope} for any solution of Boltzmann's equation (\ref{eq:BE}). Unfortunately, such
a strategy \emph{again} fails
 due to the fact that this ``$Q^+$ equation'' is not globally well-posed for
all initial data in, say, the Schwartz class.\footnote{The blow up results for the $Q^+$ equation 
do \emph{not} hold in the ``near vacuum'' regime; this is of little relevance here since we are concerned with
local solutions for initial data of arbitrary size.} \cite{RSH1987} The best we can hope for is a
\emph{local upper envelope} (described momentarily), extended for a short time interval forward from any point $t_0 \geq 0$, \emph{depending
on the solution $f\left(t_0 \right)$ of (\ref{eq:BE}) itself!} It is the
precise understanding of this local upper envelope, or simply \emph{upper envelope} (since there is not a global one
in the general case regardless), that will provide the foundation for our main
theorem.

Let us briefly elaborate on the idea of an \emph{upper envelope}, to avoid any possible confusion. Usually, an envelope of
a collection $C$ of smooth curves in in the plane is another curve which meets tangentially each element of $C$. Formally,
viewing curves as graphs of functions, the function $f \left( t \right)$ would be the envelope of a collection of functions
$\left\{ g \left( t; t^\prime \right) \right\}_{t^\prime}$ (indexed by $t^\prime$) under the conditions
\[
g \left( t; t \right) = f \left( t \right) \quad \textnormal{ and } \quad \left. f^\prime \left( t \right) =\frac{\partial g}{\partial t}
\left( t; t^\prime \right) \right|_{t^\prime = t} 
\] 
where the partial derivative is evaluated in $t$ for fixed $t^\prime$, but evaluated along the diagonal $t^\prime = t$.
(Here we simply assume that each element of the collection only intersects $f$ at a single point; precise definitions do
not matter for this discussion.) We reverse the definition, viewing the \emph{collection} as an \emph{upper envelope} for the
\emph{curve}, and relaxing \emph{equality} to \emph{inequality}, namely
\[
g \left( t; t \right) = f \left( t \right) \quad \textnormal{ and } \quad \forall \left( t \geq t^\prime \right)\;
f \left( t \right) \leq g \left( t; t^\prime \right)
\]
(we do not define $g$ for $t < t^\prime$).
In particular, in the smooth setting, 
\[
f^\prime \left(t \right) \leq
\left. \frac{\partial g}{\partial t}
\left( t; t^\prime \right) \right|_{t^\prime = t} 
\]
In our case (regarding $\left( x,v \right)$ as fixed and restricting $t$ to a suitable existence interval in time),
$f$ solves (\ref{eq:BE}), $g \left( \cdot ; t^\prime\right)$ satisfies (\ref{eq:BE}) omitting $Q^-$ (for each $t^\prime$ fixed),
and we refer to $g$ as an \emph{upper envelope} for $f$. This situation is reminiscent of the theory of viscosity solutions,
but we find no precisely analogous terminology in the literature, so we have chosen this terminology for the benefit of
visualization. Precise definitions will be introduced in our discussion of the \emph{comparison principle}
in Section \ref{sec:comparison}.

\subsection{Results.}

\begin{definition}
\label{def:distributionalSoln}
We will say that a \emph{non-negative} function
\[
f \in L^1_{\textnormal{loc}} \left( I \times \mathbb{R}^2 \times \mathbb{R}^2 \right)
\]
where $I = \left[ 0,T \right)$ with $0 < T \leq \infty$, is a \emph{distributional solution} of
(\ref{eq:BE}) provided that each
\[
Q^+ \left( f,f \right) \quad \textnormal{and} \quad Q^- \left( f,f \right) \quad \in
L^1_{\textnormal{loc}} \left( I \times \mathbb{R}^2 \times \mathbb{R}^2 \right)
\]
i.e. $Q^{\pm} \left( f,f \right)$ are each locally integrable, and that
(\ref{eq:BE}) holds in the sense of distributions. In particular, 
the trace along the $t=0$ time-slice
is well-defined for any distributional solution of (\ref{eq:BE}), and this trace will be
denoted $f_0$ and called the initial data. If $T = \infty$ then $f$ is said to be \emph{global}.
\end{definition}

\begin{definition}
For any triple of strictly positive real numbers $a,b,c$ 
 we define the (restricted) family of \emph{moving Maxwellian
distributions}
\[
m^{a,b,c} \left( t,x,v \right)
= a \exp \left( - b \left| v \right|^2 -c \left| x-vt \right|^2  \right)
\]
with initial data
\[
m^{a,b,c}_0 \left( x,v \right) = a \exp
\left(  - b\left| v \right|^2 - c \left| x \right|^2\right)
\]
In particular, $m^{a,b,c}$ is at once a solution of Boltzmann's equation (\ref{eq:BE}), and at the same time a solution
of the free transport equation (\ref{eq:freeTransport}), that is,
\[
m^{a,b,c} \left( t \right) = \mathcal{T} \left( t \right) m^{a,b,c}_0
\]
\end{definition}

\begin{theorem*}
\emph{(Main Theorem, Part One)}
Let $a,b,c$ be arbitrary strictly positive real numbers, and consider the moving Maxwellian initial data $m^{a,b,c}_0$.
Then there exists a number
\[
\varepsilon = \varepsilon \left( a,b,c \right) > 0
\]
such that if $f_0 \in L^1_{\textnormal{loc}} \left( \mathbb{R}^2 \times \mathbb{R}^2 \right)$ is non-negative and satisfies
\begin{equation}
\label{eq:momentEpsilon}
\sum_{\varphi}
\left|
\int_{\mathbb{R}^2 \times \mathbb{R}^2} \varphi \left( x,v \right)
\left( f_0 \left( x,v \right) - m_0 \left( x,v \right) \right) dx dv 
\right| < \varepsilon
\end{equation}
where the sum ranges over $\varphi \in \left\{ 1, \; v_1, \; v_2, \;
\left|v\right|^2, \; \left| x \right|^2, \; x \cdot v \right\}$, and
\begin{equation}
\label{eq:sqaureEpsilon}
\left\Vert f_0 - m^{a,b,c}_0 \right\Vert_{L^2} < \varepsilon
\end{equation}
then there exists a non-negative global distributional solution $f$ of (\ref{eq:BE}), with initial data $f_0$,
such that
\[
f \in C \left( \left[0,\infty\right), L^2 \right)
\]
Moreover, $f$ scatters, which means (here and throughout this article) that there exists a non-negative measurable
function $f_{+\infty} \in L^2$ such that
\begin{equation}
\label{eq:itScatters}
\lim_{t \rightarrow +\infty} \left\Vert f \left( t \right) -
\mathcal{T} \left( t \right) f_{+\infty} \right\Vert_{L^2} = 0
\end{equation}
\end{theorem*}

\begin{remark}
The proof of the main theorem provides no quantitative control on
$\left\Vert f \left( t \right) - m^{a,b,c} \left( t \right) \right\Vert_{L^2}$
for $t > 0$.
\end{remark}

\begin{theorem*}
\emph{(Main Theorem, Part Two)}
Under the assumptions of Part One, if in addition $f_0 \in \mathcal{S}$, then the solution $f$ is
unique (in the sense to be explained in Section \ref{sec:WkStrong}), and
\[
f \in C^1 \left( \left[ 0,\infty\right), \mathcal{S} \right)
\]
also holds.
\end{theorem*}

We remark that condition (\ref{eq:momentEpsilon}) deliberately places the absolute value bars
on the \emph{outside} of the integral:
 indeed, we could replace ``$<\varepsilon$'' by ``$=0$'' \emph{in this line}
 without altering the
conceptual substance of the theorem, since this condition does little more than provide a sense of scale (in the space-homogeneous
case it is analogous to normalizing the total kinetic energy to one separately for each $f_0,m_0$).
Note carefully that neither part of the main theorem is restricted to what may be called the
``near vacuum'' regime.\footnote{The term ``near vacuum,'' in the kinetic theory sense, means not only that 
$x$ ranges over all of $\mathbb{R}^2$ and $f$ exhibits decay (in an average sense)
as $\left| x\right|\rightarrow \infty$,
but that the initial data $f_0$ lies in a small ball of the zero function for a suitable Banach space.}
 The result is perturbative around a function which is Gaussian jointly in
space and velocity, but that underlying Gaussian may be of any size, the only restriction being that
$\varepsilon$ depends on the underlying Gaussian. Thus the theorem \emph{statement} (but not the proof!)
is very similar to an old result by Toscani, who \emph{also} considered global solutions near large moving
Maxwellians. \cite{Ti1988} However, unlike Toscani, since the perturbation  
here is only restricted with respect to the size of the $L^2$ deviation combined with finiteness of the physical quantity $L^1_2$,
even a Schwartz initial data $f_0$ may be very far removed (in $L^\infty$, say) from $m_0^{a,b,c}$.

\begin{remark}
Regardless of the regularity of $f_0$, scattering is \emph{always} defined relative to the
$L^2$ norm, precisely as indicated in (\ref{eq:itScatters}):
 at no point is convergence at long time to be claimed in any other sense.
\end{remark}

\section{Uniform square integrability}
\label{sec:squareIntegrability}

For this section, let $E$ be a measurable subset of a Euclidean space $\mathbb{R}^k$, $k \in \mathbb{N}$, equipped with
the measure $\lambda$ induced by the Lebesgue measure on $\mathbb{R}^k$; in the applications,
$E$ may be $\mathbb{R}^2 \times \mathbb{R}^2$, or $I \times \mathbb{R}^2 \times \mathbb{R}^2$ for
an interval $I \subset \mathbb{R}$. 

\subsection{Preliminary remarks.}

We will be adapting the notion of uniform integrability as
it is applied in kinetic theory, where the interpretation is closely related to, but slightly different from,
that which arises in probability theory. In particular, in kinetic theory, consideration must be
made for underlying measure spaces which are not probability spaces, such as the Lebesgue measure
on Euclidean space. Beyond that, we will be further specializing by examining the uniform
integrability of the \emph{squares} of a sequence of functions, and establishing a dominated convergence
theorem in $L^2 \left( E, \lambda \right)$.

We emphasize that the material in this section is standard; in particular, our objective (in this section only)
is a special case of the \emph{Lebesgue-Vitali convergence theorem}; e.g. see \cite{VIB2007}, Chapter 4, Corollary 4.5.5 (which
is the infinite measure case of Theorem 4.5.4 in the same reference). Our motivation for repeating the analysis (specialized
to the Euclidean case for simplicity) is twofold: first, the results are easy to prove (in our limited setting)
 yet absolutely fundamental to all that is to follow; and, second, we wish to establish a more convenient form of terminology for
 our own purposes, as the terminology of \cite{VIB2007} is rather general and somewhat
  onerous for kinetic theory applications.

 In fact, we will be interested in the $L^2$ setting (what we will refer to by 
 the term \emph{uniform square integrability}), whereas \cite{VIB2007} considers the $L^1$ setting; this is a trivial distinction
 from the abstract perspective, but it is an essential distinguishing factor for \emph{this} paper, as it is only
  the $L^1$ case which is ubiquitous in kinetic theory (specifically in the theory of renormalized
  solutions, as well as hydrodynamic limits). Uniform square integrability will be an essential tool as it allows to
execute dominated convergence arguments (in $L^2$) when the dominating functions compose a \emph{family}, instead of a singleton;
this will enable, starting in
Section \ref{sec:ptwiseConvFund},
 the usage of a powerful \emph{comparison principle}.

\subsection{Definitions}

\begin{definition}
A sequence of (not necessarily non-negative) measurable functions $\left\{ h_n \right\}_n \subset L^1 \left( E,\lambda \right)$
will be said to be \emph{uniformly integrable} if, for every $\varepsilon > 0$, there exists a 
number $\delta > 0$ and a compact set $K \subset \mathbb{R}^k$ such that: \emph{ for any measurable set
$F \subset E$ with measure $\lambda \left[ F \right] < \delta$, it holds}
\[
\sup_n \int_{F \bigcup \left( E \backslash K \right)} \left| h_n \right| d\lambda < \varepsilon
\]
\end{definition}

\begin{remark}
Technically, this definition is closer to the notion of \emph{uniform absolute continuity of integral} in \cite{VIB2007}; however,
the two concepts are equivalent for atomless measures such as Lebesgue measure (Proposition 4.5.3 of \cite{VIB2007}), and the latter
terminology is not standard in the kinetic theory literature. Note, also, that the use of the compact set $K$, in our definition of
uniform integrability, obviates the need for an additional condition when working in the whole Euclidean space.
\end{remark}

\begin{definition}
A sequence of (not necessarily non-negative) measurable functions $\left\{ h_n \right\}_n \subset L^2 \left( E,\lambda \right)$
will be said to be \emph{uniformly square integrable} if the sequence 
$\left\{ \left| h_n \right|^2 \right\}_n$ is uniformly integrable (here $\left| h_n \right|^2 \left( e \right)
= \left| h_n \left( e \right) \right|^2$ for $e \in E$).
\end{definition}

\subsection{Results.}

\begin{lemma}
\label{lem:uiCst}
Let $h$ by a measurable function on $E$ such that $h \in L^2 \left( E,\lambda \right)$. Then the constant sequence
\[
\left\{ h, h, h, \dots \right\}
\]
is uniformly square integrable.
\end{lemma}
\begin{proof}
First observe that, by monotone convergence and the square-integrability of $h$,
\[
\lim_{n \rightarrow \infty} \int_E \left(
\mathbf{1}_{e \; : \; \left| h\left( e \right) \right| > n} + \mathbf{1}_{e \; : \; \left| e \right| > n} \right) \left|
h \left( e \right) \right|^2 d\lambda \left( e \right) = 0
\]
Let $\varepsilon > 0$. Then there exists an integer $N$ such that, for all $n \geq N$,
\[
\int_E \left(
\mathbf{1}_{e \; : \; \left| h\left( e \right) \right| > n} + \mathbf{1}_{e \; : \; \left| e \right| > n} \right) \left|
h \left( e \right) \right|^2 d\lambda \left( e \right)  < \frac{\varepsilon}{4}
\]
In particular, we can take $n = N$; in that case,
for the points $e \in E$ at which the integrand \emph{vanishes,} we have
$\left| e \right| \leq N$ and $\left| h \left( e \right) \right| \leq N$. Let the set of all such points
be denoted by $E_N$.

Now let $F$ be any measurable subset of $E$ such that 
\[
\lambda \left[ F \right] < \frac{\varepsilon}{2 N^2}
\]
We decompose $F$ as $F_1 \bigcup F_2$ where $F_1 = F \bigcap E_N$ and $F_2 = F \bigcap \left( E \backslash E_N \right)$.
Clearly, since $F_2 \subset E \backslash E_N$, we have
\[
\int_E \left(
\mathbf{1}_{e\in F_2} + \mathbf{1}_{e \; : \; \left| e \right| > N} \right) \left|
h \left( e \right) \right|^2 d\lambda \left( e \right)  < \frac{\varepsilon}{2}
\]
On the other hand, we also clearly have
\[
\int_{F_1} \left| h \left( e \right) \right|^2 d\lambda \left( e \right) <
\frac{\varepsilon}{2}
\]
Therefore,
\[
\int_E \left(
\mathbf{1}_{e\in F} + \mathbf{1}_{e \; : \; \left| e \right| > N} \right) \left|
h \left( e \right) \right|^2 d\lambda \left( e \right)  < \varepsilon
\]
so we may conclude.
\end{proof}

\begin{lemma}
\label{lem:uiUpper}
If $\left\{ h_n \right\}_n \subset L^2 \left( E,\lambda \right)$ is uniformly square integrable, and $g_n$ is a sequence of
measurable functions on $E$ such that
\[
\left| g_n \left( e \right) \right| \leq \left| h_n \left( e \right) \right|
\]
 for $\lambda$-a.e.
$e \in E$, then $\left\{ g_n \right\}_n$ is uniformly square integrable.
\end{lemma}
\begin{proof}
This follows from the definition of uniform square integrability, using the same
$\delta, K$ for each sequence $\left\{ g_n \right\}_n , \left\{ h_n \right\}_n$.
\end{proof}

\begin{lemma}
\label{lem:ui}
If the sequence $\left\{ h_n \right\}_n \subset L^2 \left( E, \lambda \right)$
 converges in $L^2 \left( E,\lambda \right)$, that is, there exists
$h \in L^2 \left( E,\lambda \right)$ such that
\[
\lim_{n \rightarrow \infty} \left\Vert h_n - h \right\Vert_{L^2 \left( E , \lambda  \right)} = 0
\]
then the sequence $\left\{ h_n \right\}_n$ is uniformly square integrable.
\end{lemma}
\begin{proof}
First note that, by Lemma \ref{lem:uiCst},
the sequence $\left\{ h_n \right\}_n$ is uniformly square integrable if and only if the
sequence $\left\{ h_n - h \right\}_n$ is uniformly square integrable, because
$h \in L^2 \left( E,\lambda \right)$. Therefore, we may assume without loss that $h$ is identically zero.

Let $\varepsilon > 0$.

Then since $h_n \rightarrow 0$ in $L^2 \left( E,\lambda \right)$, there exists a number $N$ such that,
for every $n \geq N$,
\[
\int_E \left| h_n \right|^2 d\lambda < \varepsilon
\]
Therefore, we may restrict our attention to the \emph{finite} set $\left\{ h_n \right\}_{1 \leq n < N}$.
For each $n=1,2,\dots,N-1$, by Lemma \ref{lem:uiCst} there exists a number $\delta_n > 0$ and a compact set
$K_n \subset \mathbb{R}^p$ such that, for any measurable set $F \subset E$ with $\lambda\left[ F \right] < \delta_n$,
\[
\int_{F \bigcup \left( E \backslash K_n \right)} \left| h_n \right|^2 d\lambda < \varepsilon
\]
Let $K = \bigcup_{n = 1,2, \dots, N-1} K_n$ and $\delta = \min_{n = 1, 2, \dots, N-1} \delta_n$ to conclude.
\end{proof}

\begin{lemma}
\label{lem:uiDCT}
\emph{(Special case of the general form of the Lebesgue-Vitali convergence theorem.)}
If the sequence $\left\{ h_n \right\}_n \subset L^2 \left( E,\lambda \right)$ is uniformly square integrable,
and the pointwise limit $h \left( e \right) = \lim_{n \rightarrow \infty} h_n \left( e \right)$ exists
for $\lambda$-a.e. $e \in E$, then 
\[
\lim_{n \rightarrow \infty} \left\Vert h_n - h \right\Vert_{L^2 \left( E,\lambda \right)} = 0
\]
\end{lemma}
\begin{proof}
Follows immediately from Egorov's theorem.
\end{proof}

\section{An abstract theorem}
\label{sec:abstractTheorem}

We use the Banach fixed point theorem to
establish a sense of local well-posedness for nonlinear evolutionary equations associated with a certain
type of multilinear estimate. This result does \emph{not} apply directly to (\ref{eq:BE}) but it \emph{does}
apply to the $Q^+$ equation (or gain-only Boltzmann equation)
\begin{equation}
\label{eq:GOEqn}
\left( \partial_t + v \cdot \nabla_x \right) h = Q^+ \left( h, h \right)
\end{equation}
We will be relying heavily on the unique local solution $h$ of (\ref{eq:GOEqn}), the existence of which will
be established using the theorem from this section. We will be tracking all constants in this section
precisely, so that we may focus on compact intervals in time without loss of generality.

\begin{definition}
Let $J \subset \mathbb{R}$ be a compact interval, and $\mathfrak{G}$ a separable
Banach space. Then we define a
\emph{norm} on $W^{1,1} \left( J, \mathfrak{G}\right)$, distinguished by the stylized
notation
\[
\mathcal{W}^{1,1} \left( J, \mathfrak{G} \right)
\]
by
\[
\left\Vert x \right\Vert_{\mathcal{W}^{1,1} \left( J, \mathfrak{G} \right)}
=
\left\Vert x \right\Vert_{L^\infty \left( J, \mathfrak{G} \right)} +
\left\Vert \frac{dx}{dt} 
\right\Vert_{L^1 \left( J, \mathfrak{G} \right)}
\]
which is equivalent to the usual norm on $W^{1,1} \left( J, \mathfrak{G} \right)$ due
to the compactness of $J$. (Note that $W^{1,1} \left( J, \mathfrak{G} \right) \subset
C \left( J, \mathfrak{G} \right)$.)
\end{definition}

\begin{corollary}
If $t_0 \in J$ is fixed arbitrarily, then the norm defined by
\[
\left\Vert x(t_0) \right\Vert_{\mathfrak{G}} +
\left\Vert \frac{dx}{dt} 
\right\Vert_{L^1 \left( J, \mathfrak{G} \right)}
\]
is equivalent to $\mathcal{W}^{1,1} \left( J, \mathfrak{G}\right)$, the constant being
independent of each $J$ and $t_0 \in J$ (indeed a constant of $2$ suffices in either direction).
\end{corollary}

The Corollary implies that if $x(t)$ is controlled at a single point
and $\frac{dx}{dt} $ is controlled along the interval, then $x(t)$ is
controlled along the interval. 

\begin{lemma}
\label{lem:Abd}
Let $\mathfrak{G}$ be a separable Banach space over $\mathbb{R}$, and 
let $p$ be a real number with
\[
1 \leq p < \infty
\]
and let $J \subseteq \mathbb{R}$ be a compact interval. 
Furthermore, suppose 
\[
\mathcal{A} (t,x_1,\dots,x_k) : J \times
\mathfrak{G}^{\times k} \rightarrow \mathfrak{G}
\]
 is linear in $x_1, \dots, x_k$ for each $t \in J$,
and satisfies
\begin{equation}
\label{eq:Abd0}
\left\Vert \mathcal{A} (t, x_1,\dots,x_k ) \right\Vert_{L^p \left( J, \mathfrak{G} \right)}
\leq C_0 \prod_{i=1}^k \left\Vert x_i \right\Vert_{\mathfrak{G}}
\end{equation}
In particular $\mathcal{A}$ may be viewed as a multilinear map
 $\mathfrak{G}^k \rightarrow L^p \left( J, \mathfrak{G} \right)$.

Then there exists a unique mapping
\[
\tilde{\mathcal{A}} : \left( W^{1,1} \left( J, \mathfrak{G}\right) \right)^{\times k}
\rightarrow L^p \left( J, \mathfrak{G} \right)
\]
for which $\tilde{A}$ is linear in each of its inputs and satisfies
for $x_i \in \mathfrak{G}, \; \varphi_i \in C^\infty \left( J, \mathbb{R} \right)$ ($i=1,\dots,k$)
the formula
\begin{equation}
\label{eq:multilin-phi}
\tilde{\mathcal{A}} \left(  x_1 \varphi_1 , 
\dots, x_k \varphi_k \right) (t) =
\mathcal{A} \left(t, x_1, \dots, x_k \right) \prod_{i=1}^k
\varphi_i \left(t\right)
\end{equation}
Clearly $\tilde{\mathcal{A}}$ is a canonical extension of $\mathcal{A}$.

The extension $\tilde{\mathcal{A}}$ satisfies for any $x_1\left(\cdot\right),
\dots,x_k\left(\cdot\right) \in W^{1,1} \left( J, \mathfrak{G} \right)$:
\begin{equation}
\label{eq:multilin0}
\begin{aligned}
& \left\Vert \tilde{\mathcal{A}} \left(  x_1 , \dots, x_k  \right)
\right\Vert_{L^p \left( J, \mathfrak{G} \right)} \leq (k+1) C_0
\prod_{i=1}^k 
\left\Vert x_i \right\Vert_{\mathcal{W}^{1,1} \left( J, \mathfrak{G} \right)}
\end{aligned}
\end{equation}
noting carefully the stylized $\mathcal{W}$ in (\ref{eq:multilin0}).
\end{lemma}
\begin{proof}
We assume without loss that $J = \left[ 0, b \right]$ for some $0 < b< \infty$.
First we will establish the uniqueness, and in so doing we will obtain
 a formula for $\tilde{\mathcal{A}}$; then we will show that the formula
 defines a mapping which satisfies (\ref{eq:multilin0}).
 
For given $\varphi \in C^\infty \left( \mathbb{R}, \mathbb{R} \right)$ and $s \in \mathbb{R}$ let us
define the translation operator
\[
\left( \tau_s \varphi \right) (t) = \varphi (t-s)
\]
Then if $x_1,\dots,x_k \in \mathfrak{G}$ and
$\varphi_1, \dots, \varphi_k  \in C^\infty \left( \mathbb{R}, 
\mathbb{R}\right)$ then by (\ref{eq:multilin-phi}) it holds
\[
\tilde{\mathcal{A}} \left(  x_1 \tau_s \varphi_1 , 
\dots, x_k \tau_s \varphi_k \right) (t) =
\mathcal{A} \left(t, x_1, \dots, x_k \right) \prod_{i=1}^k
\varphi_i \left(t-s\right)
\]
By (\ref{eq:Abd0}) 
the right-hand side is differentiable in $s$ for almost every $t \in J$ fixed, and
we have
\begin{equation*}
\begin{aligned}
& \frac{\partial}{\partial s}
\tilde{\mathcal{A}} \left(  x_1 \tau_s \varphi_1 , 
\dots, x_k \tau_s \varphi_k \right) \left(t\right) \\
& \qquad \qquad  =
- \mathcal{A} \left(t, x_1, \dots, x_k \right)
\sum_{i=1}^k \varphi_i^\prime \left(t-s\right) \prod_{j \neq i} \varphi_j \left(t-s\right)
\end{aligned}
\end{equation*}

Therefore, applying (\ref{eq:multilin-phi}) again on the right, we have
\begin{equation*}
\begin{aligned}
& \frac{\partial}{\partial s}
\tilde{\mathcal{A}} \left(  x_1 \tau_s \varphi_1 , 
\dots, x_k \tau_s \varphi_k \right) \left(t\right) \\
& \qquad \qquad = -
\sum_{i=1}^k 
\mathcal{\tilde{A}} \left( x_1 \tau_s \varphi_1, \dots,
x_i \tau_s \varphi_i^\prime, \dots,
x_k \tau_s \varphi_k \right)\left(t\right)
\end{aligned}
\end{equation*}

By the linearity in each entry, if each $x_i \left(t\right)$ is a finite linear combination of
terms like $x_\ell \varphi_\ell \left(t\right)$, then
\begin{equation*}
\begin{aligned}
& \frac{\partial}{\partial s}
\tilde{\mathcal{A}} \left(  \tau_s x_1 \left(\cdot\right) , 
\dots, \tau_s x_k (\cdot) \right) \left(t\right) \\
& \qquad \qquad = -
\sum_{i=1}^k 
\mathcal{\tilde{A}} \left( \tau_s x_1 \left(\cdot\right), \dots,
\tau_s x_i^\prime \left(\cdot\right), \dots,
 \tau_s x_k \left(\cdot\right) \right) \left(t\right)
\end{aligned}
\end{equation*}
and since we have only taken finite combinations, for almost every
$t \in J$ the formula (as before) does hold strongly for each $s \in J$.
Hence for any such $t$ we can integrate in $s$ over a domain that
\emph{depends on $t$}, namely $0 \leq s \leq t$, to deduce
\[
\begin{aligned}
& \tilde{\mathcal{A}} \left( x_1 \left(0\right), \dots, x_k \left(0\right) \right) -
\tilde{\mathcal{A}} \left( x_1 \left(\cdot\right), \dots, x_k \left(\cdot\right) \right) \\
& \qquad\qquad= -\sum_{i=1}^k 
\int_0^t \mathcal{\tilde{A}} \left(
\tau_s x_1 \left(\cdot\right), \dots, \tau_s x_i^\prime \left(\cdot\right), \dots,
\tau_s x_k (\cdot)\right)\left(t\right) ds
\end{aligned}
\]
Again since the $x_k$ are finite sums of constant elements of $\mathfrak{G}$ times
smooth scalar-valued functions, it is acceptable to replace
$\mathcal{A}$ for $\mathcal{\tilde{A}}$ under the integral on the right to
obtain
\[
\begin{aligned}
& \tilde{\mathcal{A}} \left( x_1 \left(0\right), \dots, x_k \left(0\right) \right) \left(t\right) -
\tilde{\mathcal{A}} \left( x_1 \left(\cdot\right), \dots, x_k \left(\cdot\right) \right) \left(t\right)\\
& \qquad\qquad= -\sum_{i=1}^k 
\int_0^t \mathcal{A} \left(t,
 x_1 \left(t-s\right), \dots,  x_i^\prime \left(t-s\right), \dots,
x_k \left(t-s\right)\right)ds
\end{aligned}
\]
Equivalently, this may be written
\[
\begin{aligned}
& 
\tilde{\mathcal{A}} \left( x_1 \left(\cdot\right), \dots, x_k \left(\cdot\right) \right) \left(t\right)=
 \mathcal{A} \left(t,  x_1 \left(0\right), \dots, x_k \left(0\right) \right) \\
& \qquad \qquad \qquad+
\sum_{i=1}^k 
\int_0^t \mathcal{A} \left(t,
 x_1 \left(s\right), \dots,  x_i^\prime \left(s\right), \dots,
x_k \left(s\right)\right)ds
\end{aligned}
\]
Thus $\tilde{\mathcal{A}}$ is uniquely determined by $\mathcal{A}$ due
to a density argument; this is justified by the continuity estimates
we prove below.

Indeed, clearly we have
\begin{equation*}
\begin{aligned}
& \left\Vert \mathcal{A} \left( t, x_1 \left(0\right), \dots, x_k \left(0\right) \right)
\right\Vert_{L^p \left( J, \mathfrak{G} \right)} \\
& \qquad \qquad \qquad \leq C_0
\prod_{i=1}^k \left\Vert x_i \left(0\right) \right\Vert_{\mathfrak{G}}
\leq C_0
 \prod_{i=1}^k 
 \left\Vert x_i \right\Vert_{\mathcal{W}^{1,1} \left( J, \mathfrak{G} \right)}
\end{aligned}
\end{equation*}
by (\ref{eq:Abd0}). Additionally, using (\ref{eq:Abd0}) again,
\[
\begin{aligned}
&\left\Vert \int_0^t \mathcal{A} \left(t,
 x_1 \left(s\right), \dots,  x_i^\prime \left(s\right), \dots,
x_k \left(s\right)\right)ds
\right\Vert_{L^p \left( J, \mathfrak{G} \right)}\\
& \qquad \qquad\leq
\left\Vert \int_0^t \left\Vert 
\mathcal{A} \left(t,
 x_1 \left(s\right), \dots,  x_i^\prime \left(s\right), \dots,
x_k \left(s\right)\right)\right\Vert_{\mathfrak{G}} ds \right\Vert_{L^p \left( J, \mathbb{R} \right)}\\
& \qquad \qquad\leq
\left\Vert \int_J \left\Vert 
\mathcal{A} \left(t,
 x_1 \left(s\right), \dots,  x_i^\prime \left(s\right), \dots,
x_k \left(s\right)\right)\right\Vert_{\mathfrak{G}} ds \right\Vert_{L^p \left( J, \mathbb{R} \right)}\\
& \qquad \qquad\leq
 \int_J \left\Vert 
\mathcal{A} \left(t,
 x_1 \left(s\right), \dots,  x_i^\prime \left(s\right), \dots,
x_k \left(s\right)\right) \right\Vert_{L^p \left( J, \mathfrak{G} \right)} ds\\
& \qquad \qquad\leq C_0
 \int_J 
 \left\Vert x_i^\prime \left(s\right) \right\Vert_{\mathfrak{G}}
 \left( \prod_{j \neq i} \left\Vert x_j \left(s\right) \right\Vert_{\mathfrak{G}} \right)  ds \\
 & \qquad \qquad \leq C_0
 \left\Vert x_i^\prime \right\Vert_{L^1 \left( J, \mathfrak{G} \right)}
 \prod_{ j \neq i} \left\Vert x_j \right\Vert_{L^\infty \left( J, \mathfrak{G} \right)} \\
 & \qquad\qquad \leq C_0
 \prod_{i=1}^k 
 \left\Vert x_i \right\Vert_{\mathcal{W}^{1,1} \left( J, \mathfrak{G} \right)}
\end{aligned}
\]
\end{proof}

By abuse of notation, we will write $\mathcal{A}$ in place
of $\tilde{\mathcal{A}}$ in what follows.

\begin{theorem}
\label{thm:CritLWP}
Given a separable Banach space $\mathfrak{G}$ and a compact
 interval
$J \subseteq \mathbb{R}$ where $J = \left[ 0, b \right]$, some $0 < b < \infty$, suppose
$\mathcal{A} \left(t,x_1,x_2\right) : J \times \mathfrak{G}\times \mathfrak{G} \rightarrow \mathfrak{G}$
 is linear in 
$x_1,x_2$ for each
$t \in J$, and satisfies
\[
\left\Vert \mathcal{A} \left(t,x_1,x_2\right) 
\right\Vert_{L^1 \left( J, \mathfrak{G} \right)} \leq C_0 
\left\Vert x_1 \right\Vert_{\mathfrak{G}} \left\Vert x_2 \right\Vert_{\mathfrak{G}}
\]

Let $\varepsilon > 0$. There are numbers
$\delta_1^0, \delta_2^0 > 0$, \emph{depending only on $C_0$ and $\varepsilon$},
such that if
\[
0 < \delta_1 \leq \delta_1^0\qquad \textnormal{and}
\qquad 0 < \delta_2 \leq \delta_2^0
\]
and the following three estimates
\[
\left\Vert \mathcal{A} \left(t,x_0,x_0\right)
\right\Vert_{L^1 \left( J, \mathfrak{G} \right)} \leq \delta_1
\]
\[
\forall \left( x_2 \in \mathfrak{G}\right)\quad
\left\Vert \mathcal{A} \left(t,x_0,x_2\right)
\right\Vert_{L^1 \left( J, \mathfrak{G} \right)} \leq \delta_2
\left\Vert x_2 \right\Vert_{\mathfrak{G}} 
\]
\[
\forall \left( x_2 \in \mathfrak{G}\right)\quad
\left\Vert \mathcal{A} \left(t,x_2,x_0\right)
\right\Vert_{L^1 \left( J, \mathfrak{G} \right)} \leq \delta_2
\left\Vert x_2 \right\Vert_{\mathfrak{G}}
\]
all hold for some $x_0 \in \mathfrak{G}$, then the following holds as well:

There exists a unique function
\[
x \in W^{1,1} \left(J,\mathfrak{G}\right)
\]
 such that for all $t \in J$ there holds
\[
x\left(t\right) = x_0 + \int_0^t \mathcal{A} \left( s, x\left(s\right), x\left(s\right) \right) ds
\]
and also
\[
\left\Vert x - x_0 \right\Vert_{L^\infty \left( J, \mathfrak{G} \right)} +
\left\Vert \frac{d}{dt} \left\{ x - x_0 \right\} 
\right\Vert_{L^1 \left( J, \mathfrak{G} \right)} \leq \varepsilon
\]
In particular, because $\frac{d}{dt} x = \mathcal{A} \left(
t, x\left(t\right), x\left(t\right) \right)$ and $x_0$ is a constant, we have
\[
\left\Vert \mathcal{A} \left( t, x\left(t\right), x\left(t\right) 
\right) \right\Vert_{L^1 \left( J, \mathfrak{G}\right)} \leq \varepsilon
\]
\end{theorem}

\begin{proof}
Fix $x_0 \in \mathfrak{G}$ and define the map
\[
\mathfrak{F} : W^{1,1} \left( J, \mathfrak{G} \right) \rightarrow
W^{1,1} \left( J, \mathfrak{G} \right)
\]
via
\[
\left[ \mathfrak{F} \left( x \right) \right] \left(t\right)
= x_0 + \int_0^t \mathcal{A} \left(s, x\left(s\right), x\left(s\right) \right) ds
\]

Then  we have
\[
\left[ \mathfrak{F} \left( x \right) \right] \left(t\right) - x_0
= \int_0^t \mathcal{A} \left(s, \left( x\left(s\right)-x_0 \right) + x_0, \left( x\left(s\right)-x_0\right)+x_0 \right) ds
\]
that is
\begin{equation}
\label{eq:F00}
\begin{aligned}
& \left[ \mathfrak{F} \left( x \right) \right] \left(t\right) - x_0
= \int_0^t \mathcal{A} \left(s, x_0, x_0 \right) ds
+ \int_0^t \mathcal{A} \left(s, x\left(s\right)-x_0, x_0 \right) ds\\
&\qquad+ \int_0^t \mathcal{A} \left(s, x_0, x\left(s\right)-x_0 \right) ds
+ \int_0^t \mathcal{A} \left(s, x\left(s\right)-x_0,x\left(s\right)-x_0 \right) ds 
\end{aligned}
\end{equation}

Note that the right-hand side is zero when $t=0$.
Hence by  Lemma \ref{lem:Abd} and the bounds assumed in the 
statement of the Theorem, there holds
\[
\begin{aligned}
&\left\Vert \mathfrak{F} \left( x \right)  - x_0
\right\Vert_{\mathcal{W}^{1,1} \left( J, \mathfrak{G} \right)}\\
& \qquad\qquad \leq
2 \delta_1 + 8 \delta_2 \left\Vert x - x_0
\right\Vert_{\mathcal{W}^{1,1} \left( J, \mathfrak{G} \right)} +
6 C_0 \left\Vert x - x_0 \right\Vert^2_{\mathcal{W}^{1,1} \left( J, \mathfrak{G} \right)}
\end{aligned}
\]
For instance, in the first term, we have an extra factor of $2$ because
$\mathcal{W}^{1,1}$ counts an $L^\infty \left( J, \mathfrak{G} \right)$ and an $L^1 \left( J, \mathfrak{G} \right)$, and the 
$L^\infty \left( J, \mathfrak{G} \right)$
is precisely bounded by the $L^1 \left( J, \mathfrak{G} \right)$ since the initial value is zero.
Similar logic holds for the remaining terms.

The Lipschitz estimate from (\ref{eq:F00}) reads as
\[
\begin{aligned}
& \left\Vert \mathfrak{F} \left( x \right)   - 
\mathfrak{F} \left( \tilde{x} \right) 
\right\Vert_{\mathcal{W}^{1,1} \left( J, \mathfrak{G}\right)} \\
& \leq \left[ 8 \delta_2 + 6 C_0 \left(
\left\Vert x - x_0 \right\Vert_{\mathcal{W}^{1,1} \left( J, \mathfrak{G} \right)}
+ 
\left\Vert \tilde{x} - x_0 \right\Vert_{\mathcal{W}^{1,1} \left( J, \mathfrak{G} \right)}
 \right)\right] \left\Vert
x - \tilde{x} \right\Vert_{\mathcal{W}^{1,1} \left( J, \mathfrak{G}\right)}
\end{aligned}
\]

To conclude, we apply the Banach fixed point theorem in the metric space
\[
\mathcal{B}_\varepsilon = \left\{ x \in W^{1,1} \left( J, \mathfrak{G} \right) \; \left|
\; 
\left\Vert x - x_0 \right\Vert_{\mathcal{W}^{1,1} \left( J, \mathfrak{G} \right)}
\leq \varepsilon
\right. \right\}
\]
the metric provided by the $\mathcal{W}^{1,1} \left( J, \mathfrak{G} \right)$ norm (of the difference
between any two elements).
We may without loss assume $\varepsilon$ is suffiently small.

The constraints are
\[
2 \delta_1 + 8 \delta_2 \varepsilon + 6 C_0 
\varepsilon^2 \leq \varepsilon
\]
and
\[
8 \delta_2 + 12 C_0 \varepsilon < 1
\]
The second constraint is satisfied once $\varepsilon < \frac{1}{48 C_0}$
 and $\delta_2 < \frac{1}{16}$. To satisfy the first constraint, it then
 suffices to further require that $\delta_1 < \frac{1}{8} \varepsilon$.
\end{proof}

\section{Dispersive estimates}
\label{sec:dispersiveEstimates}

\subsection{Castella-Perthame.}

First let us recall the family of homogeneous kinetic Strichartz estimates from Castella and Perthame, along with the
key dispersive estimates upon which they rely.
\cite{CaP1996,Ar2011} (We will not require the corresponding inhomogeneous Strichartz estimates.)

\begin{lemma}
\label{lem:CaP}
For any $1 \leq r \leq p \leq \infty$, if
\[
f_0 \in L^{r}_x L^p_v \left( \mathbb{R}^2 \times \mathbb{R}^2 \right)
\]
then for any $t \in \mathbb{R} \setminus \left\{ 0 \right\}$ it holds
\[
\mathcal{T} \left( t \right) f_0 \in L^p_x L^{r}_v \left( \mathbb{R}^2 \times \mathbb{R}^2 \right)
\]
and we have the estimate
\[
\left\Vert \mathcal{T} \left( t \right) f_0
\right\Vert_{L^p_x L^{r}_v \left( \mathbb{R}^2 \times \mathbb{R}^2 \right)} \leq
C \left| t \right|^{-2 \left( \frac{1}{r} - \frac{1}{p}\right)}
\left\Vert f_0 \right\Vert_{L^{r}_x L^p_v \left( \mathbb{R}^2 \times \mathbb{R}^2 \right)}
\]
where $C$ is an absolute constant independent of $t$ and $p$.
\end{lemma}

\begin{proposition}
\label{prop:CaP}
Whenever $r, p \in \left[1,\infty\right]$ are such that $r>2$ and
$\frac{1}{r} = 1 - \frac{2}{p}$, for any $f_0 \in L^2$ there holds
\[
\mathcal{T} f_0 \in L^r_t L^p_x L^{p^\prime}_v
\left( \mathbb{R} \times \mathbb{R}^2 \times \mathbb{R}^2 \right) 
\]
Moreover, we have the following estimate:
\[
\left\Vert \mathcal{T} f_0\right\Vert_{L^r_t L^p_x L^{p^\prime}_v
\left( \mathbb{R} \times \mathbb{R}^2 \times \mathbb{R}^2 \right)} \leq C_r
\left\Vert f_0 \right\Vert_{L^2}
\]
the constant $C_r$ depending only on $r$.
\end{proposition}

\subsection{Intuition.}

We are nearly ready to discuss the meaning of (\ref{eq:crucialEstimateSummary}), i.e.
\begin{equation}
L^2 \times L^2 \rightarrow L^1 \left( \mathbb{R}, L^2 \right)
\end{equation}
In fact, what we really mean is that this bilinear estimate holds for the gain operator $Q^+$ composed with free
transport, that is,
\begin{equation}
\label{eq:crucialEstimateExplain}
\left\Vert Q^+ \left( \mathcal{T} f_0, \mathcal{T} h_0
\right) \right\Vert_{L^1 \left( \mathbb{R}, L^2 \right)}
\leq C \left\Vert f_0 \right\Vert_{L^2}
\left\Vert h_0 \right\Vert_{L^2}
\end{equation}
whenever $f_0, h_0 \in L^2$. This is a \emph{scaling-critical bilinear Strichartz estimate}, and the space
$L^1 \left( \mathbb{R}, L^2 \right)$ is essentially a stand-in for the missing scaling-critical endpoint of the
classical theory of Bourgain spaces, i.e. in the usual notation
$X^{s,b}$ with $\left( s,b \right) = \left( 0, - \frac{1}{2} \right)$, for which general theory does not exist without
refinement of the functional setting.

Before we begin, let us explain heuristically why (\ref{eq:crucialEstimateExplain}) \emph{should} be true
(since it is not entirely obvious at first glance). Let us introduce the classical convolution acting
in the \emph{velocity} variable only,
\[
\left( f \ast_v h \right) \left( t,x,v \right) =
\int_{\mathbb{R}^2} f \left( t, x, u \right) h \left( t , x , v - u \right) du
\]
then if $f, h \in C \left( I, \mathcal{S} \right)$ then we have for each $\left(t,x\right)
 \in I\times \mathbb{R}^2$ the Young's inequality
\begin{equation}
\label{eq:YoungsInequality}
\left\Vert f \ast_v h \right\Vert_{L^2_v \left( \mathbb{R}^2 \right)}
\leq C \left\Vert f \right\Vert_{L^{\frac{4}{3}}_v \left( \mathbb{R}^2 \right)}
 \left\Vert h \right\Vert_{L^{\frac{4}{3}}_v \left( \mathbb{R}^2 \right)}
\end{equation}
Now since $Q^+$ apparently has a convolutive structure (but taken over manifolds respecting the
 energy and momentum constraints),
and in addition the collision kernel at hand is bounded and homogeneous \textbf{of degree zero}\footnote{since
homogeneity of any other degree would impact the numerology of convolution inequalities}, we \emph{might} expect
(\ref{eq:YoungsInequality}) to hold again for $Q^+$:
\begin{equation}
\label{eq:youngsQPlus}
\left\Vert Q^+ \left( f, h \right) \right\Vert_{L^2_v \left( \mathbb{R}^2 \right)}
\leq C \left\Vert f \right\Vert_{L^{\frac{4}{3}}_v \left( \mathbb{R}^2 \right)}
 \left\Vert h \right\Vert_{L^{\frac{4}{3}}_v \left( \mathbb{R}^2 \right)}
\end{equation}
and it turns out (\ref{eq:youngsQPlus}) is \emph{true!} It has been proven, and studied in detail (in far greater generality), by
Alonso and Carneiro using Fourier methods \cite{AC2010}, and by
Alonso, Carneiro and Gamba using a weighted convolution formulation of $Q^+$ \cite{ACG2010},
 and was also proven for $Q_b^+$ with a restricted
class of collision kernels $b$ by Arsenio \cite{Ar2011} using the weak formulation of $Q_b^+$ on the kinetic side.
(Note that Alonso and Carneiro \cite{AC2010} used a radial symmetrization technique on the Fourier transform of
$f$ in the \emph{proof}, but their theorem makes no assumption of radiality for $f$.)

So let us combine (\ref{eq:youngsQPlus}) with the homogeneous Strichartz estimates of Proposition \ref{prop:CaP} 
to ``prove'' (\ref{eq:crucialEstimateExplain}):
\begin{equation*}
\begin{aligned}
& \left\Vert Q^+ \left( \mathcal{T} f_0, \mathcal{T} h_0 \right) \right\Vert_{L^1 \left( \mathbb{R}, L^2 \right)} \\
& \qquad \qquad \leq
C \left\Vert \mathcal{T} f_0 \right\Vert_{L^2 \left( I, L^4_x L^{\frac{4}{3}}_v \left(
\mathbb{R}^2 \times \mathbb{R}^2 \right)\right)}
\left\Vert \mathcal{T} h_0 \right\Vert_{L^2 \left( I, L^4_x L^{\frac{4}{3}}_v \left(
\mathbb{R}^2 \times \mathbb{R}^2 \right)\right)} \\
& \qquad\qquad \leq C \left\Vert f_0 \right\Vert_{L^2} \left\Vert h_0 \right\Vert_{L^2}
\end{aligned}
\end{equation*}
where we have applied (\ref{eq:youngsQPlus}) followed by H{\" o}lder's inquality (in $x$ then $t$) in the first step,
and the endpoint case $r=2$ of Proposition \ref{prop:CaP} in the second step, namely:
\begin{equation}
\label{eq:endpointStrich}
\left\Vert \mathcal{T} h_0 \right\Vert_{L^2_t L^4_x L^{\frac{4}{3}}_v
 \left( \mathbb{R} \times \mathbb{R}^2 \times \mathbb{R}^2 \right)}
\leq C \left\Vert h_0 \right\Vert_{L^2}
\end{equation}
Unfortunately, (\ref{eq:endpointStrich}) is known to be \emph{\textbf{false.}} \cite{Bennettetal2014}
A more careful analysis is required.

\begin{remark}
For detailed treatments of an approach to proving (\ref{eq:crucialEstimateExplain}) by way
of the Wigner-Weyl transform,
we refer the reader to the previous articles
of this series. \cite{CDP2017,CDP2018,CDP2019} We use an alternative approach below which does
not use the Wigner transform in its usual formulation.
\end{remark}

\subsection{The basic estimate.}

First, a simple application of the endpoint Strichartz estimates of Keel and Tao. Before we can
state the lemma, we need to formally define the partial Fourier transform acting only in $v$:
\[
\left[
\mathcal{F}_v f \right] \left( t,x,\eta \right) =
\int_{\mathbb{R}^2} e^{-2 \pi i v \cdot \eta } f \left( t,x,v \right) dv 
\]

\begin{remark}
The reader must take care to realize that the failure of endpoint Strichartz estimates for the Schr{\" o}dinger equation
in two dimensions, which is well-known, has no bearing on our application of Keel-Tao. That failure represents the
$\left( 2, \infty, 1 \right)$ edge case in Keel-Tao and it is \emph{not} the case we are using here.
\end{remark}

\begin{lemma}
\label{lem:KeelTaoStrich}
For any $f_0 \in L^2$, it holds
\[
\mathcal{F}_v \mathcal{T} f_0 \in L^2 \left( \mathbb{R}, L^4_{x,\eta} \left(
\mathbb{R}^2 \times \mathbb{R}^2 \right) \right)
\]
and we have the estimate
\[
\left\Vert \mathcal{F}_v \mathcal{T} f_0 \right\Vert_{L^2 \left( \mathbb{R}, L^4_{x,\eta} \left(
\mathbb{R}^2 \times \mathbb{R}^2 \right) \right)} \leq C
\left\Vert f_0 \right\Vert_{L^2}
\]
for some absolute constant $C$.
\end{lemma}

\begin{proof}
Let us formally define the parametrized family of operators for $t \in \mathbb{R}$
\[
\mathcal{U} \left( t \right) = 
\mathcal{F}_v \mathcal{T} \left( t \right) \mathcal{F}_v^{-1}
\]
Clearly $\mathcal{U} \left(t \right)$ acts boundedly on $L^2_{x,\eta} \left(
\mathbb{R}^2 \times \mathbb{R}^2 \right)$ for each $t \in \mathbb{R}$, with operator norm
equal to one.

The formal adjoint of $\mathcal{U} \left( t \right)$ is $\mathcal{U} \left( -t \right)$: here we are using
that the formal adjoint of $\mathcal{T} \left( t \right)$ is $\mathcal{T} \left( -t \right)$, \emph{regardless} of
whether the base field is $\mathbb{R}$ or $\mathbb{C}$. (This is due to the fact
that $\mathcal{T}$ commutes with complex conjugation.) Therefore, in order to apply the result of Keel and Tao
(\cite{KT1998}, Theorem 1.2), with indices (in their notation)
 \[
 \left( q,r,\sigma\right) = \left( 2,4,2\right)
 \]
we have only to prove for any Schwartz function $\zeta_0 \left( x, \eta \right)$ defined for $\left( x, \eta \right)
\in \mathbb{R}^2 \times \mathbb{R}^2$ the estimate
\begin{equation}
\label{eq:KTclaim}
\left\Vert \mathcal{U} \left( t \right) \zeta_0 \right\Vert_{L^\infty_{x,\eta} \left( \mathbb{R}^2 \times \mathbb{R}^2 \right)}
\leq C \left| t \right|^{-2} \left\Vert \zeta_0 \right\Vert_{L^1_{x,\eta} \left( \mathbb{R}^2 \times \mathbb{R}^2 \right)}
\end{equation}
any $t \neq 0$ to conclude.

In fact (\ref{eq:KTclaim}) follows from Lemma \ref{lem:CaP} (with $\left( p,r \right) = \left( \infty,1 \right)$ in the
notation of the Lemma as quoted above) interleaving \emph{two} careful applications of
 the Hausdorff-Young inequality, as we now show:
 \begin{equation*}
 \begin{aligned}
 \left\Vert \mathcal{U} \left( t \right) \zeta_0
 \right\Vert_{L^\infty_{x,\eta} \left( \mathbb{R}^2 \times \mathbb{R}^2 \right)} & =
 \left\Vert \mathcal{F}_v \mathcal{T} \left(t \right) \mathcal{F}_v^{-1} \zeta_0
 \right\Vert_{L^\infty_{x,\eta} \left( \mathbb{R}^2 \times \mathbb{R}^2 \right)} \\
 & \leq C \left\Vert \mathcal{T} \left( t \right) \mathcal{F}_v^{-1} \zeta_0
 \right\Vert_{L^\infty_x L^1_v \left( \mathbb{R}^2 \times \mathbb{R}^2 \right)} \\
 & \leq C \left| t \right|^{-2} \left\Vert \mathcal{F}_v^{-1} \zeta_0
 \right\Vert_{L^1_x L^\infty_v \left( \mathbb{R}^2 \times \mathbb{R}^2 \right)} \\
 & \leq C \left| t \right|^{-2} \left\Vert \zeta_0 \right\Vert_{L^1_{x,\eta} \left( \mathbb{R}^2 \times \mathbb{R}^2 \right)}
 \end{aligned}
 \end{equation*}

 We conclude by observing that the square-integrability of $\mathcal{F}_v f_0$ is equivalent to the
 square-integrability of $f_0$, by Plancherel.
\end{proof}

We turn to the main estimate upon which this entire article rests (originally obtained in the previous article of this series by
a slightly different proof \cite{CDP2019}).

\begin{proposition}
\label{prop:QplusBound}
For any $f_0, h_0 \in L^2$ there holds
\[
\left\Vert Q^+ \left( \mathcal{T} f_0, \mathcal{T} h_0
\right) \right\Vert_{L^1 \left( \mathbb{R}, L^2 \right)}
\leq C \left\Vert f_0 \right\Vert_{L^2}
\left\Vert h_0 \right\Vert_{L^2}
\]
\end{proposition}
\begin{proof}
By a result of Alonso and Carneiro (\cite{AC2010}, Theorem 1, with $\alpha=0$, $n=2$, $p=q=4$ and $r=2$), it holds
\begin{equation}
\begin{aligned}
& \left\Vert \mathcal{F}_v 
Q^+ \left( \mathcal{T} f_0, \mathcal{T} h_0 \right) \right\Vert_{L^2_\eta \left( \mathbb{R}^2 \right)}
\leq C \left\Vert \mathcal{F}_v \mathcal{T} f_0 \right\Vert_{L^4_\eta \left( \mathbb{R}^2 \right)}
\left\Vert \mathcal{F}_v \mathcal{T} h_0 \right\Vert_{L^4_\eta \left( \mathbb{R}^2 \right)}
\end{aligned}
\end{equation}
(Note carefully that in \cite{AC2010}, a radial symmetrization technique was used in the
\emph{proof}, but the theorem there makes no assumption of radiality.)
Therefore, by applying H{\" o}lder's inequality in $x$ followed by $t$, it holds for any interval $I$
\begin{equation}
\label{eq:fourierLebesgueBoundQPlus}
\begin{aligned}
& \left\Vert \mathcal{F}_v 
Q^+ \left( \mathcal{T} f_0, \mathcal{T} h_0 \right) \right\Vert_{L^1 \left( I, L^2_{x,\eta}
\left( \mathbb{R}^2 \times \mathbb{R}^2 \right)\right)} \\
& \qquad \qquad
\leq C \left\Vert \mathcal{F}_v \mathcal{T} f_0 \right\Vert_{L^2 \left( I, L^{4}_{x,\eta} \left(
\mathbb{R}^2 \times \mathbb{R}^2\right) \right)}
\left\Vert \mathcal{F}_v \mathcal{T} h_0 \right\Vert_{L^2 \left( I, L^{4}_{x,\eta} \left(
\mathbb{R}^2 \times \mathbb{R}^2\right) \right)}
\end{aligned}
\end{equation}
the constant being independent of $I$.
Of course by Plancherel
\[
\left\Vert \mathcal{F}_v
Q^+ \left( \mathcal{T} f_0, \mathcal{T} h_0 \right) \right\Vert_{L^1 \left( I, L^2_{x,\eta}
\left( \mathbb{R}^2 \times \mathbb{R}^2 \right)\right)}
= \left\Vert  
Q^+ \left( \mathcal{T} f_0, \mathcal{T} h_0 \right) \right\Vert_{L^1 \left( I, L^2 \right)}
\]

Combining (\ref{eq:fourierLebesgueBoundQPlus}) with Lemma \ref{lem:KeelTaoStrich} provides
\begin{equation}
\label{eq:sigmaBoundTXVSecond}
\begin{aligned}
& \left\Vert
 Q^+ \left( \mathcal{T} f_0 ,
\mathcal{T} h_0 \right) \right\Vert_{L^1\left( I, L^2\right)} \\
& \qquad\qquad\leq C
\left\Vert \mathcal{F}_v \mathcal{T} f_0 \right\Vert_{L^2 \left( I, L^{4}_{x,\eta}
\left( \mathbb{R}^2 \times \mathbb{R}^2 \right)\right)} \left\Vert 
\mathcal{F}_v \mathcal{T} h_0 \right\Vert_{L^2 \left( I, L^{4}_{x,\eta}
\left(  \mathbb{R}^2 \times \mathbb{R}^2 \right)\right)} \\
& \qquad\qquad\qquad\qquad \leq C
\left\Vert f_0 \right\Vert_{L^2} \left\Vert h_0 \right\Vert_{L^2}
\end{aligned}
\end{equation}
the conclusion being the special case $I=\mathbb{R}$.
\end{proof}

Small time versions will also be required, again having been first obtained in the preceding article.

\begin{proposition}
\label{prop:QplusSmallTime}
Let $f_0 \in L^2$.
There is a real-valued function 
\[
\delta_{f_0} \left(T\right) > 0
\]
defined for $T>0$, which (as indicated) depends only on
$f_0$, such that for $f_0$ fixed there holds
\[
\limsup_{T \rightarrow 0^+} \delta_{f_0} (T) = 0
\]
and for any $h_0 \in L^2$
there holds
\begin{equation}
\label{eq:QplusSmallTimeFirst}
\left\Vert Q^+ \left( \mathcal{T} f_0, \mathcal{T} h_0
\right) \right\Vert_{L^1 \left( J \left( T \right), L^2 \right)}
\leq \delta_{f_0} \left(T\right)
\left\Vert h_0 \right\Vert_{L^2}
\end{equation}
and
\begin{equation}
\label{eq:QplusSmallTimeSecond}
\left\Vert Q^+ \left( \mathcal{T} h_0, \mathcal{T} f_0
\right) \right\Vert_{L^1 \left(  J \left( T \right), L^2 \right)}
\leq \delta_{f_0} \left(T\right)
\left\Vert h_0 \right\Vert_{L^2}
\end{equation}
where $J \left( T \right) = \left[ -T,T\right]$.
\end{proposition}
\begin{proof}
This is a simple refinement of Proposition \ref{prop:QplusBound}.
Indeed, considering just (\ref{eq:QplusSmallTimeFirst}) (the proof of (\ref{eq:QplusSmallTimeSecond}) being similar),
 taking again the \emph{first} half of (\ref{eq:sigmaBoundTXVSecond}) now with $I = J\left( T \right) =
 \left[ -T,T \right]$, and to \emph{only} $h_0$ applying Lemma \ref{lem:KeelTaoStrich} followed by Plancherel, we have
\begin{equation*}
\begin{aligned}
& \left\Vert
 Q^+ \left( \mathcal{T} f_0 ,
\mathcal{T} h_0 \right) \right\Vert_{L^1\left( J\left(T\right), L^2\right)} \\
& \qquad\qquad\leq C
\left\Vert \mathcal{F}_v \mathcal{T} f_0 \right\Vert_{L^2 \left( J\left(T\right), L^{4}_{x,\eta}
\left( \mathbb{R}^2 \times \mathbb{R}^2 \right)\right)} \left\Vert 
\mathcal{F}_v \mathcal{T} h_0 \right\Vert_{L^2 \left( J\left(T\right), L^{4}_{x,\eta}
\left(  \mathbb{R}^2 \times \mathbb{R}^2 \right)\right)} \\
& \qquad\qquad\qquad\qquad \leq C
\left\Vert \mathcal{F}_v \mathcal{T} f_0
 \right\Vert_{L^2 \left( J\left(T\right), L^{4}_{x,\eta}
 \left( \mathbb{R}^2 \times \mathbb{R}^2 \right)\right)} \left\Vert h_0 \right\Vert_{L^2}
\end{aligned}
\end{equation*}
Then again, by Lemma \ref{lem:KeelTaoStrich} and Plancherel we have
\begin{equation*}
\left\Vert 
 \mathcal{F}_v \mathcal{T} f_0
\right\Vert_{L^2 \left( \mathbb{R}, L^{4}_{x,\eta}
\left( \mathbb{R}^2 \times \mathbb{R}^2 \right)\right)} \leq C
\left\Vert f_0 \right\Vert_{L^2}
\end{equation*}
so our hypothesis
\[
f_0 \in L^2
\]
implies
\[
\limsup_{T \rightarrow 0^+}
\left\Vert 
 \mathcal{F}_v \mathcal{T} f_0
\right\Vert_{L^2 \left( J\left( T \right) , L^{4}_{x,\eta}
\left( \mathbb{R}^2 \times \mathbb{R}^2 \right)\right)} = 0
\]
hence we may conclude.
\end{proof}

\subsection{Weights.}

We require weighted versions of the above estimates, particularly for the discussion
of weak-strong uniqueness in Section \ref{sec:WkStrong}. Indeed by
conservation of energy there holds for $\alpha \geq 0$
\[
\begin{aligned}
|v|^\alpha & \leq \left( |v|^2 + |v_*|^2 \right)^{\frac{1}{2}\alpha}\\
& = \left( \left| v^\prime \right|^2 + \left| v_*^\prime\right|^2
\right)^{\frac{1}{2}\alpha} \\
& \leq C_\alpha \left( \left| v^\prime \right|^\alpha +
\left| v_*^\prime \right|^\alpha \right)
\end{aligned}
\]
therefore
\begin{equation}
|v|^\alpha Q^+ \left(f,h\right) \leq C_\alpha \cdot \left(
Q^+ \left( |v|^\alpha \left| f \right|, \left| h\right| \right) + 
Q^+ \left( \left| f \right|, |v|^\alpha \left| h\right| \right) \right)
\end{equation}
hold pointwise a.e. $\left( t,x,v \right)$.
Hence, the following weighted estimates follow from the unweighted versions:

\begin{proposition}
\label{prop:QplusWeighted}
Let $\alpha \geq 0$. Then if $f_0, h_0$ are such that
\[
\left< v \right>^\alpha f_0, \; \left< v\right>^\alpha h_0 \in L^2
\]
then
\[
\left\Vert \left< v \right>^\alpha Q^+ \left(
\mathcal{T} f_0, \mathcal{T} h_0 \right)
\right\Vert_{L^1 \left( \mathbb{R}, L^2 \right)} \leq C_\alpha
\left\Vert \left< v \right>^\alpha f_0 \right\Vert_{L^2}
\left\Vert \left< v \right>^\alpha h_0 \right\Vert_{L^2}
\]
the constant $C_\alpha$ depending only on $\alpha$.
\end{proposition}

\begin{proposition}
\label{prop:QplusWeightedSmallTime}
Let $\alpha \geq 0$ and
let $f_0$ be such that
\[
\left< v \right>^\alpha f_0 \in L^2
\]
There is a real-valued function 
\[
\delta_{\alpha,f_0} \left(T\right) > 0
\]
defined for $T>0$, which  depends only on
$f_0$ and $\alpha$, such that for $f_0 , \alpha$ fixed there holds
\[
\limsup_{T \rightarrow 0^+} \delta_{\alpha,f_0} \left(T\right) = 0
\]
and for any $h_0$ with
$\left< v \right>^\alpha h_0 \in L^2$
there holds
\[
\left\Vert \left< v \right>^\alpha Q^+ \left( \mathcal{T} f_0, \mathcal{T} h_0
\right) \right\Vert_{L^1 \left( J \left( T \right), L^2 \right)}
\leq \delta_{\alpha,f_0} \left( T \right)
\left\Vert \left< v \right>^\alpha h_0 \right\Vert_{L^2}
\]
and
\[
\left\Vert \left< v \right>^\alpha Q^+ \left( \mathcal{T} h_0, \mathcal{T} f_0
\right) \right\Vert_{L^1 \left( J \left( T \right), L^2 \right)}
\leq \delta_{\alpha,f_0} \left(T\right)
\left\Vert \left< v \right>^\alpha h_0 \right\Vert_{L^2}
\]
where $J \left( T \right) = \left[ -T,T\right]$.
\end{proposition}

\subsection{Truncated weights.}
\label{ssec:truncatedWeights}

The weighted estimates can be truncated at large velocities: indeed, if we denote for $R>0$ the weight
\[
\nu_R \left( v \right) = \min \left( \left< v \right>, R \right)
\]
via pointwise minimum, and similarly for $\alpha \geq 0$ the shorthand
\[
\nu_R^\alpha =
\nu_R^\alpha \left( v \right)
= \nu_R \left( v \right)^\alpha = \min \left( \left< v \right>^\alpha, R^\alpha \right)
\]
then it is possible to show that
\begin{equation}
\label{eq:truncatedWeight}
\nu_R^\alpha \left( v \right) \leq C_\alpha \left( 
\nu_R^\alpha \left( v^\prime\right) +
\nu_R^\alpha \left( v_*^\prime \right) \right)
\end{equation}
To see this, consider first the case
\[
\max \left( \left< v^\prime \right>, \left< v_*^\prime \right> \right) < R
\]
in which case we can compute
\[
\nu_R^\alpha \left( v \right) \leq \left< v \right>^\alpha \leq
C_\alpha \left( \left< v^\prime \right>^\alpha +
\left< v_*^\prime \right>^\alpha \right) =
C_\alpha \left( \nu_R^\alpha \left( v^\prime \right) +
\nu_R^\alpha \left( v_*^\prime \right) \right)
\]
In the alternative case, we have
\[
\max \left( \left< v^\prime \right>, \left< v_*^\prime \right> \right) \geq R
\]
which implies we at least have one of
$\nu_R^\alpha \left( v^\prime \right) = R^\alpha$ or $\nu_R^\alpha \left( v_*^\prime \right) = R^\alpha$, so
we can similarly compute
\[
\nu_R^\alpha \left( v \right) \leq R^\alpha \leq
\max \left( \nu_R^\alpha \left( v^\prime \right), \nu_R^\alpha \left(
v_*^\prime \right) \right) \leq 
C_\alpha \left(
\nu_R^\alpha \left( v^\prime \right) +
\nu_R^\alpha \left( v_*^\prime \right)
\right)
\]
where we assume without loss of generality that $C_\alpha \geq 1$ in the last step.

Hence we have as before
\begin{proposition}
\label{prop:QplusTruncated}
Let $\alpha \geq 0$. Then if $f_0, h_0 \in L^2$ 
then for each $R > 0$ it holds
\[
\left\Vert \nu_R^\alpha Q^+ \left(
\mathcal{T} f_0, \mathcal{T} h_0 \right)
\right\Vert_{L^1 \left( \mathbb{R}, L^2 \right)} \leq C_\alpha
\left\Vert \nu_R^\alpha f_0 \right\Vert_{L^2}
\left\Vert \nu_R^\alpha h_0 \right\Vert_{L^2}
\]
the constant $C_\alpha$ depending only on $\alpha$; in particular, $C_\alpha$ is independent of $R$.
\end{proposition}

The small-time version of Proposition \ref{prop:QplusTruncated} is far more subtle. Indeed observe that
we need to have a single $\delta \left( T \right)$ that applies \emph{independent of $R$}, which does not
immediately follow from the proof of Proposition \ref{prop:QplusWeightedSmallTime} since that proof relies on an argument
involving the continuity of the integral, and would therefore have to be applied separately for
each value of $R$, yielding a $\delta \left(T\right)$ that implicitly depends on $R$. Instead,  to guarantee
the independence of $\delta \left(T \right)$ from $R$,
we elect to assume \emph{once and for all}
that
\[
\left< v \right>^\alpha f_0 \in L^2
\]
in other words that we \emph{do not} truncate $f_0$. In that case, $h_0$ can be freely truncated and therefore
it suffices to assume that $h_0 \in L^2$.

\begin{proposition}
\label{prop:QplusTruncatedSmallTime}
Let $\alpha \geq 0$ and
let $f_0$ be such that
\[
\left< v \right>^\alpha f_0 \in L^2
\]
There is a real-valued function 
\[
\delta_{\alpha,f_0} \left(T\right) > 0
\]
defined for $T>0$,  depending only on
$f_0$ and $\alpha$, such that for $f_0, \alpha$ fixed there holds
\[
\limsup_{T \rightarrow 0^+} \delta_{\alpha,f_0} \left(T\right) = 0
\]
and for any $h_0 \in L^2$
there holds, simultaneously for all $R>1$,
\begin{equation}
\label{eq:QplusTruncatedBdFirst}
\left\Vert \nu_R^\alpha Q^+ \left( \mathcal{T} f_0, \mathcal{T} h_0
\right) \right\Vert_{L^1 \left( J \left( T \right), L^2 \right)}
\leq \delta_{\alpha,f_0} \left( T \right)
\left\Vert \nu_R^\alpha h_0 \right\Vert_{L^2}
\end{equation}
and
\begin{equation}
\label{eq:QplusTruncatedBdSecond}
\left\Vert \nu_R^\alpha Q^+ \left( \mathcal{T} h_0, \mathcal{T} f_0
\right) \right\Vert_{L^1 \left( J \left( T \right), L^2 \right)}
\leq \delta_{\alpha,f_0} \left(T\right)
\left\Vert \nu_R^\alpha h_0 \right\Vert_{L^2}
\end{equation}
where $J \left( T \right) = \left[ -T,T\right]$.
\end{proposition}
\begin{proof}
Letting $f = \mathcal{T} f_0$ and $h = \mathcal{T} h_0$, we have
by (\ref{eq:truncatedWeight}) the \emph{pointwise} bound
\[
\nu_R^\alpha Q^+ \left(f,h\right) \leq C_\alpha \cdot \left(
Q^+ \left( \nu_R^\alpha \left| f \right|, \left| h\right| \right) + 
Q^+ \left( \left| f \right|, \nu_R^\alpha \left| h\right| \right) \right)
\]
But in the first term on the right-hand side we can bound $\nu_R^\alpha \leq \left< v \right>^\alpha$
in the first entry, and $1$ by $\nu_R^\alpha$ in the second entry (since $R > 1$),
so we obtain
\[
\nu_R^\alpha Q^+ \left(f,h\right) \leq C_\alpha \cdot \left(
Q^+ \left( \left< v \right>^\alpha \left| f \right|, \nu_R^\alpha \left| h\right| \right) + 
Q^+ \left( \left| f \right|, \nu_R^\alpha \left| h\right| \right) \right)
\]
Again, for the first entry of the second term we can bound $1$ by $\left< v \right>^\alpha$ so, multiplying
the constant by two, we obtain
\[
\nu_R^\alpha Q^+ \left(f,h\right) \leq C_\alpha 
Q^+ \left( \left< v \right>^\alpha \left| f \right|, \nu_R^\alpha \left| h\right| \right) 
\]
Therefore, for any compact interval $J \subset \mathbb{R}$ there holds
\begin{equation*}
\begin{aligned}
& \left\Vert \nu_R^\alpha Q^+ \left(\mathcal{T} f_0,
\mathcal{T} h_0\right) \right\Vert_{L^1 \left( J, L^2 \right)} \\
& \qquad \qquad \quad \leq C_\alpha 
\left\Vert
Q^+ \left( \mathcal{T} \left( \left< v \right>^\alpha \left| f_0 \right| \right), 
\mathcal{T} \left( \nu_R^\alpha \left| h_0 \right| \right) \right) 
\right\Vert_{L^1 \left( J, L^2 \right)}
\end{aligned}
\end{equation*}
where we have used the fact that $\mathcal{T}$ commutes with taking absolute values, and \emph{also}
commutes with multiplication by
 any scalar function of $\left| v \right|$. Finally, applying Proposition
 \ref{prop:QplusSmallTime} with 
 \[
 \left< v \right>^\alpha f_0 \quad \textnormal{ in place of }
 \quad f_0
 \]
(noting that $\left< v \right>^\alpha f_0 \in L^2$ by hypothesis), and
 \[
 \nu_R^\alpha h_0 \quad \textnormal{ in place of } \quad h_0
 \]
 implies (\ref{eq:QplusTruncatedBdFirst}). The proof of (\ref{eq:QplusTruncatedBdSecond}) is similar.
\end{proof}

\subsection{Time-dependent estimates.}

We can estimate $Q^+$ even when the arguments depend on time, not simply given by the free flow.

\begin{lemma}
\label{lem:timeDependentEstimate}
Let $0 \leq a < b < \infty$, $I = \left[ a, b \right]$,
and let $f_1, f_2, \zeta_1, \zeta_2$ be measurable functions such that
\[
f_1, f_2, \zeta_1, \zeta_2 \in
L^1_{\textnormal{loc}} \left( I\times \mathbb{R}^2 \times \mathbb{R}^2 \right)
\]
\[
\forall \left( i \in \left\{ 1,2 \right\} \right) \quad
f_i \in C \left( I, L^2 \right)
\]
\[
\forall \left( i \in \left\{ 1, 2\right\} \right) \quad
\zeta_i \in L^1 \left( I , L^2 \right)
\]
\[
\forall \left( i \in \left\{ 1,2 \right\} \right) \quad
\left( \partial_t + v \cdot \nabla_x \right) f_i = \zeta_i
\]

Then $Q^+ \left( f_1, f_2 \right) \in L^1 \left( I, L^2 \right)$ and we have the bound
\begin{equation}
\label{eq:mixedEstimateOne}
\begin{aligned}
& \left\Vert Q^+ \left( f_1, f_2 \right) \right\Vert_{L^1 \left( I, L^2 \right)} \\
& \quad \leq C \prod_{i \in \left\{ 1,2 \right\}} \left\Vert f_i \left( a \right) \right\Vert_{L^2} +
\sum_{i \in \left\{1,2\right\}}
q_i \left\Vert \zeta_i \right\Vert_{L^1 \left( I, L^2 \right)} +
C \prod_{i \in \left\{ 1,2 \right\}} \left\Vert \zeta_i
\right\Vert_{L^1 \left( I, L^2 \right)}
\end{aligned}
\end{equation}
where
\[
q_1 = 
\sup \left\{
\left\Vert Q^+ \left( 
\mathcal{T} \left( t-a \right) h_0,
\mathcal{T} \left( t-a \right) f_2 \left( a \right)
\right) \right\Vert_{L^1 \left( I, L^2 \right)}
\; : \; \left\Vert h_0 \right\Vert_{L^2} \leq 1
\right\}
\]
\[
q_2 = 
\sup \left\{
\left\Vert Q^+ \left( \mathcal{T} \left( t-a \right) f_1 \left( a \right),
\mathcal{T} \left( t-a \right) h_0
\right) \right\Vert_{L^1 \left( I, L^2 \right)}
\; : \; \left\Vert h_0 \right\Vert_{L^2} \leq 1
\right\}
\]
In particular, by Proposition \ref{prop:QplusBound},
\begin{equation}
\label{eq:mixedEstimateTwo}
\left\Vert Q^+ \left( f_1, f_2 \right) \right\Vert_{L^1 \left( I, L^2 \right)} \leq
C \prod_{i \in \left\{1,2 \right\}}
\left( \left\Vert f_i \right\Vert_{L^\infty \left( I, L^2 \right)} +
\left\Vert \zeta_i \right\Vert_{L^1 \left( I, L^2 \right)}\right)
\end{equation}
\end{lemma}

\begin{proof}
Expanding each $f_1, f_2$ by Duhamel's formula, we can decompose
\[
Q^+ \left( f_1, f_2 \right) = \mathcal{I}_1 + \mathcal{I}_2 + \mathcal{I}_3 + \mathcal{I}_4
\]
where
\[
\mathcal{I}_1 = Q^+ \left( \mathcal{T} \left( t-a \right) f_1 \left( a \right),
\mathcal{T} \left( t-a \right) f_2 \left( a \right) \right)
\]
\[
\mathcal{I}_2 = \int_a^t Q^+ \left( \mathcal{T} \left( t-\tau \right) \zeta_1 \left( \tau \right),
 \mathcal{T} \left( t-a \right) f_2 \left(a \right)
 \right)d\tau
\]
\[
\mathcal{I}_3 = \int_a^t Q^+ \left( \mathcal{T} \left( t-a \right) f_1 \left(a \right),
\mathcal{T} \left( t-\tau \right) \zeta_2 \left( \tau \right) \right) d\tau
\]
\[
\mathcal{I}_4 = \int_a^t \int_a^t Q^+ \left(
\mathcal{T} \left( t-\tau_1 \right) \zeta_1 \left( \tau_1 \right),
\mathcal{T} \left( t-\tau_2 \right) \zeta_2 \left( \tau_2 \right)
\right) d\tau_1 d\tau_2
\]
Proposition \ref{prop:QplusBound} provides
\[
\left\Vert \mathcal{I}_1 \right\Vert_{L^1 \left( I, L^2 \right)} \leq
C \prod_{i \in \left\{ 1,2 \right\}} \left\Vert f_i \left( a \right) \right\Vert_{L^2}
\]
The definitions of $q_i$, combined with Minkowski's inequality and the fact that free transport preserves
the $L^2$ norm, give us
\[
\left\Vert \mathcal{I}_2 \right\Vert_{L^1 \left( I, L^2 \right)} \leq
q_1 \left\Vert \zeta_1 \right\Vert_{L^1 \left( I, L^2 \right)}
\]
and
\[
\left\Vert \mathcal{I}_3 \right\Vert_{L^1 \left( I, L^2 \right)} \leq
q_2 \left\Vert \zeta_2 \right\Vert_{L^1 \left( I, L^2 \right)}
\]
For example,
\begin{equation*}
\begin{aligned}
& \left\Vert \mathcal{I}_2 \right\Vert_{L^1 \left( I, L^2 \right)} \\
& \quad = \left\Vert 
\int_a^t Q^+ \left( \mathcal{T} \left( t-\tau \right) \zeta_1 \left( \tau \right),
 \mathcal{T} \left( t-a \right) f_2 \left(a \right)
 \right)d\tau \right\Vert_{L^1 \left( I, L^2 \right)} \\
 & \quad \leq \left\Vert \int_a^t 
 \left\Vert Q^+ \left( \mathcal{T} \left( t-\tau \right) \zeta_1 \left( \tau \right),
 \mathcal{T} \left( t-a \right) f_2 \left(a \right)
 \right)\right\Vert_{L^2} d\tau \right\Vert_{L^1 \left( I, \mathbb{R} \right)} \\
 & \quad \leq \left\Vert \int_I
 \left\Vert Q^+ \left( \mathcal{T} \left( t-\tau \right) \zeta_1 \left( \tau \right),
 \mathcal{T} \left( t-a \right) f_2 \left(a \right)
 \right)\right\Vert_{L^2} d\tau \right\Vert_{L^1 \left( I, \mathbb{R} \right)} \\
  & \quad \leq \int_I
 \left\Vert Q^+ \left( \mathcal{T} \left( t-\tau \right) \zeta_1 \left( \tau \right),
 \mathcal{T} \left( t-a \right) f_2 \left(a \right)
 \right)\right\Vert_{L^1 \left( I, L^2\right)} d\tau  \\
 & \quad \leq \int_I q_1 
 \left\Vert \mathcal{T} \left( - \left( \tau - a \right) \right) \zeta_1 \left( \tau \right)
 \right\Vert_{L^2} d\tau \\
 & \quad = q_1 \left\Vert \zeta_1 \right\Vert_{L^1 \left( I, L^2 \right)}
 \end{aligned}
\end{equation*}
and $\mathcal{I}_3$ is similar.

For $\mathcal{I}_4$ we can use a similar estimate:
\begin{equation*}
\begin{aligned}
& \left\Vert \mathcal{I}_4 \right\Vert_{L^1 \left( I, L^2 \right)} \\
& \; = \left\Vert \int_a^t \int_a^t Q^+ \left(
\mathcal{T} \left( t-\tau_1 \right) \zeta \left( \tau_1 \right), \mathcal{T}
\left( t-\tau_2 \right) \zeta_2 \left( \tau_2 \right) \right)
d\tau_1 d\tau_2 \right\Vert_{L^1 \left( I, L^2 \right)} \\
& \; \leq \left\Vert \int_a^t \int_a^t \left\Vert Q^+ \left(
\mathcal{T} \left( t-\tau_1 \right) \zeta \left( \tau_1 \right), \mathcal{T}
\left( t-\tau_2 \right) \zeta_2 \left( \tau_2 \right) \right)\right\Vert_{L^2}
d\tau_1 d\tau_2 \right\Vert_{L^1 \left( I, \mathbb{R} \right)} \\
& \; \leq \left\Vert \int_I \int_I \left\Vert Q^+ \left(
\mathcal{T} \left( t-\tau_1 \right) \zeta \left( \tau_1 \right), \mathcal{T}
\left( t-\tau_2 \right) \zeta_2 \left( \tau_2 \right) \right)\right\Vert_{L^2}
d\tau_1 d\tau_2 \right\Vert_{L^1 \left( I, \mathbb{R} \right)} \\
& \; \leq \int_I \int_I \left\Vert Q^+ \left(
\mathcal{T} \left( t-\tau_1 \right) \zeta \left( \tau_1 \right), \mathcal{T}
\left( t-\tau_2 \right) \zeta_2 \left( \tau_2 \right) \right)\right\Vert_{L^1 \left( I, L^2 \right)}
d\tau_1 d\tau_2 \\
& \; \leq C \int_I \int_I \left\Vert
\mathcal{T} \left( - \left( \tau_1 - a \right) \right) \zeta_1 \left( \tau_1 \right)\right\Vert_{L^2}
\left\Vert
\mathcal{T} \left( - \left( \tau_2 - a \right) \right) \zeta_2 \left( \tau_2 \right)\right\Vert_{L^2}
d\tau_1 d\tau_2 \\
& \; \leq C \left\Vert \zeta_1 \right\Vert_{L^1 \left( I, L^2 \right)}
\left\Vert \zeta_2 \right\Vert_{L^1 \left( I, L^2 \right)}
\end{aligned}
\end{equation*}
\end{proof}

\subsection{Large time.} We will need the following variant of Proposition
\ref{prop:QplusSmallTime} for our discussion of scattering, namely the proof of Lemma \ref{lem:scattering-lemma}.

\begin{proposition}
\label{prop:QplusLargeTime}
Let $\varepsilon > 0$, $f_{+\infty} \in L^2$, and 
\[
I = \left[ 0, \infty \right)
\]
be provided.

Then there exist numbers $\delta > 0$, $T>0$, each $\delta,T$ depending only on
$\varepsilon,f_{+\infty}$, such that whenever $h_0 \in L^2$ satisfies
\[
\exists \left( t_0 \geq T \right) \quad
 \left\Vert h_0 - \mathcal{T} \left( t_0 \right) f_{+\infty}
\right\Vert_{L^2} < \delta
\]
then each of the following bounds hold:
\begin{equation}
\label{eq:QplusLargeTimeFirst}
\left\Vert Q^+ \left( \mathcal{T} h_0, \mathcal{T} h_0 \right) \right\Vert_{L^1 \left( I, L^2 \right)} <
\varepsilon
\end{equation}
\begin{equation}
\label{eq:QplusLargeTimeSecond}
\forall \left( g_0 \in L^2 \right) \quad
\left\Vert Q^+ \left( \mathcal{T} h_0, \mathcal{T} g_0 \right) \right\Vert_{L^1 \left( I , L^2 \right)} <
\varepsilon \left\Vert g_0 \right\Vert_{L^2}
\end{equation}
\begin{equation}
\label{eq:QplusLargeTimeThird}
\forall \left( g_0 \in L^2 \right) \quad
\left\Vert Q^+ \left( \mathcal{T} g_0, \mathcal{T} h_0 \right) \right\Vert_{L^1 \left( I , L^2 \right)} <
\varepsilon \left\Vert g_0 \right\Vert_{L^2}
\end{equation}
\end{proposition}
\begin{proof}
Assume without loss of generality that 
\[
\left\Vert f_{+\infty} \right\Vert_{L^2} = \frac{1}{2}
\]
Then (\ref{eq:QplusLargeTimeFirst}) follows from (\ref{eq:QplusLargeTimeSecond}) simply by taking $g_0 = h_0$,
as long as $\delta$  is at most $\frac{1}{2}$.
Therefore, we only need to
prove (\ref{eq:QplusLargeTimeSecond}), the proof of (\ref{eq:QplusLargeTimeThird}) being similar.

From (\ref{eq:fourierLebesgueBoundQPlus}) and Plancherel we know
\begin{equation}
\begin{aligned}
& \left\Vert 
Q^+ \left( \mathcal{T} h_0 , \mathcal{T} g_0 \right) \right\Vert_{L^1 \left( I, L^2 \right)} \\
& \qquad \qquad
\leq C \left\Vert \mathcal{F}_v \mathcal{T} h_0 \right\Vert_{L^2 \left( I, L^{4}_{x,\eta} \left(
\mathbb{R}^2 \times \mathbb{R}^2\right) \right)}
\left\Vert \mathcal{F}_v \mathcal{T} g_0 \right\Vert_{L^2 \left( I, L^{4}_{x,\eta} \left(
\mathbb{R}^2 \times \mathbb{R}^2\right) \right)}
\end{aligned}
\end{equation}
thus applying Lemma \ref{lem:KeelTaoStrich} to $g_0$ we obtain
\begin{equation}
\left\Vert 
Q^+ \left( \mathcal{T} h_0 , \mathcal{T} g_0 \right) \right\Vert_{L^1 \left( I, L^2 \right)} 
\leq C \left\Vert \mathcal{F}_v \mathcal{T} h_0 \right\Vert_{L^2 \left( I, L^{4}_{x,\eta} \left(
\mathbb{R}^2 \times \mathbb{R}^2\right) \right)}
\left\Vert g_0 \right\Vert_{L^2}
\end{equation}
So let us compute, using the triangle inequality followed by Lemma \ref{lem:KeelTaoStrich},
denoting $f_0 = \mathcal{T} \left( t_0 \right) f_{+\infty}$:
\begin{equation*}
\begin{aligned}
& \left\Vert \mathcal{F}_v \mathcal{T} h_0 \right\Vert_{L^2 \left( I, L^{4}_{x,\eta} \left(
\mathbb{R}^2 \times \mathbb{R}^2\right) \right)} \\
& \quad\leq \left\Vert \mathcal{F}_v \mathcal{T} \left( h_0 - f_0 \right) \right\Vert_{L^2 \left( I, L^{4}_{x,\eta} \left(
\mathbb{R}^2 \times \mathbb{R}^2\right) \right)} + 
\left\Vert \mathcal{F}_v \mathcal{T} f_0 \right\Vert_{L^2 \left( I, L^{4}_{x,\eta} \left(
\mathbb{R}^2 \times \mathbb{R}^2\right) \right)} \\
& \quad \leq C \left\Vert h_0 - f_0 \right\Vert_{L^2} + 
\left\Vert \mathcal{F}_v \mathcal{T} f_0 \right\Vert_{L^2 \left( I, L^{4}_{x,\eta} \left(
\mathbb{R}^2 \times \mathbb{R}^2\right) \right)} \\
& \quad < C \delta + \left\Vert \mathcal{F}_v \mathcal{T} \mathcal{T} \left( t_0 \right) f_{+\infty}
 \right\Vert_{L^2 \left( I, L^{4}_{x,\eta} \left(
\mathbb{R}^2 \times \mathbb{R}^2\right) \right)}
\end{aligned}
\end{equation*}
(note carefully the double $\mathcal{T}$ in the second term is \emph{not} a typo!) Thus provided
\[
C\delta < 2^{-1} \varepsilon
\]
 and picking a large enough $T$ that
\[
\left\Vert \mathcal{F}_v \mathcal{T} \mathcal{T} \left( T \right) f_{+\infty}
 \right\Vert_{L^2 \left( I, L^{4}_{x,\eta} \left(
\mathbb{R}^2 \times \mathbb{R}^2\right) \right)} < 2^{-1} \varepsilon
\]
(which is possible by Lemma \ref{lem:KeelTaoStrich} and monotone convergence as
$T \rightarrow \infty$, in view of the group property of $\mathcal{T}$) implies the result. Note carefully that once
$T$ is chosen sufficiently large, any $t_0 \geq T$ suffices to carry out the previous estimate: this justifies the
\emph{order} of quantifiers in the Lemma statement.
\end{proof}

\begin{remark}
It is interesting to note that the proof of Proposition \ref{prop:QplusLargeTime} tells us slightly more: namely (and
perhaps surprisingly), $\delta$ only depends on $\left\Vert f_{+\infty} \right\Vert_{L^2}$ (due to the normalization
condition at the start of the proof). It is only $T$ that depends on the profile of $f_{+\infty}$
 (as it must, by the scaling-criticality
of $L^2$).
\end{remark}

\subsection{Local temporal decomposition.} We can adapt the proof of Proposition \ref{prop:QplusSmallTime}
to handle \emph{intervals}, as opposed to \emph{neighborhoods of a point}, by decomposing any compact interval 
$\left[ 0, T \right]$ into
$N$ nonuniformly-sized sub-intervals, saving $\mathcal{O} \left( \varepsilon \right)$ on each interval by letting
$N$ be sufficiently large depending on $\varepsilon$. 
This will seem unmotivated here but will become crucial when we consider propagation of higher regularity, the
\emph{second} part of our main Theorem, and the decomposition leads naturally to propagation estimates like
\[
\left( 1 - \mathcal{O}\left(\varepsilon\right) \right)^{-\mathcal{O} \left( N \right)}
\]
so that a finite bound on $N$ is available for every $\varepsilon$ sufficiently small.

\begin{proposition}
\label{prop:QplusTimeDecomposition}
Let $0 < T < \infty$, $I = \left[ 0, T \right]$, and let 
\[
h, g \in C \left( I, L^2 \right)
\]
 be such that
\[
\left( \partial_t + v \cdot \nabla_x \right) h, \quad
\left( \partial_t + v \cdot \nabla_x \right) g \quad \in L^1 \left( I, L^2 \right)
\]
and define the constant
\[
C_0 \left( g \right) = \left\Vert g \right\Vert_{L^\infty \left( I , L^2 \right)} +
\left\Vert \left( \partial_t + v \cdot \nabla_x \right) g \right\Vert_{L^1 \left( I, L^2 \right)}
\]

Let $\varepsilon > 0$. Then there exists a number $N \in \mathbb{N}$ and a partition
\[
0 = t_0 < t_1 < t_2 < \dots < t_{N-1} < t_N = T
\]
the cardinality $N$ \emph{and} endpoints  $\left\{ t_j \right\}_j$ 
all depending on $g$ and $\varepsilon$ but \textbf{\emph{not}} on $h$, such that
denoting $I_j = \left[ t_j, t_{j+1} \right]$, $j = 0, 1, \dots, N-1$, there holds for each $j$ the estimate
\begin{equation*}
\label{eq:timeDecompositionBound}
\begin{aligned}
& \left\Vert Q^+ \left( h, g \right) \right\Vert_{L^1 \left( I_j, L^2 \right)} +
\left\Vert Q^+ \left( g, h \right) \right\Vert_{L^1 \left( I_j, L^2 \right)} \\
& \qquad \qquad \quad \leq C C_0 \left( g \right) \times \left( 
\left\Vert h \left( t_j \right) \right\Vert_{L^2} +
\varepsilon \left\Vert \left( \partial_t + v \cdot \nabla_x \right) h
\right\Vert_{L^1 \left( I_j, L^2 \right)}\right)
\end{aligned}
\end{equation*}
with $C$ an absolute constant (independent of $h, g, T, \varepsilon, N$ and all the $t_j$).
\end{proposition}
\begin{proof}
We may assume without loss of generality that each $g, h$ are non-negative almost everywhere, namely
\begin{equation}
\label{eq:nonNegativeAssumption001}
0 \leq g \left( t,x,v \right) \quad
\textnormal{a.e.} \; \left( t,x,v\right) \in I \times \mathbb{R}^2 \times \mathbb{R}^2
\end{equation}
\begin{equation}
\label{eq:nonNegativeAssumption002}
0 \leq h \left( t,x,v \right) \quad
\textnormal{a.e.} \; \left( t,x,v\right) \in I \times \mathbb{R}^2 \times \mathbb{R}^2
\end{equation}
 for, if we
have established that case, then for general $g,h$ we can simply apply the Lemma to $\left| g \right|, \left|
h \right|$, keeping in mind the \emph{pointwise} identities
\[
\left| \left( \partial_t + v \cdot \nabla_x \right) g \right| = \left|
\left( \partial_t + v \cdot \nabla_x \right) \left| g \right| \right|
\]
and
\[
\left| \left( \partial_t + v \cdot \nabla_x \right) h \right| = \left|
\left( \partial_t + v \cdot \nabla_x \right) \left| h \right| \right|
\]
as well as the \emph{pointwise} inequalities
\[
Q^+ \left( h, g \right) \leq Q^+ \left( \left| h \right|, \left| g \right| \right)
\]
and
\[
Q^+ \left( g, h \right) \leq Q^+ \left( \left| g \right|, \left| h \right| \right)
\]

Viewing $\tilde{g}$ as fixed, consider the linear operator
\[
\mathfrak{L}_{\tilde{g}} \tilde{h} = \mathfrak{L}\left\{ \tilde{g} \right\}
\tilde{h} =
Q^+ \left( \tilde{h}, \tilde{g} \right) + Q^+ \left( \tilde{g}, \tilde{h} \right)
\]
We will show, associating $g$ with $C_0 \left(g \right)$ as in the statement of the Proposition, that
\begin{equation*}
\begin{aligned}
& \left\Vert \mathfrak{L}_g  h  \right\Vert_{L^1 \left( I_j, L^2 \right)} \\
& \qquad \quad  \leq C C_0 \left( g \right) \times \left( 
\left\Vert h \left( t_j \right) \right\Vert_{L^2} +
\varepsilon \left\Vert \left( \partial_t + v \cdot \nabla_x \right) h
\right\Vert_{L^1 \left( I_j, L^2 \right)}\right)
\end{aligned}
\end{equation*}
for a suitable partition of $I = \left[ 0, T \right]$, as in the statement of the Proposition. Then in view
of (\ref{eq:nonNegativeAssumption001}-\ref{eq:nonNegativeAssumption002}) we have
\[
0 \leq Q^+ \left( h, g \right) \leq \mathfrak{L}_g h
\]
and
\[
0 \leq Q^+ \left( g, h \right) \leq \mathfrak{L}_g h
\]
so the conclusion follows.

Recall from the proof of Proposition
\ref{prop:QplusSmallTime} that for any interval 
\[
J = \left[ a,b \right] \subset I
\]
 and for any
$\tilde{g}_0, \tilde{h}_0 \in L^2$ and any $a^\prime, a^{\prime\prime} \in \mathbb{R}$ 
(neither being necessarily equal to $a$, which is
crucial) it holds
 \[
 \begin{aligned}
 &\left\Vert Q^+ \left( \mathcal{T} \left(t-a^\prime
 \right) \tilde{g}_0 , \mathcal{T}\left(t-a^{\prime\prime} \right) \tilde{h}_0
 \right) \right\Vert_{L^1 \left( J , L^2 \right)} \\
 &\qquad\qquad\leq
 C \left\Vert \mathcal{F}_v \left[ \mathcal{T} \left(t-a^\prime \right) \tilde{g}_0
 \right] \right\Vert_{L^2 \left( J, L^4_{x,\eta}
 \left( \mathbb{R}^2 \times \mathbb{R}^2 \right) \right)}
\left\Vert \tilde{h}_0 \right\Vert_{L^2}
 \end{aligned} 
 \]
 and
 \[
 \begin{aligned}
 &\left\Vert Q^+ \left( \mathcal{T} \left(t-a^{\prime\prime}
 \right) \tilde{h}_0 , \mathcal{T}\left(t-a^{\prime} \right) \tilde{g}_0
 \right) \right\Vert_{L^1 \left( J , L^2 \right)} \\
 &\qquad \qquad \leq
 C \left\Vert \mathcal{F}_v \left[ \mathcal{T} \left(t-a^\prime \right) \tilde{g}_0
 \right] \right\Vert_{L^2 \left( J, L^4_{x,\eta}
  \left( \mathbb{R}^2 \times \mathbb{R}^2 \right) \right)}
\left\Vert \tilde{h}_0 \right\Vert_{L^2}
 \end{aligned} 
 \]
 where $\mathcal{F}_v$ is the Fourier transform in $v$. Together these imply
 \begin{equation}
 \label{eq:frakLboundZero}
 \begin{aligned}
 &\left\Vert 
 \mathfrak{L} \left\{ \mathcal{T} \left( t - a^\prime \right) \tilde{g}_0 \right\}
 \left( \mathcal{T} \left( t - a^{\prime\prime} \right) \tilde{h}_0 \right) 
  \right\Vert_{L^1 \left( J , L^2 \right)} \\
 &\qquad \qquad \leq
 C \left\Vert \mathcal{F}_v \left[ \mathcal{T} \left(t-a^\prime \right) \tilde{g}_0
 \right] \right\Vert_{L^2 \left( J, L^4_{x,\eta}
  \left( \mathbb{R}^2 \times \mathbb{R}^2 \right) \right)}
\left\Vert \tilde{h}_0 \right\Vert_{L^2}
 \end{aligned} 
 \end{equation}
 and hence, by Lemma \ref{lem:KeelTaoStrich}, also
  \begin{equation}
 \label{eq:frakLboundOne}
 \begin{aligned}
 &\left\Vert 
 \mathfrak{L} \left\{ \mathcal{T} \left( t - a^\prime \right) \tilde{g}_0 \right\}
 \left( \mathcal{T} \left( t - a^{\prime\prime} \right) \tilde{h}_0 \right) 
  \right\Vert_{L^1 \left( J , L^2 \right)} \leq
 C \left\Vert \tilde{g}_0 \right\Vert_{L^2}
\left\Vert \tilde{h}_0 \right\Vert_{L^2}
 \end{aligned} 
 \end{equation}
We shall define
\[
\zeta = \left( \partial_t + v \cdot \nabla_x \right) g
\]
and
\[
\xi = \left( \partial_t + v \cdot \nabla_x \right) h
\]
which in particular provides
\begin{equation}
\label{eq:zetaBoundZero00}
\zeta, \xi \in L^1 \left( I, L^2 \right)
\end{equation}
by hypothesis. Moreover we may write
\begin{equation}
\label{eq:zetaBoundOne01}
C_0 \left( g \right) =
\left\Vert g\right\Vert_{L^\infty \left( I, L^2 \right)} +
\left\Vert \zeta \right\Vert_{L^1 \left( I, L^2 \right)}
\end{equation}

Let us decompose the interval $I = \left[0,T\right]$, for a sufficiently large integer
 $K \in \mathbb{N}$ to be chosen later, as
 \[
 0 = \tau_0 < \tau_1 < \tau_2 < \dots < \tau_{K-1} < \tau_K = T
 \]
 where 
 \[
 \left\Vert \zeta \right\Vert_{L^1 \left( J_k , L^2 \right)}
 = \frac{1}{K} 
 \left\Vert \zeta \right\Vert_{L^1 \left( I , L^2\right)}
 \]
 with $J_k = \left[ \tau_k, \tau_{k+1}\right]$.
 This is possible due to (\ref{eq:zetaBoundZero00}); observe, in particular, that the
 partition $\left\{ \tau_k \right\}_k$ depends on $g$ (which is in accordance with the statement of
 the Proposition). Now from (\ref{eq:zetaBoundOne01}) we have
 \begin{equation}
\label{eq:zetaBoundTwo02}
\left\Vert \zeta \right\Vert_{L^1 \left( J_k , L^2 \right)} \leq
K^{-1} C_0 \left( g \right)
\end{equation}

For each $k$ pick an positive integer $L \left(k \right)$, sufficiently large to be
chosen later, and times $\tau_k^\ell$ such that
\[
\tau_k = \tau_k^0 < \tau_k^1 < \tau_k^2 < \dots < \tau_k^{L \left( k \right)-1} < \tau_k^{L \left( k \right)}
= \tau_{k+1} 
\]
and 
\[
\begin{aligned}
& \left\Vert \mathcal{F}_v \left[ 
\mathcal{T} \left(t-\tau_k \right) g \left( \tau_k \right) \right]
\right\Vert^2_{L^2 \left( J_k^\ell , L^4_{x,\eta} \left( \mathbb{R}^2 \times \mathbb{R}^2 \right)\right)} \\
 &\qquad\qquad\qquad\qquad =
\frac{1}{L \left( k \right)}
\left\Vert \mathcal{F}_v \left[ 
\mathcal{T} \left(t-\tau_k\right) g\left(\tau_k\right) \right]
\right\Vert^2_{L^2 \left( J_k, L^4_{x,\eta}
\left( \mathbb{R}^2 \times \mathbb{R}^2 \right)\right)} 
\end{aligned}
\]
where the intervals $\left\{ J_k^\ell \right\}_\ell$,
 $J_k^\ell = \left[ \tau_k^\ell , \tau_k^{\ell+1}\right]$, partition $J_k$. (Note carefully the \emph{squares} in
 the defining relation for $J_k^\ell$.)
In particular, letting
\[
L = \inf \left\{ L \left( k \right) \; : \; k \in \left\{ 0, 1, 2, \dots, K-1 \right\} \right\}
\]
we have by Lemma \ref{lem:KeelTaoStrich}
\[
\left\Vert \mathcal{F}_v \left[ 
\mathcal{T} \left(t-\tau_k \right) g \left( \tau_k \right) \right]
\right\Vert^2_{L^2 \left( J_k^\ell , L^4_{x,\eta} \left( \mathbb{R}^2 \times \mathbb{R}^2 \right)\right)} \\
\leq \frac{C}{L} \left\Vert g \left( \tau_k \right) \right\Vert^2_{L^2}
\]
hence
\begin{equation}
\label{eq:fourierSmallBound001}
\left\Vert \mathcal{F}_v \left[ 
\mathcal{T} \left(t-\tau_k \right) g \left( \tau_k \right) \right]
\right\Vert_{L^2 \left( J_k^\ell , L^4_{x,\eta} \left( \mathbb{R}^2 \times \mathbb{R}^2 \right)\right)} \\
\leq C L^{-\frac{1}{2}}
C_0 \left( g \right)
\end{equation}
\begin{equation}
\end{equation}

Duhamel's formula for $t \in J_k$ reads
\[
g(t) = \mathcal{T} \left(t-\tau_k\right) g \left(\tau_k\right) +
\int_{\tau_k}^t \mathcal{T} \left(t-s\right) \zeta \left( s \right) ds
\]
Additionally, for $t \in J_k^\ell \subset J_k$, we have
\[
h \left( t \right) =
\mathcal{T}\left(t-\tau_k^\ell \right) h \left(\tau_k^\ell\right) +
\int_{\tau_k^\ell}^t
\mathcal{T} \left(t-s\right)\xi \left(s\right) ds
\]

For $t \in J_k^\ell$ we may plug the two Duhamel formulas recorded above into $\mathfrak{L}_g h$:
\[
\mathfrak{L}_g h = \mathcal{I}_1 + \mathcal{I}_2 + \mathcal{I}_3 + \mathcal{I}_4
\]
where
\[
\mathcal{I}_1 =
\mathfrak{L} \left\{
 \mathcal{T} \left( t - \tau_k \right) g \left( \tau_k \right)
\right\}
\left( \mathcal{T} \left( t - \tau_k^\ell \right) h \left( \tau_k^\ell \right)  \right)
\]
\[
\mathcal{I}_2 =
\int_{\tau_k}^t ds \mathfrak{L} \left\{
 \mathcal{T} \left( t-s \right) \zeta \left( s \right)
 \right\}
\left(
\mathcal{T} \left( t - \tau_k^\ell\right) h \left( \tau_k^\ell \right) \right)
\]
\[
\mathcal{I}_3 = 
\int_{\tau_k^\ell}^t ds 
\mathfrak{L} \left\{  \mathcal{T} \left( t-\tau_k \right) g \left( \tau_k \right)  \right\}
\left( \mathcal{T} \left( t-s \right) \xi \left( s\right)\right)
\]
\[
\mathcal{I}_4 = \int_{\tau_k}^t ds \int_{\tau_k^\ell}^t ds^\prime
\mathfrak{L} \left\{ \mathcal{T} \left( t-s \right) \zeta \left( s \right)  \right\} \left(
\mathcal{T} \left( t - s^\prime \right)
\xi \left( s^\prime \right)\right)
\]

In what follows we will freely reduce (without comment) expressions like
 $\left\Vert \mathcal{T} \left( \alpha \right)
\left( \cdot \right) \right\Vert_{L^2}$ to simply
$\left\Vert \cdot \right\Vert_{L^2}$ for any $\alpha \in \mathbb{R}$, for the sake of brevity.
Also, since $t \in J_k^\ell$,
 we will freely replace integrals like $\int_{\tau_k}^t$ resp. $\int_{\tau_k^\ell}^t$ by
$\int_{J_k}$ resp. $\int_{J_k^\ell}$, as warranted by Minkowski's inequality applied to the
inner $L^2$ alone.

In the case of $\mathcal{I}_1$ we may simply use (\ref{eq:frakLboundOne}):
\[
\left\Vert \mathcal{I}_1 \right\Vert_{L^1 \left( J_k^\ell , L^2 \right)} \leq
C C_0 \left( g \right) \left\Vert h \left( \tau_k^\ell \right) \right\Vert_{L^2}
\]
Similarly for $\mathcal{I}_2$ we again have (\ref{eq:frakLboundOne}):
\[
\left\Vert \mathcal{I}_2 \right\Vert_{L^1 \left( J_k^\ell, L^2 \right)} \leq
\int_{J_k} ds C \left\Vert \zeta \left( s \right) \right\Vert_{L^2} 
\left\Vert h \left( \tau_k^\ell \right) \right\Vert_{L^2} \leq 
C C_0 \left( g \right) \left\Vert h \left( \tau_k^\ell \right) \right\Vert_{L^2}
\]

For $\mathcal{I}_4$, by (\ref{eq:frakLboundOne}) again, along with
(\ref{eq:zetaBoundTwo02}),
\begin{equation*}
\begin{aligned}
\left\Vert \mathcal{I}_4 \right\Vert_{L^1 \left( J_k^\ell , L^2 \right)} & \leq
\int_{J_k} ds
\int_{J_k^\ell} ds^\prime 
C \left\Vert \zeta \left( s \right)  \right\Vert_{L^2}
\left\Vert \xi \left( s^\prime \right) \right\Vert_{L^2} \\
& \leq C 
K^{-1} C_0 \left( g \right) \left\Vert \xi
\right\Vert_{L^1 \left( J_k^\ell , L^2 \right)}
\end{aligned}
\end{equation*}

Lastly, and most technically, for $\mathcal{I}_3$, by
(\ref{eq:frakLboundZero}) and (\ref{eq:fourierSmallBound001}), we have
\begin{equation*}
\begin{aligned}
&  \left\Vert \mathcal{I}_3 \right\Vert_{L^1 \left( J_k^\ell , L^2 \right)} \\
 & \qquad  \leq
\int_{J_k^\ell} ds
C \left\Vert \mathcal{F}_v \left[ \mathcal{T} \left(t-\tau_k \right) g \left( \tau_k \right)
 \right] \right\Vert_{L^2 \left( J_k^\ell , L^4_{x,\eta}
  \left( \mathbb{R}^2 \times \mathbb{R}^2 \right) \right)}
  \left\Vert \xi \left( s \right) \right\Vert_{L^2} \\
  & \qquad \leq C L^{-\frac{1}{2}} C_0 \left( g \right) \left\Vert \xi \right\Vert_{L^1 \left( J_k^\ell, L^2 \right)}
\end{aligned}
\end{equation*}

Altogether we have
\begin{equation*}
\begin{aligned}
& \left\Vert \mathcal{L}_g h \right\Vert_{L^1 \left( J_k^\ell, L^2 \right)} \\
& \qquad \leq
C C_0 \left( g \right) \times \left(
\left\Vert h \left( \tau_k^\ell \right) \right\Vert_{L^2} +
\left( K^{-1} + L^{-\frac{1}{2}} \right)
\left\Vert \xi \right\Vert_{L^1 \left( J_k^\ell , L^2 \right)}
\right)
\end{aligned}
\end{equation*}
Recalling that $\xi = \left( \partial_t + v \cdot \nabla_x \right) h$,
letting $K^{-1}$ and $L^{-\frac{1}{2}}$ each be smaller
than $2^{-1} \varepsilon$, and identifying the partition
$\left\{ I_j \right\}_j$ of cardinality $N = K  L$ with the partition
$\left\{ J_k^\ell \right\}_{k,\ell}$ provides the result.
\end{proof}

\begin{corollary}
\label{cor:QplusTimeDecompositionSecond}
Fix an integer $M \in \mathbb{N}$. Then Proposition \ref{prop:QplusTimeDecomposition} holds again under the
added constraint that, for each $j$,
\[
\left| t_{j+1} - t_j \right| < \frac{1}{M}
\]
\end{corollary}
\begin{proof}
Partition $I = \left[ 0,T \right]$ into $M$ intervals $I_m$ where
\[
I_m = \left[ \frac{m}{M} T, \frac{m+1}{M} T \right]
\]
Then apply Proposition \ref{prop:QplusTimeDecomposition} to the intervals $I_m$ in succession, starting
with $m=0$ and ending with $m = M-1$.
\end{proof}

\section{Estimates with non-negativity} 
\label{sec:nonNegativeEstimates}

Lemma \ref{lem:timeDependentEstimate} can be refined under non-negativity assumptions: we do not need to
assume that
\[
\left( \partial_t + v \cdot \nabla_x \right) f_i \in L^1 \left( I, L^2 \right)
\]
as  long as the $f_i$ are each non-negative and we have some control \emph{from above} in
Duhamel's formula. This will be useful
for the proof of weak-strong uniqueness, Theorem \ref{thm:unique}.

\begin{lemma}
\label{lem:nonnegativeEstimate}
Let $0 \leq a < b < \infty$, $I = \left[ a, b \right]$,
and let $f_1, f_2, \zeta_1, \zeta_2$ be non-negative measurable functions such that
\[
f_1, f_2, \zeta_1, \zeta_2 \in
L^1_{\textnormal{loc}} \left( I\times \mathbb{R}^2 \times \mathbb{R}^2 \right)
\]
\[
\forall \left( i \in \left\{ 1,2 \right\} \right) \quad
0 \leq f_i \in C \left( I, L^2 \right)
\]
\[
\forall \left( i \in \left\{ 1, 2\right\} \right) \quad
0 \leq \zeta_i \in L^1 \left( I , L^2 \right)
\]
and that for almost every $\left(t,x,v \right) \in I \times \mathbb{R}^2 \times
\mathbb{R}^2$ we have the pointwise bounds for each $i \in \left\{ 1,2 \right\}$
\[
0 \leq f_i \left( t \right)
\leq \mathcal{T} \left( t-a \right) f_i \left( a \right) +
\int_a^t \mathcal{T} \left( t-\tau \right) \zeta_i \left( \tau \right) d\tau
\]

Then $Q^+ \left( f_1, f_2 \right) \in L^1 \left( I, L^2 \right)$ and we have the bound
\begin{equation}
\label{eq:nonnegativeEstimateOne}
\begin{aligned}
& \left\Vert Q^+ \left( f_1, f_2 \right) \right\Vert_{L^1 \left( I, L^2 \right)} \\
& \quad \leq C \prod_{i \in \left\{ 1,2 \right\}} \left\Vert f_i \left( a \right) \right\Vert_{L^2} +
\sum_{i \in \left\{1,2\right\}}
q_i \left\Vert \zeta_i \right\Vert_{L^1 \left( I, L^2 \right)} +
C \prod_{i \in \left\{ 1,2 \right\}} \left\Vert \zeta_i
\right\Vert_{L^1 \left( I, L^2 \right)}
\end{aligned}
\end{equation}
where
\[
q_1 = 
\sup \left\{
\left\Vert Q^+ \left( 
\mathcal{T} \left( t-a \right) h_0,
\mathcal{T} \left( t-a \right) f_2 \left( a \right)
\right) \right\Vert_{L^1 \left( I, L^2 \right)}
\; : \; \left\Vert h_0 \right\Vert_{L^2} \leq 1
\right\}
\]
\[
q_2 = 
\sup \left\{
\left\Vert Q^+ \left( \mathcal{T} \left( t-a \right) f_1 \left( a \right),
\mathcal{T} \left( t-a \right) h_0
\right) \right\Vert_{L^1 \left( I, L^2 \right)}
\; : \; \left\Vert h_0 \right\Vert_{L^2} \leq 1
\right\}
\]
In particular, by Proposition \ref{prop:QplusBound},
\begin{equation}
\label{eq:nonnegativeEstimateTwo}
\left\Vert Q^+ \left( f_1, f_2 \right) \right\Vert_{L^1 \left( I, L^2 \right)} \leq
C \prod_{i \in \left\{1,2 \right\}}
\left( \left\Vert f_i \right\Vert_{L^\infty \left( I, L^2 \right)} +
\left\Vert \zeta_i \right\Vert_{L^1 \left( I, L^2 \right)}\right)
\end{equation}
\end{lemma}
\begin{proof}
For $i=1,2$ let us define for $t \in I$
\[
h_i \left( t \right) =
\mathcal{T} \left( t-a \right) f_i \left( a \right) +
\int_a^t \mathcal{T} \left( t - \tau \right)
\zeta_i \left( \tau \right) d\tau
\]
Then for almost every $\left(t,x,v \right) \in I \times \mathbb{R}^2 \times
\mathbb{R}^2$ and each $i=1,2$ we have the pointwise bound
\[
0 \leq f_i \leq h_i
\]
so it suffices to show
\begin{equation*}
\begin{aligned}
& \left\Vert Q^+ \left(h_1, h_2\right) \right\Vert_{L^1 \left( I, L^2 \right)} \\
& \quad \leq C \prod_{i \in \left\{ 1,2 \right\}} \left\Vert f_i \left( a \right) \right\Vert_{L^2} +
\sum_{i \in \left\{1,2\right\}}
q_i \left\Vert \zeta_i \right\Vert_{L^1 \left( I, L^2 \right)} +
C \prod_{i \in \left\{ 1,2 \right\}} \left\Vert \zeta_i
\right\Vert_{L^1 \left( I, L^2 \right)}
\end{aligned}
\end{equation*}
but this now follows from Lemma \ref{lem:timeDependentEstimate}.
\end{proof}

\section{The $Q^+$ equation}
\label{sec:QplusSection}

A local solution of the Boltzmann equation with gain term only, or
\emph{gain-only Boltzmann equation}, provides (in suitable regularly classes)
a \emph{local upper envelope} to solutions of
(\ref{eq:BE}) with the same initial data. (The same can be said for a small forward
interval of any $t_0$, say $\left[ t_0, t_0 + \varepsilon \right)$,
 taking the solution $f \left( t_0 \right)$ of (\ref{eq:BE}) at
time $t_0$ as the initial data for the $Q^+$ equation.) 
The main objective of this section is to provide a
detailed understanding of the gain-only Boltzmann eqaution, as a means for characterizing
such a local upper envelope.

\subsection{The gain-only equation.}

The $Q^+$ equation, or gain-only Boltzmann equation, or simply \emph{the gain-only equation}, refers to the following
evolutionary equation:
\begin{equation}
\label{eq:QplusEq}
\left( \partial_t + v \cdot \nabla_x \right) h = Q^+ \left( h,h \right)
\end{equation}
and this equation (\ref{eq:QplusEq}) will be the sole concern of this section.
 Note carefully that the space $L^2$, \emph{not}
$L^2 \bigcap L^1_2$, will be the relevant functional setting for the study of (\ref{eq:QplusEq}).

\begin{theorem}
\label{thm:QPlusLWP}
Given any $0 \leq h_0 \in L^2$, the gain-only equation (\ref{eq:QplusEq}) admits a unique local solution
\[
h \in C \left( \left[ 0, T \right], L^2 \right)
\]
satisfying
\begin{equation}
\label{eq:KMboundZero}
Q^+ \left( h,h \right) \in L^1 \left( \left[ 0, T\right], L^2 \right)
\end{equation}
and $h \left( t=0 \right) = h_0$,
the time $T$ depending on the profile of $h_0$. (In particular, the uniqueness assertion is conditional
on the bound (\ref{eq:KMboundZero}) for $Q^+$ as applied to any candidate solution of (\ref{eq:QplusEq}): the constructed
solution satisfies (\ref{eq:KMboundZero}) regardless.)
Additionally, for any $r,p \in \left[ 1,\infty \right]$ such that $r > 2$ and
$\frac{1}{r} = 1 - \frac{2}{p}$, it holds
\begin{equation}
\label{eq:GOstrich}
h \in L^r_t L^p_x L^{p^\prime}_v \left( \left[ 0, T \right] \times
\mathbb{R}^2_x \times \mathbb{R}^2_v \right)
\end{equation}
 There is a number $\eta_0$, $0 < \eta_0 < \infty$,
such that if $\left\Vert h_0 \right\Vert_{L^2} < \eta_0$ then we may take $T = \infty$.
\end{theorem} 

\begin{remark}
The small data regime, characterized by the number $\eta_0$ in Theorem
\ref{thm:QPlusLWP}, was previously studied in \cite{CDP2019}.
\end{remark}

\begin{proof}
This follows from
Proposition \ref{prop:QplusBound}, Proposition \ref{prop:QplusSmallTime},
and Theorem \ref{thm:CritLWP}, taking $\mathfrak{G} = L^2$ and
\[
\mathcal{A} (t, f_0, h_0) =
\mathcal{T} (-t) Q^+ \left( \mathcal{T} (t) f_0,
\mathcal{T} (t) h_0 \right)
\]
where we have implicitly employed the change of variables 
\[
\tilde{h} \left( t \right) = \mathcal{T} 
\left( -t \right) h \left( t \right)
\]
 to formally write for any solution $h$ of (\ref{eq:QplusEq}) that
\[
\partial_t \tilde{h} \left(t \right) = \mathcal{A} \left( t,
 \tilde{h} \left( t \right), \tilde{h} \left( t \right) \right)
\]

To see that $h \in C \left( \left[ 0, T \right], L^2 \right)$, observe by Duhamel's formula
\[
\mathcal{T} \left( -t \right) h \left( t \right) - \mathcal{T} \left( -s \right) h \left( s \right) =
\int_s^t \mathcal{T} \left( - \sigma \right) Q^+ \left( h , h \right) \left( \sigma \right) d\sigma
\]
we can bound by Minknowski's inequality
\[
\left\Vert
\mathcal{T} \left( -t \right) h \left( t \right) - \mathcal{T} \left( -s \right) h \left( s \right)
\right\Vert_{L^2}\leq
\int_s^t
\left\Vert  Q^+ \left( h , h \right) \left( \sigma \right)
\right\Vert_{L^2} d\sigma
\]
where we have used the fact that $\mathcal{T}$ preserves the $L^2$ norm. Therefore the time-continuity
of $\mathcal{T} \left( -t \right) h \left( t \right)$, and hence $h$ itself, follows from
(\ref{eq:KMboundZero}).

The bound (\ref{eq:GOstrich}) follows from Proposition \ref{prop:CaP},
as follows: first, note that by Duhamel's formula the solution $h$ of (\ref{eq:QplusEq}) satisfies
for $0 \leq t \leq T$
\[
\begin{aligned}
h \left( t \right) &
= \mathcal{T} \left( t \right) h_0 + \int_0^t \mathcal{T} \left( t-s \right) Q^+ \left(h,h\right)\left(s\right) ds \\
& \leq 
\mathcal{T} \left( t \right) h_0+ \int_0^{T} \mathcal{T} \left( t-s \right)
 Q^+ \left( h,h\right)\left(s\right) ds 
\end{aligned}
\]
where we have replaced $t$ by $T$ in the limits of integration; hence, by Minkowski's inequality
\[
\begin{aligned}
& \left\Vert h \right\Vert_{L^r_t L^p_x L^{p^\prime}_v 
\left( \left[ 0,T\right] \times \mathbb{R}^2_x \times \mathbb{R}^2_v \right)}\\
& \leq
\left\Vert \mathcal{T} h_0 
\right\Vert_{L^r_t L^p_x L^{p^\prime}_v 
\left( \left[ 0,T \right] \times \mathbb{R}^2_x \times \mathbb{R}^2_v \right)}\\
& \qquad \qquad + \int_0^{T}
\left\Vert \mathcal{T}\left( t-s \right) Q^+ \left(h,h\right) \left(s\right) 
\right\Vert_{L^r_t L^p_x L^{p^\prime}_v 
\left( \left[ 0,T \right] \times \mathbb{R}^2_x \times \mathbb{R}^2_v \right)} ds \\
& \leq \left\Vert h_0 \right\Vert_{L^2} +
\int_0^{T} \left\Vert Q^+ \left(h,h\right) \left(s\right) \right\Vert_{L^2} ds
\end{aligned}
\]
and recall that $Q^+ \left(h,h\right) \in L^1 \left( \left[0,T\right], L^2 \right)$.
\end{proof}

\begin{remark}
Since the initial data $h_0$ is non-negative almost
everywhere, the solution of the gain-only equation is again
non-negative almost everywhere for positive times inside the domain of
existence. To see this, expand the solution $h \left( t \right)$ in ``powers'' of
$h_0$ by iterating Duhamel's formula \emph{ad infinitum}. Every term of the resulting series
is non-negative by the non-negativity of $h_0$ and $Q^+$,
and the series is guaranteed to converge to $h$ by the
proof of Theorem \ref{thm:CritLWP}.
\end{remark}

\begin{definition}
Given any $0 \leq h_0 \in L^2$, let $S \left( h_0 \right)$ be the set of numbers $T \in
\left( 0, \infty \right)$ such that there
exists a solution $h$ of the gain-only equation (\ref{eq:QplusEq}) with
\[
h \in C \left( \left[ 0, T \right], L^2 \right)
\]
satisfying
\[
Q^+ \left( h,h \right) \in L^1 \left( \left[ 0, T \right], L^2 \right)
\]
and $h \left( t=0 \right) = h_0$. (Note that $h$ is, as before, necessarily non-negative.)

We also note that $S \left( h_0 \right)$ is a connected subset of $\left( 0, \infty \right)$ with
nonempty interior, by Theorem \ref{thm:QPlusLWP}. 

We shall denote by
\[
T_{\textnormal{g.o.}} \left( h_0 \right) = \sup S \left( h_0 \right) \in
\left( 0, \infty \right]
\]
what we shall call the \textbf{scaling-critical
time of existence for the gain-only equation} for the initial data $h_0$.
\end{definition}

\begin{remark}
By the definition of $T_{\textnormal{g.o.}} \left(h_0\right)$ and uniqueness,
 the solution $h \left(t\right)$ guaranteed by
Theorem \ref{thm:QPlusLWP} is continued for $0 \leq t \leq T$, any
$0 < T< T_{\textnormal{g.o.}} \left( h_0\right)$. It is obvious from the proof
of Theorem \ref{thm:QPlusLWP} and the definition of $T_{\textnormal{g.o.}}
\left( h_0 \right)$ that
(\ref{eq:GOstrich}) holds for any $0 < T < T_{\textnormal{g.o.}} \left(h_0\right)$.
\end{remark}

Henceforth we shall always
 take the initial data $h_0$ for the gain-only equation
  (\ref{eq:QplusEq}) to be non-negative at almost every point of its domain.
For $0 \leq t < T_{\textnormal{g.o.}} \left( h_0 \right)$ we define
\[
\mathfrak{Z}_{\textnormal{g.o.}} \left( h_0 \right) \left(t \right)
\]
to be the unique solution of the gain-only equation 
(\ref{eq:QplusEq}), as specified in the definition
of $T_{\textnormal{g.o.}}\left( h_0 \right)$, corresponding to the initial data $h_0$. In particular,
\[
\left( \partial_t + v \cdot \nabla_x \right) \left\{
\mathfrak{Z}_{\textnormal{g.o.}} \left( h_0 \right) \left( t \right) \right\} =
Q^+ \left( \mathfrak{Z}_{\textnormal{g.o.}} \left( h_0 \right) \left( t \right),
\mathfrak{Z}_{\textnormal{g.o.}} \left( h_0 \right) \left( t \right) \right)
\]
and $\mathfrak{Z}_{\textnormal{g.o.}} \left( h_0 \right) \left( 0 \right) = h_0$. 
Therefore, $\mathfrak{Z}_{\textnormal{g.o.}}$
satisfies a restricted version of the semigroup property, which holds precisely to the extent that the
flow is defined as above; we refer to this property as simply \emph{the semigroup property} of
$\mathfrak{Z}_{\textnormal{g.o.}}$.

\subsection{Lower semi-continuity}
\label{sec:lsc}

For what follows we define $L^{2,+}$ to be the set of
functions $h_0 \in L^2$ such that $ h_0 \left( x,v \right) \geq 0$ a.e. 
$\left( x,v \right)$. $L^{2,+}$ is topologized by the $L^2$ norm of the pointwise difference between
two elements, unless stated
otherwise. When we refer to \emph{lower semi-continuity} without further qualification, we \emph{always}
(from here to the end of the article) mean this term in reference to the $L^2$ \emph{\textbf{norm}} topology.

Our ultimate goal is to prove that $T_{\textnormal{g.o.}}$ is lower semi-continuous, and that the solution
map $\mathfrak{Z}_{\textnormal{g.o.}}$ is itself continuous in a suitable sense. The first step will be
the construction of a family of \emph{lower semi-continuous lower bounds} for $T_{\textnormal{g.o.}}$, parameterized
by $\varepsilon > 0$. In other words, once we fix an $\varepsilon$, we can obtain from this a lower semi-continuous
function which bounds $T_{\textnormal{g.o.}}$ from below, and satisfies an additional $\varepsilon$-dependent
bound. This function, to be constructed momentarily, shall be denoted $F^{\left( \varepsilon \right)}$.

It will be convenient to abbreviate
\[
Q^+ \left( f,f \right)
\]
as 
\[
Q^+ \left( f \right)
\]
and we will do so without further comment.

\begin{lemma}
\label{lem:lsc}
Let $\varepsilon > 0$. Then there exists a function
\[
F^{\left(\varepsilon\right)} : L^{2,+} \rightarrow \mathbb{R} \bigcup 
\left\{ +\infty \right\} 
\]
such that each of the following is true:
\begin{enumerate}
\item For any $h_0 \in L^{2,+}$,
\[
0 < F^{\left(\varepsilon\right)} \left(h_0\right) \leq T_{\textnormal{g.o.}} \left(h_0\right)
\]
\item If $h_0 \in L^{2,+}$ and $h_{0,k} \in L^{2,+}$ for $k = 1,2,3,\dots$ then
\[
\lim_{k \rightarrow \infty} \left\Vert h_{0,k} - h_0 \right\Vert_{L^2} = 0
\quad \implies \quad F^{\left(\varepsilon\right)} \left(h_0\right) \leq
\liminf_{k \rightarrow \infty} F^{\left(\varepsilon\right)} \left( h_{0,k} \right)
\]
\item For any $h_0 \in L^{2,+}$,
\begin{equation}
\label{eq:F-eps-bd}
\int_0^{F^{\left(\varepsilon\right)} \left(h_0\right)}
\left\Vert Q^+ \left( \mathfrak{Z}_{\textnormal{g.o.}} \left(h_0\right) \left(t\right) \right)
\right\Vert_{L^2} dt \leq \varepsilon
\end{equation}
\end{enumerate}
\end{lemma}
\begin{proof}
First observe that if $T_{\textnormal{g.o.}} \left(h_0\right) < \infty$ then
\begin{equation}
\label{eq:GOInfinity}
\int_0^{T_{\textnormal{g.o.}}\left( h_0 \right)} \left\Vert Q^+ \left( \mathfrak{Z}_{\textnormal{g.o.}} \left(h_0\right) 
\left(t\right) \right)
\right\Vert_{L^2} dt = \infty
\end{equation}
for, if this were not so, then by Duhamel's formula and Minkowski's integral
inequality we would have
\[
\left\Vert \mathcal{T} \left(-t\right) h \left(t\right) -
\mathcal{T} \left(-s\right) h\left(s\right) \right\Vert_{L^2}
\leq \int_s^t \left\Vert Q^+ \left(h \left( \sigma \right) \right) \right\Vert_{L^2} d\sigma
\]
where $h \left(t\right) = \mathfrak{Z}_{\textnormal{g.o.}} \left(h_0\right) \left(t\right)$. In particular, 
letting $\left| t-s \right| \rightarrow 0$, we find that the map
$t \mapsto \mathcal{T} \left( -t \right) h\left( t \right)$ then extends uniquely to a function in 
\[
C \left( \left[ 0, T_{\textnormal{g.o.}} \left(h_0\right) \right], L^2 \right)
\]
hence $h$ does so extend as well,
and we can apply the local well-posedness theorem, Theorem \ref{thm:QPlusLWP}, with initial data
$h \left( T_{\textnormal{g.o.}} \left( h_0 \right) \right)$ to produce a solution
of (\ref{eq:QplusEq}), with initial data $h_0$ but extended past $T_{\textnormal{g.o.}} \left(h_0\right)$, in contradiction
with the definition of $T_{\textnormal{g.o.}} \left(h_0\right)$.

Hence we may define an extended-real-valued function $\tau^{\left(\varepsilon\right)} \left(h_0\right)$,
\[
0 < \tau^{\left(\varepsilon\right)} \left(h_0\right) \leq \infty
\]
on $L^{2,+}$ by the formula
\[
\tau^{\left(\varepsilon\right)} \left(h_0\right) = \sup \left\{
t > 0 \; : \; \int_0^\tau
\left\Vert Q^+ \left( \mathfrak{Z}_{\textnormal{g.o.}} \left(h_0\right) \left(t\right) \right)
\right\Vert_{L^2} dt < \varepsilon
\; \right\}
\]
Then by (\ref{eq:GOInfinity}) we see that
\[
0 < \tau^{\left(\varepsilon\right)} \left(h_0\right) \leq
T_{\textnormal{g.o.}} \left( h_0 \right)
\]
and, moreover,
\[
\int_0^{\tau^{\left(\varepsilon\right)} \left(h_0\right)}
\left\Vert Q^+ \left( \mathfrak{Z}_{\textnormal{g.o.}} \left(h_0\right) \left(t\right) \right)
\right\Vert_{L^2} dt \leq \varepsilon
\]
and the equality prevails whenever $T_{\textnormal{g.o.}} \left( h_0 \right) < \infty$.

It will be proven that for $h_{0,k}, h_0 \in L^{2,+}$,
\begin{equation}
\label{eq:tauliminf}
\lim_k \left\Vert h_{0,k} - h_0 \right\Vert_{L^2} = 0 \quad
\implies \quad \liminf_k \tau^{\left(\varepsilon\right)} \left( h_{0,k} \right) > 0
\end{equation}
Then if we write as $B_r \left(h_0\right) \subset L^2$ the \emph{open} ball in $L^2$ of radius $r>0$
centered about $h_0 \in L^2$ then defining
\[
F^{\left(\varepsilon\right)} \left(h_0\right) = \quad
\sup_{r > 0} \left( \quad
\inf \left\{ \tau^{\left(\varepsilon\right)} \left( \tilde{h}_0 \right)
\quad \left. :  \quad
\tilde{h}_0 \in B_r \left( h_0 \right) \bigcap L^{2,+}
\right. \right\} \quad \right)
\]
allows us to conclude.

We turn to the proof of (\ref{eq:tauliminf}). Assume that
\[
\lim_k \left\Vert h_{0,k} - h_0 \right\Vert_{L^2} = 0
\]
We need to place an asymptotic lower bound on $\tau^{\left(\varepsilon\right)} \left(h_{0,k}\right)$,
the bound itself possibly depending on $h_0$.
By Theorem \ref{thm:CritLWP},
it suffices to show that for any $\eta > 0$ there exists a 
$0 < \delta < \infty$ and an $r>0$  (each depending on $\eta$ and $h_0$) such that 
if $\tilde{h}_0 \in B_r \left( h_0 \right)$ then
\begin{equation}
\label{eq:lsc-main-step}
\forall \left( f_0 \in L^2 \right) \qquad
\left\Vert Q^+ \left( \mathcal{T}f_0, \mathcal{T} \tilde{h}_0
\right) \right\Vert_{L^1 \left( \left[ - \delta,\delta\right], L^2
\right)} \leq \eta \left\Vert f_0 \right\Vert_{L^2}
\end{equation}
and symmetrically reversing the two entries of $Q^+$. The point is that $\delta$ must be
uniform across a ball (of radius $r$); in that case, once $\eta$ is taken sufficiently small
(depending on $h_0$), it holds that for all large enough $k$, it 
must be that $\tau^{\left( \varepsilon \right)} \left( h_{0,k} \right) \geq 2^{-1} \delta$, hence the conclusion.

But by Proposition \ref{prop:QplusSmallTime} applied to the limiting function $h_0$, we can assume
\[
\forall \left( f_0 \in L^2 \right) \qquad
\left\Vert Q^+ \left( \mathcal{T} f_0, \mathcal{T} h_0
\right) \right\Vert_{L^1 \left( \left[ - \delta,\delta\right], L^2
\right)} \leq \frac{1}{2} \eta \left\Vert f_0 \right\Vert_{L^2}
\]
and we also have
\[
\left\Vert Q^+ \left( \mathcal{T}  f_0, \mathcal{T} 
\left( h_0 - \tilde{h}_0 \right)
\right) \right\Vert_{L^1 \left( \left[ - \delta,\delta\right], L^2
\right)} \leq C \left\Vert h_0 - \tilde{h}_0 \right\Vert_{L^2}
 \left\Vert f_0 \right\Vert_{L^2}
\]
so (\ref{eq:lsc-main-step}) holds when $r \leq \left(2C\right)^{-1} \eta$.
\end{proof}

\begin{corollary}
\label{cor:Zdisco}
If $h_0 \in L^{2,+}$ and $T_{\textnormal{g.o.}} \left(h_0\right) < \infty$ then the set
\[
\mathfrak{Z}_{\textnormal{g.o.}} \left(h_0\right) \left( \left[ 0, 
T_{\textnormal{g.o.}} \left(h_0\right)\right)\right)
\]
is not pre-compact in $L^2$.
\end{corollary}
\begin{proof}
Suppose otherwise: that is, the image of the set
\[
\left[ 0, T_{\textnormal{g.o.}} \left( h_0 \right) \right)
\]
by the map
\[
t \mapsto \mathfrak{Z}_{\textnormal{g.o.}} \left( h_0 \right) \left( t \right)
\]
is pre-compact in $L^2$. Let us denote the \emph{closure} (in $L^2$) of this image
by $\mathcal{K}$; then $\mathcal{K}$ is a compact subset of $L^2$. Therefore, the
lower-semicontinuous function  $F^{\left(1\right)}$ attains a minimum value on $\mathcal{K}$.
However, $F^{\left(1\right)} > 0$ everywhere, so it follows that $F^{\left(1\right)}$ is
bounded away from zero on $\mathcal{K}$.

Therefore, there exists an $\eta > 0$ such that
\[
\forall 0 \leq t < T_{\textnormal{g.o.}} \left(h_0\right),
\qquad
F^{\left(1\right)} \left( \mathfrak{Z}_{\textnormal{g.o.}} \left(h_0\right) \left(t\right) \right) \geq \eta
\] 
Hence we may cover $\left[ 0, T_{\textnormal{g.o.}} \left(h_0\right) \right]$ by a finite set of open intervals of size
$\leq \eta$ and use the defining properties of $F^{\left(1\right)}$ and the
semigroup property of $\mathfrak{Z}_{\textnormal{g.o.}}$ to conclude that
\[
Q^+ \left( \mathfrak{Z}_{\textnormal{g.o.}} \left(h_0\right) \left(t\right) \right) \in L^1
\left( \left[ 0, T_{\textnormal{g.o.}} \left(h_0\right) \right), L^2 \right)
\]
and observe that this contradicts (\ref{eq:GOInfinity}).
\end{proof}

\begin{corollary}
\label{cor:gain-only-blowup}
For any $h_0 \in L^{2,+}$, if $T_{\textnormal{g.o.}} \left( h_0 \right) < \infty$ then
\[
\lim_{t \rightarrow T_{\textnormal{g.o.}} \left(h_0\right)^-} \left\Vert
\mathfrak{Z}_{\textnormal{g.o.}} \left( h_0 \right) \left(t\right) 
\right\Vert_{L^2} = \infty
\]
\end{corollary}
\begin{proof}
Suppose the contrary; then there exists an increasing sequence of 
numbers $t_k \rightarrow T_{\textnormal{g.o.}} \left(h_0\right)^-$
and a number $0 < C < \infty$ such that
\[
\sup_{k} \left\Vert \mathfrak{Z}_{\textnormal{g.o.}} \left(h_0\right) \left(t_k\right)
\right\Vert_{L^2} < C
\]
Since free transport preserve the $L^2$ norm, we have
\[
\sup_{k} \left\Vert \mathcal{T} \left(-t_k\right) \left\{
\mathfrak{Z}_{\textnormal{g.o.}} \left(h_0\right) 
 \left(t_k\right) \right\}
\right\Vert_{L^2} \leq C
\]
By Duhamel's formula,
\[
\mathcal{T}\left(-t_k\right) \left\{\mathfrak{Z}_{\textnormal{g.o.}} 
\left(h_0\right) \left(t_k\right) \right\}
= h_0 + \int_0^{t_k} 
\mathcal{T}\left(-s\right) Q^+ \left( \mathfrak{Z}_{\textnormal{g.o.}} 
\left(h_0\right) \left(s\right) \right) ds
\]
therefore since $h_0 \in L^{2,+}$ we have
\[
\sup_k \left\Vert \int_0^{t_k} 
\mathcal{T}\left(-s\right) Q^+ \left( \mathfrak{Z}_{\textnormal{g.o.}} \left(h_0\right) \left(s\right) \right) ds
\right\Vert_{L^2} \leq C
\]
up to increasing $C$.
Then again, by Duhamel's formula and non-negativity (to increase the bounds of integration in the last line),
 for $0 \leq s < t$ it holds
\[
\begin{aligned}
& \left\Vert \mathcal{T} \left(-t\right) \left\{\mathfrak{Z}_{\textnormal{g.o.}} \left(h_0 \right)\left(t\right)\right\} -
\mathcal{T} \left(-s\right) \left\{\mathfrak{Z}_{\textnormal{g.o.}} \left(h_0\right) \left(s\right) \right\} \right\Vert_{L^2} \\
& \qquad \qquad \qquad \qquad \leq 
\left\Vert \int_s^t 
\mathcal{T}\left(-\sigma\right) Q^+ \left( \mathfrak{Z}_{\textnormal{g.o.}} \left(h_0\right) 
\left(\sigma\right) \right) d\sigma  \right\Vert_{L^2} \\
& \qquad \qquad \qquad \qquad \leq 
\left\Vert \int_s^{T_{\textnormal{g.o.}}\left( h_0 \right)}
\mathcal{T}\left(-\sigma\right) Q^+ \left( \mathfrak{Z}_{\textnormal{g.o.}} \left(h_0\right) 
\left(\sigma\right) \right) d\sigma
\right\Vert_{L^2}
\end{aligned}
\]
so by dominated convergence (letting $s \rightarrow T_{\textnormal{g.o.}} \left( h_0 \right)^{-}$ in the last
line and expanding the definition of the $L^2$ norm to apply the dominated convergence theorem, taking
care \emph{not} to apply Minkowski's inequality),
 we find that the function
 \[
 t \mapsto 
 \mathcal{T} \left(-t\right) \left\{\mathfrak{Z}_{\textnormal{g.o.}} \left(h_0\right)\left(t\right)\right\}
 \]
  admits
 a continuous extension from $\left[ 0, T_{\textnormal{g.o.}} \left(h_0\right) \right]$ to $L^2$. 
 In particular, $\mathfrak{Z}_{\textnormal{g.o.}} \left(h_0\right) \left(t\right)$ also admits a continuous
 extension from $\left[ 0, T_{\textnormal{g.o.}} \left(h_0\right) \right]$ into $L^2$, 
 in contradiction with Corollary \ref{cor:Zdisco}.
\end{proof}

\begin{lemma}
\label{lem:continuityZero}
Let $h_0 \in L^{2,+}$; then, there exist numbers $\sigma, r > 0$, depending only
 on $h_0$, such that the following holds:

For any $\varepsilon > 0$, there exists a $\delta > 0$ such that whenever
\[
\tilde{h}_0^{\left( 1 \right)}, \tilde{h}_0^{\left( 2 \right)} \quad \in
L^{2,+}
\]
are chosen to satisfy
\[
\forall \left( i \in \left\{ 1, 2 \right\} \right) \quad
\left\Vert \tilde{h}_0^{\left(i\right)} - h_0 \right\Vert_{L^2} < r
\]
and
\[
\left\Vert \tilde{h}_0^{\left( 1 \right)} - \tilde{h}_0^{\left( 2 \right)} 
\right\Vert_{L^2} < \delta
\]
then it follows
\[
\left\Vert 
\mathfrak{Z}_{\textnormal{g.o.}} \left( \tilde{h}_0^{\left( 1 \right)} \right)
\left( t \right) - \mathfrak{Z}_{\textnormal{g.o.}} \left( \tilde{h}_0^{\left( 2 \right)} \right)
\left( t \right) \right\Vert_{L^\infty \left( J, L^2 \right)} < \varepsilon
\]
where $J = \left[ 0, \sigma \right]$.
\end{lemma}
\begin{proof}
Let $\tilde{h}_0^{\left( i \right)}$, $i=1,2$, be chosen as in the statement of the Lemma,
for some $r > 0$ to be determined later. We can assume by our choice of $r,\sigma$, at the very least, that
\begin{equation}
\label{eq:continuityHelper}
\inf \left\{ T_{\textnormal{g.o.}} \left( \tilde{h}_0 \right) \quad : \quad
\left\Vert \tilde{h}_0 - h_0 \right\Vert_{L^2} < r \right\} \; > \; \sigma
\end{equation}
in view of Lemma \ref{lem:lsc}.

Letting $\tilde{h}^{\left(i\right)} \left( t \right) = \mathfrak{Z}_{\textnormal{g.o.}} \left( 
\tilde{h}_0^{\left(i\right)} \right)
\left( t \right)$, define
\[
w(t) = \tilde{h}^{\left(1\right)} \left(t\right) - \tilde{h}^{\left(2\right)} \left(t\right)
\]
then it holds
\begin{equation}
\label{eq:continuityDifferenceEquation}
\left( \partial_t + v \cdot \nabla_x \right) w =
Q^+ \left( \tilde{h}^{\left(1\right)}, w \right) + Q^+ \left( w, \tilde{h}^{\left(2\right)} \right)
\end{equation}
and we denote $w_0 = w \left( t=0 \right)$.
Consider just the first term on the right; the second is handled similarly.

We may write
\begin{equation}
\begin{aligned}
& Q^+ \left( \tilde{h}^{\left(1\right)}, w \right) \\
&  \qquad =
Q^+ \left( \mathcal{T} \tilde{h}_0^{\left(1\right)}, w \right) +
Q^+ \left( \tilde{h}^{\left(1\right)} - \mathcal{T}
 \tilde{h}_0^{\left(1\right)}, w \right) \\
 & \qquad \quad = \mathcal{I}_1 + \mathcal{I}_2
\end{aligned}
\end{equation}
Let us denote $J_\sigma = \left[ 0, \sigma \right]$.

We have previously seen (e.g. from the proof of Lemma \ref{lem:lsc}, specifically (\ref{eq:lsc-main-step}))
that, by choosing $r,\sigma$ small depending on the
small parameter $\eta$, we may have simultaneously for
all $\tilde{h}_0^{\left(1\right)}$ within $L^2$-distance $r$ of $h_0$
and all $u_0 \in L^2$ that
\[
\left\Vert Q^+ \left( \mathcal{T} \tilde{h}_0^{\left(1\right)}, \mathcal{T} u_0
\right) \right\Vert_{L^1 \left( J_\sigma, L^2 \right)} \leq
\eta \left\Vert u_0 \right\Vert_{L^2}
\]
 This estimate
suffices to handle term $\mathcal{I}_1$: indeed, it implies by Duhamel's formula applied to $w$ that
\[
\left\Vert
Q^+ \left( \mathcal{T} \tilde{h}_0^{\left(1\right)}, w \right)
\right\Vert_{L^1 \left( J_\sigma , L^2 \right)} \leq
\eta \left( \left\Vert w_0\right\Vert_{L^2} + \left\Vert 
\left( \partial_t + v \cdot \nabla_x \right) w \right\Vert_{L^1 \left( J_\sigma , L^2 \right)} \right)
\]
the right-hand side being finite by (\ref{eq:continuityHelper}), since $w$ is simply the difference between
two solutions which each have lifetimes strictly larger than $\sigma$.

Also, by Duhamel's formula
\[
\tilde{h}^{\left(1\right)} \left( t \right) - \mathcal{T} \left( t \right) \tilde{h}_0^{\left(1\right)} = \int_0^t
\mathcal{T} \left(t-s\right) Q^+ \left( \tilde{h}^{\left(1\right)} \left( s \right) \right) ds
\]
and Proposition \ref{prop:QplusBound}, it holds for any $u_0 \in L^2$
\[
\left\Vert Q^+ \left( \tilde{h}^{\left( 1 \right)} - \mathcal{T} \tilde{h}_0^{\left( 1 \right)} \; , \; \mathcal{T} u_0
\right) \right\Vert_{L^1 \left( J_\sigma , L^2 \right)} \leq C
\left\Vert Q^+ \left( \tilde{h}^{\left(1\right)} \right) \right\Vert_{L^1 \left( J_\sigma , L^2 \right)}
 \left\Vert u_0 \right\Vert_{L^2}
\]
Note carefully we have substituted Duhamel's formula into the first entry of
$Q^+$, so that $Q^+$ is acting on another $Q^+$ and a $u_0$; it is to the outer $Q^+$ that we apply
Proposition \ref{prop:QplusBound}.
 By Lemma \ref{lem:lsc} with $\varepsilon$ (the $\varepsilon$ of Lemma \ref{lem:lsc}, not related to the $\varepsilon$ appearing
 in the statement of the present lemma), for any $\eta > 0$ there exist $r,\sigma > 0$ such that again,
 simultaneously for
all $\tilde{h}_0^{\left(1\right)}$ within $L^2$-distance $r$ of $h_0$
and all $u_0 \in L^2$, it holds
 \[
 C
\left\Vert Q^+ \left( \tilde{h}^{\left(1\right)} \right) \right\Vert_{L^1 \left( J_\sigma , L^2 \right)}
\leq \eta
 \]
So for any $u_0 \in L^2$ we may now write
 \[
\left\Vert Q^+ \left( \tilde{h}^{\left( 1 \right)} - \mathcal{T} \tilde{h}_0^{\left( 1 \right)} \; , \; \mathcal{T} u_0
\right) \right\Vert_{L^1 \left( J_\sigma , L^2 \right)} \leq \eta
 \left\Vert u_0 \right\Vert_{L^2}
\]
so that, once again,
  \begin{equation*}
  \begin{aligned}
& \left\Vert Q^+ \left( \tilde{h}^{\left( 1 \right)} - \mathcal{T} \tilde{h}_0^{\left( 1 \right)} \; , \; w
\right) \right\Vert_{L^1 \left( J_\sigma , L^2 \right)} \\
& \qquad\qquad\qquad\leq \eta \left( \left\Vert w_0 \right\Vert_{L^2} +
 \left\Vert\left( \partial_t + v \cdot \nabla_x \right) w \right\Vert_{L^1 \left( J_\sigma , L^2 \right)}\right)
\end{aligned}
\end{equation*}
which suffices for $\mathcal{I}_2$.

 To conclude, let us denote
 \[
 a\left( \sigma \right) =
 \left\Vert w \right\Vert_{L^\infty \left( J_\sigma , L^2 \right)} +
 \left\Vert \left( \partial_t + v \cdot \nabla_x \right)
 w \right\Vert_{L^1 \left( J_\sigma,  L^2 \right)}
 \]
 which we recall is finite in any case, and observe that
 \[
 a\left( \sigma \right) \leq
\left\Vert w_0 \right\Vert_{L^2} +
 2 \left\Vert \left( \partial_t + v \cdot \nabla_x \right)
 w \right\Vert_{L^1 \left( J_\sigma,  L^2 \right)} \leq 2 a \left( \sigma \right)
 \]
Hence by (\ref{eq:continuityDifferenceEquation}) and the above estimates on $\mathcal{I}_1$ and $\mathcal{I}_2$
 we now have
 \[
a \left( \sigma \right) \leq
 \left\Vert \tilde{h}_0^{\left( 1 \right)} - \tilde{h}_0^{\left(2\right)} \right\Vert_{L^2} +
 16 \eta a \left( \sigma \right)
 \]
Letting $\eta = \frac{1}{32}$, with the corresponding contraints on $r,\sigma$ as specified above, 
yields by the definition of $a \left( \sigma \right)$ that
\[
\left\Vert w 
\right\Vert_{L^\infty \left( J_\sigma , L^2 \right)} \leq
2  \left\Vert \tilde{h}_0^{\left( 1 \right)} - \tilde{h}_0^{\left(2\right)} \right\Vert_{L^2} 
\]
as claimed.
\end{proof}

For the next lemma we denote by $B_r \left( h_0 \right)$ the ball of radius $r$ in $L^2$ centered
about $h_0 \in L^{2,+}$.

\begin{lemma}
\label{lem:continuityOne}
Let $K \subset L^{2,+}$ be compact. Then there exists a $\sigma > 0$, depending only on $K$, such
that the following is true:

For every $h_0 \in K$, there exists an $r > 0$
such that
\[
\forall \left(\tilde{h}_0 \in 
B_r \left( h_0 \right)\bigcap L^{2,+} \right)\quad
\sigma < T_{\textnormal{g.o.}} \left( \tilde{h}_0 \right)
\]
and such that, denoting $J_\sigma = \left[ 0, \sigma \right]$, the map
\[
B_r \left( h_0 \right) \bigcap L^{2,+} \rightarrow C \left( J_\sigma, L^2 \right),\qquad
\tilde{h}_0 \mapsto \mathfrak{Z}_{\textnormal{g.o.}} \left( \tilde{h}_0 \right) 
\left( \cdot \right)
\]
is continuous.
\end{lemma}
\begin{remark}
It is convenient for the proof to let $r$ possibly depend on $h_0 \in K$, although it is possible
to show by the compactness of $K$ that $r$ need not depend on $h_0$, even if we have only proven the claim
allowing $r$ to depend on $h_0$. Indeed, choosing $r$ for each $h_0$ as in the Lemma, cover $K$ by open
balls of radius $\frac{r_i}{2}$ about $h_0^{\left(i\right)}$ as $i$ ranges over a finite set.
\end{remark}
\begin{proof}
For any $h_0 \in L^{2,+}$ we will write
\[
0 < \sigma \in A \left( h_0 \right) \subset \mathbb{R}
\]
if and only if both the following hold: first, that there exists $r > 0$, depending on $\sigma$ and $h_0$,
such that
\[
\forall \left( \tilde{h}_0 \in B_r \left(h_0\right) \bigcap L^{2,+}\right) \quad
\sigma \leq 2^{-1} T_{\textnormal{g.o.}} \left( \tilde{h}_0 \right)
\]
and second, that the map
\[
B_r \left( h_0 \right) \bigcap L^{2,+} \rightarrow C \left( \left[ 0,\sigma \right], L^2 \right),\qquad
\tilde{h}_0 \mapsto \mathfrak{Z}_{\textnormal{g.o.}} \left( \tilde{h}_0 \right) 
\left( \cdot \right)
\]
is continuous .

Also let us write
\[
a \left( h_0 \right) = \sup A \left( h_0 \right)
\]
the least upper bound of the set $A \left( h_0 \right)$.
By Lemma \ref{lem:continuityZero}, $a \left( h_0 \right) > 0$ for each
$h_0 \in L^{2,+}$.

We have to show that for any compact $K \subset L^{2,+}$,
\[
\inf \left\{ a \left( h_0 \right) \; : \; h_0 \in K 
\right\} \; > \; 0
\]
By way of contradiction, suppose that there are points $h_{0,k} \in K$,
$k = 1, 2, 3, \dots$,
such that $a \left( h_{0,k} \right)\rightarrow 0$  as $k\rightarrow \infty$.
By the compactness of $K$, we can pass to a subsequence converging in $L^2$, say
$h_{0, k^\prime} \rightarrow h_0$ for some $h_0 \in K$. Applying Lemma \ref{lem:continuityZero}
to $h_0$ we find that there must exist a number $k_0$ such that
 $a \left( h_{0, k^\prime} \right)$ is bounded from below
uniformly in $k^\prime \geq k_0$, hence the contradiction.
\end{proof}

\begin{remark}
Observe that in the proof of Lemma \ref{lem:continuityOne}, we have relied on the fact
that Lemma \ref{lem:continuityZero} provides continuity of the solution map not just
at $h_0$, but across a small ball surrounding $h_0$, for a time bounded uniformly from
below on said ball. In particular, we obtain continuity on a relatively open set
$\mathcal{O} \subset L^{2,+}$ with $K \subset \mathcal{O}$, the existence time being bounded
from below uniformly on $\mathcal{O}$.
\end{remark}

\begin{theorem}
\label{thm:lsc2}
$T_{\textnormal{g.o.}}$ is \emph{lower semi-continuous}: that is, if $h_0 \in L^{2,+}$ and $h_{0,k} \in L^{2,+}$ for
$k = 1, 2, 3, \dots$, then
\[
\lim_{k \rightarrow \infty} \left\Vert h_{0,k} - h_0 \right\Vert_{L^2} = 0 \quad \implies \quad
T_{\textnormal{g.o.}} \left( h_0 \right) \leq
\liminf_{k \rightarrow \infty} T_{\textnormal{g.o.}} \left( h_{0,k} \right)
\]
 Moreover, the solution map $\mathfrak{Z}_{\textnormal{g.o.}}$ for
 (\ref{eq:QplusEq}) is \emph{continuous}, in the following sense:
 
 Denoting for $h_0 \in L^2$ the open ball
 \[
 B_r \left( h_0 \right) = \left\{ \tilde{h}_0 \in L^2 \; : \;
 \left\Vert \tilde{h}_0 - h_0 \right\Vert_{L^2} < r \right\}
 \]
it holds that  for any $h_0 \in L^{2,+}$ and any compact interval $J = \left[ 0, T \right]$,
where $ 0 < T < T_{\textnormal{g.o.}} \left(h_0\right)$ is chosen arbitrarily,
there exists an $r>0$, depending only on $T$ and $h_0$,
such that the map
\[
B_r \left( h_0 \right) \bigcap L^{2,+} \rightarrow C \left( J, L^2 \right),\qquad
\tilde{h}_0 \mapsto \mathfrak{Z}_{\textnormal{g.o.}} \left( \tilde{h}_0 \right) 
\left( \cdot \right)
\]
is continuous.
\end{theorem}
\begin{proof}
First observe that for any $0 < T < T_{\textnormal{g.o.}} \left(h_0\right)$ the set
\[
K = \mathfrak{Z}_{\textnormal{g.o.}} \left(h_0\right) \left( J \right)
\]
is compact, being the image of the compact interval $J = \left[ 0, T \right]$ by the continuous  map
$\mathfrak{Z}_{\textnormal{g.o.}} \left( h_0 \right) \left( \cdot \right)$. Thus we may apply
Lemma \ref{lem:continuityOne} to the set $K$.

For each $t_0 \in J$ let $B^{t_0}$ be the $L^2$ ball centered on
$\mathfrak{Z}_{\textnormal{g.o.}} \left(h_0\right) \left(t_0\right)$ guaranteed by Lemma \ref{lem:continuityOne}:
\emph{note carefully that we are taking the solution  at time $t_0 \in J$, that is
$\mathfrak{Z}_{\textnormal{g.o.}} \left(h_0\right) \left(t_0\right)$,
 as our initial
data in the application of Lemma \ref{lem:continuityOne}}.
In particular, by the compactness of $K$, the solution map
$\mathfrak{Z}_{\textnormal{g.o.}}$
 is continuous on $B^{t_0}$ for a time
$\sigma$ that is uniform in $t_0 \in J$. Assume without loss (up to a possibly smaller choice of the constant $\sigma$)
that $T = M \sigma$, $M$ an integer.

The proof is by an induction \emph{backwards} in time, starting at $T$.
The starting point is the unit $L^2$ ball centered on 
$\mathfrak{Z}_{\textnormal{g.o.}} \left(h_0\right) \left(T\right)$. Take the preimage of this ball by  the
(partially defined) gain-only flow, at time $\sigma$, and call $U_1$ the intersection
of this preimage with $B^{T - \sigma}$. Then $U^1$ is open for the
subspace topology of $L^{2,+} \subset L^2$. Repeat the process, taking
the preimage of $U_k$ by the time $\sigma$ flow and intersecting with
$B^{T - (k+1)\sigma}$ to produce $U_{k+1}$. Eventually we will have
$U_M$, a relatively open subset of $L^{2,+}$ that contains $h_0$;
moreover, by construction, for any $\tilde{h}_0 \in U_M \subset L^{2,+}$ the flow
is defined for $0 \leq t \leq T$, and the flow is continuous on
$U_M$ for $0 \leq t \leq T$.  
\end{proof}

\section{The comparison principle}
\label{sec:comparison}

Any smooth solution $f$ of (\ref{eq:BE}) with sufficient decay for large $\left( x,v \right)$ is bounded
from above \emph{pointwise} at positive times by the solution of the $Q^+$ equation (\ref{eq:QplusEq}) with the same initial data,
for the full lifespan of the solution of (\ref{eq:QplusEq}).
Thus, under such assumptions, we 
may view the solution of (\ref{eq:QplusEq}) as an \emph{upper envelope} for the solution of
(\ref{eq:BE}), at least on a small time interval.
Setting aside ``near vacuum'' results, 
solutions of the $Q^+$ equation are not global in general even for smooth data with rapid decay
 \cite{RSH1987}; nevertheless, we can take
$f \left( t_0 \right)$ as initial data in (\ref{eq:QplusEq}) to obtain, once again, an upper envelope
valid for $t \in \left[ t_0, t_0 + \sigma\right)$ for some small $\sigma > 0$ depending  on 
 $f \left( t_0 \right)$ (\emph{note: \textbf{not} the $L^2$ norm of $f \left( t_0 \right)$, but the full profile}).
This \emph{comparison principle}, the invocation of which is \emph{defined to mean} that we may obtain
an upper envelope along sufficiently small half-open intervals starting from  \emph{any} $t_0$ in the (larger but still
half-open) domain
of interest, is a fundamental property of any
Boltzmann equation satisfying the Grad cut-off condition (the principle is obviously meaningless in
the non-cutoff case). Now it is not at all clear whether the renormalized solutions of 
DiPerna and Lions \cite{DPL1989} satisfy
a version of the comparison principle in general. However, in the
$L^2$ setting, we \emph{can} make sense of (\ref{eq:QplusEq}) by Theorem \ref{thm:QPlusLWP}. 
Since the comparison principle is the foundation of everything to follow,
we devote this section to formalizing the comparison principle to the extent that we require.

\begin{definition}
Let $f \left( t, x, v \right)$ be a non-negative measurable 
function \emph{(not necessarily solving Boltzmann's equation (\ref{eq:BE}))}, defined
in the domain
\[
I \times \mathbb{R}^2 \times \mathbb{R}^2
\]
where $I = \left[ a, b \right)$ and $-\infty < a < b \leq \infty$.
Let us assume that for any compact set $K$ of the product form 
\[
K =  A \times B \times C
 \subset I \times \mathbb{R}^2 \times \mathbb{R}^2
 \]
 (namely
 $A \subset I$, and $B, C \subset \mathbb{R}^2$), there holds
\[
 \left. f \right|_K \in C \left( A, L^1 \left( B \times C \right) \right)
\]
In particular, the pointwise evaluation in time, $f \left( t_0 \right)$, is
well-defined for each $t_0 \in I$.

For any $t_0 \in I$ such that $f \left( t_0 \right) \in L^2$, we shall write
\[
f \in \mathfrak{B}^{I}_{\left\{ t_0 \right\}}
\]
if for any $t \in \mathbb{R}$ such that
\[
t \in I \quad \textnormal{ and } \quad
t_0 \leq t < t_0 + T_{\textnormal{g.o.}} \left( f\left(t_0\right) \right)
\]
we have
\[
f \left(t\right) \leq \mathfrak{Z}_{\textnormal{g.o.}} \left( f \left( t_0\right) \right) \left( t-t_0 \right)
\]
for almost every $\left( x,v \right)$.

For any subset $F\subseteq I$ we will write
\[
f \in \mathfrak{B}_F^I
\]
if
\[
\forall\left(
t_0\in F\right) \quad
f \in \mathfrak{B}^I_{\left\{ t_0 \right\}}
\]
That is,
\[
\mathfrak{B}_F^I = \bigcap_{t_0 \in F}
\mathfrak{B}^I_{\left\{ t_0 \right\}}
\]

Similarly, if $J = \left[ a, b \right]$ is a \emph{compact} interval, then letting $I = \left[ a, b \right)$,
for any $t_0 \in J$ we write
\[
f \in \mathfrak{B}^J_{\left\{ t_0 \right\}}
\]
if either (i) $t_0 = b$ and $f \left( b \right) \in L^2$, or  (ii)
\[
f \in \mathfrak{B}^I_{\left\{ t_0 \right\}}
\]
For any subset $F \subset J$ we write
\[
f \in \mathfrak{B}^J_F
\]
if
\[
\forall \left( t_0 \in F \right) \quad f \in \mathfrak{B}^J_{\left\{ t_0 \right\}}
\]
thus
\[
\mathfrak{B}^J_F = \bigcap_{t_0 \in F} \mathfrak{B}^J_{\left\{ t_0 \right\}}
\]
\end{definition}

\begin{lemma}
\label{lem:comparisonExtension}
If $0 < T_1 < T_0$, $I_1 = \left[ 0, T_1 \right)$ and $I_2 =
\left[ T_1, T_0\right)$, and if
\begin{equation}
f \in \mathfrak{B}^{I_1}_{I_1} \quad \textnormal{ and }\quad
f\in \mathfrak{B}^{I_2}_{I_2}
\end{equation}
then
\begin{equation}
f \in \mathfrak{B}^{I_3}_{I_3}
\end{equation}
where $I_3 = \left[ 0, T_0 \right)$.
\end{lemma}
\begin{proof}
This is an immediate consequence of the semigroup property of
$\mathfrak{Z}_{\textnormal{g.o.}}$ combined with the fact that $Q^+$ is \emph{monotonic}, i.e.
\[
0 \leq f_0 \leq h_0 \quad \implies \quad 0 \leq Q^+ \left( f_0 \right) \leq Q^+ \left( h_0 \right)
\]
In fact, this monotonicity property of $Q^+$ implies that the gain-only flow $\mathfrak{Z}_{\textnormal{g.o.}}$ is
monotonic as well (for $t$ fixed):
\[
0 \leq f_0 \leq h_0 \quad \implies \quad 0 \leq \mathfrak{Z}_{\textnormal{g.o.}} \left( f_0 \right) \left( t \right)
 \leq \mathfrak{Z}_{\textnormal{g.o.}} \left( h_0 \right) \left( t \right)
\]
whenever this makes sense (this can be established by writing $\mathfrak{Z}_{\textnormal{g.o.}}$ in terms of the
initial data using an infinite iterated Duhamel expansion, which is guaranteed to converge on a small time interval
 by the Banach contraction
used in the construction of $\mathfrak{Z}_{\textnormal{g.o.}}$). Combining the monotonicity and semigroup properties
 of $\mathfrak{Z}_{\textnormal{g.o.}}$ with the definition of $\mathfrak{B}^I_F$ establishes the Lemma with a few lines
 of straightforward algebra, which we recount next:
 
 Indeed, it suffices to consider the case
 \[
 t_1 \in I_1, \; \; t_2 \in I_2
 \]
 such that
 \[
 t_2 < t_1 + T_{\textnormal{g.o.}} \left( f \left( t_1 \right) \right)
 \]
 In that case, it immediately follows each
 \[
 T_1 < t_1 + T_{\textnormal{g.o.}} \left( f \left( t_1 \right) \right)
 \]
 and
 \[
 t_2 < T_1 + T_{\textnormal{g.o.}} \left( f \left( T_1 \right) \right)
 \]
 by the semigroup property. Moreover, from the definition of $\mathfrak{B}^I_F$ we may deduce
 \[
f \left( t_2 \right) \leq \mathfrak{Z}_{\textnormal{g.o.}} \left( f \left( T_1 \right) \right) \left( t_2 - T_1 \right)
 \]
 using $f \in \mathfrak{B}^{I_2}_{I_2}$, and also
 \[
 f \left( T_1 \right) \leq \mathfrak{Z}_{\textnormal{g.o.}} \left( f \left( t_1 \right) \right) \left( T_1 - t_1 \right)
 \]
using $f \in \mathfrak{B}^{I_1}_{I_1}$ and continuity in time. Therefore, applying the monotonicity of
$\mathfrak{Z}_{\textnormal{g.o.}}$ followed by the semigroup property, we have
\[
\begin{aligned}
f \left( t_2 \right) & \leq \mathfrak{Z}_{\textnormal{g.o.}} \left( f\left(T_1 \right) \right) \left( t_2 - T_1 \right) \\
& \leq \mathfrak{Z}_{\textnormal{g.o.}} \left[ \mathfrak{Z}_{\textnormal{g.o.}} \left( f \left( t_1 \right) \right)
 \left( T_1 - t_1 \right) \right]
\left( t_2 - T_1 \right) \\
& = \mathfrak{Z}_{\textnormal{g.o.}} \left( f \left( t_1 \right) \right) \left( t_2 - t_1 \right)
\end{aligned}
\]
as required.
\end{proof}

\begin{proposition}
\label{prop:comparisonZero}
If $0 < T < \infty$ and
\[
f \in
C \left( J, L^2 \right) \bigcap
 \mathfrak{B}^{J}_{J}
\]
where $J = \left[ 0, T \right]$, then
\[
Q^+ \left(f,f\right)  \in L^1 \left( J,  L^2 \right)
\]
\end{proposition}
\begin{proof}
By Lemma \ref{lem:lsc} and the compactness of $J$, since
$f \in C \left( J,  L^2 \right)$ we have
\[
\inf_{t \in J} T_{\textnormal{g.o.}} \left( f\left(t\right) \right) \geq
\inf_{t \in J} F^{\left(1\right)} \left( f\left(t\right) \right) = \eta > 0
\]
where $F^{\left(1\right)}$ is the function from
the statement of Lemma \ref{lem:lsc} in the case $\varepsilon = 1$.
Hence we can use $f \in \mathfrak{B}^J_J$ to estimate, by the monotonicity of $Q^+$,
\[
\begin{aligned}
\int_{J_\eta}
\left\Vert Q^+ \left( f(t) \right)
\right\Vert_{L^2} dt & \leq 
\int_0^{F^{(1)} \left( f(t_0)\right)} 
\left\Vert Q^+ \left( \mathfrak{Z}_{\textnormal{g.o.}} \left( f(t_0)\right)
\left(\tau\right) \right) \right\Vert_{L^2} d\tau \\
& \leq 1
\end{aligned}
\]
where $J_\eta = J \bigcap \left[ t_0, t_0+\eta \right]$, and we have set $\varepsilon = 1$ on each side of
(\ref{eq:F-eps-bd}) to establish the last line.
Since $\eta$ is independent of $t_0$, we can conclude by covering $J$ by a finite collection of
closed intervals, each of size at most $\eta$.
\end{proof}

Note carefully that Proposition \ref{prop:comparisonZero} relies on continuity
into $L^2$ but does \emph{not} require $f$ to solve Boltzmann's equation (\ref{eq:BE}).
For functions $f$ which actually satisfy (\ref{eq:BE}), at least to the point where
Duhamel's formula is valid, we have the following
converse to Proposition \ref{prop:comparisonZero} (which we first establish on a small time interval,
followed by longer time intervals):

\begin{lemma}
\label{lem:comparison-delta}
If $f\geq 0$ solves Boltzmann's equation (\ref{eq:BE}) on $I = \left[ 0,T \right)$ in such
a  way that Duhamel's formula holds, and in addition
\[
f \in C \left( I, L^2 \right)
\]
and
\[
Q^+ \left(f\right) \in L^1 \left( J , L^2 \right)
\]
for each compact sub-interval $J \subset I$,
then for some $\sigma > 0$ there holds
\[
f \in \mathfrak{B}^{I_\sigma}_{I_\sigma}
\]
where $I_\sigma = \left[ 0, \sigma \right)$.
\end{lemma}
\begin{proof}
Obviously by the hypotheses for any $\varepsilon > 0$ there is a $\sigma > 0$
such that
\begin{equation}
\label{eq:comparison-delta-epsilon}
\int_0^\sigma \left\Vert Q^+ \left( f\left(s\right) \right)
\right\Vert_{L^2} ds < 2^{-1} \varepsilon
\end{equation}
but we leave the choice of a particular $\varepsilon$ for later.
We will use (\ref{eq:comparison-delta-epsilon}) in combination with the proof of Theorem \ref{thm:CritLWP}
to close a Banach fixed point iteration for the gain-only equation,
the limit of which coincides with $\mathfrak{Z}_{\textnormal{g.o.}}$ by uniqueness, and show that
$f$ lies below the function so constructed.
Hence we shall show that $f \in \mathfrak{B}^{I_\sigma}_{\left\{t_0\right\}}$,
each $t_0 \in I_\sigma$.

Fix $t_0 \in I_\sigma$.
The new iteration is defined for $t \in \left[ t_0, \sigma \right)$ by the formulas
\[
h^{\left(1\right)} \left(t\right) = \mathcal{T} \left(t-t_0\right) f \left( t_0 \right) + \int_{t_0}^t 
\mathcal{T}\left(t-s\right) Q^+ \left( f\left(s\right) \right) ds
\]
\[
h^{\left(k+1\right)} \left(t\right) = \mathcal{T} \left(t-t_0\right) f\left( t_0 \right) + \int_{t_0}^t 
\mathcal{T}\left(t-s\right) Q^+ \left( h^{\left(k\right)} \left(s\right) \right) ds
\]
In particular, it follows that
\[
\left( \partial_t + v \cdot \nabla_x \right)
\left( h^{\left(1\right)} \left(t\right) - \mathcal{T}\left(t-t_0\right) f\left(t_0\right) \right)
= Q^+ \left( f \left(t \right)\right)
\]
with $h^{\left( 1 \right)} \left( t_0 \right) = f \left( t_0 \right)$;
hence, by (\ref{eq:comparison-delta-epsilon}), there holds
\[
\begin{aligned}
&\left\Vert h^{\left(1\right)} - \mathcal{T}\left(t-t_0\right) f \left( t_0 \right) \right\Vert_{L^\infty
\left( \tilde{I}_\sigma, L^2 \right)} \\
& \qquad \qquad +
\left\Vert \left( \partial_t + v \cdot \nabla_x \right)
\left( h^{\left(1\right)} - \mathcal{T}\left(t-t_0\right) f\left(t_0\right) \right)
\right\Vert_{L^1 \left( \tilde{I}_\sigma , L^2\right)} \\
& \qquad\qquad\qquad\qquad\leq 2\cdot \left( 2^{-1} \varepsilon \right) = \varepsilon
\end{aligned}
\]
where $\tilde{I}_\sigma = \left[ t_0, \sigma \right)$; but we may now notice that the norm on the left
(the sum of \emph{both} terms) is exactly the one used to define the ball $\mathcal{B}_\varepsilon$ appearing
in the proof of Theorem \ref{thm:CritLWP}. Therefore, for small enough $\varepsilon$
we have the convergence of $h^{\left(k\right)}$ in $L^2$ as $k \rightarrow \infty$, and
by uniqueness the limit is equal
to 
\[
\mathfrak{Z}_{\textnormal{g.o.}} \left(f\left(t_0\right)\right) \left(t-t_0\right)
\]
 each $t \in \left[ t_0, \sigma \right)$.

To conclude, let us show  by induction 
that $f \left( t \right) \leq h^{(k)} \left(t\right)$ for each $t \in \left[ t_0, \sigma \right)$ and each $k$.
By Duhamel's formula and the non-negativity of $f$,
\[
f\left(t\right) \leq \mathcal{T} \left(t-t_0\right) f\left(t_0\right) + \int_{t_0}^t \mathcal{T} \left(t-s\right) 
Q^+ \left( f\left(s\right) \right) ds
\]
and the expression on the right is just $h^{\left(1 \right)}$ by definition,
so 
\[
f\left(t\right)\leq h^{\left(1\right)}\left(t\right)
 \]
for such $t$. Now suppose, for some $k$, that we have 
\[
f \left( t \right) \leq h^{\left(k\right)}\left(t\right)
\]
for each such $t$, then by Duhamel's formula and the monotonicity of $Q^+$ we also have
\[
\begin{aligned}
f\left(t\right) & \leq \mathcal{T} \left(t-t_0\right) f\left(t_0\right) + \int_{t_0}^t
 \mathcal{T} \left(t-s\right) Q^+ \left(
f\left(s\right) \right) ds \\
& \leq \mathcal{T} \left(t-t_0\right) f\left(t_0\right) + \int_{t_0}^t \mathcal{T} \left(t-s\right)
Q^+ \left( h^{\left(k\right)} \left(s\right) \right) ds \\
& = h^{\left(k+1\right)} (t)
\end{aligned}
\]
Passing to the limit in $k$ we find that for any $t_0 \leq t < \sigma$
it holds
\[
f \left(t\right) \leq \mathfrak{Z}_{\textnormal{g.o.}} \left(f\left(t_0\right)\right) \left(t-t_0\right)
\]
almost every $\left(x,v\right)$.
\end{proof}

\begin{proposition}
\label{prop:comparisonConverse}
If $f \geq 0$ solves Boltzmann's equation (\ref{eq:BE}) on $I = \left[ 0,T \right)$ in such
a  way that Duhamel's formula holds, and in addition
\[
f \in C \left( I, L^2 \right)
\]
and
\[
Q^+ (f) \in L^1 \left( J, L^2 \right)
\]
for each compact sub-interval $J \subset I$,
then
\[
f \in \mathfrak{B}^{I}_{I}
\]
\end{proposition}
\begin{proof}
Define
\[
\zeta = \sup \left\{ \;\sigma \in \left( 0, T \right) \;
: \; f \in \mathfrak{B}^{I_\sigma}_{I_\sigma} \; \right\}
\]
where $I_\sigma = \left[ 0, \sigma \right)$.
By Lemma \ref{lem:comparison-delta} we have $\zeta > 0$. Suppose
\[
\zeta < T
\]
by way of contradiction.
We can show from definitions that 
\[
f \in 
\mathfrak{B}^{I_\zeta}_{I_\zeta}
\]
 Then again, by
Lemma \ref{lem:comparison-delta}, we also have for some
$\delta > 0$ that 
\[
f \in \mathfrak{B}^{I_{\zeta+\delta}\setminus I_\zeta}_{I_{\zeta+\delta}\setminus I_\zeta}
\]
Hence Lemma \ref{lem:comparisonExtension} implies that
\[
f \in \mathfrak{B}^{I_{\zeta+\delta}}_{I_{\zeta+\delta}}
\]
contradicting
the definition of $\zeta$.
\end{proof}

\section{Pointwise convergence and the fundamental lemma}
\label{sec:ptwiseConvFund}

The following Lemma utilizes the uniform square integrability results from Section \ref{sec:squareIntegrability}
to pass to pointwise limits in the comparison principle, under suitable conditions. The Lemma also allows us
to propagate $L^2$ convergence from one point in time to a later point in time, under the same conditions. 
We will use this Lemma both in the construction (by compactness) of ($*$)-solutions, and similarly, the passage
to limits of ($*$)-solutions, in Sections \ref{sec:ExistenceProof} and \ref{sec:starLimitsProof}, respectively.

\begin{lemma}
\label{lem:fundamental}
\emph{(the Fundamental Lemma)}
Consider the interval $I = \left[ a, b \right)$ where $-\infty < a < b \leq \infty$, and
let $f_n, f$ be measurable, non-negative functions \emph{(not necessarily solving Boltzmann's equation)} with common domain
\[
I \times \mathbb{R}^2 \times \mathbb{R}^2
\]
such that, for any compact set $K$ of the product form
\[
K = A \times B \times C
 \subset I \times \mathbb{R}^2 \times \mathbb{R}^2
 \]
 (namely
 $A \subset I$, and $B, C \subset \mathbb{R}^2$), it holds
\[
\left. f_n\right|_K , \left. f \right|_K \in C \left( A, L^1 \left( B \times C \right) \right)
\]
In particular, pointwise evaluation in time is well-defined.
We also require
\[
f_n \left( a \right), f \left( a \right) \in L^2_{x,v} \left( \mathbb{R}^2 \times \mathbb{R}^2 \right)
\]

Furthermore, let us assume $f_n$ satisfy
\begin{equation}
\label{eq:fundAssumptionOne}
\forall\left( n \in \mathbb{N} \right) \quad 
f_n \in \mathfrak{B}^{I}_{\left\{a\right\}}
\end{equation}
making no such assumption for $f$. 

Finally, assume that there holds
\begin{equation}
\label{eq:fundAssumptionThree}
\lim_{n \rightarrow \infty} \left\Vert f_n \left( a \right)
- f \left( a \right) \right\Vert_{L^2} = 0
\end{equation}
as well as the pointwise convergence
\begin{equation}
\label{eq:fundAssumptionTwo}
f_n \rightarrow f \quad \textnormal{a.e.} \quad \left( t, x, v \right) \in  
I \times \mathbb{R}^2 \times \mathbb{R}^2 
\end{equation}

Then, given all the above, we may conclude that
\begin{equation}
\label{eq:fundConclusionOne}
f \in \mathfrak{B}^{I_0 }_{\left\{a\right\}}
\end{equation}
where 
\[
I_0 = I \bigcap \left[ a, a + T_{\textnormal{g.o.}} \left( f \left( a \right) \right) \right)
\]
and we have
\begin{equation}
\label{eq:fundConclusionTwo}
\lim_{n \rightarrow \infty} \left\Vert f_n
- f  \right\Vert_{L^2 \left( J , L^2 \right) } = 0
\end{equation}
for any compact sub-interval $J \subset I_0$.
\end{lemma}

\begin{proof}
Let $I_0$ be as in the statement of the lemma, and let $J \subset I_0$ be a compact sub-interval.
By Theorem \ref{thm:lsc2}, (\ref{eq:fundAssumptionThree}) implies that
\[
\lim_{n \rightarrow \infty} \left\Vert \mathfrak{Z}_{\textnormal{g.o.}} \left( f_n \left( a \right) \right) \left( \cdot - a \right)
- \mathfrak{Z}_{\textnormal{g.o.}} \left( f \left( a \right) \right)
 \left( \cdot - a\right) \right\Vert_{L^2 \left( J , L^2 \right)} = 0
\]
where we have used the compactness of $J$ to drop from $L^\infty$ to $L^2$ in the time variable.
Therefore, by Lemma \ref{lem:ui} with 
\[
E = J \times \mathbb{R}^2 \times \mathbb{R}^2
\]
we find that the sequence
\[
\left\{ \mathfrak{Z}_{\textnormal{g.o.}} \left( f_n \left( a \right) \right) \left( \cdot - a\right) \right\}_{n}
\]
is uniformly square integrable in $J \times\mathbb{R}^2 \times \mathbb{R}^2$. In particular, by Lemma \ref{lem:uiUpper}
and (\ref{eq:fundAssumptionOne}), the sequence
\[
\left\{ f_n \left( \cdot \right) \right\}_n
\]
is uniformly square integrable in $J \times \mathbb{R}^2 \times \mathbb{R}^2$.
Therefore, by Lemma \ref{lem:uiDCT} and (\ref{eq:fundAssumptionTwo}), we immediately deduce
\begin{equation}
\lim_{n \rightarrow \infty} \left\Vert f_n 
- f \right\Vert_{L^2 \left( J , L^2  \right)} = 0
\end{equation}
which is (\ref{eq:fundConclusionTwo}). In particular, we have
\[
f \in L^2 \left( J, L^2 \right)
\]
So there only remains to prove (\ref{eq:fundConclusionOne}).

Recall again that
\[
\lim_{n \rightarrow \infty} \left\Vert \mathfrak{Z}_{\textnormal{g.o.}} \left( f_n \left( a \right) \right) \left( \cdot - a\right)
- \mathfrak{Z}_{\textnormal{g.o.}} \left( f \left( a \right) \right)
 \left( \cdot - a\right) \right\Vert_{L^2 \left( J, L^2 \right)} = 0
\]
Therefore, passing to a subsequence in $n$, say $n_m$, $m = 1,2,3,\dots$, we find that in the limit
$m \rightarrow \infty$ we have the pointwise convergence
\[
\mathfrak{Z}_{\textnormal{g.o.}} \left( f_{n_m} \left( a \right) \right) \left( \cdot - a\right)
\rightarrow \mathfrak{Z}_{\textnormal{g.o.}} \left( f \left( a \right) \right) \left( \cdot - a \right)
\; \textnormal{a.e.} \;
\left( t, x, v \right) \in J \times \mathbb{R}^2 \times \mathbb{R}^2
\]
Combining this pointwise convergence of $\mathfrak{Z}_{\textnormal{g.o.}}$ with
the pointwise convergence from (\ref{eq:fundAssumptionTwo}), and the fact that
$J$ is an arbitrary compact subinterval of
$I_0$, we find that
\[
f \in \mathfrak{B}^{I_0}_{\left\{a\right\}}
\]
which is (\ref{eq:fundConclusionOne}). Indeed, since $f_{n_m} \in \mathfrak{B}^I_{\left\{a\right\}}$,
\[
\begin{aligned}
& f \left( t \right) - \mathfrak{Z}_{\textnormal{g.o.}} \left( f \left( a \right) \right) \left( t-a \right) \\
& \leq \left[ f\left(t\right) - \mathfrak{Z}_{\textnormal{g.o.}} \left( f \left( a \right) \right) \left( t-a \right) \right]
- \left[ f_{n_m} \left( t \right) - \mathfrak{Z}_{\textnormal{g.o.}} \left( f_{n_m} \left( a \right) \right)
\left( t-a \right) \right] \\
& = \left[ f \left( t \right) - f_{n_m} \left( t \right) \right] -
\left[ \mathfrak{Z}_{\textnormal{g.o.}} \left( f \left( a \right) \right) \left( t-a \right)
- \mathfrak{Z}_{\textnormal{g.o.}} \left( f_{n_m} \left( a \right) \right)
\left( t-a \right) \right]
\end{aligned}
\]
and both terms on the last line tend to zero pointwise almost every $\left( t,x,v \right)$ as $m \rightarrow \infty$.
\end{proof}

\section{Entropy and entropy dissipation}
\label{sec:entropyDissipation}

For any non-negative measurable function $h_0 \left( x,v \right)$ such that 
\[
\mathbf{1}_{0 < h_0 < 1} h_0 \log h_0 \in L^1
\]
 the \emph{entropy}
$H\left( h_0 \right)\in \left( -\infty, +\infty \right]$ is defined by 
\[
H \left( h_0 \right) = \int_{\mathbb{R}^2 \times \mathbb{R}^2} h_0 \left(x,v\right) \log
h_0 \left( x,v \right) dx dv
\]
where the real-valued function $s \mapsto s \log s \; \left( s \geq 0 \right)$
 is understood, by continuity, to take the value $0$ at $s=0$.
 
More generally, we will decompose
\[
H \left( h_0 \right) = H^+ \left( h_0 \right) - H^- \left(h_0 \right)
\]
where
\[
H^+ \left( h_0 \right) = \int_{h_0 > 1}
h_0 \left( x,v \right) \log h_0 \left( x,v \right) dx dv
\]
and
\[
H^- \left( h_0 \right) = \int_{0 < h_0 \leq 1}
h_0 \left( x,v \right) \log \frac{1}{h_0 \left( x,v \right)} dx dv
\]
Recall from (\ref{eq:defSubscriptedNorm}) the norm
\[
\left\Vert h_0 \right\Vert_{L^1_{2,t}} = \int_{\mathbb{R}^2 \times \mathbb{R}^2}
\left( 1 + \left| x - v t \right|^2 + \left| v \right|^2 \right) \left| h_0 \left( x,v \right) \right| dx dv
\]
where $t \in \mathbb{R}$, and $L^1_{2}$ is a shorthand for $L^1_{2,0}$. 
The next lemma shows that the entropy is well-defined in $L^1_2$, 
although possibly taking the value $+\infty$:
to this end, it suffices to prove that the negative part $H^- \left( h_0 \right)$ is finite.

We shall require the \emph{(unsigned) entropy densities}
defined via the functions $\alpha^{\pm} : \mathbb{R} \rightarrow \mathbb{R}$,
\[
\alpha^{-} \left( s \right) = \mathbf{1}_{0 < s < 1} \cdot s \log \frac{1}{s}
\]
\[
\alpha^{+} \left( s \right) = \mathbf{1}_{s > 1} \cdot s \log s
\]
so
\[
H^{\pm} \left( h_0 \right) = \int \alpha^{\pm} \left( h_0 \right) dx dv
\]

\begin{lemma}
\label{lem:alphaPlus}
For any $0 \leq a < b$,
\[
0 \leq
\alpha^{+} \left( b \right) - \alpha^+ \left( a \right) \leq
 \frac{1}{2} \left( b+a \right) \left( b-a \right)
\]
In particular, letting $a = 0$ and $b = s > 0$, we have
\[
\alpha^{+} \left( s \right) \leq \frac{s^2}{2}
\]
\end{lemma}
\begin{proof}
For any $s > 1$ we have
\[
\frac{d}{ds} \alpha^+ \left( s \right) =
1 + \log s \leq 1 + \left( s-1 \right) = s
\]
Hence we may compute by the fundamental theorem of calculus: for any $0 < a < b$,
\[
0 \leq
\alpha^+ \left( b \right) - \alpha^+ \left( a \right) \leq
\int_a^b s ds = \frac{1}{2} \left( b+a \right) \left( b-a \right)
\]
\end{proof}

\begin{figure}[ht]
\centering
\includegraphics[width=10cm]{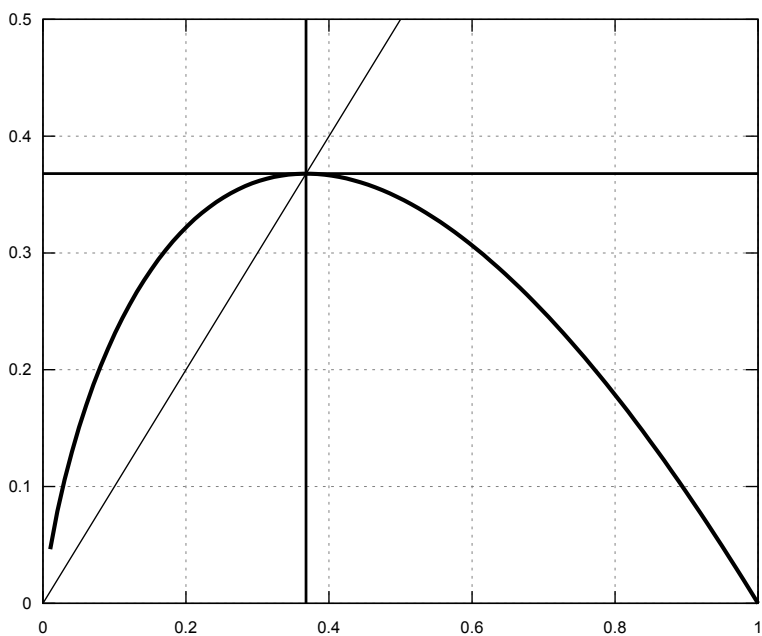}
\captionof{figure}{Graph of $\alpha^{-}$}
\label{fig:alphaminus}
\end{figure}

\begin{lemma}
\label{lem:alphaMinus}
The function $\alpha^{-}$ is continuous on the whole real line; moreover:
\begin{enumerate}
\item The restriction of $\alpha^{-}$ to $\left( 0, 1 \right)$ is
smooth and concave.

\item $\alpha^{-}$ attains a unique maximum value $\alpha^{-} \left( e^{-1} \right)
= e^{-1}$.

\item $\alpha^{-}$ is increasing on $\left(0, e^{-1}\right)$.

\item $\alpha^{-}$ is decreasing on $\left( e^{-1}, 1\right)$.

\item On compact subintervals of $\left( 0, 1 \right]$, $\alpha^{-}$ is Lipschitz.

\item Whenever $s \in \left[ e^{-1}, 1 \right]$, it holds $\alpha^{-} \left( s \right) \leq s$.
\end{enumerate}
\end{lemma}
\begin{proof}
The continuity is trivial, as is the smoothness on $\left( 0, 1 \right)$. The concavity on
$\left( 0,1 \right)$ follows from the formula
\[
\frac{d^2}{ds^2} \alpha^{-} \left( s \right) = - \frac{1}{s} < 0
\]
which in turn implies that $\alpha^{-}$ takes a unique maximum (which must lie in the
interval $\left( 0, 1 \right)$).

Since
\begin{equation}
\label{eq:alphaMinusDeriv}
\frac{d}{ds} \alpha^{-} \left(s\right) = - 1 - \log s
\end{equation}
we easily observe that $\alpha^{-}$ is (strictly) increasing on $\left( 0, e^{-1} \right)$
and (strictly) decreasing on $\left( e^{-1}, 1 \right)$. In particular, the
unique maximum is attained at $s = e^{-1}$, and we compute
\[
\alpha^{-} \left( e^{-1} \right) = e^{-1}
\]

The Lipschitz continuity on $\left( 0, 1 \right]$ follows again from (\ref{eq:alphaMinusDeriv}) and
the fact that $\log s$ is bounded on compact subsets of $\left( 0, 1 \right]$.

Since $\alpha^{-}$ is decreasing on $\left( e^{-1}, 1 \right)$ we can compute, for $s \in
\left( e^{-1}, 1 \right)$,
\[
\begin{aligned}
\alpha^{-} \left(s\right) & \leq
\alpha^{-} \left( e^{-1} \right)
= e^{-1} \leq s
\end{aligned}
\]
Hence $\alpha^{-} \left( s \right) \leq s$ for $s \in \left[ e^{-1}, 1 \right]$.
\end{proof}

\begin{lemma}
\label{lem:HminusFinite}
For any non-negative measurable function $h_0 \in L^1_2$, we have
\[
H^- \left( h_0 \right) < \infty
\]
In fact, for any $T \in \left[ 0, \infty \right)$,
\begin{equation}
\label{eq:HminusBound}
H^- \left( h_0 \right) \leq C_0 + \left\Vert h_0 \right\Vert_{L^1_{2,T}} 
\end{equation}
where the additive constant $C_0$ is given by
\begin{equation}
\label{eq:CzeroDef}
C_0 = \int_{\mathbb{R}^2 \times \mathbb{R}^2} \left( 1 + \left| x \right|^2 +
\left| v \right|^2 \right) \exp \left( - 1 - \left| x \right|^2 - \left| v \right|^2 \right) dx dv
\end{equation}
which is simply $\left| H \left( m_0 \right) \right|$
where $m_0 = \exp \left( - 1 - \left| x \right|^2 - \left| v \right|^2 \right)$.
\end{lemma}
\begin{proof}
We will require the (non-normalized) Gaussian function $m_0$ on $\mathbb{R}^2 \times \mathbb{R}^2$
defined by
\[
m_0 \left( x, v \right) = \exp \left(- 1 - \left| x \right|^2 - \left| v \right|^2 \right)
\]
and we also define via free transport (\ref{eq:freeTransport}), denoted $\mathcal{T}$, 
\[
m_t = \mathcal{T} \left( t \right) m_0
\]
which is what will allow us (by the choice $t=T$)
 to introduce the parameter $T$ in (\ref{eq:HminusBound}) \emph{without}
accepting a $T$-varying loss in constants. Note that $m_0$ (hence $m_T$) is everywhere bounded
above by $e^{-1}$.

Choose an arbitrary time $T$ with $0 \leq T < \infty$, which will be considered fixed for the
rest of the proof of this lemma.

Let us decompose the set $\left\{ 0 < h_0 \leq 1 \right\}$ into two parts, which
we will denote $A, B$, via the formulas
\[
A = \left\{ \; \left( x,v \right) \; : \; 0 < h_0 \left( x,v \right) \leq m_T \left( x,v \right)
 \; \right\}
\]
\[
B = \left\{ \; \left( x,v \right) \; : \; m_T \left( x,v \right) < h_0 \left( x,v \right) \leq 1 \; \right\}
\]
Denote the respective integrals $H^{-}_{A} \left( h_0 \right)$ and $H^{-}_{B} \left( h_0 \right)$,
 providing a decomposition of $H^{-} \left( h_0 \right)$ as their sum.

Let us first handle $H^{-}_{A} \left( h_0 \right)$. Recall that
$\left\Vert m_T \right\Vert_{L^\infty} \leq e^{-1}$. Additionally, by Lemma \ref{lem:alphaMinus}(3),
for $0 < s \leq e^{-1}$ we have that $\alpha^{-}$ is
\emph{increasing} so $H^{-}_{A} \left( h_0 \right)$ has the bound
\[
H^{-}_{A} \left( h_0 \right) \leq H^{-} \left( m_T \right)
\]
But the transport semigroup $\mathcal{T} \left( \cdot \right)$ preserves the Lebesgue measure
on $\mathbb{R}^2 \times \mathbb{R}^2$ so we have
\[
H^{-} \left( m_T \right) = H^{-} \left( m_0 \right)
\]
and, since $0 < m_0 \left( x,v \right) \leq e^{-1} < 1$,
$H^{-} \left( m_0 \right) = - H \left( m_0 \right)$ is the
constant $C_0$ appearing in (\ref{eq:HminusBound}) and (\ref{eq:CzeroDef}).

For $H^{-}_{B} \left( h \right)$, simply observe that
for all $s \in \left( 0, 1 \right]$ the function $s \mapsto \log \frac{1}{s}$ is a non-negative
\emph{decreasing} function, so we can simply bound $\log \frac{1}{h_0}$ by
$\log \frac{1}{m_T}$, that is,
\[
H^{-}_{B} \left( h_0 \right) \leq
\int_{\mathbb{R}^2 \times \mathbb{R}^2} 
\left( 1 + \left| x - v T \right|^2 + \left| v \right|^2 \right) h_0 \left( x,v \right) dx dv
\]
and the right-hand side is just $\left\Vert h_0 \right\Vert_{L^1_{2,T}}$.
\end{proof}

The (local) \emph{instantaneous entropy dissipation} $\mathcal{D}$,
 corresponding to (\ref{eq:BE}), is defined for any
non-negative measurable function $h \left( t,x,v \right)$ by
the formula
\begin{equation}
\label{eq:defEntrDis}
\left[ \mathcal{D} \left( h \right) \right] \left( t,x \right) =
\frac{1}{4} \int_{\mathbb{S}^1 \times \mathbb{R}^2 \times \mathbb{R}^2}
\left( h^\prime h_*^\prime - h h_* \right) \log
\frac{h^\prime h_*^\prime}{h h_*} d\sigma dv dv_*
\end{equation}
where $h = h \left( t,x,v \right)$, $h_* = h \left( t,x, v_* \right)$,
$h^\prime = h \left( t,x, v^\prime\right)$, and $h_*^\prime = h \left( t,x, v_*^\prime \right)$.
Now since
\[
\left( a-b \right) \log \frac{a}{b}
\]
is non-negative for each pair of positive numbers $a,b$ (since the sign of
$a-b$ is always equal to the sign of $\left(\log a - \log b\right)$), it follows that
$\mathcal{D}\left(h\right)$ is always  non-negative (although it may be infinite).

Any $0 \leq f \in C^1 \left( \left[ 0, T \right], \mathcal{S} \right)$
solving (\ref{eq:BE}) on $\left[ 0, T \right]$ with initial data
$f_0 = f \left( t=0 \right)$ is known to satisfy the \emph{entropy identity}
\begin{equation}
\label{eq:entropyIdentity}
H \left( f \left( t \right) \right) + 
\int_0^t \int_{\mathbb{R}^2} \mathcal{D} \left( f \left( \tau \right) \right) dx d\tau =
H \left( f_0 \right)
\end{equation}
each $0 \leq t \leq T$.
The (space-)time integral of the (local) instantaneous entropy dissipation is simply known as
the \emph{entropy dissipation} (at time $t$, although the $t$ dependence may be
suppressed).

\begin{remark}
The integrand in the dissipation functional is possibly ambiguous if the quantities
$h, h_*, h^\prime, h_*^{\prime}$ vanish at some point of the integration domain.
Such a situation cannot happen at $t > 0$ for classical
solutions of (\ref{eq:BE}) as long as the initial data is not \emph{identically} zero
(see \cite{Vi2002}, Chapter 2, Section 6, titled ``Lower bounds,''
and references therein). Unfortunately, it is sometimes hard to prove that $f \left(t,x,v\right) > 0$ a.e. 
$\left( x,v \right)$ for $t > 0$  at low regularity. The convention used by
DiPerna and Lions in \cite{DPL1991} is to set the integrand to infinity at any point
$\left( t, x, v, v_*, \sigma \right)$ where any
of $h, h_*, h^\prime, h_*^\prime$ vanishes, and we follow the same convention so as to make use of their results.
\end{remark}

In regimes of lesser regularity, the equality
(\ref{eq:entropyIdentity}) may be downgraded to an \emph{entropy inequality},
or fail altogether. In the $L^2$ regime, the version of the entropy inequality we shall ultimately
require is
\[
H \left( f \left( t \right) \right) + 
\int_0^t \int_{\mathbb{R}^2} \mathcal{D} \left( f \left( \tau \right) \right) dx d\tau \leq
H \left( f_0 \right)
\]
almost every $t$. To this end, we consider the
terms $H^+$ and $H^-$ separately:

\begin{lemma}
\label{lem:HplusBd}
For any non-negative measurable
 function $h_0 \in L^2$, the positive entropy integral $H^+ \left( h_0 \right)$ is finite, being bounded by
 the \emph{square} of the $L^2$ norm:
 \begin{equation}
 \label{eq:HplusBdOne}
 H^+ \left( h_0 \right) \leq \left\Vert h_0 \right\Vert_{L^2}^2
 \end{equation}
 Moreover, for any $h_{0,1}, h_{0,2} \in L^2$ we have
 \begin{equation}
 \label{eq:HplusBdTwo}
 \left| H^+ \left( h_{0,1} \right) - H^+ \left( h_{0,2} \right) \right| \leq
\left( \max_{i \in \left\{ 1, 2 \right\}}
\left\Vert h_{0,i} \right\Vert_{L^2} \right) \cdot
\left\Vert h_{0,1} - h_{0,2} \right\Vert_{L^2}
 \end{equation}
\end{lemma}
\begin{proof}
Both bounds follow immediately from Lemma \ref{lem:alphaPlus}. In particular, for
(\ref{eq:HplusBdTwo}) we use Lemma \ref{lem:alphaPlus} with the
 Cauchy-Schwarz inequality,
\[
 \left| H^+ \left( h_{0,1} \right) - H^+ \left( h_{0,2} \right) \right| \leq
 \frac{1}{2}
\left\Vert h_{0,1} + h_{0,2} \right\Vert_{L^2}
\left\Vert h_{0,1} - h_{0,2} \right\Vert_{L^2}
\]
and conclude by the triangle inequality.
\end{proof}

\begin{lemma}
\label{lem:HminusConv}
Let us be given non-negative measurable functions
\[
h_0, h_{0,n} \in  L^1_2
\]
for $n=1,2,3,\dots$, such that
\[
\sup_{n \in \mathbb{N}} \left\Vert h_{0,n} \right\Vert_{L^1_2} < \infty
\]
and
\[
h_{0,n} \left( x,v \right) \rightarrow h_0 \left( x,v \right) \quad
\textnormal{a.e.} \quad \left( x,v \right) \in \mathbb{R}^2 \times \mathbb{R}^2
\]
as $n \rightarrow \infty$. 
Then
\[
\lim_{n \rightarrow \infty} H^- \left( h_{0,n} \right) = H^- \left( h_0 \right)
\]
\end{lemma}
\begin{proof}
Let us denote
\[
E_R = \left\{ \; \left( x,v \right) \in \mathbb{R}^2 \times \mathbb{R}^2 \;\; : \;\;
1 + \left|x\right|^2 + \left|v\right|^2 \geq R^2 \; \right\}
\]
and observe that for each $R > 1$, by the continuity of $\alpha^{-}$,
 we may apply the dominated convergence theorem  on the
\emph{complement} of $E_R$ due to the
fact that
\[
\forall \left( s \in \mathbb{R} \right)
 \qquad 0 \; \leq \; \alpha^{-} \left( s \right) \; \leq \;  e^{-1}
\]
and the complement $E_R^C$ of $E_R$ is a bounded set: that is,
\[
\forall \left( R > 0 \right) \qquad
\lim_n \int_{E_R^C} \alpha^{-} \left( h_{0,n} \right) dx dv =
\int_{E_R^C} \alpha^{-} \left( h_0 \right) dx dv
\]

Hence if only we can show
\[
\lim_{R \rightarrow \infty} \sup_{n \in \mathbb{N}}
\int_{E_R}
\alpha^{-} \left( h_{0, n} \right) dx dv = 0
\]
then we will be done. To this end, we will decompose $H^-$ in a manner similar to the proof
of Lemma \ref{lem:HminusFinite}.

Let us define the \emph{non-Gaussian} function
\[
\gamma \left( x,v \right) = \exp \left[
- \left( 1 + \left|x \right|^2 + \left| v \right|^2 \right)^{\frac{1}{2}} \right]
\]
and note that $\left\Vert \gamma \right\Vert_{L^\infty} = e^{-1}$.
Let us consider separately the sets (depending on $n$) where
\[
0 \leq h_{0,n} \left( x,v \right) \leq \gamma \left( x,v \right)
\]
and
\[
\gamma \left( x,v \right) < h_{0,n} \left( x,v \right) \leq 1
\]
In the first case we have, by Lemma \ref{lem:alphaMinus}(3) and the fact
that $\left\Vert \gamma \right\Vert_{L^\infty} = e^{-1}$,
\begin{equation*}
\begin{aligned}
&  \lim_{R \rightarrow \infty} \sup_{n \in \mathbb{N}}
\int_{E_R}
\mathbf{1}_{h_{0,n} \leq \gamma} \cdot
\alpha^{-} \left( h_{0,n} \right) dx dv \\
& \qquad \qquad \qquad \qquad \qquad \qquad\qquad
\leq \lim_{R \rightarrow \infty} \int_{E_R}
\alpha^{-} \left( \gamma \right) dx dv = 0
\end{aligned}
\end{equation*}
so it only remains to
 consider the second case. Then again, since $\log \frac{1}{s}$ is a decreasing non-negative
function for $0 < s \leq 1$, for the second case we only need to show
\begin{equation}
\label{eq:HminusLastStep}
\lim_{R \rightarrow \infty} \sup_{n \in \mathbb{N}}
\int_{E_R}
\left( 1 + \left| x \right|^2 + \left| v \right|^2 \right)^{\frac{1}{2}}
h_{0,n} dx dv = 0
\end{equation}
but this follows immediately from the uniform boundedness of the sequence
$\left\{ h_{0, n} \right\}_n$ in $L^1_2$, since
\[
\left( 1 + \left| x \right|^2 + \left| v \right|^2 \right)^{\frac{1}{2}} \leq
\frac{1}{R} \left(1 + \left| x\right|^2 + \left| v \right|^2 \right)
\]
for each $\left( x,v \right) \in E_R$.
\end{proof}

\section{($*$)-solutions}
\label{sec:starSolutions}

\subsection{Definitions.}

We recall the notion of \emph{renormalized solution} as introduced by DiPerna and Lions.

\begin{definition}
(\cite{DPL1989})
Let $I = \left[ a, b \right)$, where $-\infty < a < b \leq \infty$, and suppose
\[
0 \leq f \in L^1_{\textnormal{loc}} \left( I \times \mathbb{R}^2 \times \mathbb{R}^2 \right)
\]
Then we say $f$ is a \emph{renormalized solution} of (\ref{eq:BE}) provided that
\[
\frac{1}{1+f} Q^{\pm} \left( f, f \right) \in L^1_{\textnormal{loc}} \left( I \times \mathbb{R}^2 \times \mathbb{R}^2
\right)
\]
and it holds
\[
\left( \frac{\partial}{\partial t} + v \cdot \nabla_x \right) \log \left( 1 + f \right)
= \frac{1}{1+f} \left( Q^+ \left( f, f \right) - Q^- \left( f,f \right) \right)
\]
in the sense of distributions on $I \times \mathbb{R}^2 \times \mathbb{R}^2$.
\end{definition}

Throughout this article, although not formally required in the definition of renormalized
solutions, we impose the requirement that
 solutions of Boltzmann's equation will at least be in $L^1$ uniformly in $t$, which in particular
implies that $\rho_f$ is in $L^1_x$ uniformly in $t$, i.e.
\[
f \in L^\infty \left( I, L^1 \right) \quad
\textnormal{ and } \quad
\rho_f \in L^\infty \left( I, L^1_x \left( \mathbb{R}^2 \right) \right)
\]
 In particular, given a solution $f$ defined for
$a \leq t < b$, the function $F \left( t,x, v \right)$ satisfying
\[
F^\# \left( t,x,v \right) = \int_a^t \left( \rho_f \right)^\# \left( \sigma, x, v \right) d\sigma
\]
is well-defined almost everywhere (recall that the notation $F^\#$ is defined by (\ref{eq:defHashNotation})).
 \cite{DPL1989} Thus
we can view $\rho_f$ as an \emph{integrating factor} in Boltzmann's equation to write a solution $f$ in the
form
\begin{equation}
\label{eq:BEIF}
\begin{aligned}
& f^\# \left( t,x,v \right) - f^\# \left( s,x,v \right) \exp \left( - \left( F^\# \left( t,x,v \right) -
F^\# \left( s,x,v \right) \right) \right) \\
& = \int_s^t Q^+ \left( f,f \right)^\# \left( \tau , x, v \right)
\cdot \exp \left( - \left( F^\# \left( t, x, v \right) - F^\# \left( \tau, x, v \right) \right) \right) d\tau
\end{aligned}
\end{equation}
for almost all $x,v \in \mathbb{R}^2$ and
 $a \leq s < t < b$. This form of Boltzmann's equation is particularly convenient 
because it can be stated under minimal integrability assumptions (for example, neither $Q^+$ nor
$Q^-$ need be locally integrable, as long as they can be integrated \emph{along almost every characteristic}).
 It is possible to show \cite{DPL1989} that renormalized
solutions of Boltzmann's equation (\ref{eq:BE}) (having constant collision kernel, so that the loss term is proportional to
$\rho_f$) verify (\ref{eq:BEIF}) whenever $\rho_f \in L^\infty \left( I, L^1_x \left( \mathbb{R}^2 \right) \right)$.

We are now ready to define ($*$)-solutions of (\ref{eq:BE}), although we defer til Section
\ref{sec:ExistenceProof} the proof of their existence.
Recall again, from (\ref{eq:defSubscriptedNorm}),
\[
\left\Vert h_0 \right\Vert_{L^1_2} = \int_{\mathbb{R}^2 \times \mathbb{R}^2}
\left( 1 + \left| x \right|^2 + \left| v \right|^2 \right) \left| h_0 \left( x,v \right) \right| dx dv
\]
and that the entropy is well-defined and finite on $L^2 \bigcap L^1_2$, by Lemma \ref{lem:HminusFinite}
and Lemma \ref{lem:HplusBd}.

\begin{definition}
\label{def:starSoln}
Let us be given a non-negative
measurable function 
\[
f \in L^1_{\textnormal{loc}} \left( \left[ 0, \infty \right) \times \mathbb{R}^2 \times \mathbb{R}^2 \right)
\] 
Then we will say that $f$ is a
\emph{($*$)-solution} of (\ref{eq:BE}) provided that $f$ is a renormalized solution of (\ref{eq:BE}) on 
$\left[ 0, \infty \right)$, with
\begin{equation}
\label{eq:dataCondition}
f \left( t = 0 \right) = f_0 \in L^2 \bigcap L^1_2
\end{equation}
for which
\begin{equation}
\label{eq:continuityL1}
f \in C \left( \left[ 0, \infty \right), L^1 \right)
\end{equation}
 and that
there exists a number $T^* \left( f \right)$,
\[
0 < T^* \left( f \right) \leq \infty
\]
with corresponding interval
\[
I^* \left( f \right) = \left[ 0, T^* \left( f \right) \right)
\]
such that the following holds:

For each compact sub-interval $J \subset I^* \left( f \right)$ it holds
\begin{equation}
\label{eq:continuityCondition}
f \in C \left( J, L^2 \right)
\end{equation}
and
\begin{equation}
\label{eq:kmCondition}
Q^+ \left( f,f \right) \in L^1 \left( J, L^2 \right)
\end{equation}
and, in the event $T^* \left( f \right) < \infty$, we also require
\begin{equation}
\label{eq:notinCondition}
f \notin C \left( \left[ 0, T^* \left( f \right) \right] , L^2 \right)
\end{equation}

Additionally, we require that for almost every
$t$ with
\[
0 \leq t < \infty
\]
we have each of the following estimates:
\begin{equation}
\label{eq:massCondition}
\int_{\mathbb{R}^2 \times \mathbb{R}^2} f \left( t \right) dx dv = \int_{\mathbb{R}^2 \times \mathbb{R}^2}
 f_0 dx dv
\end{equation}
\begin{equation}
\label{eq:momentumCondition}
\int_{\mathbb{R}^2 \times \mathbb{R}^2} v f \left( t \right) dx dv = \int_{\mathbb{R}^2 \times \mathbb{R}^2}
v f_0 dx dv
\end{equation}
\begin{equation}
\label{eq:energyCondition}
\int_{\mathbb{R}^2 \times \mathbb{R}^2}
 \left| v \right|^2 f \left( t \right) dx dv \leq \int_{\mathbb{R}^2 \times \mathbb{R}^2}
  \left| v \right|^2 f_0 dx dv
\end{equation}
\begin{equation}
\label{eq:dispCondition}
\int_{\mathbb{R}^2 \times \mathbb{R}^2}
 \left| x-vt \right|^2 f \left( t \right) dx dv \leq \int_{\mathbb{R}^2 \times \mathbb{R}^2}
   \left| x\right|^2 f_0 dx dv
\end{equation}
and
\begin{equation}
\label{eq:entropyCondition}
H \left( f \left( t \right) \right) +
\int_0^t \int_{\mathbb{R}^2} \mathcal{D} \left( f \left( s \right) \right) dx ds \leq
H \left( f_0 \right)
\end{equation}
\end{definition}

\begin{remark}
Note carefully that uniqueness is unknown, at present, in the class of
($*$)-solutions for a given initial data $f_0$, even on an arbitrarily
 small time interval $\left[ 0, \delta \right] \subset
I^* \left( f \right)$.
This is why the notation
$T^* \left( f \right), I^* \left( f \right)$ makes explicit reference to the solution
$f$, not only the initial data $f_0$: for a given $f_0$ there may well be multiple ($*$)-solutions
$f$ with initial data $f_0$ but different values of $T^* \left( f \right)$.
\end{remark}

\subsection{Discussion.}

A ($*$)-solution, as provided by Definition \ref{def:starSoln}, is intuitively understood as a \emph{global}
renormalized solution which happens to be (simultaneously) a \emph{distributional} solution on some
(possibly finite) interval $I^* \left(f\right)$. The solution can be viewed as an $L^2$ solution
on $I^* \left( f \right)$, but the solution is not continuous into $L^2$ on any interval $J$ containing
$I^* \left( f \right)$ as a proper subset; therefore, the solution is in this sense \emph{maximal}.
 Note carefully that maximality is for the \emph{solution},
not the \emph{data}, in view of possible non-uniqueness: two maximal solutions need not coincide for any $t > 0$, nor do
their intervals $I^*$ need to coincide. 

The idea of constructing a renormalized solution of (\ref{eq:BE}), which is also a solution in some stronger
sense on some initial interval, has been studied previously by Lions: see \cite{L1994II},
Theorem V.1. In that reference, Lions establishes a class of global renormalized solutions which
satisfy \emph{in addition} certain differential inequalities, which Lions refers to as
\emph{dissipation inequalities}; the solutions so obtained are called
\emph{dissipative solutions}. He proves the existence of such
solutions (Theorems IV.1 and IV.2 of the same reference); however,
 general renormalized solutions are \emph{not} guaranteed to 
satisfy such dissipation inequalities.
 The differential inequalities are defined via
\emph{testing} a dissipative solution 
$f$ against functions drawn from a class of higher integrability and decay. (Here \emph{testing} is
not meant in a distributional sense, but a different sense reminiscent of viscosity solutions.)
Taking a \emph{classical} (or sufficiently strong) solution $\tilde{f}$ \emph{as the test function}
in the differential inequalities leads immediately to his Theorem V.1 on weak-strong uniqueness, namely
$f = \tilde{f}$ insofar as $\tilde{f}$ is defined and so controlled (i.e. on the initial
interval).

\begin{remark}
At no point in this paper do we employ dissipative
solutions, weak solutions in the sense of \cite{L1994II}, or differential inequalities so obtained,
 although we mention them in passing; note carefully that the strong compactness result of
 \cite{L1994II} does \emph{not} require dissipation inequalities in its general formulation, namely
 Theorem II.1 of that reference.
\end{remark}

\subsection{Integrability and time continuity.}
\label{ssec:integrableAndContinuous}

The objective of this sub-section is to show that the $Q^+$ bound (\ref{eq:kmCondition}) combined with the initial
data condition
(\ref{eq:dataCondition}) automatically implies
the $L^2$ continuity (\ref{eq:continuityCondition}), and that ($*$)-solutions are \emph{distributional}
solutions of (\ref{eq:BE}) on $I^* \left( f \right)$: in particular, each of $Q^+$ and $Q^-$ is in
\[
L^1_{\textnormal{loc}} \left( I^* \left( f \right) \times \mathbb{R}^2 \times \mathbb{R}^2 \right)
\]
Of course this is immediate for $Q^+$ from our assumption (\ref{eq:kmCondition}); hence, we only have
to prove local integrability for $Q^-$, and the $L^2$ time continuity of $f$.

\begin{lemma}
\label{lem:rhoL6}
For any compact interval $J \subset \mathbb{R}$,
and any $f\left(t,x,v\right)$ such that the right-hand side is finite, it holds
\[
\left\Vert \rho_f \right\Vert_{L^6 \left( J, L^{3/2}_x
\left( \mathbb{R}^2 \right)
\right)}
 \leq C
\left\Vert \left< v \right>^2 f \right\Vert_{L^\infty \left( J,
 L^1\right)}^{\frac{1}{2}}
 \left\Vert f \right\Vert_{L^3 \left( J, L^3_x L^{3/2}_v
 \left( \mathbb{R}^2 \times \mathbb{R}^2 \right)\right)}^{\frac{1}{2}}
\]
the constant depending on neither $J$ nor $f$.
\end{lemma}
\begin{proof}
By H{\" o}lder's inequality,
\[
\rho_f \left( t,x \right) = \int_{\mathbb{R}^2} f\left(t,x,v\right) dv \leq C
\left\Vert \left< v \right> f\left(t,x,v\right) \right\Vert_{L^{6/5}_v \left( \mathbb{R}^2 \right)}
\]
Also, by interpolation
\[
\left\Vert \left<  v \right> f \left( t \right) \right\Vert_{L^{6/5}_v
\left( \mathbb{R}^2 \right)} \leq
C \left\Vert \left<v \right>^2 f\left(t\right) \right\Vert_{L^1_v
\left(\mathbb{R}^2\right)}^{\frac{1}{2}}
\left\Vert f \left(t\right) \right\Vert_{L^{3/2}_v
\left(\mathbb{R}^2\right)}^{\frac{1}{2}} 
\]
hence
\[
\rho_f \left( t,x \right) \leq
C \left\Vert \left<v \right>^2 f \left( t,x,v \right)\right\Vert_{L^1_v
\left( \mathbb{R}^2 \right)}^{\frac{1}{2}}
\left\Vert f \left( t,x,v \right) \right\Vert_{L^{3/2}_v
\left( \mathbb{R}^2 \right)}^{\frac{1}{2}} 
\]
Apply the norm $L^{\frac{3}{2}}_x \left(\mathbb{R}^2\right)$ to both sides and use H{\" o}lder.
\[
\left\Vert \rho_f \left( t \right) \right\Vert_{L^{3/2}_x \left( \mathbb{R}^2 \right)} \leq
C \left\Vert \left< v \right>^2 f \left( t \right) \right\Vert_{L^1}^{\frac{1}{2}}
\left\Vert f \left(t\right) \right\Vert_{L^3_x L^{3/2}_v \left( \mathbb{R}^2 \times \mathbb{R}^2 \right)}^{\frac{1}{2}}
\]
Take the $L^6 \left( J\right)$ norm for the $t$ variable on both sides and apply
H{\" o}lder's inequality once more to conclude.
\end{proof}

Let us show that a renormalized solution on $\left[ 0, \infty \right)$ satisfying (\ref{eq:dataCondition}),
 (\ref{eq:massCondition}),
 (\ref{eq:energyCondition}), and (\ref{eq:dispCondition}),
 as well as 
(\ref{eq:kmCondition}) with $J = \left[ 0, T \right]$, automatically satisfies the local integrability
\[
Q^- \left( f,f \right) =
\rho_f f \in L^1_{\textnormal{loc}} \left(
J \times \mathbb{R}^2 \times \mathbb{R}^2 \right)
\]
We will need the Strichartz
estimates
\begin{equation}
\label{eq:strich000}
\left\Vert \mathcal{T}  h_0
\right\Vert_{L^3_t L^3_x L^{\frac{3}{2}}_v
\left( \mathbb{R} \times \mathbb{R}^2 \times \mathbb{R}^2\right)}
 \leq C \left\Vert h_0 \right\Vert_{L^2}
\end{equation}
and
\begin{equation}
\label{eq:strich001}
\left\Vert \mathcal{T} h_0
\right\Vert_{L^{\frac{7}{3}}_t L^{\frac{7}{2}}_x L^{\frac{7}{5}}_v
\left( \mathbb{R} \times \mathbb{R}^2 \times \mathbb{R}^2\right)} \leq C \left\Vert h_0 \right\Vert_{L^2}
\end{equation}
which hold for any $h_0 \in L^2$
by Proposition \ref{prop:CaP}.

First, for $t \in J$ by (\ref{eq:BEIF}) we have
the \emph{pointwise} upper bound
\begin{equation}
\label{eq:PtwiseDuhamelUpper}
f\left( t,x,v \right) \leq
\left[ \mathcal{T}\left( t \right)f_0 \right] \left( x,v \right) +
\int_0^t \left[ \mathcal{T} \left( t-s \right) Q^+ \left( f,f \right) \left( s \right)
\right] \left( x,v \right) ds
\end{equation}
Here we have used the non-negativity of $\rho_f$ to bound the exponential factors involving $F^\#$
(i.e. integrating factors) uniformly by  $1$.

In any case, substituting $T$ for $t$ in the upper limit of the Duhamel integral on the right
side of (\ref{eq:PtwiseDuhamelUpper}), and using the non-negativity of $Q^+ \left( f,f \right)$, we have
\[
f\left( t,x,v \right) \leq
\left[ \mathcal{T}\left( t \right)f_0 \right] \left( x,v \right) +
\int_0^{T} \left[ \mathcal{T} \left( t-s \right) Q^+ \left( f,f \right) \left( s \right)
\right] \left( x,v \right) ds
\]
Applying Minkowski's inequality and  (\ref{eq:strich000}), we obtain
\begin{equation}
\label{eq:strichABCD}
\begin{aligned}
& \left\Vert f \right\Vert_{L^3 \left( J, L^3_x L^{\frac{3}{2}}_v
\left( \mathbb{R}^2 \times \mathbb{R}^2 \right)\right)} \\
& \leq
C \left( \left\Vert f_0 \right\Vert_{L^2} + 
\int_0^{T} \left\Vert \mathcal{T} \left( t-s \right) 
Q^+ \left( f,f \right) \left( s \right) \right\Vert_{L^3 \left( J, L^3_x L^{\frac{3}{2}}_v
\left( \mathbb{R}^2 \times \mathbb{R}^2 \right)\right)} ds \right) \\
& \leq
C \left( \left\Vert f_0 \right\Vert_{L^2} + 
\int_0^{T} \left\Vert Q^+ \left( f,f \right) \left( s \right) \right\Vert_{L^2} ds \right)
\end{aligned}
\end{equation}
and the right-hand side is finite by hypothesis.
Combining this with Lemma \ref{lem:rhoL6} and the fact that $f \left( t \right) \in L^1_{2,t}$ each
 $t$ allows us to conclude that
\[
\rho_f \in L^6 \left( J, L^{\frac{3}{2}}_x \left( \mathbb{R}^2 \right)\right)
\]
On the other hand, by H{\" o}lder's inequality,
\[
\left\Vert \rho_f f \right\Vert_{L^{\frac{42}{25}}_t L^{\frac{21}{20}}_x L^{\frac{7}{5}}_v
\left( J \times \mathbb{R}^2 \times \mathbb{R}^2 \right)} \leq
\left\Vert \rho_f \right\Vert_{L^6_t L^{\frac{3}{2}}_x
\left( J \times \mathbb{R}^2 \right)}
\left\Vert f \right\Vert_{L^{\frac{7}{3}}_t L^{\frac{7}{2}}_x L^{\frac{7}{5}}_v
\left( J \times \mathbb{R}^2 \times \mathbb{R}^2 \right)}
\]
so arguing again by the Duhamel inequality, as in the proof of (\ref{eq:strichABCD}) above,  using
now (\ref{eq:strich001}) to place 
\[
f \in L^{\frac{7}{3}}_t L^{\frac{7}{2}}_x L^{\frac{7}{5}}_v
\left( J \times \mathbb{R}^2 \times \mathbb{R}^2 \right)
\]
we may conclude that
\[
Q^- \left( f,f \right) = \rho_f f \in L^{\frac{42}{25}}_t L^{\frac{21}{20}}_x L^{\frac{7}{5}}_v
\left( J \times \mathbb{R}^2 \times \mathbb{R}^2 \right)
\]
so, in particular, 
\[
Q^- \left( f,f \right) =
\rho_f f \in L^1_{\textnormal{loc}}
\left( J \times \mathbb{R}^2 \times \mathbb{R}^2 \right)
\]
Thus (\ref{eq:BE}) holds in the sense
of distributions on $J = \left[ 0, T \right]$.

\begin{remark}
We have actually shown more, namely that for a ($*$)-solution $f$,
\[
Q^- \left( f,f \right) \in L^p_{t,x,v,\textnormal{loc}} \left( I^* \left( f \right)\times
\mathbb{R}^2_x \times \mathbb{R}^2_v \right) 
\]
for some $p > 1$. It is \emph{also} true that
\begin{equation}
\label{eq:QplusP}
Q^+ \left( f,f \right) \in L^p_{t,x,v,\textnormal{loc}} \left( I^* \left( f \right)\times
\mathbb{R}^2_x \times \mathbb{R}^2_v \right)
\end{equation}
for some $p > 1$, 
although it does not follow immediately from (\ref{eq:kmCondition}) alone, due to the
$L^1$ integrability in time. There are many ways
to see this (e.g. using Strichartz and convolution inequalities),
but perhaps the simplest is to use conservation of mass to interpolate against
(\ref{eq:kmCondition}). Indeed,
\begin{equation}
\label{eq:Qhalf}
\left\Vert Q^{\pm} \left( f,f \right) \right\Vert_{L^\infty_t L^{\frac{1}{2}}_x
L^1_v \left( \left[ 0,\infty\right) \times \mathbb{R}^2_x \times \mathbb{R}^2_v \right)}
\leq C \left\Vert f \right\Vert_{L^\infty_t L^1_{x,v} \left(
\left[ 0,\infty \right) \times \mathbb{R}^2_x \times \mathbb{R}^2_v\right)}^2
\end{equation}
follows (by H{\" o}lder's inequality)
 immediately from the fact that, considered in the velocity variable only, due to the
constant collision kernel, $Q^{\pm}$ is continuous
as a map $L^1_v \left( \mathbb{R}^2 \right) \times L^1_v \left( \mathbb{R}^2 \right)
\rightarrow L^1_v \left( \mathbb{R}^2 \right)$. Interpolating (\ref{eq:Qhalf}) against
(\ref{eq:kmCondition}) (which remains a valid operation in fractional integrability in this case),
an epsilon away from the (\ref{eq:kmCondition}) endpoint,
provides a quantitative $p > 1$ for which (\ref{eq:QplusP}) holds. 
\end{remark}

It remains to show, again with $J = \left[ 0,T \right]$ and under the same assumptions, that
\[
f \in C \left( J, L^2 \right)
\] 
Indeed, 
we have by Duhamel's formula, for $t \in J$,
\[
\mathcal{T} \left( -t \right) f \left( t \right) +
\int_0^t \mathcal{T} \left( -s \right) Q^- \left( f,f \right) \left(s\right)ds =
f_0 + \int_0^t \mathcal{T} \left( -s \right) Q^+ \left( f,f \right)\left( s \right) ds
\]
and the terms are all non-negative (on both sides).
Since $f_0 \in L^2$ and $Q^+ \left( f,f \right) \in L^1 \left( J, L^2\right)$,
 we therefore have
\[
\mathcal{T} \left( -t \right) Q^- \left( f,f \right) \left( t \right) \in
L^2 \left( \mathbb{R}^2_x \times \mathbb{R}^2_v, L^1_t \left( J, \mathbb{R} \right) \right)
\]
Of course we also have
\[
\mathcal{T} \left( -t \right) Q^+ \left( f,f \right)  \left( t \right) \in
L^2 \left( \mathbb{R}^2_x \times \mathbb{R}^2_v, L^1_t \left( J, \mathbb{R} \right) \right)
\]
which follows directly from our hypothesis $Q^+ \left( f,f \right) \in
L^1 \left( J, L^2 \right)$ and Minkowski's inequality.

Now by Duhamel again, for $0 \leq s \leq t \leq T$ we have
\[
\mathcal{T} \left( -t \right) f \left( t \right) -
\mathcal{T} \left( -s \right) f \left( s \right) =
\int_s^t \mathcal{T} \left( - \tau\right) \left\{ Q^+ \left( f,f \right) -
Q^- \left( f, f \right) \right\} \left( \tau \right) d\tau
\]
so taking the $L^2$ norm of both sides (\emph{without} applying Minkowski) it follows from dominated
convergence in time\footnote{expand the
$L^2$ norm of the Duhamel integral to obtain a double integral involving \emph{two} time
variables, say $\tau$ and $\tau^\prime$, and let $t,s$ each be drawn from a shrinking
family of open neighborhoods of some fixed $t_0 \in J$} that the map
\[
t \mapsto \mathcal{T} \left( -t \right) f \left( t \right)
\]
is in the class
\[
 C \left( J, L^2 \right)
\]
But the continuity of $\mathcal{T} \left( -t \right) f \left( t \right)$ is equivalent to the
continuity of $f \left( t \right)$, so we find
that $f \in C \left( J, L^2 \right)$.

We also have:

\begin{proposition}
\label{prop:comparisonStar}
If $f$ is a ($*$)-solution of (\ref{eq:BE}) then
\[
f \in \mathfrak{B}^{I^* \left( f \right)}_{I^* \left( f \right)}
\]
\end{proposition}
\begin{proof}
Follows immediately from Proposition \ref{prop:comparisonConverse} and the definition of ($*$)-solution.
\end{proof}

\section{Criterion on finite-time breakdown of continuity}
\label{sec:breakdownCriterion}

The criterion (\ref{eq:notinCondition}) in the definition of ($*$)-solutions implies that
($*$)-solutions are in some sense \emph{maximal} (indeed, verifying this maximality plays a central role
in the proof of \emph{existence} of ($*$)-solutions, as we shall see in Section \ref{sec:ExistenceProof}).
One might conjecture, based on Corollary \ref{cor:gain-only-blowup}, that
\begin{equation}
\label{eq:unknownCriterion}
\lim_{t \rightarrow T^* \left( f \right)^-} \left\Vert f \left( t \right) \right\Vert_{L^2} = \infty
\end{equation}
whenever $T^* \left( f \right) < \infty$; however, it is not at all clear whether (\ref{eq:unknownCriterion}) holds
for every ($*$)-solution $f$ of (\ref{eq:BE}) with $T^* \left( f \right) < \infty$.
Indeed (\ref{eq:unknownCriterion}) cannot follow simply from the local existence theory\footnote{or uniqueness, for that
matter, should it hold}, due to the
scaling-criticality of $L^2$ for (\ref{eq:BE}).

Nevertheless, there \emph{are} several scaling-critical criteria which one can prove for finite-time breakdown
of continuity of (\ref{eq:BE}): the next Theorem establishes two such criteria, one stated in terms of the gain-only
flow (namely $T_{\textnormal{g.o.}}$), the other in terms of a time integral for the gain term $Q^+$,
 reminiscent of functional settings studied by Klainerman and Machedon. \cite{KM1993,KM2008}

\begin{theorem}
\label{thm:breakdownRegularity}
Let $f$ be a ($*$)-solution of (\ref{eq:BE}) corresponding to some initial data 
\[
0 \leq f_0 \in L^2 \bigcap L^1_2
\]
Then each of the following is true:
\begin{enumerate}
\item
For any compact sub-interval $J \subset I^* \left( f \right)$,
\begin{equation}
\label{eq:TgoBddBelowCpt}
\inf_{t \in J} T_{\textnormal{g.o.}} \left( f \left( t \right) \right) > 0
\end{equation}

\item
Either $T^* \left( f \right) = \infty$ or each of the following holds:
\begin{enumerate}
\item
For any $t \in I^* \left( f \right)$ there holds
\begin{equation}
\label{eq:TgoUpperBd}
T_{\textnormal{g.o.}} \left( f \left( t \right) \right) \leq T^* \left( f \right) - t
\end{equation}
hence
\[
\inf_{t \in I^* \left( f \right)} T_{\textnormal{g.o.}} \left( f \left( t \right) \right) = 0
\]

\item
There holds
\begin{equation}
\label{eq:QplusBlowUp}
\int_{I^* \left( f \right)} 
\left\Vert Q^+ \left( f,f \right) \left( t \right)
\right\Vert_{L^2} dt = \infty
\end{equation}
\end{enumerate}
\end{enumerate}
\end{theorem}

\begin{proof}
The (unconditional)
 first claim (\ref{eq:TgoBddBelowCpt}) follows immediately from the lower semi-continuity of $T_{\textnormal{g.o.}}$, since
$f$ is continuous into $L^2$ on compact subintervals of $I^* \left( f \right)$. Moreover, the time-continuity
argument from subsection \ref{ssec:integrableAndContinuous} shows that, subject to the
condition $T^* \left( f \right) < \infty$, the
the $Q^+$ blow-up (\ref{eq:QplusBlowUp}) must hold, since otherwise we would have
continuity on the \emph{compact} interval
\[
f \in C \left( \left[ 0, T^* \left( f \right) \right], L^2 \right)
\]
in contradiction with the definition of ($*$)-solution. 
 Thus we only need to show that if $T^* \left(f \right) < \infty$
then (\ref{eq:QplusBlowUp}) implies (\ref{eq:TgoUpperBd}).

 Suppose (\ref{eq:TgoUpperBd}) fails to hold; then
there exists a time $t_0 \in I^* \left( f \right)$ such that
\[
T_{\textnormal{g.o.}} \left( f \left( t_0 \right) \right) > T^* \left( f \right) - t_0
\]
This implies by the definition of $T_{\textnormal{g.o.}}$ that
\[
\mathcal{J} = \int_0^{T^* \left( f \right) - t_0} 
\left\Vert Q^+ \left( \mathfrak{Z}_{\textnormal{g.o.}} \left(
f \left( t_0 \right) \right) \left( s \right) \right)
\right\Vert_{L^2} ds < \infty
\]
Hence, by the comparison principle Proposition \ref{prop:comparisonStar},
for 
\[
 t_0 < t < T^* \left( f \right)
\]
it holds
\[
\int_{t_0}^{t}
\left\Vert Q^+ \left( 
f , f \right) \left(s\right) \right\Vert_{L^2} dt \leq \mathcal{J}
\]
Therefore, by monotone convergence,
\[
\int_{t_0}^{T^* \left( f \right)}
\left\Vert Q^+ \left( 
f , f \right) \left(s\right) \right\Vert_{L^2} dt \leq \mathcal{J} < \infty
\]
in contradiction with (\ref{eq:QplusBlowUp}).
\end{proof}

\section{Existence of ($*$)-solutions}
\label{sec:ExistenceProof}

\subsection{The truncation scheme.}
\label{ssec:truncationScheme}

In order to construct local solutions of Boltzmann's equation at low regularity,
we will be relying on a compactness argument based on a modified equation
which is known to be globally well-posed. This will be essentially the same scheme as appears
in the original work of DiPerna and Lions (\cite{DPL1989} Section VIII and references therein),
where both the evolution and
the initial data are modified. Crucially, for the purposes of this paper,
the modified collision kernel must be bounded
from above \emph{pointwise} by the uniform constant determined by the normalization
of (\ref{eq:BE}): this is required because later we will need to prove the comparison principle
for the \emph{modified} equation whereas our definition of $\mathfrak{B}_I^I$ is in
reference to the \emph{standard} version of the gain-only flow. Recall again that the
definition of $\mathfrak{B}_I^I$ does not require $f_n$ to solve Boltzmann's equation.

Let us recall the spatial density
\[
\rho_f \left( t,x \right) = \int_{\mathbb{R}^2} f \left( t,x,v \right) dv
\]
and formally set
\begin{equation}
\label{eq:truncatedEquation}
\left( \partial_t + v \cdot \nabla_x \right) f_n = \frac{1}{1 + n^{-1} \rho_{f_n}}
\left\{
Q_{b_n}^+ \left( f_n, f_n \right) - Q_{b_n}^- \left( f_n, f_n \right) \right\}
\end{equation}
where $0 \leq f_n \left( t=0 \right) = f_{n,0} \in \mathcal{S}$ approaches $f_0$ in a sense to be specified later,
and
\begin{equation}
\label{eq:truncatedCollKernel}
b_n \in C_0^{\infty} \left( \mathbb{R}^2 \times \mathbb{S}^1 \right)
\subset L^1 \left( \mathbb{R}^2 \times \mathbb{S}^1 \right)
\bigcap L^\infty \left( \mathbb{R}^2 \times \mathbb{S}^1 \right)
\end{equation}
refers to a smooth compactly supported
 collision kernel (depending only radially on the relative velocity for each $n$)
satisfying the \emph{pointwise} constraints
\begin{equation}
\label{eq:truncatedCollTwoPi}
\forall \left( n \in \mathbb{N} \right) \quad 0 \leq b_n \leq \frac{1}{2\pi}
\end{equation}
and
\begin{equation}
\label{eq:truncatedCollConv}
b_n \rightarrow \frac{1}{2\pi} \quad \textnormal{ almost everywhere }
\end{equation}
as $n \rightarrow \infty$,
having defined $Q^{\pm}_{b_n}$ by substituting $b_n$ for $b$ in (\ref{eq:defQbPlus})
and (\ref{eq:defQbMinus}).

For technical reasons, we shall also assume that, for each $n \in \mathbb{N}$, there exists
a number $\delta_n > 0$ (tending to zero as $n \rightarrow \infty$) such that
\begin{equation}
\label{eq:truncSptVel}
\min \left( \left| z \right|, \left| z \right|^{-1} \right) < \delta_n \qquad
\implies \qquad b_n \left( z, \sigma \right) = 0
\end{equation}
and, for $z \neq 0$,
\begin{equation}
\label{eq:truncSptDefl}
\min \left( \frac{\left| z \cdot \sigma\right|}{\left| z \right|},
1 - \frac{\left|z \cdot \sigma\right|}{\left| z \right|} \right) < \delta_n
\qquad \implies \qquad b_n \left( z, \sigma \right) = 0
\end{equation}
These conditions intuitively forbid scattering events with small or large relative speed,
or those residing inside a set of deflection angles, that set being defined explicitly and
 having small measure.

We recall below (cf. \cite{DPL1989, Ar2011})
 a simple global well-posedness result for the truncated equation
(\ref{eq:truncatedEquation}):
in fact, for the proof,
 it will be slightly modified further still by initially
  substituting $\rho_{\left| f_n \right|}$ for
$\rho_{f_n}$, since we do not know \emph{a priori} that $f_n$ is non-negative for positive values of $t$.

We turn to the basic global well-posedness result for (\ref{eq:truncatedEquation}).

\begin{theorem}
\label{thm:truncated}
For $n \in \mathbb{N}$, let $b_n$ be given as above and
let $f_{n,0} \in \mathcal{S}$ be a non-negative function; furthermore, assume that for
each $n$ there exists $c_n > 0$ such that, for all $\left( x,v \right) \in \mathbb{R}^2 \times
\mathbb{R}^2$,
\begin{equation}
\label{eq:dataLowerBd}
f_{n,0} \left( x,v \right) \geq c_n 
\exp \left( - \frac{1}{2} \left| x\right|^2 - 
\frac{1}{2} \left|v\right|^2 \right)
\end{equation}
 Then there exists a unique non-negative mild solution
\[
f_n \in C^1 \left( \left[0,\infty\right), \mathcal{S} \right)
\]
of the truncated Boltzmann equation (\ref{eq:truncatedEquation})
such that $f_n \left( t=0 \right) = f_{n,0}$; moreover, for all
$\left( t,x,v \right) \in \left[ 0,\infty\right) \times \mathbb{R}^2 \times \mathbb{R}^2$,
\[
f_n \left( t,x,v \right) > 0
\]
Additionally, for each $t \geq 0$, we have the global conservation of mass,
\begin{equation}
\label{eq:truncatedMass}
\int_{\mathbb{R}^2 \times \mathbb{R}^2} f_n \left(t,x,v\right) dx dv = 
\int_{\mathbb{R}^2 \times \mathbb{R}^2} f_{n,0} \left(x,v\right) dx dv
\end{equation}
the global conservation of momentum,
\begin{equation}
\label{eq:truncatedMomentum}
\int_{\mathbb{R}^2 \times \mathbb{R}^2} v f_n \left(t,x,v\right) dx dv = 
\int_{\mathbb{R}^2 \times \mathbb{R}^2} v f_{n,0} \left(x,v\right) dx dv
\end{equation}
the global conservation of kinetic energy,
\begin{equation}
\label{eq:truncatedVelMoment}
\int_{\mathbb{R}^2 \times \mathbb{R}^2} \left|v\right|^2 f_n \left(t,x,v\right) dx dv =
\int_{\mathbb{R}^2 \times \mathbb{R}^2} \left|v\right|^2 f_{n,0} \left(x,v\right) dx dv
\end{equation}
a similar conservation law for spatial moments,
\begin{equation}
\label{eq:truncatedPosMoment}
\int_{\mathbb{R}^2 \times \mathbb{R}^2} 
 \left|x-vt\right|^2 f_n \left(t,x,v\right) dx dv = \int_{\mathbb{R}^2 \times \mathbb{R}^2} 
  \left|x\right|^2 f_{n,0} \left(x,v\right) dx dv
\end{equation}
and
\begin{equation}
\label{eq:truncatedPosCross}
\int_{\mathbb{R}^2 \times \mathbb{R}^2}
 \left( x - v t \right) \cdot v\;
  f_n \left(t,x,v\right) dx dv = \int_{\mathbb{R}^2 \times \mathbb{R}^2} 
  x\cdot v\; f_{n,0} \left(x,v\right) dx dv
\end{equation}
and the entropy identity
\begin{equation}
\label{eq:truncatedEntropyIdentity}
H \left( f_n \left( t \right) \right) +
\int_0^t \int_{\mathbb{R}^2}
\frac{1}{1 + n^{-1} \rho_{f_n}}
 \mathcal{D}_{b_n} \left( f_n \left( \tau \right) \right) dx d\tau =
H \left( f_{n,0} \right)
\end{equation}
where $\mathcal{D}_{b_n}$ refers to the entropy dissipation defined in reference to the
collision kernel $b_n$, namely
\begin{equation}
\label{eq:defEntrDisGeneral}
\mathcal{D}_{b_n} \left( h \right) =
\frac{1}{4} \int_{\mathbb{S}^1 \times \mathbb{R}^2 \times \mathbb{R}^2}
b_n 
\left( h^\prime h_*^\prime - h h_* \right) \log
\frac{h^\prime h_*^\prime}{h h_*} d\sigma dv dv_*
\end{equation}
where $b_n$ denotes $b_n \left( v - v_* \right)$ (the integrand is everywhere finite since
$f_n$ is nowhere vanishing).
\end{theorem}

\begin{proof}
See Appendix \ref{app:truncated}.
\end{proof}

\subsection{The comparison principle.}
\label{ssec:truncatedComparison}

We aim to show that the Schwartz solutions $f_n$ from subsection \ref{ssec:truncationScheme} satisfy the
comparison principle:
\begin{equation}
\label{eq:comp-fn}
f_n \in \mathfrak{B}^I_I
\end{equation}
where $I = \left[ 0, \infty\right)$. Now due to the fact that $f_n$ is Schwartz we clearly have
\[
Q^+ \left( f_n ,f_n \right) \in L^1 \left( J, L^2 \right)
\]
for any compact $J \subset I$, and that the proof of Proposition \ref{prop:comparisonConverse} only depends
on Lemma \ref{lem:comparison-delta}.
The only problem is that $f_n$ does not satisfy (\ref{eq:BE}), but rather (\ref{eq:truncatedEquation}).
But the proof of Lemma \ref{lem:comparison-delta} does not actually
require $f_n$ to satisfy (\ref{eq:BE}): the proof carries through (simply replacing $f$ by $f_n$ everywhere)
if only it holds the \emph{pointwise} upper bound for $0 \leq t_0 < t < \infty$
\[
f_n \left( t \right) \leq \mathcal{T} \left( t-t_0 \right) f_n \left( t_0 \right) +
\int_{t_0}^t \mathcal{T} \left( t-s \right) Q^+ \left( f_n \left( s \right) \right) ds
\]
noting carefully $Q^+$ is that of (\ref{eq:BE}), \emph{not} (\ref{eq:truncatedEquation}). But we can verify
this inequality directly from Duhamel's formula:
\begin{equation*}
\begin{aligned}
f_n \left( t \right) & = \mathcal{T} \left( t-s \right) f_n \left( s \right) +
\int_s^t \mathcal{T} \left( t-\tau\right) 
\frac{
Q_{b_n}^+ \left( f_n, f_n \right) - Q_{b_n}^- \left( f_n, f_n \right)}{1+n^{-1}\rho_{f_n}}
\left( \tau \right) d\tau \\
& \leq \mathcal{T} \left( t-s \right) f_n \left( s \right) +
\int_s^t \mathcal{T} \left( t-\tau\right) 
\frac{
Q_{b_n}^+ \left( f_n, f_n \right)}{1+n^{-1} \rho_{f_n}}
\left( \tau \right) d\tau \\
\\
& \leq \mathcal{T} \left( t-s \right) f_n \left( s \right) +
\int_s^t \mathcal{T} \left( t-\tau\right) 
Q_{b_n}^+ \left( f_n, f_n \right)
\left( \tau \right) d\tau \\
\\
& \leq \mathcal{T} \left( t-s \right) f_n \left( s \right) +
\int_s^t \mathcal{T} \left( t-\tau\right) 
Q^+ \left( f_n, f_n \right)
\left( \tau \right) d\tau 
\end{aligned}
\end{equation*}
where we have used the uniform bound $b_n \leq \left( 2\pi \right)^{-1}$ in the last
step. Hence we may conclude (\ref{eq:comp-fn}).

\subsection{The convergence argument.} We are ready to prove:
\begin{theorem}
\label{thm:starSolnExists}
For any
\[
0 \leq f_0 \in L^2 \bigcap L^1_2
\]
there exists a ($*$)-solution of (\ref{eq:BE}) corresponding to the initial data $f_0$.
\end{theorem}
\begin{proof}
To begin, consider the unique solutions $f_n$ from Theorem \ref{thm:truncated}, corresponding to 
non-negative Schwartz initial
data $f_{n,0}$ which we assume to satisfy each of the following:
\begin{equation}
\label{eq:starAssumptionOne}
\lim_{n \rightarrow \infty} \left\Vert
f_{n,0} - f_0 \right\Vert_{L^2 \bigcap L^1_2} = 0
\end{equation}
\begin{equation}
\label{eq:starAssumptionThree}
\lim_{n \rightarrow \infty} f_{n,0} \left( x,v \right) = f_0 \left( x,v \right)
\quad \textnormal{ a.e }
\left( x,v \right) \in \mathbb{R}^2 \times \mathbb{R}^2
\end{equation}
\begin{equation}
\label{eq:starAssumptionFour}
f_{n,0} \geq c_n \exp \left( - \frac{1}{2} \left| x \right|^2 -
\frac{1}{2} \left| v\right|^2 \right)
\end{equation}
with $c_n \rightarrow 0$ and $n \rightarrow \infty$,
no other conditions being imposed on the sequence $f_{n,0}$. Note that (\ref{eq:starAssumptionOne})
implies 
\begin{equation}
\label{eq:starAssumptionTwo}
\lim_{n \rightarrow \infty} 
\int_{\mathbb{R}^2 \times \mathbb{R}^2 }
\left( \left| x \right|^2 + \left| v \right|^2 \right)
\left| f_{n,0} - f_0 \right| dx dv  = 0
\end{equation}
 Such a sequence can be constructed
by first producing a sequence of smooth compactly supported approximants $\tilde{f}_{n,0}$ 
via convolution and
truncation, and then writing $f_{0,n}$ as the sum of $\tilde{f}_{n,0}$ and the function on the
right-hand side of (\ref{eq:starAssumptionThree}) with, say, $c_n = \frac{1}{n}$.
Passing to a subsequence, also denoted
$f_{n,0}$, provides (\ref{eq:starAssumptionThree}).

It follows immediately that
\[
\sup_{n \in \mathbb{N}} \left\Vert f_{n,0} \right\Vert_{L^2 \bigcap L^1_2} < \infty
\]
and hence by Lemmas \ref{lem:HminusFinite} and \ref{lem:HplusBd} we also have
\[
\sup_{n \in \mathbb{N}} H^{\pm} \left( f_{n,0} \right) < \infty
\]
Then following the DiPerna-Lions argument \cite{DPL1989,DPL1991} one shows the weak compactness for the solution
sequence $f_n$, and that any limit point is a renormalized solution of (\ref{eq:BE}). Moreover,
passing to a subsequence $n_m$ ($m \in \mathbb{N}$), and using the $L^1$ (norm topology) compactness result of Lions \cite{L1994II}, 
we may assume (aside from the usual weak convergence) the \emph{pointwise} convergence
\[
f_{n_m}
 \rightarrow f \quad \textnormal{a.e.} \quad \left(t,x,v\right) \in \left[ 0,\infty\right) \times \mathbb{R}^2 \times \mathbb{R}^2
\]
as $m \rightarrow \infty$,
where $f$ is a renormalized solution of (\ref{eq:BE}). 
Our claim is that the limiting function $f$ is, in fact, a ($*$)-solution of (\ref{eq:BE}).

We note that the $L^1$ time continuity (\ref{eq:continuityL1}) follows from the DiPerna-Lions argument.
Let us turn to the bounds on moments and entropy. 

Let us begin with the kinetic energy bound. Since $f$ is the weak limit of the sequence $f_n$,
it follows for any non-negative function $\varphi \left( t,x,v \right)$, smooth and compactly supported
in all variables, with $\left\Vert \varphi \right\Vert_{L^\infty} \leq 1$, and assuming $\varphi$ is
supported in a time interval $\left( a,b \right)$ of size $\tau$,
\[
\begin{aligned}
& \int_{\left( 0,\infty \right) \times \mathbb{R}^2 \times \mathbb{R}^2}
\left|v \right|^2 \varphi \left( t,x,v \right) f \left( t,x,v \right) dt dx dv \\ 
& \qquad \qquad \qquad = \lim_{n \rightarrow \infty} 
\int_{\left( 0,\infty \right) \times \mathbb{R}^2 \times \mathbb{R}^2}
\left|v \right|^2 \varphi \left( t,x,v \right) f_n \left( t,x,v \right) dt dx dv \\
& \qquad \qquad \qquad \leq
\tau \liminf_{n \rightarrow \infty} \sup_{t \in \left( a,b \right)}
\int_{\mathbb{R}^2 \times \mathbb{R}^2}
\left|v \right|^2  f_n \left( t,x,v \right)  dx dv \\
& \qquad \qquad \qquad = \tau \liminf_{n \rightarrow \infty} \int_{\mathbb{R}^2 \times \mathbb{R}^2}
\left|v \right|^2  f_{n,0} \left( x,v \right) dx dv
\end{aligned}
\]
where we have used (\ref{eq:truncatedVelMoment}) in the last step. By
(\ref{eq:starAssumptionTwo}) and
 the arbitrariness of $\varphi$
we deduce (\ref{eq:energyCondition}). Similarly we deduce (\ref{eq:dispCondition}) from
(\ref{eq:truncatedPosMoment}).

Turn now to the mass bound; we only sketch the proof. For any compact set $K \subset
\mathbb{R}^2_x \times \mathbb{R}^2_v$, we can decompose
\[
\begin{aligned}
& \int_{\mathbb{R}^2 \times \mathbb{R}^2} f_n \left( t,x,v \right) dx dv \\
& \qquad \qquad \qquad =
\int_{K} f_n \left( t,x,v \right) dx dv +
\int_{\left(\mathbb{R}^2 \times \mathbb{R}^2\right)\setminus K} f_n \left( t,x,v \right) dx dv
\end{aligned}
\]
The first term on the right converges (in a suitable sense) to
$\int_K f \left(t\right) dx dv$, and the second can be made small uniformly in $n$ by suitable
choice of $K$, due to (\ref{eq:truncatedVelMoment}) and (\ref{eq:truncatedPosMoment}). Hence
we deduce (\ref{eq:massCondition}) as a consequence of (\ref{eq:truncatedMass}) and the bounds
on second moments in $x$ and $v$. Similarly, we can use (\ref{eq:truncatedMomentum}) to
deduce (\ref{eq:momentumCondition}).

The entropy inequality is far more subtle and has been studied by DiPerna and Lions
in \cite{DPL1991}. In that reference it was proven that, for a sequence of renormalized solutions
(or solutions of the truncated model, etc.), and ignoring notational details for simplicity,
\[
\int_0^t \int_{\mathbb{R}^2} \mathcal{D} \left( f \left( t \right) \right) dx d\tau \leq
\liminf_{n \rightarrow \infty} 
\int_0^t \int_{\mathbb{R}^2} \mathcal{D} \left( f_n \left( t \right) \right) dx d\tau
\]
The proof is non-trivial but it is based on convexity arguments combined with a careful definition
for dissipation functional. We also have, by convexity,
\[
H \left( f \left( t \right) \right) \leq \liminf_{n \rightarrow \infty}
H \left( f_n \left( t \right) \right)
\]
Therefore, to deduce (\ref{eq:entropyCondition}) from (\ref{eq:truncatedEntropyIdentity}), 
the key is to prove the \emph{limits} at the initial time,
\[
\lim_{n \rightarrow \infty} H^{\pm} \left( f_{n,0} \right) = H^{\pm} \left( f_0 \right) 
\]
This follows from (\ref{eq:starAssumptionOne}) and (\ref{eq:starAssumptionThree}), using
Lemma \ref{lem:HplusBd} for $H^+$ and Lemma \ref{lem:HminusConv} for $H^-$.

It remains to identify a $T^* \left( f \right) \in \left( 0, \infty \right]$ which
verifies
(\ref{eq:continuityCondition}), (\ref{eq:kmCondition}) and (\ref{eq:notinCondition}).

The fundamental lemma, Lemma \ref{lem:fundamental}, implies that for some $\delta > 0$,
\[
Q^+ \left( f,f \right) \in L^1 \left( \left[ 0, \delta \right), L^2  \right)
\]
since (by the lemma) $f$ is controlled \emph{pointwise} by the gain-only flow based at $f_0$
for $ t \in \left[ 0,\delta \right)$ (some
$\delta$ depending only on $f_0$),
whereas the gain-only flow has the requisite $Q^+$ bound in $L^1_t L^2_{x,v}$ for small enough time intervals.

So let us define
\begin{equation}
\label{eq:TstarDef}
T^* \left( f \right) = \sup \left\{ \; T \in \left( 0, \infty \right) \; : \; Q^+ \left( f, f \right) \in
L^1 \left( \left[ 0, T \right], L^2 \right) \; \right\}
\end{equation}
and
\[
I^* \left( f \right) = \left[ 0, T^* \left( f \right) \right)
\]
Clearly (\ref{eq:kmCondition}) follows trivially from the
definition of $T^* \left( f \right)$.
Now even though we have not yet proven that $f$ is a ($*$)-solution, 
 we can still apply
the arguments from subsection \ref{ssec:integrableAndContinuous} to conclude
from (\ref{eq:kmCondition}) that, on any compact sub-interval
$J \subset I^* \left( f \right)$, it holds
\[
f \in C \left( J, L^2 \right)
\]
hence we have (\ref{eq:continuityCondition}). So it only remains to prove (\ref{eq:notinCondition}).

Before we proceed to prove (\ref{eq:notinCondition}), let us prove a preliminary result.
Let $T$ be a real number with $0 < T < T^* \left( f \right)$; then $f \in C \left( \left[ 0, T \right], L^2\right)$,
so by the lower semi-continuity of $T_{\textnormal{g.o.}}$, we know that there exists $\eta_T$ with
\[
0 < \eta_T < \inf_{t \in \left[ 0, T \right]} T_{\textnormal{g.o.}} \left( f \left( t \right) \right)
\]
Therefore, by partitioning $\left[ 0, T \right]$ into suitable consecutive sub-intervals of size
\[
\textnormal{ between } \frac{\eta_T}{4} \textnormal{ and } \frac{\eta_T}{2}
\]
 and inductively applying Lemma \ref{lem:fundamental} finitely many times (using our freedom to wait to choose the next
 interval of the partition until \emph{after} the previous invocation of the lemma),
 we can deduce that up to extraction of a further
subsequence still denoted $f_{n_{m}}$, there holds
\begin{equation}
\label{eq:step3141592}
\lim_{m \rightarrow \infty} \left\Vert f_{n_m} \left( t \right) - f \left( t \right) \right\Vert_{L^2} = 0
\end{equation}
for almost every $t \in \left[ 0, T \right]$. In particular, by the arbitrariness of $T$ and diagonalization,
 the same can be said for almost
every $t \in I^* \left( f \right)$.

To complete the proof, let us suppose that (\ref{eq:notinCondition}) fails; that is,
\[
f \in C \left( \left[ 0, T^* \left( f \right) \right], L^2 \right)
\]
Then by the lower-semicontinuity of $T_{\textnormal{g.o.}} \left( \cdot \right)$ we may choose $r$ such that
\[
0 < r < \inf_{t \in I^* \left( f \right)} T_{\textnormal{g.o.}} \left( f \left( t \right) \right)
\]

Let us pick an intermediate time $t_0$ with
\[
T^* \left( f \right) - \frac{r}{2} < t_0 < T^* \left( f \right)
\]
for which (\ref{eq:step3141592}) holds.
Then applying Lemma \ref{lem:fundamental} one last time, we can conclude 
from (\ref{eq:fundConclusionOne})
that
$f$ is bounded \emph{pointwise} by the gain-only flow based at $f \left( t_0 \right)$, up to a slightly larger time than $T^* \left( f \right)$,
say $\tilde{t}$ where
\[
T^* \left( f \right) < \tilde{t} < T^* \left( f \right) + \frac{r}{4}
\]
hence for some $T^\prime > T^* \left( f \right)$ we have
\[
Q^+ \left( f,f \right) \in L^1
\left( \left[ 0, T^\prime \right], L^2 \right)
\]
which contradicts (\ref{eq:TstarDef}).
\end{proof}

\section{Limits of ($*$)-solutions}
\label{sec:starLimitsProof}

For the next theorem, we consider a sequence $f_n$ of ($*$)-solutions to (\ref{eq:BE}), corresponding simply
to initial data $f_{n,0} \in L^2 \bigcap L^1_2$, without assuming any higher regularity or decay for
$f_n$ or $f_{n,0}$.
 We shall assume that we have prepared the sequence $f_n$ by passing to subsequences,
prior to the application of the theorem, so as to simplify the statement of the theorem itself.

\begin{theorem}
\label{thm:limits}
For each $n \in \mathbb{N}$ let
 $f_n$ be a ($*$)-solution of (\ref{eq:BE}) with initial data 
 \[
 f_{n,0} = f_n \left( t=0 \right) \in L^2 \bigcap L^1_2
 \]
Furthermore, let us assume that, for some renormalized solution $f$ of (\ref{eq:BE}),
\[
f_n \rightharpoonup f
\]
where the convergence is (at least) in the weak topology of $L^1 \left( K \right)$
for each compact $K \subset \left[ 0, \infty \right) \times \mathbb{R}^2 \times \mathbb{R}^2$,
 (cf. \cite{DPL1989}), and that there holds the convergence of the initial data
\[
\lim_{n \rightarrow \infty} \left\Vert f_{n,0} - f_0 \right\Vert_{L^2} = 0
\]
and
\begin{equation}
\begin{aligned}
& \lim_{n \rightarrow \infty} \sum_{\varphi \in
\left\{ 1, v_1, v_2, \left| v \right|^2, \left| x \right|^2, x \cdot v \right\}} \left|
\int_{\mathbb{R}^2 \times \mathbb{R}^2 } \varphi \cdot
\left( f_{n,0} - f_0 \right) dx dv \right|  = 0
\end{aligned}
\end{equation}
where $f_0 = f \left( t=0 \right)$, and additionally that (cf. \cite{L1994II})
\[
f_n \rightarrow f \quad \textnormal{a.e.} \quad \left( t,x,v \right) \in \left[ 0, \infty \right) \times \mathbb{R}^2
\times \mathbb{R}^2
\]

Then it follows that $f$ is a ($*$)-solution of (\ref{eq:BE}) with
\[
0 < T^* \left( f \right) \leq \liminf_{n \rightarrow \infty} T^* \left( f_n \right)
\]
the $\liminf$ being necessarily non-zero (but possibly infinite). (But even if each
$T^* \left( f_n \right)$ is finite we do not exclude the possibility $T^* \left( f \right) = \infty$, provided
the $\liminf$ is infinite, as indicated.)

Moreover, there exists a subsequence $n_m$ such that both the following hold: first,
for each compact sub-interval $J \subset I^* \left( f \right)$,
\[
\lim_{m \rightarrow \infty} \left\Vert f_{n_m} - f \right\Vert_{L^2 \left( J \times
\mathbb{R}^2 \times \mathbb{R}^2 \right)} = 0
\]
and, second, for almost every $t \in I^* \left( f \right)$, it holds
\[
\lim_{m \rightarrow \infty} \left\Vert f_{n_m} \left( t \right) - f \left( t \right) \right\Vert_{L^2} = 0
\]
\end{theorem}
\begin{proof}
Clearly we may assume without loss of generality, by passing to a further sequence (still denoted $f_n$) which 
saturates the $\liminf$ in the theorem statement, that
for some $\tilde{T}$ with $0 < \tilde{T} \leq \infty$, the \emph{limit}
\[
\tilde{T} = \lim_{n \rightarrow \infty} T^* \left( f_n \right)
\]
exists in the extended real line. By lower-semicontinuity of
$T_{\textnormal{g.o.}} \left( \cdot \right)$ and the strong $L^2$ convergence at $t=0$, along with
(\ref{eq:TgoUpperBd}), we  have
\[
\tilde{T} \geq T_{\textnormal{g.o.}} \left( f_0 \right) > 0
\]
which follows from the chain of (in)equalities
\[
\tilde{T} = \lim_{n \rightarrow \infty} T^* \left( f_n \right) \geq
\liminf_{n \rightarrow \infty} T_{\textnormal{g.o.}} \left( f_{n,0} \right)
\geq T_{\textnormal{g.o.}} \left( f_0 \right) > 0
\]
 Let us furthermore define
\[
T_0 = \sup \left\{ \; T \in \left( 0, \infty \right) \; : \; Q^+ \left( f, f \right) \in
L^1 \left( \left[ 0, T \right], L^2 \right) \; \right\}
\]
where the set is non-empty by passage to the limit in
 the comparison principle following Lemma \ref{lem:fundamental}: indeed,
 $T_0 \geq T_{\textnormal{g.o.}} \left( f_0 \right) > 0$.
In what follows we will assume that each $T_0, \tilde{T}$ is finite: the proof is simpler in the
case that $\tilde{T} = \infty$ and $T_0 < \infty$. (There are two cases remaining: that each
$T_0, \tilde{T}$ is infinite, and that $T_0 = \infty$ and $\tilde{T}$ is finite; but, there is nothing to
show in the first case, and the proof below shows that the second case is impossible.)

Let us denote the shorthand
\[
I_0 = \left[ 0, T_0 \right) \qquad J_0 = \left[ 0, T_0 \right]
\]
\[
\tilde{I} = \left[ 0, \tilde{T} \right)\qquad
\tilde{J} = \left[ 0, \tilde{T} \right]
\]
By Proposition \ref{prop:comparisonStar},
\[
\forall \left( n \in \mathbb{N} \right) \quad
f_n \in \mathfrak{B}^{I^* \left( f_n \right)}_{I^* \left( f_n \right)}
\]
hence by the definition of $\tilde{T}$, we find that for any compact
subinterval $J \subset \tilde{I}$ there exists an integer
$N \in \mathbb{N}$, depending on $J$, such that
\[
\forall \left( n \in \mathbb{N} \; : \; n \geq N \right) \quad
f_n \in \mathfrak{B}^J_J
\]

Before we turn to the core of the proof,
let us pass to an even further subsequence (still denoted $f_n$) such that both the following hold:
first, for any compact sub-interval $J \subset I_0 \bigcap \tilde{I}$,
\begin{equation}
\label{eq:convOneLimits}
\lim_{n \rightarrow \infty} \left\Vert f_n - f \right\Vert_{L^2 \left( J \times
\mathbb{R}^2 \times \mathbb{R}^2 \right)} = 0
\end{equation}
and, second, for almost every $t \in I_0 \bigcap \tilde{I}$,
\begin{equation}
\label{eq:convTwoLimits}
\lim_{n \rightarrow \infty} \left\Vert f_n \left( t \right) - f \left( t \right) \right\Vert_{L^2} = 0
\end{equation}
This is possible by inductively applying Lemma \ref{lem:fundamental} as in the proof
of Theorem \ref{thm:starSolnExists}. Note that we need $I_0$ to guarantee the square-integrability (with time continuity)
of $f$ along $J$, whereas we need $\tilde{I}$ to guarantee the comparison principle on $J$ for
$f_n$ for all large enough $n$ depending on $J$: these two, with the necessary convergence
 at $t=0$, are the keys to inductively applying Lemma \ref{lem:fundamental}. We can moreover conclude that
 \[
 f \in \mathfrak{B}^{I_0 \bigcap \tilde{I}}_{\left\{ t_0 \right\}}
 \]
 for almost every $t_0 \in I_0 \bigcap \tilde{I}$.
 
We must also prove the moment bounds and entropy inequality. The key to proving
(\ref{eq:massCondition}-\ref{eq:dispCondition}) is that we have assumed,
for $\varphi \in \left\{ 1, v_1, v_2, \left|v \right|^2, \left| x \right|^2, x \cdot v \right\}$,
\begin{equation}
\label{eq:dataConvPf}
\lim_{n \rightarrow \infty} \int_{\mathbb{R}^2 \times \mathbb{R}^2} \varphi f_{n,0} \; dx dv =
\int_{\mathbb{R}^2 \times \mathbb{R}^2} \varphi f_0 dx dv
\end{equation}
which, by non-negativity of $f_{n,0} , f_0$
 and combined with the assumption that $f_{n,0} \rightarrow f_0$ strongly
in $L^2$, provides us
\begin{equation}
\label{eq:dataBdPf}
\sup_{n \in \mathbb{N}} \left\Vert f_{n,0} \right\Vert_{L^2 \bigcap L^1_2} < \infty
\end{equation}
but note carefully that we are \emph{neither} assuming \emph{nor} asserting that $f_{n,0}$ converges
to $f_0$ strongly in $L^2 \bigcap L^1_2$, contrary to the proof of Theorem 
\ref{thm:starSolnExists}. In any case, using (\ref{eq:dataConvPf}) and the known estimates for
the ($*$)-solutions $f_n$, we can deduce (\ref{eq:massCondition}-\ref{eq:dispCondition}) similarly
to the proof of Theorem \ref{thm:starSolnExists}. Similarly, using again the results of 
DiPerna and Lions from \cite{DPL1991} as in the proof of Theorem \ref{thm:starSolnExists},
we obtain (\ref{eq:entropyCondition}) by noting that
\[
\lim_{n \rightarrow \infty} H^{\pm} \left( f_{n,0} \right) = H^{\pm} \left( f_0 \right)
\]
using, as before, Lemma \ref{lem:HplusBd} and Lemma \ref{lem:HminusConv}, and our assumptions
on $f_{n,0}$ (namely strong $L^2$ convergence, the boundedness in $L^1_2$, and pointwise convergence,
all at $t=0$).

We have only to show that $T_0 \leq \tilde{T}$ (where $T_0$ comes from the $Q^+$ integral for
$f$ whereas $\tilde{T}$ comes from the sequence $f_n$), and that
\begin{equation}
\label{eq:toShowNotIn}
Q^+ \left( f, f \right) \notin L^1 \left( I_0, L^2 \right)
\end{equation}
which encodes the maximality property of ($*$)-solutions.
Indeed, given (\ref{eq:toShowNotIn}), assume that
$f$ is continuous from $\left[ 0, T^* \left( f \right) \right]$ into $L^2$ and deduce a contradiction with the
comparison principle cf.  
the proof of Proposition \ref{prop:comparisonZero}.

Let us begin by proving instead the  statement
\begin{equation}
\label{eq:toShowNotInWeaker}
Q^+ \left( f, f \right) \notin L^1 \left( \tilde{I}, L^2 \right)
\end{equation}
Indeed, if this were \emph{not} the case, then arguing as in subsection \ref{ssec:integrableAndContinuous},
 the reader can verify
 that we would have continuity on
the \emph{closed} interval
\[
f \in C \left( \tilde{J}, L^2 \right)
\]
and moreover that $T_0 \geq \tilde{T}$.
From this we obtain that, by lower-semicontinuity of $T_{\textnormal{g.o.}} \left( \cdot \right)$, we may
choose $r$ such that
\[
0 < r < \inf_{t \in \tilde{J}}
T_{\textnormal{g.o.}} \left( f \left( t \right) \right)
\]
So pick a time $\tilde{t}$ with
\begin{equation}
\label{eq:tildeStar}
\tilde{T} - \frac{r}{2} < \tilde{t} < \tilde{T}
\end{equation}
such that (\ref{eq:convTwoLimits}) holds. Now by (\ref{eq:TgoUpperBd}), for each large enough $n$ we have
\[
T_{\textnormal{g.o.}} \left( f_n \left( \tilde{t} \right) \right) \leq
T^* \left( f_n \right) - \tilde{t}
\]
We wish to let $n \rightarrow \infty$ in this inequality; indeed, on the right we simply obtain
\[
 \tilde{T} - \tilde{t} 
 \]
 whereas on the left, by lower semi-continuity of $T_{\textnormal{g.o.}} \left( \cdot \right)$
and the fact that $\tilde{t}$ verifies (\ref{eq:convTwoLimits}), we find that
\[
T_{\textnormal{g.o.}} \left( f \left( \tilde{t} \right) \right) \leq
\liminf_{n \rightarrow \infty} T_{\textnormal{g.o.}} \left( f_n \left( \tilde{t} \right) \right) 
\]
hence
\[
T_{\textnormal{g.o.}} \left( f \left( \tilde{t} \right) \right) \leq \tilde{T} - \tilde{t}
\]
The quantity on the left is no less than $r$, hence
\[
r \leq \tilde{T} - \tilde{t}
\]
which contradicts (\ref{eq:tildeStar}).

We conclude that (\ref{eq:toShowNotInWeaker}) holds; this immediately implies that
$T_0 \leq \tilde{T}$. But (\ref{eq:toShowNotInWeaker}) \emph{also} implies the following:
\emph{in the case} that $T_0 = \tilde{T}$, we immediately have (\ref{eq:toShowNotIn}),
 so there is nothing more to
show in that case. Therefore, to conclude the proof, we are free to prove
(\ref{eq:toShowNotIn}) under the simplifying assumption that $T_0 < \tilde{T}$.

Suppose the desired conclusion fails. Then we have
\[
Q^+ \left( f,f \right) \in L^1 \left( I_0 , L^2 \right)
\]
and in particular, continuity on the \emph{closed} interval $J_0 = I_0 \bigcup \left\{
T_0 \right\}$, i.e.
\[
f \in C \left( J_0 , L^2 \right)
\]
so choose, as before, an $r_0$ satisfying
\[
0 < r_0 < \inf_{t \in J_0}
T_{\textnormal{g.o.}} \left( f \left( t \right) \right)
\]
As before, pick a time $t_0$ with
\[
T_0 - \frac{r_0}{2} < t_0 < T_0
\]
such that (\ref{eq:convTwoLimits}) holds. Then by Lemma \ref{lem:fundamental} and using
that $T_0 < \tilde{T}$, we can conclude
that
\[
f \in \mathfrak{B}^I_{\left\{ t_0 \right\}}
\]
where $I = \left[ t_0 , b \right)$ with $b = \min \left( \tilde{T}, t_0 + r_0 \right) > T_0$.
In particular, by our choice of $r_0$ as (less than) an inf over $J_0$ and that $t_0 \in J_0$,
\[
t_0 + T_{\textnormal{g.o.}} \left( f \left( t_0 \right) \right) \geq
t_0 + r_0 \geq b > T_0
\]
Hence for any compact subinterval $J$ of $\left[ 0, b \right)$,
\[
Q^+ \left( f,f \right) \in L^1 \left( J, L^2 \right)
\]
which contradicts the definition of $T_0$.
\end{proof}

\section{Scattering}
\label{sec:scatteringBasic}

\subsection{The scattering lemma.}

The Lemma to follow expresses a type of stability against perturbations of scattering states.

\begin{lemma}
\label{lem:scattering-lemma}
Suppose
\[
 f_{+\infty} \in L^2
\]
Then there exist numbers $\varepsilon, T > 0$, each depending only on $f_{+\infty}$, such that the following
holds:

For any $t_0 \geq T$,
\[
\left\Vert h_0 - \mathcal{T} \left( t_0 \right) f_{+\infty} \right\Vert_{L^2} 
< \varepsilon \;\; \implies \;\;
\begin{cases}
T_{\textnormal{g.o.}} \left(h_0\right) = \infty \\
\textnormal{and}\\
\int_0^\infty \left\Vert Q^+ \left( \mathfrak{Z}_{\textnormal{g.o.}} \left(h_0\right) \left(t\right)
\right) \right\Vert_{L^2} dt < \infty
\end{cases}
\]
\end{lemma}
\begin{proof}
An immediate consequence of Theorem \ref{thm:CritLWP} with Proposition \ref{prop:QplusLargeTime}, cf. the proof
of Theorem \ref{thm:QPlusLWP}.
\end{proof}

\subsection{The scattering criterion.} We are ready to characterize scattering solutions of (\ref{eq:BE}).

\begin{theorem}
\label{thm:ScatteringCriterion}
Let $f$ be a ($*$)-solution of (\ref{eq:BE}); then the following are equivalent:
\begin{enumerate}
\item
$T^* \left(f\right) = \infty$ and $f$ scatters

\item
\[
\int_{I^* \left( f  \right)}
\left\Vert Q^+ \left( f,f \right) \left( t \right) \right\Vert_{L^2}
dt < \infty
\]
\end{enumerate}
\end{theorem}
\begin{proof}
\underline{\emph{(1)$\implies$(2)}.} Since $f$ scatters by hypothesis, there exists an 
\[
f_{+\infty} \in L^2
\]
 such that
\[
\lim_{t \rightarrow + \infty} \left\Vert f\left( t \right) - \mathcal{T} \left( t \right) f_{+\infty} \right\Vert_{L^2}
= 0
\]
Let $\varepsilon, T$ be as in the statement of Lemma \ref{lem:scattering-lemma}. 
Pick a number $\tilde{T}$ such that
\[
\forall \left( t \geq \tilde{T} \right) \quad
\left\Vert f\left( t \right) - \mathcal{T} \left( t \right) f_{+\infty} \right\Vert_{L^2}
< \varepsilon
\]
and let
\[
t_0 = \min \left( T, \tilde{T} \right)
\]
Then $t_0 \geq T$ and
\[
\left\Vert f\left( t_0 \right) - \mathcal{T} \left( t_0 \right) f_{+\infty} \right\Vert_{L^2}
< \varepsilon
\]
hence by the Lemma we have 
\[
T_{\textnormal{g.o.}} \left( f \left( t_0 \right) \right) = \infty
\]
and
\[
\int_0^\infty \left\Vert Q^+ \left( \mathfrak{Z}_{\textnormal{g.o.}} \left(f \left( t_0 \right) \right) \left(t\right)
\right) \right\Vert_{L^2} dt < \infty
\]
Thus by Proposition \ref{prop:comparisonStar} we have
\[
\int_{t_0}^\infty
\left\Vert Q^+ \left( f \left( t \right) \right)\right\Vert_{L^2} dt \leq
\int_0^\infty \left\Vert Q^+ \left( \mathfrak{Z}_{\textnormal{g.o.}} \left(f \left( t_0 \right) \right) \left(t\right)
\right) \right\Vert_{L^2} dt < \infty
\]
and $Q^+ \left( f \right) \in L^1 \left( \left[ 0, t_0 \right] , L^2 \right)$ since
$T^* \left( f \right) = \infty$,
so by adding the two time integrals, we may conclude.

\underline{\emph{(2)$\implies$(1)}.} Since we have assumed
\[
\int_0^{T^* \left( f \right)}
\left\Vert Q^+ \left( f \left( t \right) \right)\right\Vert_{L^2} dt < \infty
\]
it follows from Theorem \ref{thm:breakdownRegularity} that 
\[
T^* \left( f \right) = \infty
\]
that is
\begin{equation}
\label{eq:scatteringCriterionConverse}
\int_0^\infty
\left\Vert Q^+ \left( f \left( t \right) \right)\right\Vert_{L^2} dt < \infty
\end{equation}
Also, we have Duhamel's formula, for $0 < s < t$,
\[
\mathcal{T} \left( -t \right) f \left( t \right) - \mathcal{T} \left( -s \right) f \left( s \right) =
\int_s^t \mathcal{T} \left( - \tau \right) Q^+ \left( f \left( \tau \right) \right) d\tau
\]
hence
\[
\left\Vert
\mathcal{T} \left( -t \right) f \left( t \right) - \mathcal{T} \left( -s \right) f \left( s \right)\right\Vert_{L^2}
\leq \int_s^\infty
\left\Vert Q^+ \left( f \left( t \right) \right)\right\Vert_{L^2} dt 
\]
the right-hand side tending to zero as $s\rightarrow \infty$ by monotone convergence and (\ref{eq:scatteringCriterionConverse}).
Thus there exists $f_{+\infty} \in L^2$ such that
\[
\lim_{t \rightarrow +\infty} \left\Vert \mathcal{T} \left( -t \right) f \left( t \right) - f_{+\infty} 
\right\Vert_{L^2} = 0
\]
which is equivalent to
\[
\lim_{t \rightarrow +\infty} \left\Vert f \left( t \right) - \mathcal{T} \left( t \right) f_{+\infty} 
\right\Vert_{L^2} = 0
\]
so we may conclude.
\end{proof}

\section{Exclusive scattering}
\label{sec:scatteringAdvanced}

\subsection{Definition.}

\begin{definition}
\label{def:ExSc}
A non-negative measurable function $f_0 \in L^2 \bigcap L^1_2$ will be said to be
\emph{exclusively scattering} if, for every ($*$)-solution $f$ with initial data
$f \left( t=0 \right) = f_0$, it holds that
\[
T^* \left( f \right) = \infty \quad \textnormal{ and } \quad
f \textnormal{ scatters}
\]
and in such case we write $f_0 \in \mathcal{E}$.
\end{definition}

\begin{remark}
Observe that the definition of $\mathcal{E}$ makes no mention of uniqueness; in particular, it is a
property of the \emph{initial data} $f_0$, not of a ($*$)-solution (since there might be many
($*$)-solutions corresponding to any given $f_0 \in \mathcal{E}$). When we say that $f_0$ is
exclusively scattering, or equivalently $f_0 \in \mathcal{E}$, we are simply saying that it is not possible
to identify a ($*$)-solution of (\ref{eq:BE}) with initial data $f_0$ that does not scatter.
\end{remark}

\subsection{Perturbations.}

We begin with a simple lemma.

\begin{lemma}
\label{lem:standard}
Let $\left( Z, d_Z \right)$ be a metric space (not necessarily complete). Suppose that for a subset $U \subset Z$, 
it holds that for every $u \in U$ and for every sequence $\left\{ z_n \right\}_n \subset Z$ with
\[
\lim_{n \rightarrow \infty} d_Z \left( z_n, u \right) = 0
\]
there exists a subsequence $\left\{ z_{n_m}\right\}_m$ such that
\[
\left( \forall m \right) \; z_{n_m} \in U
\]
Then $U$ is open in $Z$.
\end{lemma}
\begin{proof}
If $U$ is not open then there must be a point $u \in U$ and a
 sequence $\left\{ z_n \right\}_n \subset Z \backslash U$ such that
$z_n \rightarrow u$ in $Z$.
\end{proof}

Recall from (\ref{eq:defXnorm}) the $X$ norm
\[
\left\Vert h \right\Vert_X \coloneqq
\left\Vert h \right\Vert_{L^2} + \sum_{\varphi \in \left\{
1, \; v_1, \; v_2, \; \left| v \right|^2, \; \left| x \right|^2 , \; x \cdot v \right\}}
\left| \int_{\mathbb{R}^2 \times \mathbb{R}^2} 
\varphi \left( x,v \right) h \left( x,v \right) dx dv \right|
\]
which leads to define (\ref{eq:defXspace}) the \emph{incomplete} metric space
\[
X = \left( L^{2,+} \bigcap L^1_2 \; , \; d_X \right)
\]
where $L^{2,+}$ is the set of non-negative functions in $L^2$
and
\[
d_X \left( h, \tilde{h} \right) = \left\Vert h - \tilde{h} \right\Vert_X
\]
and we are ready to show:

\begin{theorem}
\label{thm:exScOpen}
$\mathcal{E}$ is open in $X$.
\end{theorem}
\begin{proof}
Let $\left\{ f_{0,n} \right\}_n \subset X$ be a sequence such that
\[
\lim_{n \rightarrow \infty} \left\Vert f_{0,n} - f_0 \right\Vert_{X} = 0
\]
for some $f_0 \in \mathcal{E}$. By Lemma \ref{lem:standard}, it suffices to show that there exist infinitely
many $n$ for which $f_{0,n} \in \mathcal{E}$. 

So suppose the opposite: then, there exists $N$ such that $f_{0,n} \notin \mathcal{E}$ for each $n \geq N$.
Now the sequence $f_{0,n}$ is clearly uniformly bounded in $L^2 \bigcap L^1_2$; in particular, we also have
uniform bounds on entropy and entropy dissipation for any ($*$)-solutions associated with the $f_{0,n}$.
For each $n \geq N$ let us pick a ($*$)-solution $f_n$ such that $f_n \left( t=0 \right) = f_{0,n}$ and
$f_n$ is not a global scattering solution (that is, either $T^* \left( f_n \right) < \infty$, or
$T^* \left( f_n \right) = \infty$ but $f_n$ does not scatter). This is possible because, for
$n \geq N$, we have $f_{0,n} \notin \mathcal{E}$. Passing to a subsequence, applying Theorem \ref{thm:limits},
and passing to a further subsequence, we can eventually find a subsequence $n_m$ such that all the following hold:

\begin{enumerate}
\item
The sequence $\left\{f_{n_m}\right\}_m$ converges, weakly and for a.e. $(t,x,v)$, and for a.e. $\left( x,v \right)$ at $t=0$,
 to a ($*$)-solution $f$ with
$f \left( t=0 \right) = f_0$.

\item
\begin{equation}
\label{item:second}
T^* \left( f \right) \leq \liminf_{m \rightarrow \infty} T^* \left( f_{n_m} \right)
\end{equation}

\item
For a.e. $t$ with $0 < t < T^* \left( f \right)$,
\begin{equation}
\label{item:third}
\lim_{m \rightarrow \infty}
\left\Vert f_{n_m} \left( t \right) - f \left( t \right) \right\Vert_{L^2} = 0
\end{equation}

\item
For each $m$: 
\begin{equation}
\label{item:fourth}
 \textnormal{either } T^* \left( f_{n_m} \right) < \infty\textnormal{, or } T^* \left( f_{n_m} \right) = \infty
 \textnormal{ but }
f_{n_m} \textnormal{ does not scatter.}
\end{equation}
\end{enumerate}

But now we see that, since $f$ is a ($*$)-solution with initial data $f_0$, and by hypothesis we
have $f_0 \in \mathcal{E}$, it follows from the definition of $\mathcal{E}$ that
\[
T^* \left( f \right) = \infty
\]
and $f$ scatters. In particular, by (\ref{item:second}),
\[
\liminf_{m \rightarrow \infty} T^* \left( f_{n_m} \right) = \infty
\]
Moreover, by the scattering lemma, Lemma \ref{lem:scattering-lemma}, there exist numbers $T,\varepsilon$, depending only
on the solution $f$ just identified\footnote{which need not be unique!},
such that any ($*$)-solution $\tilde{f}$ which comes within an $\varepsilon$-ball of $f$ in
$L^2$ at any one
time at least $T$ necessarily satisfies $T^* \left( \tilde{f} \right) = \infty$ and $\tilde{f}$ scatters.
But now we see that (\ref{item:third}) implies
 that
 \[
 \exists\left( \tilde{t} \in \left[ T, T+ 1\right]\right) \;\;
  \exists \left( M \in \mathbb{N} \right) \;\;
 \forall\left( m>M\right) \;\;
 \left\Vert f_{n_m} \left( \tilde{t} \right) - 
 f \left( \tilde{t} \right) \right\Vert_{L^2} < \varepsilon
 \]
so for all $m > M$ we have that
$T^* \left( f_{n_m} \right) = \infty$ and $f_{n_m}$ scatters, which contradicts (\ref{item:fourth}).
\end{proof}

\section{Weak-strong uniqueness}
\label{sec:WkStrong}

\subsection{Propagation of weighted estimates.}

We know by now that ($*$)-solutions exist for any non-negative $f_0 \in L^2 \bigcap L^1_2$. However,
if $f_0$ is chosen from a more restrictive functional space, then we can say more. We begin with the
gain-only equation, then we upgrade the result to the full Boltzmann equation.

\begin{lemma}
\label{lem:MomentPropGO}
Let $0 < \alpha < \infty$.
Assume $f_0$ is such that 
\[
\left< v \right>^\alpha f_0 \in L^2
\]
and let
\[
0 < T < T_{\textnormal{g.o.}} \left(f_0\right)
\]
Then the solution $h\left(t\right)$ 
of the gain-only
Boltzmann equation with initial data $f_0$ i.e.
\[
h \left( t \right) = \mathfrak{Z}_{\textnormal{g.o.}} \left( f_0 \right) \left( t \right)
\]
 satisfies
\[
\left< v \right>^\alpha h \in L^\infty \left( [0,T], L^2 \right)
\quad \textnormal{and} \quad
\left<v \right>^\alpha Q^+ \left(h,h\right) \in L^1 \left( [0,T], L^2 \right) 
\]
\end{lemma}
\begin{proof}
Fixing $0 < T < T_{\textnormal{g.o.}} \left( f_0 \right)$ with $I = \left[ 0, T \right]$ we may define
\[
C_0 \left( T \right) = \left\Vert h \right\Vert_{L^\infty \left( I, L^2 \right)} + 
\left\Vert Q^+ \left( h,h \right) \right\Vert_{L^1 \left( I, L^2 \right)} < \infty
\]
and observe that $Q^+ \left( h,h \right)$ is exactly $\left( \partial_t + v \cdot \nabla_x \right) h$.
If, as in subsection \ref{ssec:truncatedWeights}, we write
\[
\nu_R^\alpha = \min \left( \left< v \right>^\alpha, R^\alpha \right)
\]
then we have each
\[
\left( \partial_t + v \cdot \nabla_x \right) h = Q^+ \left( h,h \right)
\]
and
\[
\left( \partial_t + v \cdot \nabla_x \right) \left\{
\nu_R^\alpha h \right\} = \nu_R^\alpha  Q^+ \left( h,h \right)
\]

Let us apply Proposition \ref{prop:QplusTimeDecomposition}, viewing $g$ as $h$ and $h$ as
$\nu_R^\alpha h$, to deduce the existence of a finite partition $I = \bigcup_j I_j$, 
$I_j = \left[ t_j, t_{j+1} \right]$, such that
\begin{equation*}
\begin{aligned}
& \left\Vert Q^+ \left( \nu_R^\alpha h, h \right) \right\Vert_{L^1 \left( I_j, L^2 \right)} +
\left\Vert Q^+ \left( h, \nu_R^\alpha h \right) \right\Vert_{L^1 \left( I_j, L^2 \right)} \\
& \qquad \qquad \leq C_1 C_0 \left( T \right) \times \left(
\left\Vert \nu_R^\alpha h \left( t_j \right) \right\Vert_{L^2} + \varepsilon
\left\Vert \nu_R^\alpha Q^+ \left( h,h \right) \right\Vert_{L^1 \left( I_j, L^2 \right)} \right)
\end{aligned}
\end{equation*}
where we label $C_1$ to fix the constant once and for all. Now according to Proposition
\ref{prop:QplusTimeDecomposition}, the partition depends on $h$ but not on
$\nu_R^\alpha h$; this may seem paradoxical since the pointwise quotient of these two is the known
function $\nu_R^\alpha$, but what it \emph{really} means in this context is that the partition does
not depend on the parameters $\alpha,R$. Crucially, $\nu^\alpha_R$ is bounded above by $R^\alpha$ so we
know that
\[
\nu_R^\alpha h \in L^\infty \left( I, L^2 \right) \;\; \textnormal{ and } \;\;
\nu_R^\alpha Q^+ \left( h,h \right) \in L^1 \left( I, L^2 \right)
\]
Also, as in the discussion of subsection \ref{ssec:truncatedWeights}, we may write 
\begin{equation*}
\begin{aligned}
& \left\Vert \nu_R^\alpha Q^+ \left( h,h \right) \right\Vert_{L^1 \left( I_j, L^2 \right)} \\
& \qquad\quad \leq 2^{2 + \frac{\alpha}{2}} \left(
\left\Vert Q^+ \left( \nu_R^\alpha h, h \right) \right\Vert_{L^1 \left( I_j, L^2 \right)} +
\left\Vert Q^+ \left( h, \nu_R^\alpha h \right) \right\Vert_{L^1 \left( I_j, L^2 \right)}\right)
\end{aligned}
\end{equation*}
therefore
\begin{equation*}
\begin{aligned}
& \left\Vert \nu_R^\alpha Q^+ \left( h,h \right) \right\Vert_{L^1 \left( I_j, L^2 \right)} \\
& \qquad \quad \leq 2^{2 + \frac{\alpha}{2}} C_1 C_0 \left( T \right) \times \left(
\left\Vert \nu_R^\alpha h \left( t_j \right) \right\Vert_{L^2} + \varepsilon
\left\Vert \nu_R^\alpha Q^+ \left( h,h \right) \right\Vert_{L^1 \left( I_j, L^2 \right)} \right)
\end{aligned}
\end{equation*}

Let us assume that
\[
2^{2 + \frac{\alpha}{2}} C_1 C_0 \left( T \right)  \varepsilon = \frac{1}{2}
\]
so that
\begin{equation*}
\begin{aligned}
 \left\Vert \nu_R^\alpha Q^+ \left( h,h \right) \right\Vert_{L^1 \left( I_j, L^2 \right)} 
\leq 2^{3 + \frac{\alpha}{2}} C_1 C_0 \left( T \right)
\left\Vert \nu_R^\alpha h \left( t_j \right) \right\Vert_{L^2} 
\end{aligned}
\end{equation*}
On the other hand,
\[
\left\Vert \nu_R^\alpha h \left( t_j \right) \right\Vert_{L^2}  \leq 
\left\Vert \nu_R^\alpha f_0 \right\Vert_{L^2} + \sum_{i=0}^{j-1}
\left\Vert \nu_R^\alpha Q^+ \left( h,h \right) \right\Vert_{L^1 \left( I_i, L^2 \right)}
\]
Therefore
\begin{equation*}
\begin{aligned}
 & \left\Vert \nu_R^\alpha Q^+ \left( h,h \right) \right\Vert_{L^1 \left( I_j, L^2 \right)}  \\
 &   \qquad \leq  2^{3 + \frac{\alpha}{2}} C_1 C_0 \left( T \right) \left(
\left\Vert \nu_R^\alpha f_0 \right\Vert_{L^2} + \sum_{i=0}^{j-1}
\left\Vert \nu_R^\alpha Q^+ \left( h,h \right) \right\Vert_{L^1 \left( I_i, L^2 \right)}
\right)
\end{aligned}
\end{equation*}

We conclude by a finite induction in $j$. Indeed, suppose that
\[
\left< v \right>^\alpha Q^+ \left( h,h \right) \in
\bigcap_{i=0}^{j-1} L^1 \left( I_i, L^2 \right)
\]
then we have
\begin{equation*}
\begin{aligned}
 & \left\Vert \nu_R^\alpha Q^+ \left( h,h \right) \right\Vert_{L^1 \left( I_j, L^2 \right)}  \\
 &   \qquad \leq  2^{3 + \frac{\alpha}{2}} C_1 C_0 \left( T \right) \left(
\left\Vert \left< v \right>^\alpha f_0 \right\Vert_{L^2} + \sum_{i=0}^{j-1}
\left\Vert \left< v \right>^\alpha Q^+ \left( h,h \right) \right\Vert_{L^1 \left( I_i, L^2 \right)}
\right)
\end{aligned}
\end{equation*}
therefore by monotone convergence in $R$ as $R \rightarrow \infty$ it follows
\[
\left< v \right>^\alpha Q^+ \left(h,h\right) \in
\bigcap_{i=0}^{j} L^1 \left( I_i, L^2 \right)
\]
so we finally obtain
\[
\left< v \right>^\alpha Q^+ \left( h,h \right) \in
L^1 \left( I, L^2 \right)
\]
which in turn implies
\[
\left< v \right>^\alpha h \in L^\infty \left( I, L^2 \right)
\]
since $\left< v \right>^\alpha f_0 \in L^2$.
\end{proof}

\begin{proposition}
\label{prop:propagationL2Velocity}
Let $\alpha > 0$. Assume $f$ is a ($*$)-solution of (\ref{eq:BE}) with initial
data $0 \leq f_0 \in L^2 \bigcap L^1_2$ such that
\[
\left< v \right>^\alpha f_0 \in L^2
\]
Then for any compact sub-interval $J \subset I^* \left( f \right)$, 
\[
\left< v \right>^\alpha f \in L^\infty \left( J, L^2 \right) \quad \textnormal{ and }
\quad \left< v \right>^\alpha Q^+ \left( f,f \right) \in L^1 \left( J, L^2 \right)
\]
\end{proposition}

\begin{remark}
Note carefully that Proposition \ref{prop:propagationL2Velocity} neither requires uniqueness, nor
does the proof imply uniqueness. All it says is that if the initial data satisfies a certain $L^2$-based weighted estimate, then
any ($*$)-solution $f$ corresponding to $f_0$ enjoys the same estimate on compact subintervals
of $I^* \left( f \right)$.
\end{remark}

\begin{proof}
Let $T$ be any real number such that
\[
0 < T < T^* \left( f \right)
\]
Since $f \in C \left( \left[ 0, T \right], L^2 \right)$, by lower semi-continuity of
$T_{\textnormal{g.o.}}$ we may pick $r$ with
\[
0 < r < \inf_{t \in \left[ 0, T \right]} T_{\textnormal{g.o.}} \left( f \left( t \right) \right)
\]
We may assume without loss of generality that
\[
T = k r
\]
for some $k \in \mathbb{N}$. Let us define, for $j = 0, 1, 2, \dots, k-1$,
\[
I_j = \left[ j r, \left( j+1 \right) r \right]
\]
Denote by $P_j$ the statement
\[
\left< v \right>^\alpha f \in L^\infty \left( I_j, L^2 \right) \quad \textnormal{ and }
\quad \left< v \right>^\alpha Q^+ \left( f,f \right) \in L^1 \left( I_j, L^2 \right)
\]

Combining Lemma \ref{lem:MomentPropGO} with Proposition \ref{prop:comparisonStar} and the assumption
\[
\left< v \right>^\alpha f_0 \in L^2
\]
immediately
lets us conclude $P_0$.  Similarly, if 
\[
P_0, P_1, P_2, \dots, P_{\ell - 1}
\]
 all hold, then Lemma \ref{lem:MomentPropGO} combined with Proposition \ref{prop:comparisonStar} 
 imply $P_\ell$.
\end{proof}

\subsection{Weak-strong uniqueness.} Uniqueness holds in the ($*$)-solution class assuming the existence
of a classical solution, up to the time $T^* \left( f \right)$ where continuity breaks down.
 More precisely, we have the following:

\begin{theorem}
\label{thm:unique}
Let $f$ be a ($*$)-solution of (\ref{eq:BE}), corresponding to some initial data
$0 \leq f_0 \in L^2 \bigcap L^1_2$. Furthermore, assume that
\[
\left< v \right>^2 f_0 \in L^2
\]
and also assume that
\[
\forall \left( 0 < T < T^* \left( f \right) \right) \quad
\left< v \right>^2 f \in
L^2 \left( \left[ 0, T \right], L^\infty_x L^2_v \left(
\mathbb{R}^2 \times \mathbb{R}^2 \right)\right)
\]
Then the following uniqueness holds in the class of ($*$)-solutions:

For any ($*$)-solution $h$ of (\ref{eq:BE}), corresponding to the same $f_0$, it holds
\[
T^* \left( h \right) = T^* \left( f \right)
\]
and for almost every $\left( t,x,v \right) \in I^* \left( f \right) \times \mathbb{R}^2 \times
\mathbb{R}^2$, 
\[
h \left( t,x,v \right) = f \left( t,x,v \right)
\]
There is no claim of uniqueness for $t > T^* \left( f \right)$.
\end{theorem}

\begin{remark}
Note carefully that Theorem \ref{thm:unique} does \emph{not} address uniqueness in the class of
renormalized solutions. That is, even on $I^* \left( f \right)$, we \emph{do not} exclude (by this
argument) the possibility
that there exist renormalized solutions for the initial data $f_0$ that do not coincide with $f$, regardless of
the particular bounds we have assumed for $f$ alone. From the proof below,
we can \emph{only} say that any such renormalized solution
does not possess an $L^1 \left( J, L^2 \right)$ bound for $Q^+ \left( h \right)$ on compact
subintervals $J \subset I^* \left( f \right)$. That is, precisely as written, uniqueness is only shown to hold
in the class of ($*$)-solutions, and only on $I^* \left( f \right)$.
\end{remark}

\begin{proof}
The proof is a standard Gronwall-type argument on the difference equation (and relying, in particular,
on the non-negativity of $f,h$). Let us define
\[
w = h - f
\]
and let $T$ be such that
\[
0 < T < \min \left( T^* \left( f \right), T^* \left( h \right) \right)
\]
and denote $I = \left[ 0, T \right]$. Due to the characterization of breakdown of continuity, namely
Theorem \ref{thm:breakdownRegularity}, it suffices to show that $w \left( t,x,v \right) = 0$ for almost
every $\left( t,x,v \right) \in I \times \mathbb{R}^2 \times \mathbb{R}^2$, whenever $T$ is so chosen.

Clearly $w \in C \left( I, L^2 \right)$ and $w \left( t=0, x, v \right) = 0$ a.e. $\left( x,v \right)$.
Also, by Proposition \ref{prop:propagationL2Velocity} we have
\begin{equation}
\label{eq:fWeightedInteg}
\left< v \right>^2 f \in L^\infty \left( I, L^2 \right) \quad \textnormal{ and }
\quad \left< v \right>^2 Q^+ \left( f,f \right) \in L^1 \left( I, L^2 \right)
\end{equation}
\begin{equation}
\label{eq:hWeightedInteg}
\left< v \right>^2 h \in L^\infty \left( I, L^2 \right) \quad \textnormal{ and }
\quad \left< v \right>^2 Q^+ \left( h,h \right) \in L^1 \left( I, L^2 \right)
\end{equation}
so $w=h-f$ immediately provides
\begin{equation}
\label{eq:differenceL2Bd}
\left< v \right>^2 w \in L^\infty \left( I, L^2 \right)
\end{equation}
We have by Duhamel's formula
\[
\left< v \right>^2 f \left( t \right) \leq \mathcal{T} \left( t \right)
\left(\left<v\right>^2 f_0\right) + \int_0^t
\mathcal{T} \left( t-\tau \right) \left\{
\left< v \right>^2 
Q^+ \left( f , f  \right) \left(\tau\right)
\right\} d\tau
\]
\[
\left< v \right>^2 h \left( t \right) \leq \mathcal{T} \left( t \right)\left(
\left<v\right>^2 f_0\right) + \int_0^t
\mathcal{T} \left( t-\tau \right) \left\{
\left< v \right>^2 
Q^+ \left( h, h \right)\left(\tau\right)
\right\} d\tau
\]
therefore by Lemma \ref{lem:nonnegativeEstimate} we may deduce
\[
Q^+ \left( \left< v \right>^2 f, \left< v \right>^2 h \right), \;\;
Q^+ \left( \left< v \right>^2  h,\left< v \right>^2  f \right) \;\; \in
L^1 \left( I, L^2 \right)
\]
therefore
\[
\left< v \right>^2 Q^+ \left( f, h \right), \;\;
\left< v \right>^2 Q^+ \left( h, f \right) \;\; \in
L^1 \left( I, L^2 \right)
\]
which in turn implies (by expanding $w = h - f$)
\begin{equation}
\label{eq:differenceTimeBd}
\left< v \right>^2 Q^+ \left( w, h \right), \;\;
\left< v \right>^2 Q^+ \left( f, w \right) \;\; \in
L^1 \left( I, L^2 \right)
\end{equation}

Moreover, $w$ satisfies the following \emph{difference equation} in the sense of distributions:
\[
\left( \partial_t + v \cdot \nabla_x \right) w = 
Q^+ \left( w, h \right) + Q^+ \left(f, w \right) - w \rho_h - f \rho_w
\]
We can equivalently write
\begin{equation}
\label{eq:differenceEq}
\left( \partial_t + v \cdot \nabla_x + \rho_h \right) w =
Q^+ \left( w, h \right) + Q^+ \left( f, w \right) - f \rho_w
\end{equation}
and view $\rho_h$ as an integrating factor in Duhamel's formula, precisely as is done in (\ref{eq:BEIF}).
In particular, since
$h \geq 0$ a.e. $\left( t,x,v \right)$, we find that $\rho_h \geq 0$ a.e. $\left( t,x \right)$ so
that, as long as we work purely in mixed Lebesgue spaces (which we will), the term $\rho_h$ is completely harmless
(the fact that the terms on the right of
(\ref{eq:differenceEq}) need not be non-negative is irrelevant: we will be estimating each
in absolute value).

\begin{remark}
Technically we have not shown that $w \rho_h$ is locally integrable. However, it turns out $w \rho_h$ is, indeed,
locally integrable: this is because the estimates to follow indirectly imply that
$f \rho_w$ is locally integrable, and we may write
\[
w \rho_h  = h \rho_h - f \rho_f - f \rho_w
\]
and the first two terms on the right are just the losses $Q^- \left( h,h \right)$ resp.
$Q^- \left( f,f \right)$, which we have already shown to be locally integrable on compact sub-intervals of
$I^* \left( h \right)$ resp. $I^* \left( f \right)$.
\end{remark}

Let us multiply the right-hand side of (\ref{eq:differenceEq}) by $\textnormal{sgn}\left(w \right)$ (as if
to write an energy estimate for $\left| w \right|$) and
decompose into its three terms: namely,
\[
\mathcal{M} = \mathcal{M}_1 + \mathcal{M}_2 - \mathcal{M}_3
\]
where
\[
\mathcal{M}_1 = \textnormal{sgn} \left( w \right) Q^+ \left( w,h \right)
\]
\[
\mathcal{M}_2 = \textnormal{sgn} \left( w \right) Q^+ \left( f,w \right)
\]
and
\[
\mathcal{M}_3 = \textnormal{sgn} \left(w \right) f \rho_w
\]
so that
\[
\left( \partial_t + v \cdot \nabla_x + \rho_h \right) \left| w \right| = \mathcal{M}
\]

Since $w = h - f \in C \left( I, L^2 \right)$, we see that $\left\Vert w \left(t\right)\right\Vert_{L^2}$ is
a continuous function of $t \in I$. Moreover, since $f$ and $h$ coincide when $t=0$, we see that
$w\left( t=0 \right)$ is zero almost everywhere. Let us assume that 
$\left\Vert w \left( t \right) \right\Vert_{L^2}$ is not identically zero for all $t \in I$ and
derive a contradiction. In that case, we can define
\[
t_0 = \inf \left\{ t \in \left[ 0, T \right] \; : \;
\left\Vert w\left( t \right) \right\Vert_{L^2} > 0 \right\}
\]
and observe that $0 \leq t_0 < T$ (the case $t_0 = 0$ being permitted at this stage),
 and $w  = 0$ for $0 \leq t \leq t_0$ due to the time continuity
of $w$ into $L^2$. In particular, $w \left( t=t_0,x,v\right)=0$ a.e. $\left( x,v \right)$.
 To obtain the contradiction, we shall show that $w = 0$ for
$0 \leq t < t_1$ for some $t_1$ strictly larger than $t_0$.

The style of argument is to estimate an integral in terms of itself, the constant being less than one over
any small enough time
 interval: in particular, this type of argument relies on the \emph{finiteness} of the integral, and such
estimates generally imply ``if it is finite, then it is zero.'' 
Therefore, before we begin, it will be useful to establish that 
\begin{equation}
\label{eq:UniquenessEnergyIsFinite}
\left< v \right>^2 \mathcal{M} \in L^1 \left( I, L^2 \right)
\end{equation}
To this end, let us show that
\[
\left< v \right>^2 \mathcal{M}_i \in L^1 \left( I, L^2 \right)
\]
for $i \in \left\{ 1,2,3 \right\}$. For $\mathcal{M}_1$ and $\mathcal{M}_2$, this follows immediately
from (\ref{eq:differenceTimeBd}).
For $\mathcal{M}_3$, we have by H{\" o}lder's inequality
\[
\begin{aligned}
\left\Vert \left< v \right>^2 \mathcal{M}_3 \right\Vert_{L^1 \left( I, L^2 \right)} & \leq
\left\Vert \left< v \right>^2 f \right\Vert_{L^2 \left( I, L^\infty_x L^2_v
\left( \mathbb{R}^2 \times \mathbb{R}^2 \right)
\right)}
\left\Vert \rho_{\left| w \right|} \right\Vert_{L^2 \left( I, L^2_x \left( \mathbb{R}^2 \right)\right)} \\
& \leq C
\left\Vert \left< v \right>^2  f \right\Vert_{L^2 \left( I, L^\infty_x L^2_v
\left( \mathbb{R}^2 \times \mathbb{R}^2 \right)\right)}
\left\Vert \left< v \right>^2 w \right\Vert_{L^2 \left( I, L^2 \right)}
\end{aligned}
\]
where we have used that
\begin{equation}
\label{eq:rhoHolder}
 \rho_{\left| w \right|}
 = \left\Vert w \right\Vert_{L^1_v \left( \mathbb{R}^2 \right)}
  \leq C \left\Vert \left< v \right>^2 w \right\Vert_{L^2_v
 \left( \mathbb{R}^2 \right)}
\end{equation}
We know that $\left< v \right>^2 w \in L^2 \left( I, L^2 \right)$ by
(\ref{eq:differenceL2Bd}) and the compactness of $I$, and it is a hypothesis of the Theorem that
\begin{equation}
\label{eq:fHyp}
\left< v \right>^2 f \in
L^2 \left( I, L^\infty_x L^2_v \left(
\mathbb{R}^2 \times \mathbb{R}^2 \right)\right)
\end{equation}
so we may conclude (\ref{eq:UniquenessEnergyIsFinite}).

By Duhamel's formula with $w \left( t=t_0 \right) = 0$, for $t \in \left[ t_0, T\right]$ we may write
\begin{equation}
\label{eq:DifferenceDuhamel}
\left| w \right| \left( t \right) \leq
\int_{t_0}^t \mathcal{T} \left( t-\tau\right) \left| \mathcal{M} \right| \left( \tau \right) d\tau
\end{equation}
hence, multiplying through by $\left< v \right>^2$ and commuting with the free transport, we have
\begin{equation}
\label{eq:WeightedDifferenceDuhamel}
\left< v \right>^2 \left| w \right| \left( t \right) \leq
\int_{t_0}^t \mathcal{T} \left( t-\tau\right) 
\left\{ \left< v \right>^2 \left| \mathcal{M} \right|  \left( \tau \right) \right\} d\tau
\end{equation}
Therefore, letting $J_\kappa = \left[ t_0, \kappa \right]$ for $\kappa \in \left[ t_0, T \right]$,
\begin{equation}
\label{eq:DifferenceL2}
\left\Vert \left< v \right>^2 w \right\Vert_{L^\infty \left( J_\kappa, L^2 \right)}
\leq \left\Vert \left< v \right>^2 \mathcal{M}  \right\Vert_{L^1 \left( J_\kappa , L^2 \right)} 
\end{equation}
Let us define for $\kappa \in \left[ t_0, T \right]$
\[
e \left(\kappa \right) =
\left\Vert 
\left< v \right>^2
\mathcal{M} \right\Vert_{L^1 \left( J_\kappa ,  L^2 \right)}
\]
so that
\[
e\left( \kappa \right) \leq \sum_{i=1}^3 
\left\Vert 
\left< v \right>^2
\mathcal{M}_i \right\Vert_{L^1 \left( J_\kappa ,  L^2 \right)}
\]
We will show that $e \left( t_1 \right) = 0$ for some $t_1 > t_0$ to conclude the Theorem.

Let us first estimate $\mathcal{M}_3$ since it is the easiest term. Indeed
\[
\left| \mathcal{M}_3 \right| \leq f \rho_{\left| w \right|}
\]
so recalling (\ref{eq:rhoHolder}) and (\ref{eq:DifferenceL2}) we have
\[
\begin{aligned}
& \left\Vert \left< v \right>^2 \mathcal{M}_3 \right\Vert_{L^1 \left( J_\kappa , L^2 \right)} \leq
\left\Vert \left< v \right>^2 f \right\Vert_{L^2 \left( J_\kappa, L^\infty_x L^2_v
\left( \mathbb{R}^2 \times \mathbb{R}^2 \right)\right)}
\left\Vert \rho_{\left| w \right|} \right\Vert_{L^2 \left( J_\kappa, L^2_x \left( \mathbb{R}^2 \right)\right)} \\
& \qquad \qquad \leq C \left\Vert \left< v \right>^2 f \right\Vert_{L^2 \left( I, L^\infty_x L^2_v
\left( \mathbb{R}^2 \times \mathbb{R}^2 \right)\right)}
\left\Vert \left< v \right>^2 w \right\Vert_{L^2 \left( J_\kappa, 
L^2 \right)} \\
& \qquad \qquad \leq C \left( \kappa - t_0 \right)^{\frac{1}{2}}
\left\Vert \left< v \right>^2 f \right\Vert_{L^2 \left( I, L^\infty_x L^2_v
\left( \mathbb{R}^2 \times \mathbb{R}^2 \right)\right)}
\left\Vert \left< v \right>^2 w \right\Vert_{L^\infty \left( J_\kappa, 
L^2 \right)} \\
& \qquad \qquad \leq C
\left( \kappa - t_0 \right)^{\frac{1}{2}}
\left\Vert \left< v \right>^\alpha f \right\Vert_{L^2 \left( I, L^\infty_x L^2_v
\left( \mathbb{R}^2 \times \mathbb{R}^2 \right)\right)}
e \left( \kappa \right) 
\end{aligned}
\]
so if $\left( \kappa - t_0 \right)$ is sufficiently small then by (\ref{eq:fHyp}) we have
\[
\left\Vert \left< v \right>^2 \mathcal{M}_3 \right\Vert_{L^1 \left( J_\kappa , L^2 \right)} \leq
\frac{1}{4} e \left( \kappa \right)
\]

We now turn to $\mathcal{M}_1$ (the estimate for $\mathcal{M}_2$ is similar, by substituting $f$ for $h$).
Let us denote
\[
B = \left\{ \; \zeta_0 \in L^1_{\textnormal{loc}} \left( \mathbb{R}^2 \times \mathbb{R}^2 \right) \; : \;
\left\Vert \zeta_0 \right\Vert_{L^2} \leq 1 \; \right\}
\]
and then let us additionally define for $\kappa \in \left[ t_0, T \right]$
with $J_\kappa = \left[ t_0, \kappa \right]$
\[
q \left( \kappa \right) = 
\sup_{\zeta_0 \in B} 
\left\Vert Q^+ \left( \mathcal{T} \left( t - t_0 \right) \zeta_0 ,
\mathcal{T} \left( t - t_0 \right)\left\{ \left< v \right>^2 h \left( t_0 \right) \right\}
\right) \right\Vert_{L^1 \left( J_\kappa , L^2 \right)}
\]
Then since $\left< v \right>^2 h \left( t_0 \right) 
\in L^2$, by 
Proposition \ref{prop:QplusSmallTime} we have
\begin{equation}
\label{eq:qlim}
\lim_{\kappa \rightarrow t_0^+} q \left( \kappa \right) = 0
\end{equation}
We will apply Lemma \ref{lem:nonnegativeEstimate} to estimate
\[
Q^+ \left( \left< v \right>^2 \left| w \right|, \left< v \right>^2 h \right)
\]
which can only be larger than (a constant times) $\left< v \right>^2 \left| \mathcal{M}_1 \right|$.
Indeed, we know that
\[
\left< v \right>^2 \left| w \right| \left( t \right) \leq
\int_{t_0}^t \mathcal{T} \left( t-\tau\right) 
\left\{ \left< v \right>^2 \left| \mathcal{M} \right|  \left( \tau \right) \right\} d\tau
\]
and also
\begin{equation*}
\begin{aligned}
 \left< v \right>^2 h \left( t \right) & \leq \mathcal{T} \left( t -t_0 \right)\left\{
\left<v\right>^2 h \left( t_0 \right) \right\} \\
& \qquad\qquad + \int_{t_0}^t
\mathcal{T} \left( t-\tau \right) \left\{
\left< v \right>^2 
Q^+ \left( h, h \right)\left(\tau\right)
\right\} d\tau
\end{aligned}
\end{equation*}
in particular $w \left( t=t_0 \right) = 0$. Hence by Lemma \ref{lem:nonnegativeEstimate} we may write,
again with $J_\kappa = \left[ t_0, \kappa \right]$, 
\begin{equation*}
\begin{aligned}
&\left\Vert Q^+ \left( \left< v \right>^2 \left| w \right|, \left< v \right>^2 h
\right) \right\Vert_{L^1 \left( J_\kappa, L^2 \right)} \\
& \qquad \qquad\leq
q \left(\kappa\right) 
\left\Vert \left< v \right>^2 \mathcal{M} \right\Vert_{L^1 \left( J_\kappa, L^2 \right)} \\
& \qquad \qquad \qquad \qquad 
+ C \left\Vert \left< v \right>^2 Q^+ \left( h,h \right) \right\Vert_{L^1 \left( J_\kappa, L^2 \right)}
\left\Vert \left< v \right>^2 \mathcal{M} \right\Vert_{L^1 \left( J_\kappa, L^2 \right)} \\
& \qquad\qquad \leq \left( q \left( \kappa \right) +
C \left\Vert \left< v \right>^2 Q^+ \left( h,h \right) \right\Vert_{L^1 \left( J_\kappa, L^2 \right)}
\right) e \left( \kappa \right)
\end{aligned}
\end{equation*}
Then by (\ref{eq:qlim}) and (\ref{eq:hWeightedInteg}) we have
\[
\lim_{\kappa \rightarrow t_0^+}
\left( q \left( \kappa \right) +
C \left\Vert \left< v \right>^2 Q^+ \left( h,h \right) \right\Vert_{L^1 \left( J_\kappa, L^2 \right)}
\right)
= 0
\]
therefore for $\left( \kappa - t_0 \right)$ sufficiently small it holds
\[
\left\Vert \left< v \right>^2 \mathcal{M}_1
\right\Vert_{L^1 \left( J_\kappa, L^2 \right)}  \leq
\frac{1}{4} e \left( \kappa \right)
\]

Altogether we find that for all $\left( \kappa - t_0 \right)$ sufficiently small it holds
\[
e \left( \kappa \right) \leq \frac{3}{4} e \left( \kappa \right)
\]
and since we know $e \left( \kappa \right) < \infty$ this implies $e \left( \kappa_0 \right) = 0$
for some $\kappa_0 > t_0$, reaching the desired contradiction.
\end{proof}

\subsection{Exclusive scattering.}

Weak-strong uniqueness allows us to establish exclusive scattering simply by proving the existence
of a single scattering ($*$)-solution with sufficient integrability and decay:

\begin{corollary}
\label{cor:endCorollary}
Suppose $0 \leq f_0 \in L^2 \bigcap L^1_2$ is such that 
\[
\left< v \right>^2 f_0 \in L^2
\]
and that there exists a ($*$)-solution $f$ of (\ref{eq:BE}), with initial data $f_0$, such that
\[
T^* \left( f \right) = \infty \textnormal{ and  }f \textnormal{ scatters}
\]
and
\[
\forall \left( T < \infty \right) \quad
\left< v \right>^2 f \in L^2 \left( \left[ 0, T \right], L^\infty_x L^2_v 
\left( \mathbb{R}^2 \times \mathbb{R}^2 \right) \right)
\]
Then $f_0 \in \mathcal{E}$. 
\end{corollary}
\begin{proof}
Since $f$ satisfies the conditions of the weak-strong uniqueness theorem, Theorem \ref{thm:unique},
globally in time, it follows that any ($*$)-solution with initial data $f_0$ coincides with $f$
for all $t \geq 0$. On the other hand, by hypotheses, $f$ is a global scattering ($*$)-solution. Therefore,
every ($*$)-solution with initial data $f_0$ is a global scattering ($*$)-solution (being simply $f$),
so we conclude that $f_0 \in \mathcal{E}$, by the definition of the class $\mathcal{E}$.
\end{proof}

\section{Proof of the main theorem: Part I}
\label{sec:mainTheoremOne}

Let $a,b,c > 0$ and consider the moving Maxwellian distribution
\[
m^{a,b,c} \left( t,x,v \right) = a
\exp \left( - b \left| v \right|^2 - c \left| x - v t \right|^2 \right)
\]
with initial data
\[
m_0^{a,b,c} \left( x,v \right) = a \exp \left(- b \left| v \right|^2
- c \left| x \right|^2 \right)
\]
Clearly, $m^{a,b,c}$ scatters (since it is an exact solution of the free transport equation);
moreover, since $m^{a,b,c} \in C^1 \left( \left[ 0,\infty\right), \mathcal{S}\right)$,
Theorem \ref{thm:unique} implies that any ($*$)-solution corresponding to the
initial data $m_0^{a,b,c}$ is global and coincides with $m^{a,b,c}$. Therefore,
$m_0^{a,b,c}$ is exclusively scattering, i.e.  $m_0^{a,b,c} \in \mathcal{E}$.
Hence, by Theorem \ref{thm:exScOpen}, there exists an $\varepsilon =
\varepsilon \left( a,b,c \right) > 0$ such
that if $f_0 \in X$ and
\begin{equation}
\label{eq:mainThmEqn}
\left\Vert f_0 - m_0^{a,b,c} \right\Vert_{X} < 2 \cdot \varepsilon
\end{equation}
then $f_0 \in \mathcal{E}$; by the definition of the $X$-norm
\[
\left\Vert h_0 \right\Vert_X = \left\Vert h_0 \right\Vert_{L^2} +
\sum_{\varphi \in \left\{ 1, \; v_1, \; v_2, \; \left|v\right|^2,\;
\left|x\right|^2, \; x \cdot v \right\}}
\left| \int_{\mathbb{R}^2 \times \mathbb{R}^2}
\varphi \left(x,v\right) h \left( x,v\right) dx dv \right|
\]
we see that (\ref{eq:mainThmEqn}) follows from our hypotheses
(\ref{eq:momentEpsilon}-\ref{eq:sqaureEpsilon}). On the other hand, given
$f_0 \in \mathcal{E}$, it follows from the definition of $\mathcal{E}$ that any
($*$)-solution of (\ref{eq:BE}) corresponding to initial data $f_0$ is global and
scatters; but by Theorem \ref{thm:starSolnExists}, there does indeed exist such a ($*$)-solution.

\section{Higher regularity}
\label{sec:higherRegularity}

\subsection{Preliminaries.}
\label{ssec:regularityPreliminaries}

We will be using difference quotients in order to establish propagation of regularity on the full (recall, half-open)
interval $I^* \left( f \right)$. This is slightly subtle because we are using $L^1$ in the time variable: this
turns out not to be an issue, as we shall show momentarily. Let us define the translation by $a \in \mathbb{R}$
along the unit vector $\mathbf{e} \in \mathbb{R}^2$ for $h_0 \in L^2$:
\[
\left( \tau_{\mathbf{e}}^a h_0 \right) \left( x,v \right) =
h_0 \left( x + a \mathbf{e}, v \right)
\]
Then we define the finite difference operator for $a \neq 0$
\[
D_{\mathbf{e}}^a = a^{-1} \left( \tau_{\mathbf{e}}^a - I \right)
\]
where $I$ is the identity. Fixing once and for all
an orthonormal basis $\left\{\mathbf{e}_i \right\}_{i=1,2}$ of $\mathbb{R}^2$ we denote
\[
\left| D^a h_0 \right| = \left( \sum_{i} 
\left| D_{\mathbf{e}_i}^a h_0 \right|^2 \right)^{\frac{1}{2}}
\]
and $\left\Vert D^a h_0 \right\Vert_{L^2}$ is then the $L^2$ norm of
$\left| D^a h_0 \right|$. The symbol $\nabla_x$ denotes differentiation in the sense of distributions with
respect to the variable $x \in \mathbb{R}^2$.

For this subsection (specifically the following two lemmas)
we follow the presentation of the book by Evans (\cite{LCE2010} subsection 5.8.2).

\begin{lemma}
\label{lem:diffQuotOne}
For any $h_0 \in L^2$ such that $\nabla_x h_0 \in L^2$, and for any $a \in \mathbb{R} \setminus
\left\{ 0 \right\}$,
\[
\left\Vert D^a h_0 \right\Vert_{L^2} \leq 2^{\frac{1}{2}} \left\Vert \nabla_x h_0 \right\Vert_{L^2}
\]
\end{lemma}
\begin{proof}
We have by the fundamental theorem of calculus
\[
\left( D_{\mathbf{e}_i}^a h_0 \right) \left( x,v \right) =
 \int_0^1 \left( \mathbf{e}_i \cdot \nabla_x h_0 \right) \left( x + a b \mathbf{e}_i, v \right) db
\]
therefore
\[
\left\Vert D_{\mathbf{e}_i}^a h_0 \right\Vert_{L^2} \leq \int_0^1
\left\Vert \tau_{\mathbf{e}_i}^{ab}\left( \mathbf{e}_i \cdot \nabla_x h_0\right) \right\Vert_{L^2} db \leq
\left\Vert \nabla_x h_0 \right\Vert_{L^2}
\]
\end{proof}

\begin{lemma}
\label{lem:diffQuotTwo}
Let $h_{0} \in L^2$ be such that
\[
\liminf_{0 < \left| a \right| \rightarrow 0}
\left\Vert D^a h_0 \right\Vert_{L^2} < \infty
\]
Then $\nabla_x h_0 \in L^2$ and it holds
\[
\left\Vert \nabla_x h_0 \right\Vert_{L^2} \leq 2^{\frac{1}{2}}
\liminf_{0 < \left| a \right| \rightarrow 0}
\left\Vert D^a h_0 \right\Vert_{L^2}
\]
\end{lemma}
\begin{proof}
Let us define
\[
M = \liminf_{0 < \left| a \right| \rightarrow 0}
\left\Vert D^a h_0 \right\Vert_{L^2}
\]
and pick a sequence $a_k \in \mathbb{R} \setminus \left\{ 0 \right\}$ with 
$a_k \rightarrow 0$ such that
\[
\lim_{k \rightarrow \infty} \left\Vert D^{a_k} h_0 \right\Vert_{L^2} = M
\]
Then for $i = 1,2$ it holds
\[
\limsup_{k \rightarrow \infty} \left\Vert D^{a_k}_{\mathbf{e}_i} h_0 \right\Vert_{L^2} \leq M
\]

Hence we can pass to a weak limit along a subsequence $\left\{ a_{k_n} \right\}_n$
\[
D_{\mathbf{e}_i}^{a_{k_n}} h_0 \rightharpoonup u_i \in L^2
\]
and moreover
\[
\left\Vert u_i \right\Vert_{L^2} \leq M
\]
On the other hand, by duality and the dominated convergence theorem, for any smooth and compactly
supported function $\varphi_0$ on $\mathbb{R}^2 \times \mathbb{R}^2$,
\begin{equation*}
\begin{aligned}
\int_{\mathbb{R}^2 \times \mathbb{R}^2} \varphi_0 u_i  dx dv & =
\lim_{n \rightarrow \infty}
\int_{\mathbb{R}^2 \times \mathbb{R}^2} \varphi_0 D_{\mathbf{e}_i}^{a_{k_n}} h_0 dx dv \\ 
& = - \lim_{n \rightarrow \infty} \int_{\mathbb{R}^2 \times \mathbb{R}^2}
h_0 D_{\mathbf{e}_i}^{-a_{k_n}} \varphi_0 dx dv \\
& = - \int_{\mathbb{R}^2 \times \mathbb{R}^2}
h_0 \mathbf{e}_i \cdot \nabla_x \varphi_0 dx dv
\end{aligned}
\end{equation*}
which implies
\[
u_i = \mathbf{e}_i \cdot \nabla_x h_0
\]
\end{proof}

The key is to realize that $L^1$ only occurs in the \emph{time} variable, whereas the difference
quotient only occurs in the \emph{space} variable, and apply Fatou's lemma.

\begin{lemma}
\label{lem:diffQuotThree}
Let $\zeta \in L^1 \left( I, L^2 \right)$ for some interval $I \subset \mathbb{R}$, and further suppose
that
\[
\liminf_{0 < \left| a \right| \rightarrow 0} 
\left\Vert D^a \zeta \right\Vert_{L^1 \left( I, L^2 \right)} < \infty
\]
Then $\nabla_x \zeta \in L^1 \left( I, L^2 \right)$ and it holds
\[
\left\Vert \nabla_x \zeta \right\Vert_{L^1 \left( I, L^2 \right)} \leq
2^{\frac{1}{2}} \liminf_{0 < \left| a \right| \rightarrow 0}
\left\Vert D^a \zeta \right\Vert_{L^1 \left( I, L^2 \right)}
\]
\end{lemma}
\begin{proof}
Since $\zeta \in L^1 \left( I, L^2 \right)$, we have
$\zeta \left( t \right) \in L^2$ for a.e. $t \in I$; we want to apply Lemma
\ref{lem:diffQuotTwo} for almost every such $t$. Let us define
\[
M = \liminf_{0 < \left| a \right| \rightarrow 0} 
\left\Vert D^a \zeta \right\Vert_{L^1 \left( I, L^2 \right)}
\]
and take a sequence $a_k \in
\mathbb{R} \setminus \left\{ 0 \right\}$ such that
\[
\lim_{k \rightarrow \infty}
\left\Vert D^{a_k} \zeta \right\Vert_{L^1 \left( I, L^2 \right)} = M
\]
Then by Fatou's lemma, the quantity
\[
\liminf_{k \rightarrow \infty}
\left\Vert D^{a_k} \zeta \left( t \right) \right\Vert_{L^2}
\]
is finite for a.e. $t \in I$, and we note that
\begin{equation}
\label{eq:liminfUpper}
\liminf_{0 < \left| a \right| \rightarrow 0}
\left\Vert D^a \zeta \left(t \right) \right\Vert_{L^2} \leq
\liminf_{k \rightarrow \infty}
\left\Vert D^{a_k} \zeta \left( t \right) \right\Vert_{L^2}
\end{equation}
 Therefore, since we also have
$\zeta \left( t \right) \in L^2$ for a.e. $t \in I$, by Lemma \ref{lem:diffQuotTwo}, we have
that $\nabla_x \zeta \left( t \right) \in L^2$ for a.e.e $t \in I$

Now we estimate, using Lemma \ref{lem:diffQuotTwo}, followed by (\ref{eq:liminfUpper}) and 
finally Fatou's lemma:
\begin{equation*}
\begin{aligned}
\left\Vert \nabla_x \zeta \right\Vert_{L^1 \left( I, L^2 \right)} & =
\int_I \left\Vert \nabla_x \zeta \left( t \right) \right\Vert_{L^2} dt \\
& \leq 2^{\frac{1}{2}} \int_I \liminf_{0 < \left| a \right| \rightarrow 0}
\left\Vert D^{a} \zeta \left( t \right) \right\Vert_{L^2} dt \\
& \leq 2^{\frac{1}{2}} \int_I
\liminf_{k \rightarrow \infty} \left\Vert D^{a_k} \zeta \left( t \right) \right\Vert_{L^2} dt \\
& \leq 2^{\frac{1}{2}} \liminf_{k \rightarrow \infty} \int_I
\left\Vert D^{a_k} \zeta \left( t \right) \right\Vert_{L^2} dt \\
& = 2^{\frac{1}{2}} \lim_{k \rightarrow \infty}
\left\Vert D^{a_k} \zeta \right\Vert_{L^1 \left( I, L^2 \right)} \\
& = 2^{\frac{1}{2}} M
\end{aligned}
\end{equation*}
\end{proof}

\subsection{The gain-only equation.}

Let us recall the Sobolev norms (\ref{eq:SobolevNorms}) for non-negative real numbers $\alpha,\beta$,
\[
\left\Vert f_0 \right\Vert_{H^{\alpha,\beta}} =
\left\Vert \left< v \right>^\beta \left< \nabla_x \right>^\alpha f_0 \right\Vert_{L^2}
\]
We have already propagated $H^{0,\beta}$ for (\ref{eq:BE}) for any $\beta \geq 0$
 by Proposition \ref{prop:propagationL2Velocity}. The objective of
this sub-section is to propagate $H^{2,2}$ for the \emph{gain-only} equation. Then we
will close out our treatment of regularity by propagating $H^{2,2}$ for the full equation (\ref{eq:BE}) in the subsequent
sub-section, which will turn out to be sufficient to propagate Schwartz regularity and conclude Part II of the main theorem.

Before we begin, let us observe that for some constant $C > 0$ we have the equivalence of norms
\[
C^{-1} \left\Vert f_0 \right\Vert_{H^{1,2}} \leq
\left\Vert f_0 \right\Vert_{H^{0,2}} + \left\Vert \nabla_x f_0
\right\Vert_{H^{0,2}} \leq C \left\Vert f_0 \right\Vert_{H^{1,2}}
\]
Moreover, denoting by the symbol
\[
\mathcal{D}^2_x f_0
\]
the matrix of second-order distributional derivatives of $f_0 \in L^2$ in the $x$ variable only, we have
for some other constant $C > 0$
\[
C^{-1} \left\Vert f_0 \right\Vert_{H^{2,2}} \leq
\left\Vert f_0 \right\Vert_{H^{0,2}} + \left\Vert \mathcal{D}^2_x f_0
\right\Vert_{H^{0,2}} \leq C \left\Vert f_0 \right\Vert_{H^{2,2}}
\]
The propagation proofs for the gain-ony equation
 will be similar to the proof of Lemma \ref{lem:MomentPropGO} and will also
rely on the conclusion of that Proposition.

\begin{lemma}
\label{lem:DerivGO}
Assume $f_0$ is such that 
\[
\left< v \right>^2 \left< \nabla_x \right> f_0 \in L^2
\]
and let
\[
0 < T < T_{\textnormal{g.o.}} \left(f_0\right)
\]
Then the solution $h\left(t\right)$ 
of the gain-only
Boltzmann equation with initial data $f_0$ i.e.
\[
h \left( t \right) = \mathfrak{Z}_{\textnormal{g.o.}} \left( f_0 \right) \left( t \right)
\]
 satisfies
\[
\left< v \right>^2 \left< \nabla_x \right> h \in L^\infty \left( [0,T], L^2 \right)
\]
and
\[
\left<v \right>^2 \left< \nabla_x \right> Q^+ \left(h,h\right) \in L^1 \left( [0,T], L^2 \right) 
\]
\end{lemma}

\begin{proof}
Fixing any $0 < T < T_{\textnormal{g.o.}} \left( f_0 \right)$ with $I = \left[ 0,T \right]$ we may define
\begin{equation}
\label{eq:firstDerivC000}
C_0 \left( T \right) = \left\Vert h \right\Vert_{L^\infty \left( I, H^{0,2} \right)} +
\left\Vert Q^+ \left( h,h \right) \right\Vert_{L^1 \left( I, H^{0,2} \right)}
\end{equation}
which is finite by Lemma \ref{lem:MomentPropGO}. 

Let $\mathbf{e} \in \mathbb{R}^2$ be a unit vector. We have each
\begin{equation}
\label{eq:derivative000}
\left( \partial_t + v \cdot \nabla_x \right) \left\{ \left< v \right>^2 h \right\} =
\left< v \right>^2 Q^+ \left( h,h \right)
\end{equation}
and
\begin{equation}
\label{eq:derivative001}
\begin{aligned}
& \left( \partial_t + v \cdot \nabla_x \right) \left\{ \left< v \right>^2 D_{\mathbf{e}}^a h \right\} \\
& \qquad\qquad =\left< v \right>^2 Q^+ \left( D_{\mathbf{e}}^a h,h \right) +
\tau_{\mathbf{e}}^a \left\{ \left< v \right>^2 Q^+ \left(  h, 
\tau_{\mathbf{e}}^{-a} D_{\mathbf{e}}^a h \right)\right\}
\end{aligned}
\end{equation}
where $D_{\mathbf{e}}^a$ is the finite difference operator which has been previously defined, and we have applied the
product rule to commute $D_{\mathbf{e}}^a$ with $Q^+$. Let us in particular denote
\[
\zeta^a_{\mathbf{e}} = \left( \partial_t + v \cdot \nabla_x \right) \left\{ \left< v \right>^2 D_{\mathbf{e}}^a h \right\}
\]
The key is to apply Proposition \ref{prop:QplusTimeDecomposition}, recalling that the conclusion of the Proposition is
independent of \emph{one} of the two arguments of $Q^+$: this is why it does not bother us that $a$ is a variable, nor
that the right-hand side of (\ref{eq:derivative001}) contains 
$D_{\mathbf{e}}^a h$ \emph{and} $\tau_{\mathbf{e}}^{-a} D_{\mathbf{e}}^a h$. The symbol $g $ in the Proposition will
stand for the present $\left< v \right>^2 h$ (this is why we use $H^{0,2}$ in the definition
(\ref{eq:firstDerivC000}) of $C_0 \left( T \right)$ above),
and we decompose $I = \bigcup_j I_j$, $I_j = \left[ t_j, t_{j+1} \right]$, as in the Proposition, depending on
some $\varepsilon > 0$ to be chosen later.
The claim is that if
\begin{equation}
\label{eq:derivClaimHyp}
Q^+ \left( h,h \right) \in \bigcap_{i=0}^{j-1} L^1 \left( I_i, H^{1,2} \right)
\end{equation}
then
\begin{equation}
\label{eq:derivClaimConclusion}
Q^+ \left( h,h \right) \in \bigcap_{i=0}^{j} L^1 \left( I_i, H^{1,2} \right)
\end{equation}
which allows us to conclude after finitely many inductive iterations. We remark that
\[
\left\Vert h \left( t \right) \right\Vert_{H^{1,2}} \leq
\left\Vert h_0 \right\Vert_{H^{1,2}} + \int_0^t 
\left\Vert Q^+ \left( h,h \right) \left( s \right) \right\Vert_{H^{1,2}} ds
\]
so there is nothing more to show, once the claim is established.

Let us assume (\ref{eq:derivClaimHyp}); we know, in particular, that
\[
h \left( t_j \right) \in H^{1,2}
\]
and we need to show that
\[
Q^+ \left( h,h \right) \in L^1 \left( I_j, H^{1,2} \right)
\]
In fact, since $\zeta_\mathbf{e}^a =
\left< v \right>^2 D_{\mathbf{e}}^a Q^+ \left( h,h \right)$,
by Lemma \ref{lem:diffQuotThree} we only need to show that
\[
\zeta_{\mathbf{e}}^a \in L^1 \left( I_j , L^2 \right)
\]
uniformly in $a \in \mathbb{R} \setminus \left\{ 0 \right\}$ for any unit vector
$\mathbf{e} \in \mathbb{R}^2$. Note carefully that we \emph{already know} this membership \emph{for each $a$}
because $\zeta_{\mathbf{e}}^a$ is just defined by a finite difference; therefore, it is permissible to
estimate $\zeta_{\mathbf{e}}^a$ \emph{in terms of itself}, with a small enough constant, uniformly in $a$.

We proceed by (\ref{eq:derivative001}), noting that the left hand side is just $\zeta_{\mathbf{e}}^a$:
\begin{equation*}
\begin{aligned}
& \left\Vert \zeta_{\mathbf{e}}^a \right\Vert_{L^1 \left( I_j, L^2 \right)} \\
&\qquad  = \left\Vert \left< v \right>^2 Q^+ \left( D_{\mathbf{e}}^a h,h \right)  +
\tau_{\mathbf{e}}^a \left\{ \left< v \right>^2 Q^+ \left(  h, 
\tau_{\mathbf{e}}^{-a} D_{\mathbf{e}}^a h \right)\right\} \right\Vert_{L^1 \left( I_j, L^2 \right)} \\
&\qquad  \leq \left\Vert  Q^+ \left( \left< v \right>^2 
\left| D_{\mathbf{e}}^a h \right|, \; \left< v \right>^2 h \right) 
\right\Vert_{L^1 \left( I_j, L^2 \right)} \\
& \qquad \qquad \qquad +
\left\Vert Q^+ \left( \left< v \right>^2 h, \;
\left< v \right>^2 \tau_{\mathbf{e}}^{-a} 
\left| D_{\mathbf{e}}^a h \right| \right) \right\Vert_{L^1 \left( I_j, L^2 \right)} \\
& \qquad \leq C C_0 \left( T \right) \times \left(
\left\Vert h \left( t_j \right) \right\Vert_{H^{1,2}} + \varepsilon
\left\Vert \zeta_{\mathbf{e}}^a \right\Vert_{L^1 \left( I_j, L^2 \right)} \right)
\end{aligned}
\end{equation*}
We conclude by choosing $\varepsilon$ no larger than $2^{-1} C^{-1} C_0 \left( T \right)^{-1}$.
\end{proof}

The following lemma is similar to Lemma \ref{lem:DerivGO}, both in statement and in proof, and we
only sketch the details.

\begin{lemma}
\label{lem:secondDerivGO}
Assume $f_0$ is such that 
\[
\left< v \right>^2 \left< \nabla_x \right>^2 f_0 \in L^2
\]
and let
\[
0 < T < T_{\textnormal{g.o.}} \left(f_0\right)
\]
Then the solution $h\left(t\right)$ 
of the gain-only
Boltzmann equation with initial data $f_0$ i.e.
\[
h \left( t \right) = \mathfrak{Z}_{\textnormal{g.o.}} \left( f_0 \right) \left( t \right)
\]
 satisfies
\[
\left< v \right>^2 \left< \nabla_x \right>^2 h \in L^\infty \left( [0,T], L^2 \right)
\]
and
\[
\left<v \right>^2 \left< \nabla_x \right>^2 Q^+ \left(h,h\right) \in L^1 \left( [0,T], L^2 \right) 
\]
\end{lemma}

\begin{proof}
Fixing any $0 < T < T_{\textnormal{g.o.}} \left( f_0 \right)$ with $I = \left[ 0, T \right]$ we have
\begin{equation}
\label{eq:HoneBound004}
 \left\Vert h \right\Vert_{L^\infty \left( I, H^{1,2} \right)}
+ \left\Vert Q^+ \left( h,h \right) \right\Vert_{L^1 \left( I, H^{1,2} \right)} < \infty
\end{equation}
which follows from Lemma \ref{lem:DerivGO}.

Let $\mathbf{e},\mathbf{e}^\prime \in \mathbb{R}^2$ be two orthogonal
 unit vectors, and let us denote
\[
u_{\mathbf{e}^\prime} = \mathbf{e}^\prime \cdot \nabla_x h
\]
Then we may write
\begin{equation*}
\begin{aligned}
& \left( \partial_t + v \cdot \nabla_x \right) \left\{
 \left< v \right>^2 D_{\mathbf{e}}^a u_{\mathbf{e}^\prime} \right\} \\
& \quad = \left< v \right>^2 Q^+ \left(
D_{\mathbf{e}}^a u_{\mathbf{e}^\prime} , \; h \right) +
\tau_{\mathbf{e}}^a \left\{ \left< v \right>^2 Q^+ \left( h, \;
\tau_{\mathbf{e}}^{-a} D_{\mathbf{e}}^a u_{\mathbf{e}^\prime} \right) \right\} + F
\end{aligned}
\end{equation*}
where by (\ref{eq:HoneBound004}) it holds
\[
F \in L^1 \left( I, L^2 \right)
\]
The conclusion then follows similarly to the proof of Lemma \ref{lem:DerivGO}.
\end{proof}

\subsection{The full Boltzmann equation.}

\begin{proposition}
\label{prop:regularityFullEqn}
Let $f$ be a ($*$)-solution of (\ref{eq:BE}) with initial data 
\[
0 \leq f \left( t=0 \right) = f_0
\]
Then provided
\[
f_0 \in H^{2,2}
\]
it follows that for each
\[
0 < T < T^* \left( f \right)
\]
it holds
\[
f \in L^\infty \left( \left[ 0, T \right], H^{2,2} \right)
\]
and
\[
Q^{\pm} \left( f,f \right) \in L^2 \left( \left[ 0, T \right], H^{2,2} \right)
\]
\end{proposition}
\begin{remark}
Note carefully that \emph{both} $Q^+$ and $Q^-$ are placed in $H^{2,2}$.
\end{remark}
\begin{proof}
Let us recall, to start, the following bilinear estimate from the previous article \cite{CDP2017}:
for any $h_0, \tilde{h}_0 \in H^{\alpha,\beta}$, with $\alpha,\beta$ each real numbers strictly greater
than $\frac{1}{2} \left( = \frac{d-1}{2} \right)$, it holds
\begin{equation}
\label{eq:oldBilinBound}
\left\Vert Q^{\pm} \left( \mathcal{T} h_0, \mathcal{T} \tilde{h}_0 \right) \right\Vert_{L^2 \left(
\mathbb{R}, H^{\alpha,\beta}\right)} \leq C
\left\Vert h_0 \right\Vert_{H^{\alpha,\beta}} \left\Vert \tilde{h}_0
\right\Vert_{H^{\alpha,\beta}}
\end{equation}
Combining this estimate with the $p=2$ case of
 Lemma \ref{lem:Abd} immediately implies a free upgrade to $L^2$ in time
given a bound $L^1$ in time for any such $\alpha,\beta$: for example,
\begin{equation*}
\begin{aligned}
& \sum_{\mu \in \left\{ \pm \right\}}
\left\Vert Q^\mu \left( f,f \right) \right\Vert_{L^2 \left( \left[ 0, T \right], H^{2,2} \right)} \\
& \qquad \qquad \leq
C \left( \left\Vert f_0 \right\Vert_{H^{2,2}} +\sum_{\mu \in \left\{ \pm \right\}}
\left\Vert Q^\mu \left( f,f \right) \right\Vert_{L^1 \left( \left[ 0, T \right], H^{2,2}\right)}\right)^2
\end{aligned}
\end{equation*}
Therefore we will only concern ourselves with the $L^1$ estimate.

Fix $0 < T < T^* \left( f \right)$, 
and observe that by Proposition \ref{prop:propagationL2Velocity} it holds
\[
C_0 \left( T \right) =
\left\Vert f \right\Vert_{L^\infty \left( \left[ 0, T \right], H^{0,2} \right)} +
\left\Vert Q^+ \left( f,f \right) \right\Vert_{L^1 \left( \left[ 0,T \right], H^{0,2} \right)} < \infty
\]
Moreover, since $f \in C \left( \left[ 0, T \right], L^2 \right)$, we have
\begin{equation}
0 < \inf_{t \in \left[ 0, T \right]} T_{\textnormal{g.o.}} \left( f \left( t \right) \right)
\end{equation}
So let us pick a real number $\eta > 0$ such that
\begin{equation}
0 < \eta < \inf_{t \in \left[ 0, T \right]} T_{\textnormal{g.o.}} \left( f \left( t \right) \right)
\end{equation}
Fixing any $t_0 \in \left[ 0, T \right]$ let us define an interval $I$ \emph{based at $t_0$} via the formula
\[
I = I \left( t_0 \right) = \left[ t_0 , t_0 + \eta \right]
\]
and note that $I$ is guaranteed to be a sub-interval of $I^* \left( f \right)$. We are going to
show that if $t_0 \in \left[ 0, T \right]$ is chosen such that
\[
f \left( t_0 \right) \in H^{2,2}
\]
(which is true for $t_0 = 0$ in any case), then
\[
Q^{\pm} \left( f,f \right) \in L^1 \left( I, H^{2,2} \right)
\]
which, since $f \left( t_0 \right) \in H^{2,2}$, in turn implies
\[
f \in L^\infty \left( I, H^{2,2} \right)
\]
Since $\eta$ is independent of $t_0 \in \left[ 0,T \right]$, we can then conclude
\[
Q^{\pm} \left( f,f \right) \in L^1 \left( \left[ 0,T \right], H^{2,2} \right)
\]
and
\[
f \in L^\infty \left( \left[ 0, T \right], H^{2,2} \right)
\]
which implies the Proposition since $T \in \left( 0, T^* \left( f \right) \right)$ is chosen arbitrarily.

Before we begin, we need to use the gain-only equation. Indeed, since
$f \left( t_0 \right) \in H^{2,2}$, by Lemma \ref{lem:secondDerivGO} we have
\[
\mathfrak{Z}_{\textnormal{g.o.}} \left( f \left( t_0 \right) \right) \in
L^\infty \left( I, H^{2,2} \right)
\]
hence by Sobolev embedding
\[
\left< v \right>^2 \mathfrak{Z}_{\textnormal{g.o.}} \left( f \left( t_0 \right) \right) \in
L^\infty_{t} L^2_v L^\infty_x \left( I \times \mathbb{R}^2 \times \mathbb{R}^2 \right) \subset
L^\infty_{t,x} L^2_v \left( I \times \mathbb{R}^2 \times \mathbb{R}^2 \right)
\]
so by the comparison principle
\[
\left< v \right>^2 f \in
L^\infty_{t,x} L^2_v \left( I \times \mathbb{R}^2 \times \mathbb{R}^2 \right)
\]
thus by H{\" o}lder in $v$
\[
\rho_f \in L^\infty_{t,x} \left( I \times \mathbb{R}^2 \right)
\]
So let us define the real number $B$ by
\[
B = \left\Vert \left< v \right>^2 f \right\Vert_{L^\infty_{t,x} L^2_v 
\left( I \times \mathbb{R}^2 \times \mathbb{R}^2\right)} +
\left\Vert \rho_f \right\Vert_{L^\infty_{t,x} \left( I \times \mathbb{R}^2 \right)}
\]
which we may consider a \emph{constant} for the remainder of the proof.

So let us take $M \in \mathbb{N}$ sufficiently large and $\varepsilon > 0$ sufficiently small to be chosen
later (each
$\varepsilon , M$
possibly depending on each $T,B$),
and apply Corollary \ref{cor:QplusTimeDecompositionSecond} to partition (for some $N \geq M$)
\[
I = \left[ t_0, t_0 + \eta \right] = \bigcup_{j=0}^{N-1} I_j
\]
where $I_j = \left[ t_j, t_{j+1} \right]$ and
\[
t_0 < t_1 < t_2 < \dots < t_{N-1} < t_N = t_0 + \eta
\]
and for each $j$ it holds
\begin{equation}
\label{eq:MbdOne}
\left| t_{j+1} - t_j \right| < \frac{1}{M}
\end{equation}
and additionally the estimates of Proposition \ref{prop:QplusTimeDecomposition} hold with $\varepsilon$
on each $I_j$.

Let us denote by $P^{j,\alpha}$, $\alpha \in \left\{ 1, 2 \right\}$, the statement
\begin{equation*}
\forall \left( 0 \leq i < j \right) \qquad
Q^{\pm} \left( f,f \right) \in 
L^1 \left( I_i, H^{\alpha, 2} \right)
\end{equation*}
and note that
\begin{equation*}
\begin{aligned}
& \left\Vert f \left( t \right) \right\Vert_{H^{\alpha,2}} \\
& \qquad \leq
\left\Vert f_0 \right\Vert_{H^{\alpha,2}} + \int_0^t
\left( \left\Vert Q^+ \left( f,f \right) \left( s \right) \right\Vert_{H^{\alpha,2}} +
\left\Vert Q^- \left( f,f \right) \left( s \right) \right\Vert_{H^{\alpha,2}} \right) ds
\end{aligned}
\end{equation*}
Observe that $P^{0,1}$ and $P^{0,2}$ each trivially hold, since  there is no $i$ with
\[
0 \leq i < 0
\]
We are going to show that, under the hypotheses of the Proposition, 
\[
P^{j, 2} \implies P^{j+1,1}
\]
for each $j$, and 
\[
P^{j,1} + P^{j-1,2} \implies P^{j,2}
\]
for each $j \geq 1$.
The Proposition then follows after finitely many inductive steps.

\underline{$P^{j,2} \implies P^{j+1,1}$}
Since $f \left( t_0 \right) \in H^{2,2}$, we can deduce from $P^{j,2}$ that
\[
f \in L^\infty \left( \left[ t_0 , t_j \right], H^{2,2} \right)
\]
In particular,
\[
f \left( t_j \right) \in H^{2,2}
\]
We need to show that
\[
Q^{\pm} \left( f,f \right) \in L^1 \left( I_j, H^{1,2} \right)
\]
In fact, since $f$ solves (\ref{eq:BE}), it suffices to establish each
\[
Q^+ \left( f,f \right) \in L^1 \left( I_j, H^{1,2} \right)
\]
and
\[
\left( \partial_t + v \cdot \nabla_x \right) f \in L^1 \left( I_j, H^{1,2} \right)
\]
since the difference of these two is $Q^- \left( f,f \right)$.
But in fact the second assertion implies the first
(since $f \left( t_j \right) \in H^{2,2} \subset H^{1,2}$), so we need only show
\[
\left( \partial_t + v \cdot \nabla_x \right) f \in L^1 \left( I_j, H^{1,2} \right)
\]

Let $\mathbf{e} \in \mathbb{R}^2$ be a unit vector and define for $a \in \mathbb{R} \setminus \left\{ 0 \right\}$
\[
\zeta_{\mathbf{e}}^a = \left( \partial_t + v \cdot \nabla_x \right) \left\{
\left< v \right>^2 D_{\mathbf{e}}^a f \right\}
\]
noting that the right-hand side is identical to
\[
\left< v \right>^2 D_{\mathbf{e}}^a Q^+ \left( f,f \right) -
\left< v \right>^2 D_{\mathbf{e}}^a Q^- \left( f,f \right)
\]
We know that
\[
\zeta_{\mathbf{e}}^a \in L^1 \left( I_j, L^2 \right)
\]
and we only prove the uniformity in $a$ of this estimate.

Now let us observe
\begin{equation*}
\begin{aligned}
\zeta_{\mathbf{e}}^a 
 = \left< v \right>^2 Q^+ \left( D_{\mathbf{e}}^a f, f \right) & +
\tau_{\mathbf{e}}^a \left\{ \left< v \right>^2 Q^+ \left( f,
\tau_{\mathbf{e}}^{-a} D_{\mathbf{e}}^a f \right) \right\} \\
&  \qquad \qquad \quad + \left< v \right>^2 D_{\mathbf{e}}^a  Q^- \left( f,f \right)
\end{aligned}
\end{equation*}
so, as in the proof of Lemma \ref{lem:DerivGO}, we have
\begin{equation*}
\begin{aligned}
\left\Vert \zeta_{\mathbf{e}}^a \right\Vert_{L^1 \left( I_j, L^2 \right)}&  \leq
C C_0 \left( T \right) \times \left(
\left\Vert f \left( t_j \right) \right\Vert_{H^{1,2}} + \varepsilon
\left\Vert \zeta_{\mathbf{e}}^a \right\Vert_{L^1 \left( I_j, L^2 \right)} \right) \\
& \qquad \qquad\qquad\qquad \qquad + \left\Vert 
\left< v \right>^2 D_{\mathbf{e}}^a  Q^- \left( f,f \right)
\right\Vert_{L^1 \left( I_j, L^2 \right)}
\end{aligned}
\end{equation*}
so let us estimate the last term.

\begin{equation*}
\begin{aligned}
& \left\Vert \left< v \right>^2 D_{\mathbf{e}}^a  Q^- \left( f,f \right)
\right\Vert_{L^1 \left( I_j, L^2 \right)} \\
& \qquad \leq \frac{1}{M}
\left\Vert \rho_f \cdot \left< v \right>^2 D_{\mathbf{e}}^a f \right\Vert_{L^\infty \left( I_j, L^2 \right)}
+ \frac{1}{M} \left\Vert \left< v \right>^2 f \cdot \rho_{\left| D_{\mathbf{e}}^a f \right|}
\right\Vert_{L^\infty \left( I_j, L^2 \right)} \\
& \qquad \leq \frac{1}{M} \left\Vert \rho_f \right\Vert_{L^\infty_{t,x} \left(
I_j \times \mathbb{R}^2 \right)}
\left\Vert \left< v \right>^2 D_{\mathbf{e}}^a f \right\Vert_{L^\infty \left( I_j, L^2 \right)}\\
& \qquad \qquad \qquad
+ \frac{1}{M} \left\Vert \left< v \right>^2 f 
\right\Vert_{L^\infty_{t,x} L^2_v \left( I_j \times \mathbb{R}^2 \times \mathbb{R}^2 \right)} 
 \left\Vert  \rho_{\left| D_{\mathbf{e}}^a f \right|}
\right\Vert_{L^\infty_t L^2_x \left( I_j \times \mathbb{R}^2 \right)}\\
& \qquad \leq \frac{1}{M} \left\Vert \rho_f \right\Vert_{L^\infty_{t,x} \left(
I_j \times \mathbb{R}^2 \right)}
\left\Vert \left< v \right>^2 D_{\mathbf{e}}^a f \right\Vert_{L^\infty \left( I_j, L^2 \right)}\\
& \qquad \qquad \qquad
+ \frac{1}{M} \left\Vert \left< v \right>^2 f 
\right\Vert_{L^\infty_{t,x} L^2_v \left( I_j \times \mathbb{R}^2 \times \mathbb{R}^2 \right)} 
 \left\Vert \left< v \right>^2 D_{\mathbf{e}}^a f 
\right\Vert_{L^\infty \left( I_j , L^2 \right)}\\
& \qquad \leq B M^{-1}  \left\Vert \left< v \right>^2 D_{\mathbf{e}}^a f 
\right\Vert_{L^\infty \left( I_j , L^2 \right)}\\
& \qquad \leq C B M^{-1} \times \left( \left\Vert f \left( t_j \right) \right\Vert_{H^{1,2}}
+ \left\Vert \zeta_{\mathbf{e}}^a \right\Vert_{L^1 \left( I_j, L^2 \right)}\right)
\end{aligned}
\end{equation*}
Therefore we may write
\begin{equation*}
\begin{aligned}
\left\Vert \zeta_{\mathbf{e}}^a \right\Vert_{L^1 \left( I_j, L^2 \right)}&  \leq
C C_0 \left( T \right) \times \left(
\left\Vert f \left( t_j \right) \right\Vert_{H^{1,2}} + \left(
\varepsilon + B M^{-1} \right)
\left\Vert \zeta_{\mathbf{e}}^a \right\Vert_{L^1 \left( I_j, L^2 \right)} \right) \\
\end{aligned}
\end{equation*}
so the desired implication follows by taking $\varepsilon$ sufficiently small (depending on $T$)
and $M$ sufficiently
large (depending on $B$).

\underline{$P^{j,1} + P^{j-1,2} \implies P^{j,2}$}
Combining $P^{j,1}$ with the $\left(\alpha,\beta\right)=\left(1,2\right)$ case of (\ref{eq:oldBilinBound})
along with the $p=2$ case of Lemma \ref{lem:Abd} immediately implies
\[
Q^{\pm} \left( f,f \right) \in  L^2 \left( \left[ 0, t_j \right], H^{1,2} \right)
\]
This estimate implies, in turn, that $f$ coincides with the known local $H^{1,2}$ solution \cite{CDP2017} 
of (\ref{eq:BE})
on $\left[ 0, t_j \right]$. But, on the other hand, since we have $P^{j-1,2}$, we know
$f \left( t_{j-1} \right) \in H^{2,2}$, so the known theory of propagation of regularity
(\cite{CDP2018}, Theorem 2.3\emph{(i)}) immediately implies
\[
Q^{\pm} \left( f,f \right) \in L^2 \left( I_{j-1}, H^{2,2} \right)
\]
which was what we wanted.
\end{proof}

Known propagation of regularity results allow us to promote $H^{2,2}$ to $\mathcal{S}$, as follows:

\begin{theorem}
\label{thm:schwartzRegularity}
Let $f$ be a distributional solution of (\ref{eq:BE}) on a compact interval $J = \left[ 0, T \right]$, such that
\[
\left\Vert f \right\Vert_{L^\infty \left( J, H^{2,2} \right)} < \infty \quad \textnormal{ and }
\left\Vert Q^{\pm} \left( f, f \right) \right\Vert_{L^1 \left( J, H^{2,2} \right)} < \infty
\]
and $f_0 = f \left( t=0 \right)
 \in \mathcal{S}$. Then $f \in C^1 \left( J, \mathcal{S}\right)$. Moreover, the solution is unique on all of $J$
once its initial value $f_0$ is determined.
\end{theorem}

\begin{proof}
By Theorem 2.3 \emph{(i) and (ii)} of \cite{CDP2018}, we have $f \in L^\infty \left( J, H^{k,k}\right)$ and
$Q^{\pm} \left( f, f \right) \in L^1 \left( J, H^{k,k} \right)$ for any natural number $k$; i.e.,
we propagate all derivatives in $x$ and moments in $v$. These can be traded in for moments in
$x$ and derivatives in $v$ by Theorem 2.2 \emph{(i) and (ii) (respectively)} of \cite{CDP2018}; indeed, since Theorem 2.2 of
\cite{CDP2018} is stated in terms of weights (whereas $H^{k,k}$ is defined purely by differentiation in \cite{CDP2018}
via the Wigner transform), we can also mix any number of moments in $x$ with any
number of derivatives in $v$, in any $H^{k,k}$, by the same theorem (direct analysis also suffices for the mixed case, in view
of the proof of the theorem). Hence $f \left( t \right) \in \mathcal{S}$ for every $t \in J$. Time regularity is proven
in Proposition 2.4 of \cite{CDP2018}, in $H^{k,k}$, for any natural number $k$; time derivatives of 
mixed moments and derivatives likewise
follow as discussed in Remark 2.5 of the same reference. The uniqueness assertion follows, for instance, from Proposition 2.5
of \cite{CDP2018}.
\end{proof}

\begin{theorem}
\label{thm:boltzSchwartz}
Let $f$ by a ($*$)-solution of (\ref{eq:BE}) corresponding to some Schwartz initial data
$0 \leq f_0 \in \mathcal{S}$. Then 
\[
f \in C^1 \left( I^* \left( f \right), \mathcal{S} \right)
\]
 Moreover, for any
($*$)-solution $\tilde{f}$ of (\ref{eq:BE}) corresponding to the same $f_0$, it holds that
$T^* \left( \tilde{f} \right) = T^* \left( f \right)$, and $\tilde{f} = f$ on $I^* \left( f \right)$.

\end{theorem}
\begin{proof}
By Proposition \ref{prop:regularityFullEqn}, $f$ satisfies the conditions of Theorem \ref{thm:schwartzRegularity}
on any compact sub-interval $J \subset I^* \left(f \right)$. (Likewise, Proposition \ref{prop:regularityFullEqn} and 
Theorem \ref{thm:schwartzRegularity} also apply
to any other candidate ($*$)-solution $\tilde{f}$, so the uniqueness again follows from Theorem \ref{thm:schwartzRegularity}).
\end{proof}

\section{Proof of the main theorem: Part II}
\label{sec:mainTheoremTwo}

Let $f \in C \left( \left[ 0, \infty \right), L^2 \right)$ be as in Part I of the
main theorem, corresponding to some $0 \leq f_0 \in \mathcal{S}$
satisfying (\ref{eq:momentEpsilon}) and (\ref{eq:sqaureEpsilon}). Then by
Theorem \ref{thm:boltzSchwartz}, we have 
\[
f \in C^1 \left( I^* \left( f \right), \mathcal{S} \right)
\]
Then since
\[
T^* \left( f \right) = \infty
\]
we have
\[
f \in C^1 \left( \left[ 0, \infty \right), \mathcal{S} \right)
\]
Hence, by Theorem \ref{thm:unique}, if $\tilde{f}$ is any other ($*$)-solution corresponding to
the same initial data $f_0$, we find that $T^* \left( \tilde{f} \right) =
T^* \left( f \right) = \infty$ and $\tilde{f}$ coincides with $f$.

\appendix

\section{Well-posedness for the truncated equation}
\label{app:truncated}

All the content of this appendix can be found in \cite{DPL1989}, Section VIII; we recall the
proof of Theorem \ref{thm:truncated} below for the convenience of the reader.

\subsection{Global well-posedness in $L^1$.}

We will prove the global well-posedness in
\[
C \left( \left[ 0, \infty \right), L^1 \right)
\]
 for the equation
\begin{equation}
\label{eq:truncatedEquationSecond}
\left( \partial_t + v \cdot \nabla_x \right) f_n = \left( 1 + n^{-1} \rho_{\left| f_n \right|}\right)^{-1}
\left\{
Q_{b_n}^+ \left( f_n, f_n \right) - Q_{b_n}^- \left( f_n, f_n \right) \right\}
\end{equation}
and the proof will also imply the (local in time) Lipschitz estimate for the solution map, for any $T > 0$,
\begin{equation}
\label{eq:truncatedLipMap}
\left\Vert f_n  - \tilde{f}_n 
\right\Vert_{L^\infty \left( \left[ 0, T \right], L^1_{x,v} 
\left( \mathbb{R}^2 \times \mathbb{R}^2 \right) \right)} \leq e^{5 n T}
\left\Vert f_{n,0} - \tilde{f}_{n,0} \right\Vert_{L^1_{x,v}
\left( \mathbb{R}^2 \times \mathbb{R}^2 \right)}
\end{equation}
This subsection, in fact, only uses the fact that $b_{n} \in L^\infty$; the remaining subsections
of this appendix will make use of the other technical assumptions on $b_n$.

The proof of global well-posedness
is by a fixed point argument and controlled iteration in time.
Since the collision kernel $b_n$ is bounded pointwise by $\left( 2\pi \right)^{-1}$, by collision invariants
it holds
\[
\left\Vert Q^{\pm}_{b_n} \left(f,h\right) \right\Vert_{L^1_v \left( \mathbb{R}^2 \right)} \leq 
\left\Vert f \right\Vert_{L^1_v \left( \mathbb{R}^2 \right)} \left\Vert h 
\right\Vert_{L^1_v \left( \mathbb{R}^2 \right)}
\]
hence, due to the fact that $\rho_{\left|f\right|}$ is identified with the norm $L^1_v \left( \mathbb{R}^2 \right)$, we have
\[
\left\Vert \frac{Q^{\pm}_{b_n} \left(f,f\right)}{1+n^{-1} \rho_{\left|f_n\right|}}
\right\Vert_{L^1_{x,v} \left( \mathbb{R}^2 \times \mathbb{R}^2\right)} \leq 
n \left\Vert f \right\Vert_{L^1_{x,v} \left( \mathbb{R}^2 \times \mathbb{R}^2\right)}
\]

Next, consider that if $Q_{b_n} = Q_{b_n}^+ - Q_{b_n}^-$ then the quantity
\[
\frac{Q_{b_n} \left(f,f\right)}{1+ n^{-1} \rho_{\left|f\right|}} -
\frac{Q_{b_n} \left(h,h\right)}{1+n^{-1} \rho_{\left|h\right|}}
\]
may be re-written as the sum of
\[
\mathcal{I}_1 = \frac{Q_{b_n} \left(f,f\right)-Q_{b_n} \left(h,h\right)}{\left(1+ n^{-1} \left\Vert f \right\Vert_{L^1_v
\left( \mathbb{R}^2\right)}\right)
\left( 1 + n^{-1} 
\left\Vert h \right\Vert_{L^1_v \left( \mathbb{R}^2 \right)}\right)}
\]
and
\[
\mathcal{I}_2 = \frac{1}{n} \cdot \frac{Q_{b_n} \left(f,f\right)\left\Vert h \right\Vert_{L^1_v
\left( \mathbb{R}^2 \right)}-Q_{b_n} \left(h,h\right)
\left\Vert f \right\Vert_{L^1_v \left(\mathbb{R}^2\right)}}
{\left(1+n^{-1} \left\Vert f \right\Vert_{L^1_v \left( \mathbb{R}^2 \right)}\right)
\left( 1 + n^{-1} \left\Vert h \right\Vert_{L^1_v \left( \mathbb{R}^2 \right) }\right)}
\]
But 
\[
Q_{b_n}  \left(f,f\right) - Q_{b_n} \left(h,h\right) = Q_{b_n} \left(f,f-h\right) + Q_{b_n} \left(f-h,h\right)
\]
each of which is estimated in $L^1_v \left( \mathbb{R}^2 \right)$ (pointwise in $x$) as before,
and then controlled uniformly in $x$ by a factor
in the denominator of $\mathcal{I}_1$; hence,
\[
\left\Vert \mathcal{I}_1 \right\Vert_{L^1_{x,v}
\left( \mathbb{R}^2 \times \mathbb{R}^2 \right)} \leq 2 n \left\Vert f-h \right\Vert_{L^1_{x,v}
\left( \mathbb{R}^2 \times \mathbb{R}^2 \right)}
\]

As for $\mathcal{I}_2$, it is the sum of three terms,
\[
\mathcal{I}_{2,1} = \frac{1}{n} \cdot \frac{Q_{b_n} \left(f-h,f\right)\left\Vert h \right\Vert_{L^1_v \left( \mathbb{R}^2 \right)}}
{\left(1+n^{-1} \left\Vert f \right\Vert_{L^1_v \left( \mathbb{R}^2 \right)}\right)
\left( 1 + n^{-1} \left\Vert h \right\Vert_{L^1_v \left( \mathbb{R}^2 \right)}\right)}
\]
\[
\mathcal{I}_{2,2} = \frac{1}{n} \cdot \frac{Q_{b_n} \left(h,f\right)\left(\left\Vert h \right\Vert_{L^1_v
 \left( \mathbb{R}^2 \right)}
- \left\Vert f \right\Vert_{L^1_v \left( \mathbb{R}^2 \right)}\right)}
{\left(1+n^{-1} \left\Vert f \right\Vert_{L^1_v \left( \mathbb{R}^2 \right)}\right)
\left( 1 + n^{-1} \left\Vert h \right\Vert_{L^1_v \left( \mathbb{R}^2 \right)}\right)}
\]
\[
\mathcal{I}_{2,3} = \frac{1}{n} \cdot \frac{Q_{b_n} \left(h,f-h\right)\left\Vert f \right\Vert_{L^1_v
 \left( \mathbb{R}^2 \right)}}
{\left(1+ n^{-1} \left\Vert f \right\Vert_{L^1_v \left( \mathbb{R}^2 \right)}\right)
\left( 1 + n^{-1} \left\Vert h \right\Vert_{L^1_v \left( \mathbb{R}^2 \right)}\right)}
\]
each of which satisfies as before
\[
\left\Vert \mathcal{I}_{2,i} \right\Vert_{L^1_{x,v}
\left( \mathbb{R}^2 \times \mathbb{R}^2 \right)} \leq  n \left\Vert f-h \right\Vert_{L^1_{x,v}
\left( \mathbb{R}^2 \times \mathbb{R}^2 \right)}
\]

Altogether we have
\[
\left\Vert \frac{Q_{b_n} ^{\pm} \left(f,f\right)}{1+n^{-1} \rho_{\left|f\right|}}
\right\Vert_{L^1_{x,v} \left( \mathbb{R}^2 \times \mathbb{R}^2\right)}
\leq n \left\Vert f \right\Vert_{L^1_{x,v} \left( \mathbb{R}^2 \times \mathbb{R}^2\right)}
\]
\[
\left\Vert \frac{Q_{b_n}  \left(f,f\right)}{1+ n^{-1} \rho_{\left|f\right|}}
-\frac{Q_{b_n}  \left(h,h\right)}{1+n^{-1} \rho_{\left|h\right|}}
\right\Vert_{L^1_{x,v} \left( \mathbb{R}^2 \times \mathbb{R}^2 \right)} 
\leq 5 n \left\Vert f - h\right\Vert_{L^1_{x,v}
\left( \mathbb{R}^2 \times \mathbb{R}^2 \right)}
\]
so using Duhamel's formula and Banach's fixed point theorem we conclude
the existence of a unique local mild solution on a time of order
$\mathcal{O} \left( n^{-1} \right)$ irrespective of $f_0$.
Therefore the equation is
globally well-posed for each $n$ fixed. The Lipschitz estimate (\ref{eq:truncatedLipMap}),
for the solution map,
is immediate.

\subsection{$L^\infty$ bounds.} 

By a change of variables and using our technical \emph{support assumptions}
(\ref{eq:truncSptVel}) \emph{and} (\ref{eq:truncSptDefl}), one can show (estimating
in the velocity variable only):
\[
\left\Vert Q^{\pm}_{b_n} \left( f,f \right)
\right\Vert_{L^\infty_v \left( \mathbb{R}^2 \right)} \leq
C_n \left\Vert f \right\Vert_{L^1_v \left( \mathbb{R}^2 \right)}
\left\Vert f \right\Vert_{L^\infty_v \left( \mathbb{R}^2 \right)}
\]
Since we are \emph{dividing} $Q_{b_n}$ by
\[
1 + n^{-1} \left\Vert f_n \right\Vert_{L^1_v \left( \mathbb{R}^2 \right)}
\]
at each $\left( t,x \right)$, it follows
\[
\left\Vert \frac{Q^{\pm} \left( f_n , f_n \right)}
{1 + n^{-1} \rho_{\left| f_n \right|}}
\right\Vert_{L^\infty_{x,v} \left( \mathbb{R}^2 \times \mathbb{R}^2 \right)} \leq
\tilde{C}_n \left\Vert f_n \right\Vert_{L^\infty_{x,v} \left(\mathbb{R}^2 \times \mathbb{R}^2 \right)}
\]
Therefore, by Gronwall, for each $T > 0$,
\[
f_n \in L^\infty_{t,x,v} \left(
\left[ 0,T\right] \times \mathbb{R}^2 \times \mathbb{R}^2 \right)
\]

\subsection{Gaussian lower bounds.}

For a number $K_n$ to be chosen momentarily, let us define
\[
g_n \left( t,x,v \right) = c_n \exp \left( - K_n t -\frac{1}{2} \left| x - v t \right|^2
- \frac{1}{2} \left| v \right|^2 \right)
\]
where $c_n$ is as in (\ref{eq:dataLowerBd}).  Then it follows
\[
\left( \partial_t + v \cdot \nabla_x + K_n \right)
g_n = 0
\]

For the \emph{loss term} only, we have the estimate at every $\left( t,x,v\right)$,
\[
\frac{Q^- \left( \left| f_n \right|, \left| f_n \right| \right)}
{1 + n^{-1} \rho_{\left| f_n \right|}} \leq K_n
\left| f_n \right|
\]
which defines $K_n$. Therefore, \emph{if we assume} that $f_n$ is everywhere non-negative,
then it follows
\[
\left( \partial_t + v \cdot \nabla_x + K_n \right) f_n \geq Q^+ \left( f_n , f_n \right) \geq 0
\]
hence
\[
\left( \partial_t + v \cdot \nabla_x + K_n \right) \left( f_n - g_n \right) \geq 0
\]
and clearly $f_n - g_n \geq 0$ for $t = 0$. Hence, if the \emph{solution} $f_n$ 
is everywhere non-negative, then we
deduce a \emph{quantitative} lower bound $f_n \geq g_n$, which therefore acts as
an \emph{a priori} estimate (for $n$ fixed), which implies both that $f_n$ is everywhere non-negative
and that $f_n \geq g_n > 0$.

In particular, we can replace $\rho_{\left| f_n \right|}$ by
$\rho_{f_n}$, and the integrand in the instantaneous
entropy dissipation $\mathcal{D} \left( f_n \right)$ is everywhere finite.

\subsection{Collision invariants.}

For any smooth function $\varphi = \varphi \left( t,x,v \right)$
of at most polynomial growth, and any Schwartz function $h = h \left( t,x,v \right)$,
executing a pre-post change of variables on the gain term only, and using the symmetries
of $b_n$ (see e.g. \cite{CIP1994}), it holds
\[
\begin{aligned}
 \int_{\mathbb{R}^2} \varphi Q_{b_n} \left( h, h \right) dv & =
 \int_{\mathbb{R}^2} b_n \varphi \left( h^\prime h_*^\prime - h h_* \right) d\sigma dv \\
& = \int_{\mathbb{R}^2} b_n \left( \varphi^\prime - \varphi \right)
h h_* d\sigma dv \\
& = \frac{1}{2} \int_{\mathbb{R}^2} b_n \left( \varphi^\prime + \varphi_*^\prime -
\varphi - \varphi_* \right) h h_* dv
\end{aligned}
\]
If, at each $\left(t,x\right)$, $\varphi \left( t,x,\cdot \right) \in \textnormal{span} \left\{ 1, v_1, v_2,
\left| v \right|^2 \right\}$ (the implicit constants possibly depending on $\left(t,x\right)$),
then the quantity
\[
\varphi^\prime + \varphi_*^\prime -
\varphi - \varphi_*
\]
is everywhere vanishing (due to conservation of mass, momentum, and kinetic energy across a
collision). Such  functions $\varphi$ are referred to as \emph{collision invariants} (when expressed
 in $v$ only). Thus, assuming that the solution is Schwartz (to be discussed next), we immediately
obtain (\ref{eq:truncatedMass}) by taking $\varphi \equiv 1$, and
(\ref{eq:truncatedMomentum}) by taking separately $\varphi = v_1$ and $\varphi = v_2$ , and
(\ref{eq:truncatedVelMoment}) by taking $\varphi = \left| v \right|^2$.
For example,
\[
\frac{d}{dt} \int_{\mathbb{R}^2 \times \mathbb{R}^2} 
f_n dx dv = \int_{\mathbb{R}^2 \times \mathbb{R}^2} 1 \cdot Q_{b_n} \left( f_n, f_n \right) dx dv = 0
\]
yields (\ref{eq:truncatedMass}).
 Similarly
we obtain (\ref{eq:truncatedPosMoment}) by taking
$\varphi = \left| x - v t \right|^2$, and observing that this function is both an exact solution
of the free transport equation, \emph{and} a linear combination of collision invariants at each
$\left( t,x \right)$. We similarly obtain (\ref{eq:truncatedPosCross}) by taking
$\varphi = \left( x - v t \right) \cdot v$. We obtain (\ref{eq:truncatedEntropyIdentity}) similarly by letting
$\varphi = \log f$ in the above calculation and applying collision symmetries once more
(which replaces $\frac{1}{2}$ by $\frac{1}{4}$ and thereby provides an everywhere non-negative integrand).
Note that since, by the previous subsection,
 $f_n$ is bounded from below by a Gaussian jointly in $\left( x,v \right)$ for
$0 \leq t \leq T$, it follows that the negative part of
 $\log f$ grows at most quadratically, so there is no problem
in justifying the multiplication of the equation by $\log f$.

\subsection{Schwartz class.}

First we show that all moments in $x,v$ are finite, and then that all gradients are
finite, all in $L^1$. We freely make use of the fact that $f_n \in L^1 \bigcap L^\infty$,
and use differential inequalities without careful justification (which is routine).

The moment estimate is
\[
\begin{aligned}
& \frac{d}{dt} \int_{\mathbb{R}^2 \times \mathbb{R}^2} f_n
\left( \left| x \right|^k + \left| v \right|^k \right) dx dv \\
& \lesssim
\int_{\mathbb{R}^2 \times \mathbb{R}^2}
f_n \left| x \right|^{k-1} \left| v \right| dx dv  + \int_{\mathbb{R}^2 \times \mathbb{R}^2}
\left|\frac{Q_{b_n} \left( f_n, f_n \right)}{1 + n^{-1} \rho_{f_n}}\right|
\left| v \right|^k dx dv \\
& \lesssim \int_{\mathbb{R}^2 \times \mathbb{R}^2} 
f_n \left( \left| x \right|^k + \left| v \right|^k \right) dx dv
\end{aligned}
\]
where we have used integration by parts and
 that $\left| x \right|^k$ is a collision invariant in the first
step (as it is constant in $v$),
 and the fact that $f_n \in L^\infty \left( \left[ 0,T \right], L^1 \bigcap
L^\infty \right)$ for each $T > 0$ along with the boundedness and support conditions on $b_n$
 in the second step.

Finally we estimate the first derivatives in $x$; the derivatives in $v$, as well as
all higher derivatives in $x$ and $v$, are similar.
\[
\begin{aligned}
& \frac{d}{dt} \int_{\mathbb{R}^2 \times \mathbb{R}^2} 
\left| \nabla_x f_n \right| dx dv \\
& \lesssim \int_{\mathbb{R}^2 \times \mathbb{R}^2}
\frac{\left| Q_{b_n}^{\pm} \left( \left| \nabla_x f_n \right|, f_n \right) \right|+
\left| Q_{b_n}^{\pm} \left( f_n, \left| \nabla_x f_n \right|\right)\right|}
{1 + n^{-1} \rho_{f_n}} dx dv \\
& \qquad \qquad + \int_{\mathbb{R}^2 \times \mathbb{R}^2}
\frac{\left|Q_{b_n}^{\pm} \left( f_n, f_n \right)\right|}{\left( 1 + n^{-1} \rho_{f_n}\right)^2}
\left\Vert \nabla_x f_n \right\Vert_{L^1_v \left( \mathbb{R}^2 \right)} dx dv \\
& \lesssim \int_{\mathbb{R}^2 \times \mathbb{R}^2}
\left| \nabla_x f_n \right| dx dv
\end{aligned}
\]
Here we have again used that $f_n \in L^\infty \left( \left[ 0,T \right], L^1 \bigcap
L^\infty \right)$ for each $T > 0$ along with the boundedness and
 support assumptions on $b_n$.

\section*{Acknowledgements}

N.P.  gratefully acknowledges support by the NSF through grants DMS-1840314, DMS-2009549, DMS-2052789.

T.C. gratefully acknowledges support by the NSF through grants DMS-1151414 (CAREER), DMS-1716198, and DMS-2009800.

R.D. gratefully acknowledges Professor Hua Xu and the University of Texas Health 
Science Center at Houston for their moral support and encouragement in this ongoing research, 
and also thanks Diogo Arsenio, Irene Gamba, Xuwen Chen, Jin-Cheng Jiang, and Laure Saint-Raymond for numerous 
insightful discussions, as well as Nader Masmoudi, who provided the initial inspiration long ago 
and without whom this work would not have been possible.

\bibliography{ScalingCriticalBib}

\end{document}